\documentclass[12pt,reqno]{amsart}

\usepackage{amssymb,amsmath,amsthm,amscd,ifthen,xr,array}

\author{M. McKee}
\address{Department of Mathematics, University of Iowa, 
Iowa City, IA 52242, USA}
\email{mark-mckee@uiowa.edu}
\author{A. Pasquale}
\address{Institut Elie Cartan de Lorraine (IECL UMR CNRS 7502),
Universit\'e de Lorraine, F-57045 Metz, France}
\email{angela.pasquale@univ-lorraine.fr}

\author{T. Przebinda}
\address{Department of Mathematics, University of Oklahoma, Norman, OK 73019, USA}
\email{przebinda@gmail.com}

\title[Weyl calculus and dual pairs]{Weyl calculus and dual pairs}

\def\ch{\mathop{\hbox{\rm ch}}\nolimits}
\def\n{\mathfrak n}
\def\g{\mathfrak g}
\def\z{\mathfrak z}

\def\h{\mathfrak h}
\def\sp{\mathfrak {sp}}

\def\o{\mathfrak o}

\def\u{\mathfrak u}

\def\R{\mathbb{R}}
\def\C{\mathbb{C}}
\def\Ze{\mathbb{Z}}
\def\Ha{\mathbb{H}}

\def\z{\mathfrak z}
\def\c{\mathfrak c}

\def\so{\mathfrak s_{\overline 0}}
\def\ss1{\mathfrak s_{\overline 1}}

\def\hs1{\mathfrak h_{\overline 1}}

\def\Pg{\mathrm{P}}
\def\supp{\mathrm{supp}}
\def\Op{\mathrm{Op}}
\def\Ker{\mathrm{Ker}}

\def\G{\mathrm{G}}
\def\N{\mathrm{N}}
\def\Zg{\mathrm{Z}}

\def\K{\mathrm{K}}
\def\H{\mathrm{H}}
\def\M{\mathrm{M}}
\def\Zg{\mathrm{Z}}
\def\Sg{\mathrm{S}}

\def\L{\mathrm{L}}
\def\Bbb{\mathbb}

\def\N{\mathrm{N}}

\def\H{\mathrm{H}}
\def\T{\mathrm{T}}
\def\GL{\mathrm{GL}}

\def\SO{\mathrm{SO}}
\def\Sp{\mathrm{Sp}}

\def\Og{\mathrm{O}}
\def\Ug{\mathrm{U}}

\def\Mg{\mathrm{M}}

\newcommand{\anticomm}[2]{\null^{#1}#2}
\newcommand{\danticomm}[2]{\null^{\anticomm{#1}{#2}}#2}

\def \t{\tilde}
\def \wt{\widetilde}
\newcommand{\reg}[1]{ {#1}^{reg}}

\newcommand{\OP}{\mathop{\rm{OP}}}

\def\W{\mathsf{W}}
\def\Wv{\mathrm{W}}
\def\W+{\mathrm{W}_{\BB C}}
\def\Vv{\mathrm{V}}
\def\Uv{\mathrm{U}}
\def\V{\mathsf{V}}

\def\X{\mathsf{X}}

\def\Xv{\mathrm{X}}
\def\Yv{\mathrm{Y}}
\def\Sy{\mathsf{S}}

\def\Dc{\mathbb {D}}
\def\Zb{\mathbb {Z}}

\def\End{\mathop{\hbox{\rm End}}\nolimits}
\def\diag{\mathop{\hbox{\rm diag}}\nolimits}
\def\det{\mathop{\hbox{\rm det}}\nolimits}
\def\ad{\mathop{\hbox{\rm ad}}\nolimits}
\def\Ad{\mathop{\hbox{\rm Ad}}\nolimits}
\def\Hom{\mathop{\hbox{\rm Hom}}\nolimits}
\def\Re{\mathop{\hbox{\rm Re}}\nolimits}

\def\tr{\mathop{\hbox{\rm tr}}\nolimits}

\def\sgn{\mathop{\hbox{\rm sgn}}\nolimits}

\def\vol{\mathop{\hbox{\rm vol}}\nolimits}
\newcommand\diesis[1]{#1^\sharp}

\def\lim{\mathop{\hbox{\rm lim}}\nolimits}

\newcommand\inner[2]{\langle #1,#2\rangle}

\def\supp{\mathop{\hbox{\rm supp}}\nolimits}

\def\ker{\mathop{\hbox{\rm ker}}\nolimits}

\def\Oo{\mathcal{O}}

\def\Ss{\mathcal{S}}

\def\Hs{\mathcal{H}}
\def\SHs{\mathcal{SH}}

\def\fontindex{\arabic}

\def\fonttitre{\textsf}
\newcounter{thh}

\newtheorem{thm}[thh]{\fonttitre{Theorem}}

\newtheorem{pro}[thh]{\fonttitre{Proposition}}
\newtheorem*{pro*}{\fonttitre{Proposition}}
\newtheorem{cor}[thh]{\fonttitre{Corollary}}
\newtheorem*{coro*}{\fonttitre{Corollary}}
\newtheorem{lem}[thh]{\fonttitre{Lemma}}

\newtheorem{defi}[thh]{\fonttitre{Definition}}
\newtheorem*{defi*}{\fonttitre{Définition}}

\newtheorem*{nota*}{\fonttitre{Notation}}
\newenvironment{prf}{\begin{proof}}{\end{proof}}
\def\muet{ \ifthenelse{\equal{a}{b}}}
\def\nn{\nonumber}

\def\biblio{\sloppy
\bibliographystyle{alpha}
\bibliography{article}}

\begin{document}
\thanks{The second author is grateful to the University of Oklahoma for hospitality and financial support. The third author gratefully acknowledges hospitality and financial support from the Universit\'e de Lorraine and partial support from the NSA grant H98230-13-1-0205. }

\date{}

\maketitle
\begin{abstract}
We consider a dual pair $(\G,\G')$, in the sense of Howe, with $\G$ compact acting on $\L^2(\R^n)$ for an appropriate $n$ via the Weil Representation. Let $\wt\G$ be the preimage of $\G$ in the metaplectic group. Given a genuine irreducible unitary representation $\Pi$ of $\wt\G$ we compute the Weyl symbol of orthogonal projection onto $\L^2(\R^n)_\Pi$, the $\Pi$-isotypic component. We apply the result to obtain an explicit formula for the character of the corresponding irreducible unitary representation $\Pi'$ of $\wt\G'$ and to compute of the wave front set of $\Pi'$ by elementary means. 
\end{abstract}
\tableofcontents
\section{\bf Introduction. \rm}\label{Introduction}
Let $\Wv$ be a vector space of finite dimension $2n$ over $\Bbb R$ with a non-degenerate symplectic form $\langle\cdot ,\cdot \rangle$. 
Denote by $\Sp\subseteq \GL(\Wv)$ the symplectic group and 
let $\G, \G'\subseteq \Sp=\Sp(\Wv)$ be an irreducible dual pair. Denote by $\wt\G$, $\wt\G'$ the preimages of $\G$, $\G'$ in the metaplectic group $\wt\Sp=\wt\Sp(\Wv)$. Consider irreducible admissible representations $\Pi$, $\Pi'$ of  $\wt\G$, $\wt\G'$ respectively which are in Howe's correspondence.
By definition $\Pi\otimes \Pi'$ is realized as a quotient of the space of the smooth vectors of the Weil representation $\omega$ of $\wt\Sp(\Wv)$. Hence, as explained in \cite{PrzebindaUnitary}, there is a unique, up to a non-zero constant multiple, $\G\G'$-invariant tempered distribution $f_{\Pi\otimes\Pi'}$ on $\Wv$ such that $\Pi\otimes\Pi'$ is realized on the range of the operator 
\begin{equation}\label{theoperator}
\Op\circ\mathcal K(f_{\Pi\otimes\Pi'}).
\end{equation}
(See \eqref{Op} and \eqref{K} below for the precise definitions of $\Op$ and $\mathcal K$.)
Thus, in principle, all the information about the representation $\Pi\otimes\Pi'$ is encoded in the distribution $f_{\Pi\otimes\Pi'}$. For example the group action is given by
\[
\left(\Pi(\t g)\otimes\Pi'(\t g')\right)\circ \left(\Op\circ\mathcal K(f_{\Pi\otimes\Pi'})\right)=\left(\Op\circ\mathcal K(f_{\Pi\otimes\Pi'})\right)\circ \omega(\t g \t g') \qquad (\t g\in \wt\G, \t g'\in \wt\G').
\]
This is why $f_{\Pi\otimes\Pi'}$ is called an intertwining distribution \cite{PrzebindaUnitary}. In fact $f_{\Pi\otimes\Pi'}$ happens to be the Weyl symbol, \cite{Hormander}, of the operator \eqref{theoperator}, see \eqref{theweylsymbol1} below. 

Often $f_{\Pi\otimes\Pi'}$ may be computed in terms of the distribution character $\Theta_\Pi$ of $\Pi$, \cite{PrzebindaUnitary}. If the group $\G$ is compact then the 
 distribution character $\Theta_{\Pi'}$ may also be recovered from $f_{\Pi\otimes\Pi'}$ via an explicit formula, \eqref{pullbackThetaPiprime}, \cite{PrzebindaUnipotent}. Thus we have a diagram
\begin{equation}\label{first diagram}
\Theta_\Pi\longrightarrow f_{\Pi\otimes\Pi'} \longrightarrow \Theta_{\Pi'}.
\end{equation}
The asymptotic properties of $f_{\Pi\otimes\Pi'}$ determine the associated varieties of the primitive ideals of $\Pi$ and $\Pi'$ and, under some more assumptions, the wave front sets of these representations, see \cite{PrzebindaUnitary} and \cite{PrzebindaUnipotent}. 

We believe that in general one should be able to have a diagram like \eqref{first diagram} with the arrows in arbitrary direction. In particular deciding whether two representations are in Howe's correspondence should be done by comparing the intertwining distributions obtained from their characters. 

The usual, often very successful, approach to Howe's correspondence avoids any work with the distributions on the symplectic space. Instead, one finds Langlands parameters (see \cite{Moeglinarchi}, \cite{AdamsBarbaschcomplex}, \cite{AnnegretunitaryI}, \cite{AnnegretunitaryII}, 
\cite{Annegretorthosymplectic}, \cite{Jian-ShuLi-Cheng-boZhu}), character formulas (see \cite{Adamslift}, \cite{RenardLift}, \cite{DaszPrzebindaInv}), or candidates for character formulas (as in \cite{BerPrzeCHC_inv_eig}), or one establishes preservation of unitarity (as in \cite{Jian-ShuLiSingular}, \cite{HeHongyu}, \cite{PrzebindaUnitary} or \cite{AdamsBarbaschPaulTrapaVogan}). However, in the background (explicit or not), there is the orbit correspondence induced by the moment maps
\[
\g^*\longleftarrow \Wv \longrightarrow \g'{}^*,
\]
see \eqref{tau'}. This correspondence of orbits has been studied in \cite{DaszKrasPrzebindaComplex}, \cite{DaszKrasPrzebindaK-S2} and \cite{PanReal}. 
Furthermore, in their recent work, \cite{LockMaassocvar}, H. Y. Lock and J.J. Ma computed the associated variety of the representations for the dual pairs in the stable range in terms of the orbit correspondence. The $p$-adic case was studied in detail in \cite{Moeglinarchiwave}.
Moreover, still in the stable range, R. Gomez and  C. Zhu computed their generalized Whittaker models, \cite{gomezchengbozhu}. 

Needless to say, working with the $\G\G'$-invariant distributions on $\Wv$ is a more direct approach than relying on the orbit correspondence.
In this paper we consider the dual pairs with $\G$ compact.  Then the representations $\Pi$ and  $\Pi'$ are the irreducible unitary highest weight representations. They have been defined by Harish-Chandra in \cite{HC1955c} and were classified in \cite{EnrightHoweWallach}.  They have been studied in terms of Zuckerman functors in \cite{Wallachholomorphic}, \cite{Adamsdiscrete} and \cite{Adamshighestweight}.

As a complementary contribution to all this work, we compute the intertwining distributions $f_{\Pi\otimes\Pi'}$ explicitly. Our formula for the intertwining distribution $f_{\Pi\otimes\Pi'}$ is explicit enough to find its asymptotics, see Theorem \ref{the dilation limit of intertwining distribution}. These allow us to compute the wave front set of the representation $\Pi'$ within the Classical Invariant Theory, without using \cite{VoganGelfand}. See Corollary \ref{WF of Pi'} below. Also, in the case when both groups are compact, we have the diagram \eqref{first diagram} with the arrows in arbitrary direction. Therefore we see which representations $\Pi$ and $\Pi'$ occur in Howe's correspondence, which corresponds to which and we recover the fact that $\Pi\otimes\Pi'$ occurs with multiplicity one without using \cite{HoweRemarks} or \cite{WeylBook}. This is a stepping stone for understanding the more general situation.

In order to describe more precisely the content of this paper we need to introduce some notation.

Denote by $\sp$ the Lie algebra of $\Sp$. 
Fix a compatible positive complex structure $J$ on $\Wv$. Hence $J\in \sp$ is such that $J^2=-1$ (minus the identity)  and the symmetric bilinear form $\langle J\cdot ,\cdot \rangle$ is positive definite on $\Wv$.
Let $dw$ be the Lebesgue measure on $\Wv$ such that the volume of the unit cube with respect to this form is $1$. (Since all positive complex structures are conjugate by elements of $\Sp$, this normalization does not depend on the particular choice of $J$.) Let $\Wv=\Xv\oplus \Yv$ be a complete polarization. We normalize the Lebesgue measures on $\Xv$ and on $\Yv$ similarly.

Each element $K\in \Ss^*(\Xv\times \Xv)$ defines an operator
$\Op(K)\in \Hom(\Ss(\Xv),\Ss^*(\Xv))$ by
\begin{equation}\label{Op}
\Op(K)v(x)=\int_\Xv K(x,x')v(x')\,dx'.
\end{equation}
Here $\Ss(\Vv)$ and $\Ss^*(\Vv)$ denote the Schwartz space on the vector space $\Vv$ and the space of tempered distributions on $\Vv$,  respectively. 

The map 
$
\Op: \Ss^*(\Xv\times \Xv)\to \Hom(\Ss(\Xv),\Ss^*(\Xv))
$
is an isomorphism of linear topological spaces. This is known as the Schwartz Kernel Theorem, \cite[Theorem 5.2.1]{Hormander}. 

Fix the unitary character $\chi(r)=e^{2\pi i r}$, $r\in \mathbb R$, and
recall the Weyl transform $\mathcal K:\Ss^*(\Wv)\to \Ss^*(\Xv\times \Xv)$ given for $f \in \Ss(\Wv)$ by 
\begin{equation}\label{K}
\mathcal K(f)(x,x')=\int_\Yv f(x-x'+y)\chi\big(\frac{1}{2}\langle y, x+x'\rangle\big)\,dy.
\end{equation}
The Weyl symbol of the operator $\Op\circ \mathcal K (f)$  is the symplectic Fourier transform $\widehat f$ of $f$, defined by
\begin{equation}\label{theweylsymbol1}
\widehat f(w')=2^{-n}\int_\Wv f(w)\chi(\tfrac{1}{2}\langle w, w'\rangle)\,dw \qquad (w'\in \Wv).
\end{equation}
A theorem of Calderon and Vaillancourt asserts that the operator $\Op\circ\mathcal K(f)$ is bounded on $\L^2(\Xv)$ if its Weyl symbol and all its derivatives are bounded functions on $\Wv$, \cite[Theorem 3.1.3]{HoweQuantum}. The intertwining distributions we compute are
 Weyl symbols of some bounded operators which naturally come from the Representation Theory of Real Reductive Groups. Many of these  symbols turn out to be singular distributions. In order to introduce them we recall the Weil representation. 

For an element $g\in \Sp$, let $J_g=J^{-1}(g-1)$. Then its adjoint with respect to the form $\langle J \cdot,\cdot\rangle$ is $J_g^*=Jg^{-1}(1-g)$. In particular  $J_g$ and $J_g^*$ have the same kernel. Hence the image of $J_g$ is
\[
J_g\Wv=(\Ker J_g^*)^\perp=(\Ker J_g)^\perp
\]
where $\perp$ denotes the orthogonal with respect to $\langle J \cdot,\cdot\rangle$.
Therefore, the restriction of $J_g$ to $J_g\Wv$ defines an invertible element. Thus it makes sense to talk about $\det(J_g)_{J_g\Wv}^{-1}$, the reciprocal of the determinant of the restriction of $J_g$ to $J_g\Wv$. 
Let
\begin{equation}\label{metaplectic group}
\wt{\Sp}=\{\t g=(g,\xi)\in \Sp\times \C,\ \ \xi^2=i^{\dim (g-1)\Wv}\det(J_g)_{J_g\Wv}^{-1}\}\,.
\end{equation}
Then there exists a $2$-cocycle $C:\Sp\times \Sp \to \C$, so that $\wt{\Sp}$ is a group, the metaplectic group, with respect to the multiplication 
\begin{equation}\label{multiplicationSp}
(g_1,\xi_1)(g_2,\xi_2)=(g_1g_2,\xi_1\xi_2 C(g_1,g_2))\,.
\end{equation} 
In fact, by \cite[Lemma 52]{AubertPrzebinda_omega},
\begin{equation}\label{|cocycle|}
|C(g_1,g_2)|=\sqrt{\left|\frac{\det(J_{g_1})_{J_{g_1}\Wv}\det(J_{g_2})_{J_{g_2}\Wv}}{\det(J_{g_1g_2})_{J_{g_1g_2}\Wv}}\right|}
\end{equation} 
and by \cite[Proposition 46 and formula (109)]{AubertPrzebinda_omega},
\begin{equation}\label{cocycle/|cocycle|}
\frac{C(g_1,g_2)}{|C(g_1,g_2)|}=\chi(\frac{1}{8}\sgn(q_{g_1,g_2})),
\end{equation} 
where $\sgn(q_{g_1,g_2})$ is the signature of the symmetric form 
\begin{eqnarray}\label{gamag1g2}
q_{g_1,g_2}(u',u'')&=&\frac{1}{2}\langle (g_1+1)(g_1-1)^{-1}u',u''\rangle+\frac{1}{2}\langle (g_2+1)(g_2-1)^{-1}u',u''\rangle\\ 
&&(u',u''\in (g_1-1)\Wv\cap (g_2-1)\Wv.\nn
\end{eqnarray} 
By the signature of a (possibly degenerate) symmetric form we understand the difference between the maximal dimension of a subspace where the form is positive definite and the maximal dimension of a subspace where the form is negative definite.

Let
\begin{equation} \label{eq:chicg}
\chi_{c(g)}(u)=\chi\big(\frac{1}{4}\langle (g+1)(g-1)^{-1}u, u\rangle\big) \qquad (u=(g-1)w,\ w\in\Wv).
\end{equation}
(In particular, if $g-1$ is invertible on $\Wv$, then 
$
\chi_{c(g)}(u)=\chi(\frac{1}{4}\langle c(g)u, u\rangle 
$ where $c(g)=(g+1)(g-1)^{-1}$ is the usual Cayley transform.)
For $\t g=(g,\xi)\in\wt\Sp$ define
\begin{equation}\label{the omega}
\Theta(\t g)=\xi,\qquad T(\t g)=\Theta(\t g)\chi_{c(g)}\mu_{(g-1)\Wv},\qquad \omega(\t g)=\Op\circ \mathcal K \circ T(\t g),
\end{equation}
where $\mu_{(g-1)\Wv}$ is the Lebesgue measure on the subspace $(g-1)\Wv$ normalized so that the volume of the unit cube with respect to the form $\langle J \cdot,\cdot\rangle$ is $1$. In these terms, $(\omega, \L^2(\Xv))$ is the Weil representation of $\wt\Sp$ attached to the character $\chi$. A proof of this fact based on previous work of Ranga Rao \cite{RangaRaoWeil} may be found in \cite{TerujiWeyl}. Conversely, one may take the above definition of $\omega$ and check directly that it is a representation with all the required properties. This was done in \cite[Theorem 60]{AubertPrzebinda_omega}.

We consider a dual pair $(\G,\G')$, in the symplectic group $\Sp$, with $\G$ compact.
Let $\wt\G$ be the preimage of $\G$ in the metaplectic group, equipped with the Haar measure of total mass $1$.
Fix an irreducible unitary representation $\Pi$ of $\wt\G$ which occurs in the restriction of $\omega$ to $\wt\G$ and let $\Theta_{\Pi}$ be its (distribution) character. 
Then the operator
\[
\omega(\check\Theta_{\Pi})=\int_{\wt\G}\Theta_{\Pi}(\t g^{-1})\omega(\t g)\,d\t g
\]
is $(\dim \Pi)^{-1}$ times the orthogonal projection $\L^2(\Xv)\to \L^2(\Xv)_\Pi$ onto the $\Pi$-isotypic component of $\L^2(\Xv)$. The Weyl symbol of
this projection is equal to a constant multiple of
\begin{equation}\label{0.1}
T(\check\Theta_{\Pi})=\int_{\wt\G}\Theta_{\Pi}(\t g^{-1})T(\t g)\,d\t g.
\end{equation}
This is precisely the intertwining distribution we introduced before: 
\[
f_{\Pi\otimes\Pi'}=T(\check\Theta_{\Pi}).
\]
Here we are using that the symplectic transform of $T(\check\Theta_{\Pi})$ is $\pm T(\check\Theta_{\Pi})$, \cite[(5.2) and (5.4.2)]{PrzebindaUnipotent}.

For example, if $\G=\Og_1=\{\pm 1\}$ and $\G'=\Sp$, then 
\[
\wt\G=\{(1,1),\ (1,-1),\ (-1, i^n2^{-n}),\ (-1, -i^n2^{-n})\}
\]
with the multiplication given by $(g_1,\xi_1)(g_2,\xi_2)=(g_1g_2,\xi_1\xi_2C(g_1,g_2))$, where 
\[
C(1,\pm1)=C(\pm 1,1)=1\ \ \text{and}\ \ C(-1,-1)=2^{2n}. 
\]
In these terms, the following two one-dimensional representations of $\wt\G$ occur in $\omega$.
\begin{equation}\label{PipmO1Sp2N}
\Pi_+(g,\eta)=|\eta|^{-1}\eta,\ \ \Pi_-(g,\eta)=g|\eta|^{-1}\eta
\end{equation}
A straightforward computation shows that
\begin{equation}\label{1}
T(\check\Theta_{\Pi_\pm})=\frac{1}{2}\left(\delta_0\pm 2^{-n}dw\right),
\end{equation}
where $\delta_0$ is the Dirac delta at the origin in $\Wv$.

In general, Classical Invariant Theory says that the space $\L^2(\Xv)_\Pi$ is irreducible under the joint action of $\wt\G$ and $\wt\G'$, \cite{HoweRemarks}. Hence
$\L^2(\Xv)_\Pi=\L^2(\X)_{\Pi\otimes\Pi'}$ for an irreducible unitary representation $\Pi'$ of $\wt\G'$. We are interested in the character $\Theta_{\Pi'}$ of $\Pi'$.

The unnormalized moment map
\begin{equation}\label{tau'}
\tau':\Wv\to{\g'}^*,\ \tau'(w)(y)=\langle yw,w\rangle \qquad (w\in\Wv,\ y\in \g')
\end{equation}
is a quadratic polynomial map with compact fibers. Hence the pull-back
\[
\Ss(\g')\ni\psi\to \psi\circ\tau'\in\Ss(\Wv)
\]
is well defined and continuous, \cite[Lemma 6.1]{PrzebindaUnipotent}. Therefore, by dualizing, we get a push-forward of distributions
\begin{equation}\label{pushforwardbytau'}
\tau'_*: \Ss^*(\Wv)\to\Ss^*(\g').
\end{equation}
Recall the Cayley transform $c(x)=(x+1)(x-1)^{-1}$, which we view as a rational map from the Lie algebra $\sp$ into the group $\Sp$.
In particular $c(0)=-1$.
Let 
\begin{equation}\label{eq:tildec}
\t c:\sp\to\wt\Sp
\end{equation}
be a real analytic lift  of $c$. Set $\t c_-(x)=\t c(x) \t c(0)^{-1}$. Then $\t c_-(0)$ is the identity in the group $\wt\Sp$.

Let $\t c_-^*\Theta_{\Pi'}$ be the pullback of $\Theta_{\Pi'}$ by $\t c_-$, 
\cite[Theorem 6.1.2]{Hormander}.
Then, as shown in \cite[Theorem 6.7]{PrzebindaUnipotent}, for an appropriately defined Fourier transform $\mathcal{F}$ on $\g'$, 
\begin{equation}\label{pullbackThetaPiprime}
\frac{1}{\Theta\circ \t c} \; \t c_-^*\Theta_{\Pi'}=\frac{(\text{central character of $\Pi$})(\t c(0))}{\dim \Pi} \;\mathcal F(\tau'_*(T(\check\Theta_{\Pi})))\,.
\end{equation}
Formula (\ref{pullbackThetaPiprime}) allows us to determine $\Theta_{\Pi'}(\Psi)$ for every $\Psi\in C_c^\infty(\wt G')$ supported in the image of $\t c_-$. 
Indeed $\Theta_{\Pi'}(\Psi)=\big(\t c_-^*\Theta_{\Pi'}\big)(\psi)$ where $\psi(x)=\Psi(\t c_-(x))\ch^{-2r}(x)$ for $x \in \g'$. Here $\ch^{-2r}$ is the Jacobian of $\t c_-$; see (\ref{ch}) and (\ref{second step}) in section \ref{Intertwining distributions}.

For example, if $\G=\Og_1$ then for $\psi \in \Ss(\g')$ we have
\begin{equation*}
\tau'_*(T(\check\Theta_{\Pi}))(\psi)=T(\check\Theta_{\Pi})(\psi\circ \tau')
=\frac{1}{2} \Big( \delta(\psi) +2^{-n} \int_\Wv \psi(\tau'(w)) \,dw\Big)\,.
\end{equation*}
Recall that $\tau'(\Wv\setminus\{0\})\subseteq \g'$ is one of the two non-zero minimal  nilpotent orbits in $\g'=\sp$, which we denote $\Oo_{\rm min}$. 
Hence
\begin{equation}\label{pullbackThetaPi'pm}
\frac{1}{\Theta\circ \t c} \; \t c_-^*\Theta_{\Pi_\pm'}=\Pi_\pm(\t c(0)) \mathcal F\Big(\frac{1}{2}(\delta\pm \mu_{\Oo_{\rm min}})\Big),
\end{equation}
where $\mu_{\Oo_{\rm min}}=\tau'_*(2^{-n}dw)$ is an invariant measure on ${\Oo_{\rm min}}$ and $\Pi_\pm'$ are the corresponding two irreducible pieces of the Weil representation $\omega$. Notice that, since $c(0)=-1$, the definition (\ref{PipmO1Sp2N}) gives 
$\Pi_-(\t c(0))=-\Pi_-(\t c(0))$. It follows that 
\begin{equation}\label{pullbackThetaomega}
\frac{1}{\Theta\circ \t c} \; \t c_-^*\Theta_\omega=\Pi_+(\t c(0)) \mathcal F(\mu_{\Oo_{\rm min}})\,.
\end{equation}
This formula gives the character of $\omega$ as the Fourier transform of an invariant measure of a nilpotent orbits in $\sp$. This is what Kirillov's orbit method would aim at (if it worked on semisimple Lie groups).

\smallskip

In this paper we compute explicitly the distribution $f_{\Pi\otimes\Pi'}=T(\check\Theta_{\Pi})$ in terms of the $\G\G'$-orbital integrals on $\Wv$, see Theorems  \ref{main thm for l<l'} and \ref{main thm for l>=l'}. In particular we see that $T(\check\Theta_{\Pi})$ is a smooth function if and only if $(\G,\G')$ is a pair of compact unitary groups, see Proposition \ref{ul, ul' 3}. Also, modulo a few exceptions, $T(\check\Theta_{\Pi})$ is a locally integrable function if and only if the rank of $\G$ is greater or equal to the rank of $\G'$. Our results on the orbital integrals are based on the corresponding results of Harish-Chandra, which are transferred from the Lie algebra $\g'$ to the symplectic space $\Wv$ via a theorem of G. Schwarz, \cite{Schwarz74}. We hope to circumvent it in the future in order to treat the case of a non-compact group $\G$.

Let $\tau:\Wv\to\g^*$ be the unnormalized moment map for $\g$ given, as in (\ref{tau'}),  by $\tau(w)(x)=\langle x w,w\rangle$. 
The variety $\tau^{-1}(0)\subseteq \Wv$ is the closure of a single orbit $\Oo$; see e.g. \cite[Lemma 2.16]{PrzebindaUnipotent}. There is a positive $\G\G'$-invariant measure $\mu_{\Oo}$ on this orbit which defines a homogeneous tempered distribution. We denote its degree by $\deg \mu_\Oo$. For $t>0$ let 
$
M_t(w)=tw 
$, $w\in \Wv$.
Denote by $M_t^*$ the corresponding pullback of distributions. In particular
$
M_t^*\mu_\Oo=t^{\deg \mu_\Oo}\mu_\Oo
$.
We show that
\[
t^{\deg \mu_\Oo}M_{t^{-1}}^*T(\check\Theta_{\Pi}) \underset{t\to 0}\to const\,\mu_{\Oo},
\]
as tempered distributions, where $const$ is a non-zero constant, see Proposition \ref{the dilation limit of intertwining distribution}. This last statement leads to an elementary proof of the equality
\begin{equation}\label{WF1andtau'tau}
WF_1(\Theta_{\Pi'})=\tau'\tau^{-1}(0)\,,
\end{equation}
see Corollary \ref{WF of Pi'}. Here $WF_1(\Theta_{\Pi'})$ denotes the fiber of the wave front set $WF(\Theta_{\Pi'})$ over the identity $1 \in \wt{G'}$. As proven by Rossmann in \cite{Ross95}, $WF_1(\Theta_{\Pi'})$ agrees with $WF(\Pi')$, the wave front set of the representation $\Pi'$ in the sense of Howe \cite{HoweWave}.

The equality (\ref{WF1andtau'tau}) was already verified  in \cite[Theorem 6.11]{PrzebindaUnipotent}, but the proof used a theorem of Vogan concerning the restriction of a representation to a maximal compact subgroup, \cite{VoganGelfand}, which is not needed in our present approach. In order to stay within the theory of the almost semisimple orbital integrals on the symplectic space, see section \ref{An almost semisimple orbital integral on the symplectic space}, we consider only representations $\Pi$ such that the distribution character $\Theta_\Pi$ is supported in the preimage $\wt{\G_1}$ of the Zariski identity component $\G_1$ of $\G$. This eliminates some representations of the groups $\Og_{2l}$.
For reader's convenience one should mention here that there is a notion of an associated variety of a presentation introduced by Vogan \cite{Vogan89} and that associated variety determines the wave front set of a representation \cite{SchmidVilonen2000}. In this context a recent work of Lock and Ma \cite{LockMaassocvar} provides a vast generalization of our formula for the wave front set of $\Pi'$, for dual pairs in the stable range with $\G$-the smaller member. Needless to say this approach is much more sophisticated and much less direct than ours. 
We point out that the Gelfand-Kirillov dimension of $\Pi'$ (=$\frac{1}{2}\dim WF(\Pi')$) was determined in \cite[Theorem C.1]{PrzebindaUnipotent} and rediscovered later independently in    \cite{notyk} and in \cite[Theorem 6]{EnrightWillenbring2004}. Also, \cite{notyk} gives a formula for the associated variety of $\Pi'$ in the case $\G$-compact and without the stable range assumption.

In section \ref{two unitary groups} we demonstrate that our computation are precise enough to recover Weyl's theorem, saying that if both $\G$ and $\G'$ are compact then the restriction of the Weil representation to $\wt{\G}\times\wt{\G'}$ decomposes with multiplicity one.

In section \ref{Limits of orbital integrals in the stable range.} we compute limits of the almost semisimple orbital integrals using van der Corput lemma rather than techniques based on unpublished work of Ranga Rao and obtain  in that case a stronger version of Rossmann's limit formula, \cite{RossmannNilpotent}.
\section{\bf The group\rm\ $\wt\G$\bf\ and the Weyl denominator.\rm}\label{The group G and the Weyl denominator}
We keep the notation from the introduction. In particular, $(\G,\G')$ is a dual pair with $\G$ compact and $J$ is a fixed compatible positive complex structure on $\Wv$.

In this section we describe the restriction of the covering
\begin{equation}\label{The group G0}
\wt\Sp\ni (g,\xi)\to g\in \Sp
\end{equation}
to the group $\G$.  This is then applied to study the analytic lift of the Weyl denominator.  

Let $\Sp^J\subseteq \Sp$ denote the centralizer of $J$. Since $\Sp^J$ 
is a maximal compact subgroup of $\Sp$, we may assume that $\G\subseteq \Sp^J$ and begin by studying the restriction of the map (\ref{The group G0}) to $\wt{\Sp^J}$.

Let $\Wv_\C$ be the complexification of $\Wv$. Denote by the same symbol $\langle\cdot ,\cdot \rangle$ the complex bilinear extension of the symplectic form from $\Wv$ to $\Wv_\C$.
Let $\Wv_\C^+\subseteq \Wv_\C$ be the $i$-eigenspace for $J$. Denote by $\Wv_\C\ni w\to\overline w\in \Wv_\C$ the conjugation with respect to the real form $\Wv\subseteq \Wv_\C$. Then $\Wv_\C^-=\overline{\Wv_\C^+}$ is the $(-i)$-eigenspace for $J$ and 
\begin{eqnarray}\label{The group G1}
\Wv_\C=\Wv_\C^+\oplus \Wv_\C^-
\end{eqnarray}
is a complete polarization. The formula
\begin{eqnarray}\label{The form H}
H(w,w')=i\langle w, \overline{w'}\rangle \qquad (w, w'\in \Wv_\C)
\end{eqnarray}
defines a non-degenerate hermitian form on $\Wv_\C$ and the map
\begin{eqnarray}\label{half1minusiJ}
\frac{1}{2}(1-iJ):\Wv\to \Wv_\C^+
\end{eqnarray}
is an $\R$-linear isomorphism. Moreover, if $w=\frac{1}{2}(1-iJ)w_0$ and $w'=\frac{1}{2}(1-iJ)w_0'$ for some $w_0, w_0'\in\Wv$, then 
\begin{eqnarray}\label{half1minusiJ 1}
H(w,w')=\frac{1}{2}(\langle J w_0,w_0'\rangle + i\langle w_0, w_0'\rangle).
\end{eqnarray}
In particular, the restriction $H|_{\Wv_\C^+}$ of $H$ to $\Wv_\C^+$ is positive definite. 
Let $\Ug\subseteq \End(\Wv_\C^+)$ denote the isometry group of the form $H|_{\Wv_\C^+}$.

Let $g \in \Sp^J$. Then $g$ can be extended to a complex linear endomorphism, still denoted by $g$, which belongs to $\Sp(\Wv_\C)^J$. Clearly $g(\Wv_\C^+) = \Wv_\C^+$. Set 
$u=g|_{\Wv_\C^+}$. Then for every $w \in \Wv_\C^+$ with $w=\frac{1}{2}(1-iJ)w_0$ for $w_0 \in \Wv$, we have
$uw=g\big[  \frac{1}{2}(1-iJ)w_0\big]=\frac{1}{2}(1-iJ)gw_0$, i.e. 
$$u=\frac{1}{2}(1-iJ) \circ g \circ [\frac{1}{2}(1-iJ)]^{-1}\,.$$
It follows that $u \in U$.

Let
\begin{eqnarray}\label{wtU}
\wt{\Ug}=\{(u,\zeta);\ \det\,u=\zeta^2,\ u\in \Ug\}\subseteq \GL(\Wv_\C^+)\times \C^\times.
\end{eqnarray}
endowed with the coordinate-wise multiplication. 
This is a connected two-fold covering group of $\Ug$. 

\begin{pro} \label{lemma:isoSpJU}
The group isomorphism 
\begin{eqnarray}\label{isoSpJU}
\Sp^J \in g \to u=g|_{\Wv_\C^+}=\frac{1}{2}(1-iJ) \circ g \circ [\frac{1}{2}(1-iJ)]^{-1} \in U
\end{eqnarray}
lifts to a group isomorphism 
\begin{eqnarray}\label{liftisoSpJU}
\wt{\Sp^J}\ni (g,\xi)\to(u,\xi \det(g-1)_{(g-1)\Wv_\C^+})\in \t\Ug\,.
\end{eqnarray}
Therefore the restriction of the covering (\ref{The group G0}) to $\wt{\Sp}^J$ is isomorphic to the covering
\begin{eqnarray}\label{wtUcoverningU}
\wt\Ug\ni (u,\zeta)\to u\in \Ug.
\end{eqnarray}
\end{pro}
\begin{proof}
Notice that any element $g\in \Sp(\Wv_\C)^J$ preserves the decomposition (\ref{The group G1}) and satisfies the following formula
\begin{eqnarray}\label{The group G4.1}
\langle g^{-1}w,w'\rangle = \langle w,g w'\rangle
\qquad (w\in \Wv_\C^+,\ w'\in\Wv_\C^-).
\end{eqnarray}
Since the symplectic form identifies $\Wv_\C^+$ with the dual of $\Wv_\C^-$ and since the determinant of the adjoint linear map is equal to the determinant of the original linear map, (\ref{The group G4.1}) shows that 
$$
\det(g-1)_{\Wv_\C^-}=\det(g^{-1}-1)_{\Wv_\C^+},
$$
where  we take the determinant of the linear map restricted to the indicated subspace.
Hence,
\begin{eqnarray}\label{det-gminusid-onWandWC+}
\det(g-1)=(-1)^n \det(g-1)_{\Wv_\C^+}^2 \det(g)_{\Wv_\C^+}^{-1}.
\end{eqnarray}
Let  $g \in \Sp^J$. 
The restriction $\inner{\cdot}{\cdot}_0$ of $\inner{\cdot}{\cdot}$ to $(g-1)\Wv$ is nondegenerate. 
Indeed, since $g$ and $J$ commute, the orthogonal complement $((g-1)\Wv)^{\perp}$ of $(g-1)\Wv$ with respect to  $\inner{\cdot}{\cdot}$ coincides with the orthogonal complement of $(g-1)\Wv$ with respect 
to the positive definite form $\inner{J\cdot}{\cdot}$. Hence $((g-1)\Wv)^{\perp} \cap (g-1)\Wv=0$.

Consider $(g-1)\Wv$ as a symplectic space with $\inner{\cdot}{\cdot}_0$ as a symplectic form, and let $\Sp_0$ be the corresponding symplectic group.
 Since $J$ preserves $(g-1)\Wv$, we can consider its restriction $J_0$ to $(g-1)\Wv$. It satisfies 
$J_0^2=-1$ and the bilinear form $\inner{J_0\cdot}{\cdot}_0$ is symmetric and positive definite. 
So $J_0$ is a positive compatible complex structure on  $(g-1)\Wv$. In particular, $J_0 \in \Sp_0$. So $\det(J)_{(g-1)\Wv}=\det(J_0)=1$.

Formula  (\ref{det-gminusid-onWandWC+}) applied to $g\in \Sp((g-1)\Wv_\C)^{J_0}$ 
shows that
\begin{equation*}
\det(J_g)_{(g-1)\Wv}=\det(J^{-1}(g-1))_{(g-1)\Wv}=i^{\dim (g-1)\Wv} \det(g-1)_{(g-1)\Wv_\C^+}^2 \det(g)_{(g-1)\Wv_\C^+}^{-1}.
\end{equation*}
Recall the notation $u=g|_{\Wv_\C^+}$.
Since $\Wv_\C^+$ is the direct sum of $(g-1)\Wv_\C^+$ and the subspace where $g=1$, 
\begin{equation} \label{detgdetu}
\det(g)_{(g-1)\Wv_\C^+}=\det(g)_{\Wv_\C^+}=\det(u).
\end{equation}
Thus, in terms of (\ref{metaplectic group})
\begin{eqnarray}\label{detgdetu1}
\det(J_g)_{(g-1)\Wv}^{-1}=i^{\dim (g-1)\Wv} \det(g-1)_{(g-1)\Wv_\C^+}^{-2} \det(u).
\end{eqnarray}
Hence, the map (\ref{liftisoSpJU}) is a well defined bijection. We now prove that it is a group homomorphism. 

We see from \eqref{multiplicationSp} that the map (\ref{liftisoSpJU}) is a group homomorphism if and only if
\begin{eqnarray}\label{formula for cocycle 1}
C(g_1, g_2)=\frac{\det(g_1-1)_{(g_1-1)\Wv_\C^+}{\det(g_2-1)_{(g_2-1)\Wv_\C^+}}}{\det(g_1g_2-1)_{(g_1g_2-1)\Wv_\C^+}}.
\end{eqnarray}
The equations \eqref{detgdetu1} and \eqref{|cocycle|} show that the absolute values of the two sides of \eqref{formula for cocycle 1} are equal.
 
In order to shorten the notation, let us write $c(g)u=(g+1)(g-1)^{-1}u$ for $u$ in the image of $g-1$. Set
\begin{eqnarray*}
h_{g_1,g_2}(w',w'')=H(-ic(g_1)w',w'')&+&H(-ic(g_2)w',w'') \\ 
&&(w',w''\in (g_1-1)\Wv_\C^+\cap (g_2-1)\Wv_\C^+.\nn
\end{eqnarray*}
Then, since $g_1$ and $g_2$ commute with $J$,
\begin{eqnarray*}
h_{g_1,g_2}(w',w'')&=&\frac{1}{2}(\langle c(g_1)w_0',w_0''\rangle -i \langle J c(g_1)w_0',w_0''\rangle)\\
&+&\frac{1}{2}(\langle c(g_2)w_0',w_0''\rangle -i \langle J c(g_2)w_0',w_0''\rangle),
\end{eqnarray*}
where $w'$ and $w''$ are the images of $w_0'$ and $w_0''$ under the map \eqref{half1minusiJ} respectively. Moreover, $w_0', w_0''\in
(g_1-1)\Wv\cap(g_2-1)\Wv$. In particular we see that the form $h_{g_1,g_2}$ is hermitian and 
\[
\Re(h_{g_1,g_2}(w',w''))=q_{g_1,g_2}(w_0',w_0'').
\]
Let $w^1$, $w^2$, ..., $w^n$ be an $h_{g_1,g_2}$-orthogonal basis of the complex vector space $\Wv_\C^+$ with $h_{g_1,g_2}(w^k,w^k)=\pm 1$ or $0$. Then  $w^1_0$, $Jw^1_0$, $w^2_0$, $Jw^2_0$, ..., $w^n_0$, $Jw^n_0$ is an $q_{g_1,g_2}$-orthogonal basis of the real vector space $\Wv$ with $h_{g_1,g_2}(w^k,w^k)=q_{g_1,g_2}(w^k_0,w^k_0)=q_{g_1,g_2}(Jw^k_0,Jw^k_0)$. 
The signature of $h_{g_1,g_2}$ is the difference between the number of the positive $h_{g_1,g_2}(w^k,w^k)$ and the number of the negative $h_{g_1,g_2}(w^k,w^k)$, and similarly for the symmetric form $q_{g_1,g_2}$. Hence
\[
\sgn h_{g_1,g_2}=\frac{1}{2} \sgn q_{g_1,g_2}.
\]
Therefore
\begin{equation}\label{two signatures}
i^{\sgn h_{g_1,g_2}}=e^{\frac{\pi i}{2}\sgn h_{g_1,g_2}}=e^{\frac{\pi i}{4}\sgn q_{g_1,g_2}}=\chi(\frac{1}{8}\sgn q_{g_1,g_2}).
\end{equation}
We see from \eqref{cocycle/|cocycle|} and \eqref{two signatures} that it will suffice to prove that
\begin{eqnarray}\label{phases}
&&\frac{\det(g_1-1)_{(g_1-1)\Wv_\C^+}{\det(g_2-1)_{(g_2-1)\Wv_\C^+}}}{\det(g_1g_2-1)_{(g_1g_2-1)\Wv_\C^+}}
\left|\frac{\det(g_1g_2-1)_{(g_1g_2-1)\Wv_\C^+}}{\det(g_1-1)_{(g_1-1)\Wv_\C^+}{\det(g_2-1)_{(g_2-1)\Wv_\C^+}}}\right|\nn\\
&=&i^{\sgn h_{g_1,g_2}}.
\end{eqnarray}
This requires a significant amount of additional notation and therefore will be done in Appendix E. In fact, since both sides of \eqref{formula for cocycle 1} are cocycles we may (and shall) assume that $\Ker(g_1-1)=\{0\}$, see \cite[proof of Theorem 31]{AubertPrzebinda_omega}.
\end{proof}

Proposition \ref{lemma:isoSpJU} allows us to study the covering $\wt \G\to \G$ by means of the 
(explicitly given) covering $\wt \Ug \to \Ug$. 

In the following we restrict ourselves to  dual pairs $(\G,\G')$ which are irreducible, that is no nontrivial direct sum decomposition 
of the symplectic space $\Wv$ is simultaneously preserved by $\G$ and $\G'$. Irreducible dual pairs have been classified by Howe \cite{howetheta}. Those for which $\G$ is compact are all of type I. They are
\begin{equation} \label{classificationGG'}
(\Og_d, \Sp_{2m}(\R))\,,  \qquad  (\Ug_{d}, \Ug_{p,q})\,, \qquad        
(\Sp_d, \Og^*_{2m}).
\end{equation}
More precisely, given the dual pair $(\G, \G')$, there is a division algebra $\Dc=\R$, $\C$ or $\Ha$, an involution $\Bbb D\ni a \to \overline a\in\Bbb D$ over $\R$, a right $\Dc$-vector space $\Vv$ with a positive definite hermitian form $(\cdot,\cdot )$ and a left $\Dc$-vector space $\Vv'$ with a non-degenerate skew-hermitian form $(\cdot,\cdot )'$ so that $\G$ is the isometry group of the form $(\cdot,\cdot )$, $\G'$ is the isometry group of the form $(\cdot,\cdot )'$ and  the symplectic space $\Wv=\Vv\otimes_\Dc\Vv'$ with $\langle\cdot,\cdot\rangle=\tr_{\Dc/\R}((\cdot,\cdot )\otimes(\cdot,\cdot )')$, see \cite{howetheta}. The group $\G$ is viewed as a subgroup of the symplectic group via the identification $\G\ni g=g\otimes 1\in \Sp$. Similarly, the group $\G'$ is viewed as a subgroup of the symplectic group via the identification $\G'\ni g'=1\otimes g'\in \Sp$.

Since $\G$ commutes with $J$, there is $J'\in\G'$ such that $J=1\otimes J'$. Let $\Vv'_\C=\Vv'\otimes_\R\C$ and let $\Vv'_\C{}^+\subseteq \Vv'_\C$ be the $i$-eigenspace for $J'$. Then
$\Wv_\C^+=\Vv\otimes_\Dc\Vv'_\C{}^+$ and 
\[
\det(g)_{\Wv_\C^+}=\det(g\otimes 1)_{\Vv\otimes_\Dc\Vv'_\C{}^+} \qquad (g\in\G).
\]

The description of $\G$ as isometry group of the form $(\cdot,\cdot )$ on $\Vv$ provides an irreducible complex representation of $\G$ on $\Vv_\C=\Vv\otimes_\R \C$, of dimension $d=\dim \Vv$. 

Recall that the group $\Og_d$ has, up to an equivalence, one irreducible complex representation of dimension $d$. Call it $\delta$. Then  $\det(\delta(g))=\pm 1$, $g\in\G$.

The group $\Ug_d$ has two  irreducible complex representations of dimension $d$, $\delta$ and the contragredient $\delta^c$. Then 
$\det(\delta^c(g))=\det(\delta(g))^{-1}$, $g\in\G$.

The group $\Sp_d$ has, up to an equivalence, one irreducible complex representation $\delta$ of dimension $2d$. In this case $\det(\delta(g))=1$, $g\in\G$.

In these terms
\begin{equation} \label{detg-delta}
\det(g\otimes 1)_{\V\otimes_\Dc\V'_\C{}^+}=\left\{
\begin{array}{lll}
\det(\delta(g))^m\ \ &\text{if $\G'$ is isomorphic to $\Sp_{2m}(\R)$},\\
\det(\delta(g))^{p-q}\ \ &\text{if $\G'$ is isomorphic to $\Ug_{p,q}$},\\
1\ \ &\text{if $\G'$ is isomorphic to $\Og_{2m}^*$.}
\end{array}\right.
\end{equation}
We set 
\begin{equation} \label{sqrtG}
\sqrt{\G}=\{(g,\zeta);\ g \in \G, \zeta^2=\det(\delta(g))\}\,.
\end{equation}

\begin{pro}
The covering
\begin{equation}\label{The group G7}
\wt\G\to\G
\end{equation}
splits if and only if $\det(g\otimes 1)_{\V\otimes_\Dc\V'_\C{}^+}$ is a square. 
This does NOT happen if and only if
either $\G'$ is isomorphic to $\Sp_{2m}(\R)$ with $m$ odd or $\G'$ is isomorphic to $\Ug_{p,q}$ with $p+q$ odd. In these cases $\wt\G$ is isomorphic to $\sqrt{\G}$.
\end{pro}
\begin{proof}
Because of Proposition \ref{lemma:isoSpJU} we can identify the cover $\wt{\Sp^J} \to \Sp^J$ with 
$\wt\Ug\to \Ug$. Hence the cover $\wt\G\to \G$ splits if and only if there is a group homomorphism $\G\in g \to \zeta(g)\in \Ug_1 \subseteq \C^\times$ so that $\zeta(g)^2=\det(g)_{\Wv_\C^+}$. Here we are using (\ref{detgdetu}) and that $\G$ is compact. By (\ref{detg-delta}), this happens except at most in the two cases listed in the statement of the Proposition.  

Suppose that $\G'=\Sp_{2m}(\R)$, and let $\zeta:\Og_d\to \Ug_1$ be a continuous group homomorphism so that $\zeta(g)^2=\det(\delta(g))^m=(\pm 1)^m$. Then    $\zeta(\Og_d)\subseteq \{\pm 1,\pm i\}$ and it is a subgroup with at most two elements. So $\zeta(\Og_d)\subseteq \{\pm 1\}$. On the other hand, if $g \in \Og_d \setminus \SO_d$, 
then $\det(\delta(g))=-1$. Thus $\zeta(g)^2\neq \det(\delta(g))^m$ if $m$ is odd.

Suppose now that $\G'=\Ug_{p,q}$, and let $\zeta:\Ug_d\to \Ug_1$ be a continuous group homomorphism so that $\zeta(g)^2=\det(\delta(g))^{p-q}$. Restriction to $\Ug_1 \equiv \{\diag(h, 1\dots, 1): h \in \Ug_1\} \subseteq \Ug_d$ yields a continuous group homomorphism $h \in \Ug_1 \to \zeta(h) \in  \Ug_1$. Thus, there is  $k\in \Zb$ so that $\zeta(h)=h^k$ for all $h \in \Ug_1$. So $h^{2k}=\zeta(h)^2=\det(\diag(h, 1\dots, 1))^{p-q}$ implies that $p+q$ must be even. 

For the last statement, consider for $k \in\Zb$ the covering $\Mg_k=\{(g,\zeta):\zeta^2=\det(\delta(g))^{2k+1}\}$ of $\G$. Then $(g,\zeta)\to (g,\zeta (\det(\delta(g)))^{-k})$ is a covering isomorphism between $\Mg_k$ and $\Mg_0$.   
\end{proof}

Let $\diesis{\G}=\sqrt{\G}$ if the covering (\ref{The group G7}) does not split and let $\diesis{\G}=(\Zb/2\Zb)\times\sqrt{\G}$ if it does. Then we have the obvious (possibly trivial) covering
\begin{equation}\label{acceptable covering}
\diesis{\G}\to \wt\G.
\end{equation}
Let $\H\subseteq \G$ be the diagonal Cartan subgroup. 
Denote by $\diesis{\H}\subseteq \diesis{\G}$ the preimage of $\wt \H$ and let $\diesis{\H}_o\subseteq\diesis{\H}$ be the connected identity component. 
Fix a system of positive roots of $\h$ in $\g_\C$ and let $\rho\in i\h^*$ denote one half times the sum of all the positive roots. 
Then there is a group homomorphism $\xi_\rho:\diesis{\H}_o\to \C^\times$ whose derivative at the identity coincides with $\rho$. 
More generally, for any $\mu\in i\h^*$,  let $\xi_\mu:\diesis{\H}_o\to \C^\times$  denote the unique group homomorphism which has derivative at the identity equal to $\mu$, if it exists.
In particular the Weyl denominator
\begin{equation} \label{eq:Weyl-den}
\Delta(h)=\xi_\rho(h)\prod_{\alpha>0}(1-\xi_{-\alpha}(h)) \qquad (h\in \diesis{\H}_o),
\end{equation}
where the product is taken over all the positive roots $\alpha$, is well defined and analytic. 
\section{{\bf A Theorem of G. W. Schwartz.}}\label{A Theorem of G. W. Schwartz}
Since the group $\G$ is compact, the involution $\Bbb D\ni a \to \overline a\in\Bbb D$ is trivial if and only if $\Bbb D=\R$. Let $M_{d,d'}=M_{d,d'}(\Bbb D)$ denote the real vector space of the $d$ by $d'$ matrices with the entries in $\Bbb D$ and let $\Hs_{d'}=\Hs_{d'}(\mathbb D)=\{X\in M_{d',d'};\ \overline X^t=X\}$  denote the real vector space of the $\mathbb D$-hermitian matrices of size $d'$. In this section we shall be concerned with the map 
\begin{equation}\label{map beta}
\beta: M_{d,d'}\ni w\to \overline w^t w\in\Hs_{d'}.
\end{equation}
The group $\Ug_d=\Ug_d(\Bbb D)=\{g\in\GL_d(\Bbb D);\ \overline g^t=g^{-1}\}$ acts on $M_{d,d'}$ via the left multiplication and preserves the fibers of $\beta$. (In the standard notation $\Ug_d(\Bbb D)$ is equal to $\Og_d$ if $\Dc=\R$, $\Ug_d$ if $\Dc=\C$ and $\Sp_md$ if $\Dc=\Ha$.)
The title of this section refers to the following proposition, which is based on a theorem by G. Schwartz \cite{Schwarz74} .
\begin{pro}\label{SchwarzEquality}
With the above notation,
\begin{equation}\label{SchwarzEquality1}
C_c^\infty(\Hs_{d'})\circ\beta=C_c^\infty(M_{d,d'})^{\Ug_d},
\end{equation}
where $X^Y$ means the $Y$-invariants in $X$.
\end{pro}
\begin{prf}
Since each element of $\Bbb D$ may be viewed as a matrix with real entries of size equal to the dimension of $\Bbb D$ over $\R$, we have an inclusion $M_{d,d'}(\Bbb D)\subseteq M_{d\,\dim\Bbb D,d'\,\dim\Bbb D}(\R)$. 
Cauchy-Schwarz inequality shows that
\begin{equation}\label{CSinequality}
\tr_{\Bbb D/\R}\beta(w)=\tr_{\Bbb D/\R}(I\beta(w))\leq \sqrt{\tr_{\Bbb D/\R} I}\,\sqrt{\tr_{\Bbb D/\R}(\beta(w)^2)}\qquad (w\in M_{d,d'}).
\end{equation}
In particular we see that  $\beta$ is a proper map.
Since the composition with $\beta$ maps smooth functions to smooth functions, the left hand side of (\ref{SchwarzEquality1}) is contained in the right hand side.

Consider a function $\phi\in C_c^\infty(M_{d,d'})^{\Ug_d}$. Then $\beta(\supp(\phi))$ is a compact subset of $\Hs_{d'}$. Hence there exists a function 
$f\in C_c^\infty(\Hs_{d'})$ equal to $1$ in a neighborhood of $\beta(\supp(\phi))$. Also, we know from \cite{Schwarz74} that there is a function 
$\Psi\in C^\infty(\Hs_{d'})$ such that $\Psi\circ \beta=\phi$. Let $\psi=\Psi f$. Then $\psi\in C_c^\infty(\Hs_{d'})$ and $\psi\circ \beta=\phi$. Thus the right hand side of (\ref{SchwarzEquality1}) is contained in the left hand side.
\end{prf}
\begin{cor}\label{pullback via taug}
The following equality holds,  $C_c^\infty(\g')\circ\tau'=C_c^\infty(\Wv)^{\G}$.
\end{cor}
\begin{prf}
Fix a matrix $F=-\overline F^t\in\GL_{d'}(\Bbb D)$. The classification of the dual pairs, \cite{HoweTrans}, implies that we may realize the symplectic space $\Wv$ as $M_{d,d'}$ with the symplectic form
\[
\langle w',w\rangle = \tr_{\Bbb D/\R}(-F\overline w^tw') \qquad (w\in\Wv).
\]
Then $\G=\Ug_d$ acts on $\Wv$ by the left multiplication and $\G'=\{g\in\GL_{d'}(\Bbb D);\ \overline g^t F g=F\}$ via the right multiplication by the inverse. In particular,
\[
\langle y(w),w\rangle =\langle -wy,w\rangle =\tr_{\Bbb D/\R}(F\overline w^t w y) \qquad (y\in\g',\ w\in\Wv).
\]
Notice that the Lie algebra $\g'=\{FX;\ X\in\Hs_{d'}\}$. Also, if we identify $\g'=\g'{}^*$ via the trace, then
\[
\tau'(w)=F\overline w^t w=F\beta(w) \qquad (w\in\Wv).
\]
Therefore
\[
C_c^\infty(\g')\circ\tau'=C_c^\infty(\Hs_{d'})\circ\beta.
\]
Hence the corollary follows from Proposition \ref{SchwarzEquality}.
\end{prf}
\begin{cor}\label{injectivity of pushforward via tau'}
The map $\tau'_*:\Ss^*(\Wv)^\G \to \Ss^*(\g')$, (\ref{pushforwardbytau'}), is injective.
\end{cor}
\section{\bf An almost semisimple orbital integral on the symplectic space.\rm}\label{An almost semisimple orbital integral on the symplectic space}
In this section we describe the orbital integrals needed to express the distribution (\ref{0.1}). For that purpose it is convenient to view our dual pair as a supergroup as follows.

Let $\V_{\overline 0}=\Vv$ and let $\V_{\overline 1}=\Vv'$. From now on we assume that both are left vector spaces over $\Dc$. Set $\V=\V_{\overline 0}\oplus \V_{\overline 1}$ and define an element $\Sy \in \End(\V)$ by
\[
\Sy (v_0+v_1)=v_0-v_1 \qquad (v_0\in \V_{\overline 0}, v_1\in \V_{\overline 1}).
\]
Let
\begin{eqnarray*}
&&\End(\V)_{\overline 0}=\{x\in \End(\V);\ \Sy x=x\Sy \},\\
&&\End(\V)_{\overline 1}=\{x\in \End(\V);\ \Sy x=-x\Sy \},\\
&&\GL(\V)_{\overline 0}=\End(\V)_{\overline 0}\cap \GL(\V).
\end{eqnarray*}
Denote by $(\cdot ,\cdot )''$ the direct sum of the two forms $(\cdot ,\cdot )$ and $(\cdot ,\cdot )'$. Let
\begin{eqnarray}\label{super liealgebra}
&&\so=\{x\in \End(\V)_{\overline 0};\ (xu,v)''=-(u,xv)'',\ u,v\in\V\},\\
&&\ss1=\{x\in \End(\V)_{\overline 1};\ (xu,v)''=(u,\Sy xv)'',\ u,v\in\V\},\nn\\
&&\mathfrak s=\so\oplus \ss1,\nn\\
&&\Sg=\{s\in \GL(\V)_{\overline 0};\  (su,sv)''=(u,v)'',\ u,v\in\V\},\nn\\
&&\langle x, y\rangle=\tr_{\Dc/\R}(\Sy xy)\nn.
\end{eqnarray}
Then $(\Sg, \mathfrak s)$ is a real Lie supergroup, i.e. a real Lie group $\Sg$ together with a real Lie superalgebra $\mathfrak s=\so\oplus \ss1$, whose even component $\so$ is the Lie algebra of $\Sg$.
We shall write $\mathfrak s(\V)$ instead of $\mathfrak s$ whenever we shall want to specify the Lie superalgebra $\mathfrak s$ constructed as above from $\V$ and $(\cdot ,\cdot)''$. By restriction, we have the identification
\begin{eqnarray} \label{eq:WasHom}
\ss1=\Hom_\Dc(\V_{\overline 0}, \V_{\overline 1})\,.
\end{eqnarray}

The group $\Sg$ acts on $\mathfrak s$ by conjugation
and $\langle \cdot ,\cdot \rangle$ is a non-degenerate $\Sg$-invariant form on the real vector space $\mathfrak s$, whose restriction to $\so$ is symmetric and restriction to $\ss1$ is skew-symmetric. We shall employ the notation $s.x=sxs^{-1}$ for the action of $s \in \Sg$ on $x \in 
\mathfrak s$.
In terms of our previous notation,
\[
\g=\so|_{\V_{\overline 0}},\quad \g'=\so|_{\V_{\overline 1}},\quad \Wv=\ss1,\quad \G=\Sg|_{\V_{\overline 0}},\quad \G'=\Sg|_{\V_{\overline 1}},
\]
so that
\[
\so=\g\oplus \g'\quad \text{and}\quad \Sg=\G\times \G'.
\]
Notice that the action of $\Sg=\G\times \G'$ on $\ss1=\Wv$ by conjugation corresponds to the action of $\G$ on $\Wv$ by the left multiplication and of $\G'$ on $\Wv$ via the right multiplication by the inverse.
Also, we have the unnormalized moment maps
\begin{equation}\label{unnormalized moment maps}
\tau:\Wv\ni w\to w^2|_{\V_{\overline 0}}\in \g,\qquad \tau':\Wv\ni w\to w^2|_{\V_{\overline 1}}\in \g'.
\end{equation}

An element $x \in \mathfrak s$ is called semisimple (resp., nilpotent) if $x$ is semisimple (resp., nilpotent) as an endomorphism of $\V$. We say that a semisimple element $x \in \ss1$ is regular if it is nonzero and $\dim(S.x) \geq \dim(S.y)$ for all semisimple $y \in \ss1$. Let $x \in \ss1$ be fixed. The anticommutant and the double anticommutant of $x$ in $\ss1$ are  
\begin{eqnarray*}
\anticomm{x}{\ss1}&=&\{y \in \ss1:\{x,y\}=0\}\,,\\
\danticomm{x}{\ss1}&=&\bigcap_{y \in \anticomm{x}{\ss1}} \anticomm{y}{\ss1}\,,
\end{eqnarray*}
respectively. A Cartan subspace $\hs1$ of $\ss1$ as the double anticommutant of a regular semisimple element $x \in \ss1$.
We denote by $\reg{\hs1}$ the set of regular elements in $\hs1$.

Next we describe the Cartan subspaces $\hs1\subseteq\ss1$ for the supergroups associated with the irreducible dual pairs $(\G,\G')$ with $\G$ compact. We refer to \cite[\S 6]{PrzebindaLocal} and \cite[\S 4]{McKeePasqualePrzebindaSuper} for the proofs omitted here. Given a Cartan subspace $\hs1$, there are $\Ze/2\Ze$-graded subspaces $\V^j\subseteq\V$ such that the restriction of the form $(\cdot ,\cdot )''$ to each $\V^j$ is non-degenerate, $\V^j$ is orthogonal to $\V^k$ for $j\ne k$ and
\begin{equation}\label{decomposition of space for a cartan subspace}
\V=\V^0\oplus \V^1\oplus\V^2\oplus\dots \oplus \V^{l''}.
\end{equation}
The subspace $\V^0$ coincides with the intersection of the kernels of the elements of $\hs1$ (equivalently, $\V^0=\Ker(x)$ if $\hs1=\danticomm{x}{\ss1}$).
For $1\leq j\leq l''$, the subspaces $\V^j$ are described as follows. Suppose $\Dc=\R$. Then there is a basis $v_0$, $v_0'$ of $\V_{\overline 0}^j$ and  basis  $v_1$, $v_1'$ of $\V_{\overline 1}^j$ such that
\begin{eqnarray*}
&&(v_0,v_0)''=(v_0',v_0')''=1,\qquad (v_0,v_0')''=0,\\
&&(v_1,v_1)''=(v_1',v_1')''=0,\qquad (v_1,v_1')''=1.
\end{eqnarray*}
The following formulas define an element $u_j\in \ss1(\V^j)$,
\begin{eqnarray*}
&&u_j(v_0)=\frac{1}{\sqrt{2}}(v_1-v_1'),\qquad u_j(v_1)=\frac{1}{\sqrt{2}}(v_0-v_0'),\\
&&u_j(v_0')=\frac{1}{\sqrt{2}}(v_1+v_1'),\qquad u_j(v_1')=\frac{1}{\sqrt{2}}(v_0+v_0').
\end{eqnarray*}
Suppose $\Dc=\C$. Then $\V_{\overline 0}^j=\C v_0$, $\V_{\overline 1}^j=\C v_1$, where $(v_0,v_0)''=1$ and $(v_1,v_1)''=\delta_j i$, with $\delta_j=\pm 1$. The following formulas define an element $u_j\in \ss1(\V^j)$,
\begin{equation}\label{deltaj}
u_j(v_0)=e^{-i\delta_j \frac{\pi}{4}} v_1,\qquad \ u_j(v_1)=e^{-i\delta_j \frac{\pi}{4}} v_0.
\end{equation}
Suppose $\Dc=\Ha$. Then $\V_{\overline 0}^j=\Ha v_0$, $\V_{\overline 1}^j=\Ha v_1$, where $(v_0,v_0)''=1$ and $(v_1,v_1)''= i$. 
The following formulas define an element $u_j\in \ss1(\V^j)$,
\[
u_j(v_0)=e^{-i \frac{\pi}{4}} v_1, \qquad  u_j(v_1)=e^{-i\frac{\pi}{4}} v_0.
\]
In any case, by extending each $u_j$ by zero outside $\V^j$, we have
\begin{equation}\label{a cartan subspace}
\hs1=\sum_{j=1}^{l''} \R u_j\,.
\end{equation}

The formula (\ref{a cartan subspace}) describes all Cartan subspaces in $\ss1$, up to conjugation by $\Sg$. In other words it describes a maximal family of mutually non-conjugate Cartan subspaces. Notice that there is only one such subspace unless the dual pair $(\G, \G')$ is isomorphic to $(\Ug_l, \Ug_{p,q})$ with $l''=l< p+q$. In the last case there are $\min(l,p) - \max(l-q,0) +1$ such subspaces, assuming $p\leq q$. For each $m$ such that $\max(l-q,0)\leq m\leq \min(p,l)$ there is a Cartan subspace $\hs1{}_{,m}$ determined by the condition that $m$ is the number of the positive $\delta_j$ in (\ref{deltaj}). We may assume that $\delta_1=\dots=\delta_m=1$ and $\delta_{m+1}=\dots=\delta_l=-1$.
The choice of the spaces $\V_{\overline 0}^j$ may be done independently of $m$. The spaces $\V_{\overline 1}^j$ depend on $m$.

The Weyl group $W(\Sg,\hs1)$ is the quotient of the stabilizer of $\hs1$ in $\Sg$ by the subgroup $\Sg^{\hs1}$ fixing each element of $\hs1$. If $\Dc\ne \C$, then the group $W(\Sg,\hs1)$ acts by all the sign changes and all permutations of the $u_j$'s. If $\Dc=\C$, then the group $W(\Sg,\hs1)$ acts by all the sign changes of the $u_j$'s and all permutations which preserve $(\delta_1,\dots,\delta_{l''})$, see \cite[(6.3)]{PrzebindaLocal}. 

Set $\delta_j=1$ for all $1\leq j\leq l''$, if $\Dc\ne \C$. In general, let
\begin{equation}\label{complex structures for l leq l' 2}
J_j=\delta_j\tau(u_j),\qquad  J_j'=\delta_j\tau'(u_j) \qquad (1\leq j\leq l'').
\end{equation}
Then $J_j$,  $J_j'$  are complex structures on $\V_{\overline 0}^j$ and $\V_{\overline 1}^j$ respectively. Explicitly,
\begin{eqnarray}\label{explicit J_j}
&&J_j(v_0)=-v_0',\quad J_j(v_0')=v_0,\quad  J_j'(v_1)=-v_1',\quad J_j'(v_1')=v_1,\quad  \text{if}\ \ \Dc=\R,\\
&&J_j(v_0)=- i v_0,\quad J_j'(v_1)=- i v_1,\quad  \text{if}\ \ \Dc=\C\ \text{or}\ \Dc=\Ha.\nn
\end{eqnarray}
(The point of the multiplication by the $\delta_j$ in (\ref{complex structures for l leq l' 2}) is that the complex structures $J_j$, $J_j'$ do not depend on the Cartan subspace $\hs1$.) 
In particular, if $w=\sum_{j=1}^{l''} w_j u_j\in \hs1$, then
\begin{eqnarray}\label{explicit tau and tau' on cartan subspace}
\tau(w)=\sum_{j=1}^{l''} w_j^2\delta_j J_j\ \ \text{and}\ \ \tau'(w)=\sum_{j=1}^{l''} w_j^2\delta_j J_j'.
\end{eqnarray}
Let  $\hs1^2\subseteq \so$ be the subspace spanned by all the squares $w^2$, $w\in\hs1$.
Then
\[
\hs1^2=\sum_{j=1}^{l''} \R(J_j+J_j').
\]
We shall use the following identification
\begin{equation}\label{the identification}
\hs1^2|_{\V_{\overline 0}}\ni \sum_{j=1}^{l''} y_jJ_j = \sum_{j=1}^{l''}y_jJ_j'\in \hs1^2|_{\V_{\overline 1}}
\end{equation}
and denote both spaces by $\h$. 
Denote by $l$ the rank of $\g$ and by $l'$ the rank of $\g'$. Then $\h$ is an elliptic Cartan subalgebra of $\g$, if $l''=l$, and an elliptic Cartan subalgebra of $\g'$, if $l''=l'$. Let $d=\dim_\Dc\V_{\overline 0}$ and let $d'=\dim_\Dc\V_{\overline 1}$.
The proofs of the following two lemmas is straightforward and left to the reader.
\begin{lem}\label{roots for l<=l'}
Suppose $l\leq l'$. Then $l''=l$ and
one may choose the system of the positive roots of $\h$ in $\g_\C$ so that the product of all of them is given by the formula
\begin{equation}\label{product of positive roots for g}
\pi_{\g/\h}(\sum_{j=1}^l y_jJ_j)=\left\{
\begin{array}{lll}
\prod_{1\leq j<k\leq l}i(- y_j+ y_k) & \text{if} & \Dc=\C,\\
\prod_{1\leq j<k\leq l}(-y_j^2+y_k^2)\cdot \prod_{j=1}^l 2iy_j & \text{if} & \Dc=\Ha,\\
\prod_{1\leq j<k\leq l}(-y_j^2+y_k^2) & \text{if} & \Dc=\R\ \text{and}\ \g=\mathfrak s\mathfrak o_{2l},\\
\prod_{1\leq j<k\leq l}(-y_j^2+y_k^2)\cdot \prod_{j=1}^l iy_j & \text{if} & \Dc=\R\ \text{and}\ \g=\mathfrak s\mathfrak o_{2l+1}.\\
\end{array}
\right.
\end{equation}
Let $\z'\subseteq \g'$ be the centralizer of $\h$. We may choose the order of roots of $\h$ in $\g'_\C/\z'_\C$ so that the product of all of them is equal to
\begin{eqnarray}\label{product of positive roots for g'/z'}
&&\pi_{\g'/\z'}(y)=\\
&&\left\{
\begin{array}{lll}
\prod_{1\leq j<k\leq l}i(- y_j+ y_k)\cdot \prod_{j=1}^l (-iy_j)^{d'-d} & \text{if} & \Dc=\C,\\
\prod_{1\leq j<k\leq l}(-y_j^2+y_k^2)\cdot \prod_{j=1}^l (-y_j^2)^{d'-d} & \text{if} & \Dc=\Ha,\\
\prod_{1\leq j<k\leq l}(-y_j^2+y_k^2)\cdot \prod_{j=1}^l 2iy_j \cdot \prod_{j=1}^l (iy_j)^{d'-d}& \text{if} & \Dc=\R\ \text{and}\ \g=\mathfrak s\mathfrak o_{2l},\\
\prod_{1\leq j<k\leq l}(-y_j^2+y_k^2)\cdot \prod_{j=1}^l 2iy_j \cdot \prod_{j=1}^l (iy_j)^{d'-d+1}& \text{if} & \Dc=\R\ \text{and}\ \g=\mathfrak s\mathfrak o_{2l+1}.\\
\end{array}
\right.\nn
\end{eqnarray}
The product of the positive roots of $\hs1^2$ in the complexification of $\so$ evaluated at $w^2$, where $w=\sum_{j=1}^{l''}w_ju_j\in \hs1$, is equal to
\begin{eqnarray*}\label{}
&&\pi_{\so/\hs1^2}(w^2)=\pi_{\g/\h}(\tau(w))\pi_{\g'/\z'}(\tau'(w))=\\
&&\left\{
\begin{array}{lll}
\left(\prod_{1\leq j<k\leq l}i(- \delta_j w_j^2+\delta_j w_k^2)\right)^2\cdot \prod_{j=1}^l (- i\delta_j w_j^2)^{d'-d} & \text{if} & \Dc=\C,\\
\left(\prod_{1\leq j<k\leq l}(-w_j^4+w_k^4)\right)^2\cdot \prod_{j=1}^l 2iw_j^2\cdot \prod_{j=1}^l (-w_j^4)^{d'-d} & \text{if} & \Dc=\Ha,\\
\left(\prod_{1\leq j<k\leq l}(-w_j^4+w_k^4)\right)^2\cdot \prod_{j=1}^l 2iw_j^2 \cdot \prod_{j=1}^l (iw_j^2)^{d'-d}& \text{if} & \Dc=\R\ \text{and}\ \g=\mathfrak s\mathfrak o_{2l},\\
\left(\prod_{1\leq j<k\leq l}(-w_j^4+w_k^4)\right)^2\cdot  \prod_{j=1}^l iw_j^2\cdot \prod_{j=1}^l 2iw_j^2 \cdot \prod_{j=1}^l (iw_j^2)^{d'-d+1}& \text{if} & \Dc=\R\ \text{and}\ \g=\mathfrak s\mathfrak o_{2l+1}.\\
\end{array}
\right.\nn
\end{eqnarray*}
\end{lem}
\begin{lem}\label{roots for l>=l'}
Suppose $l\geq l'$. Then $l'=l''$. Set $\h'=\h$. Then
one may choose the system of the positive roots of $\h'$ in $\g'_\C$ so that the product of all of them is given by the formula
\begin{equation}\label{product of positive roots for g - bis}
\pi_{\g'/\h'}(\sum_{j=1}^{l'} y_jJ_j)=\left\{
\begin{array}{lll}
\prod_{1\leq j<k\leq l'}i(- y_j+ y_k) & \text{if} & \Dc=\C,\\
\prod_{1\leq j<k\leq l'}(-y_j^2+y_k^2) & \text{if} & \Dc=\Ha,\\
\prod_{1\leq j<k\leq l'}(-y_j^2+y_k^2)\cdot \prod_{j=1}^{l'} 2iy_j & \text{if} & \Dc=\R.
\end{array}
\right.
\end{equation}
Let $\z\subseteq \g$ be the centralizer of $\h$. We may choose the order of roots of $\h$ in $\g_\C/\z_\C$ so that the product of all of them is equal to
\begin{eqnarray}\label{product of positive roots for g'/z'}
&&\pi_{\g/\z}(y)=\\
&&\left\{
\begin{array}{lll}
\prod_{1\leq j<k\leq l'}i(- y_j+ y_k)\cdot \prod_{j=1}^{l'} (-iy_j)^{d-d'} & \text{if} & \Dc=\C,\\
\prod_{1\leq j<k\leq l'}(-y_j^2+y_k^2)\cdot \prod_{j=1}^{l'} 2iy_j \cdot \prod_{j=1}^{l'} (-y_j^2)^{d-d'} & \text{if} & \Dc=\Ha,\\
\prod_{1\leq j<k\leq l'}(-y_j^2+y_k^2)\cdot \prod_{j=1}^{l'}(iy_j)^{d-d'}& \text{if} & \Dc=\R\ \text{and}\ \g=\mathfrak s\mathfrak o_{2l},\\
\prod_{1\leq j<k\leq l}(-y_j^2+y_k^2)\cdot \prod_{j=1}^{l'}iy_j \cdot \prod_{j=1}^{l'}(iy_j)^{d-d'}& \text{if} & \Dc=\R\ \text{and}\ \g=\mathfrak s\mathfrak o_{2l+1}.\\
\end{array}
\right.\nn
\end{eqnarray}
The product of the positive roots of $\hs1^2$ in the complexification of $\so$ evaluated at $w^2$, where $w=\sum_{j=1}^{l''}w_ju_j\in \hs1$, is equal to
\begin{eqnarray*}\label{}
&&\pi_{\so/\hs1^2}(w^2)=\pi_{\g'/\h'}(\tau'(w))\pi_{\g/\z}(\tau(w))=\\
&&\left\{
\begin{array}{lll}
\left(\prod_{1\leq j<k\leq l'}i(- \delta_j w_j^2+\delta_j w_k^2)\right)^2\cdot \prod_{j=1}^{l'} (- i\delta_j w_j^2)^{d'-d} & \text{if} & \Dc=\C,\\
\left(\prod_{1\leq j<k\leq l'}(-w_j^4+w_k^4)\right)^2\cdot \prod_{j=1}^{l'} 2iw_j^2\cdot \prod_{j=1}^{l'} (-w_j^4)^{d'-d} & \text{if} & \Dc=\Ha,\\
\left(\prod_{1\leq j<k\leq l'}(-w_j^4+w_k^4)\right)^2\cdot \prod_{j=1}^{l'} 2iw_j^2 \cdot \prod_{j=1}^{l'} (iw_j^2)^{d'-d}& \text{if} & \Dc=\R\ \text{and}\ \g=\mathfrak s\mathfrak o_{2l},\\
\left(\prod_{1\leq j<k\leq l'}(-w_j^4+w_k^4)\right)^2\cdot  \prod_{j=1}^{l'} iw_j^2\cdot \prod_{j=1}^{l'} 2iw_j^2 \cdot \prod_{j=1}^{l'} (iw_j^2)^{d'-d}& \text{if} & \Dc=\R\ \text{and}\ \g=\mathfrak s\mathfrak o_{2l+1}.\\
\end{array}
\right.\nn
\end{eqnarray*}
\end{lem}

The following lemma is an immediate consequence of Lemmas \ref{roots for l<=l'} and \ref{roots for l>=l'}. 
\begin{lem}\label{|| and constant}
There is a constant $C(\hs1)$, which depends on $\hs1$, such that $|C(\hs1)|=1$ and
\begin{equation*}\label{|| and constant}
|\pi_{\so/\hs1^2}(w^2)|=C(\hs1)\,\pi_{\so/\hs1^2}(w^2) \qquad (w\in\hs1).
\end{equation*}
\end{lem}
If $\hs1$ is a Cartan subspace of $\Wv$, then 
\begin{equation} \label{eq:h1reg}
\reg{\hs1}=\{w \in \hs1: \pi_{\so/\hs1^2}(w^2)\neq 0\}\,.
\end{equation}
Fix a Cartan subspace $\hs1\subseteq\Wv$, an element $w\in\reg{\hs1}$ and a function $\phi\in \Ss(\Wv)$.
Suppose $\G=\Og_{2l+1}$ with $l<l'$. Let $w_0\in \ss1(\V^0)$ be a non-zero element. Then $w+w_0$ is a regular almost semisimple element whose centralizer in $\Sg$ is denoted by $\Sg^{\hs1+w_0}$. Set $\Oo(w)=\Sg.(w+w_0)$ and define
\begin{equation}\label{muwforoodd}
\mu_{\Oo(w),\hs1}(\phi)
=\int_{\Sg/\Sg^{\hs1+w_0}}\phi(s.(w+w_0))\,d(s\Sg^{\hs1+w_0}).
\end{equation}
Then, up to a constant multiple,
\begin{eqnarray}\label{muwforsome groups}
\mu_{\Oo(w),\hs1}(\phi)=
\int_{\Sg/\Sg^{\hs1}}\int_{\ss1(\V^0)}\phi(s.(w+w^0))\,dw^0\,d(s\Sg^{\hs1}).
\end{eqnarray}
In all the remaining cases let $\Oo(w)=\Sg.w$ and let
\begin{eqnarray}\label{muwforall othergroups}
\mu_{\Oo(w),\hs1}(\phi)
=\int_{\Sg/\Sg^{\hs1}}\phi(s.w)\,d(s\Sg^{\hs1}).
\end{eqnarray}
Let $\H\subseteq \G$ be the Cartan subgroup with the Lie algebra $\h$. Denote by $\Delta(\H)\subseteq \G\times \G'$ be the diagonal embedding. Then, explicitly, 
\begin{equation}\label{eq:Shone}
\Sg^{\hs1}=\Delta(\H)(\{1\}\times \Zg'), 
\end{equation}
where $\Zg'\subseteq \G'$ is the centralizer of $\h\subseteq \g'$.

To simplify the notation, when the Cartan subspace $\hs1$ is fixed, we shall simply write $\mu_{\Oo(w)}$ instead of $\mu_{\Oo(w),\hs1}$.
These are well defined, tempered distribution on $\Wv$, see \cite{McKeePasqualePrzebindaSuper}, which depend only on $\tau(w)$, or equivalently $\tau'(w)$ via the identification (\ref{the identification}).
Let $\mu_\Wv$ be the Lebesgue measure on $\Wv$ normalized as in the Introduction. Choose a positive Weyl chamber $\hs1^+\subseteq \reg{\hs1}$.
We shall normalize the above orbital integrals so that the Weyl integration formula reads
\begin{equation}\label{weyl int on w 1}
\mu_\Wv=\sum_{\hs1}\int_{\tau(\hs1^+)}|\pi_{\so/\hs1^2}(w^2)|\mu_{\Oo(w)}(\phi)\,d\tau(w)
\end{equation}
if $l\leq l'$, and
\begin{equation}\label{weyl int on w 2}
\mu_\Wv=\int_{\tau'(\hs1^+)}|\pi_{\so/\hs1^2}(w^2)|\mu_{\Oo(w)}(\phi)\,d\tau'(w)
\end{equation}
if $l\geq l'$.
\begin{lem}\label{relation between positive weyl chambers}
Suppose $l\leq l'$ and $\Dc=\C$. Then for $\max(l-q,0)\leq m<m'\leq \min(p,l)$,
\begin{equation*}
\tau(\h_{\overline 1, m}^{reg})\cap\tau(\h_{\overline 1, m'}^{reg})=\emptyset.
\end{equation*}
\end{lem}
\begin{prf}
We see from (\ref{explicit tau and tau' on cartan subspace}) and (\ref{eq:h1reg}) that
\begin{eqnarray}\label{the image of hs1 under tau 1}
\tau(\hs1{}_{,m}^{reg})&=&\{\sum_{j=1}^ly_jJ_j;\ y_1,\dots,y_m > 0> y_{m+1},\dots,y_l,\ \ \text{$y_j\ne y_k$ for $j\ne k$}\}.
\end{eqnarray}
Hence
\begin{eqnarray*}
&&\tau(\h_{\overline 1, m}^{reg})\cap\tau(\h_{\overline 1, m'}^{reg})\\
&=&\{\sum_{j=1}^ly_jJ_j;\ y_1,\dots,y_m> 0=y_{m+1}=\dots =y_{m'}> y_{m'+1},\dots,y_l,\ \ \text{$y_j\ne y_k$ for $j\ne k$}\}\\
&=&\emptyset.
\end{eqnarray*}
\end{prf}
\begin{defi}\label{def of HC integral on W}
Let $C_{\hs1}=C(\hs1)\cdot i^{\dim\,\g/\h}$, where $C(\hs1)$ is as in Lemma \ref{|| and constant}. Define the Harish-Chandra regular almost semisimple  orbital integral on $\Wv$ by the following formula
\[
f(y)=
\sum_{\hs1} C_{\hs1}\pi_{\g'/\z'}(y)\mu_{\Oo(w)} \qquad (y\in \bigcup_{\hs1}\tau(\reg{\hs1}),\ y=\tau(w)=\tau'(w)).
\]
(Since, by Lemma \ref{relation between positive weyl chambers}, the union is disjoint, the definition makes sense. If $l>l'$, then there is only one Cartan subspace $\hs1$ and $\z'=\h'$.)
\end{defi}

In the remainder of this section we shall extend $f$ and its partial derivatives continuously to a larger domain inside $\mathfrak h$. This domain will depend on $l$ and $l'$. If $l \leq l'$, then we will provide a continuous extension of $f$ (and its partial derivatives up to a specific order) to a distribution valued function $f:\mathfrak h\cap \tau(\Wv) \to S^*(\Wv)^S$. If $l>l'$, then $f$ and all of its partial derivatives  extend continuously to the closure of every connected component of $\mathfrak h'^{In-reg}\cap \tau'(\hs1)$. See Theorem \ref{pullback of muy 3} for the precise statement.

Let $\mu_\g$ be the Lebesgue measure on $\g$. Let us normalize the orbital integrals $\mu_{\G.y}\in\Ss^*(\g)$, $y\in\reg{\h}$, so that
\begin{equation}\label{normalization of orbital integrals in g}
\mu_\g =\int_{\h^+}|\pi_{\g/\h}(y)|^2\mu_{\G.y}\,dy,
\end{equation}
where $\h^+\subseteq \reg{\h}$ is a Weyl chamber.

Let $\Wv_\g\subseteq \Wv$ be the maximal subset such that $\tau|_{\Wv_\g}:\Wv_\g\to \g$, the restriction of $\tau$ to $\Wv_\g$, is a submersion.
Then  $\Wv_\g\ne \emptyset$ if and only if $l\leq l'$, see Appendix A.
In this case we shall assume that
\begin{equation}\label{normalization of lebesgue measures on w and g}
\tau|_{\Wv_\g}^*(\mu_\g)=\mu_\Wv|_{\Wv_\g}.
\end{equation}
\begin{lem}\label{pullback and orbital integral}
Suppose $l\leq l'$. 
Then
\[
\pi_{\g/\h}(\tau(w))\tau|_{\Wv_\g}^*(\mu_{\G.\tau(w)})=
f(y)|_{\Wv_\g} \qquad (w\in \reg{\hs1}).
\]
\end{lem}
\begin{prf}
We see from (\ref{normalization of lebesgue measures on w and g}) that
\[
|\pi_{\g/\h}(\tau(w))|^2\tau|_{\Wv_\g}^*(\mu_{\G.\tau(w)})=|\pi_{\g/\h}(\tau(w))\pi_{\g'/\z'}(\tau'(w))|\mu_{\Oo(w)}|_{\Wv_\g}\qquad (w\in \hs1^+).
\]
Hence,
\[
|\pi_{\g/\h}(\tau(w))|\tau|_{\Wv_\g}^*(\mu_{\G.\tau(w)})=|\pi_{\g'/\z'}(\tau'(w))|\mu_\Oo(w)|_{\Wv_\g}\qquad (w\in \reg{\hs1}),
\]
because both sides are $W(\Sg, \hs1)$-invariant. 
Thus,
\begin{eqnarray*}
&&\pi_{\g/\h}(\tau(w))\tau|_{\Wv_\g}^*(\mu_{\G.\tau(w)})\\
&=&\left(\frac{|\pi_{\g'/\z'}(\tau'(w))|}{\pi_{\g'/\z'}(\tau'(w))}\frac{\pi_{\g/\h}(\tau(w))}{|\pi_{\g/\h}(\tau(w))|}\right)\pi_{\g'/\z'}(\tau'(w))\mu_\Oo(w)|_{\Wv_\g}\qquad (w\in \reg{\hs1}).
\end{eqnarray*}
Let $C(\hs1)$ be the constant in Lemma \ref{|| and constant}. Then
\begin{equation*}
\frac{|\pi_{\g'/\z'}(\tau'(w))|}{\pi_{\g'/\z'}(\tau'(w))}\frac{\pi_{\g/\h}(\tau(w))}{|\pi_{\g/\h}(\tau(w))|}
=C(\hs1)\frac{\pi_{\g/\h}(\tau(w))^2}{|\pi_{\g/\h}(\tau(w))|^2}=C(\hs1)i^{\dim \g/\h}.
\end{equation*}
Hence, the lemma follows.
\end{prf}
Let
\begin{equation}\label{classical weyl group 1.1.1}
W(\G,\h)=\begin{cases}
\Sigma_l &  \text{if $\Dc=\C$},\\
\Sigma_l\ltimes \{\pm 1\}^l & \text{otherwise.}
\end{cases}
\end{equation}
Denote the elements of $\Sigma_l$ by $\sigma$ and the elements of 
$\{\pm 1\}^l$ by $\epsilon=(\epsilon_1,\epsilon_2,\dots, \epsilon_l)$, so that an arbitrary element of the group (\ref{classical weyl group 1.1.1}) looks like $\epsilon\sigma$, with $\epsilon=(1,1,\dots, 1)$, if $\Dc=\C$. This group acts on $\h$ as follows:
\begin{equation}\label{classical weyl group action}
(\epsilon\sigma)\sum_{j=1}^ly_jJ_j=\sum_{j=1}^l \epsilon_j y_{\sigma^{-1}(j)}J_j
\end{equation}
and coincides with the Weyl group, equal to the normalizer of $\h$ in $\G$ divided by the centralizer of $\h$ in $\G$, as the indicated by the notation.

Since the moment map $\tau$ intertwines the action of the Weyl group $W(\Sg,\hs1)$ with the subgroup $W(\Sg,\hs1,\h)\subseteq\Sigma_l\subseteq W(\G,\h)$ leaving the sequence $\delta_1,\ \delta_2,\ \dots,\ \delta_l$ fixed. The function $f(y)$ is invariant under that subgroup.  Furthermore,
\begin{eqnarray}\label{disjoint union  for the extension by the symmetry condition}
W(\G,\h) \bigcup_{\hs1}\tau(\reg{\hs1})=\bigcup_{\hs1}  (W(\G,\h)/W(\Sg,\hs1,\h)) \tau(\reg{\hs1}),
\end{eqnarray}
where the union on the right hand side is disjoint. 
Hence, in any case ($l\leq l'$ or $l> l'$) we may extend the function $f$ uniquely to
$W(\G,\h) \bigcup_{\hs1}\tau(\reg{\hs1})$ so that the extension satisfies the following symmetry condition
\begin{eqnarray}\label{extension by the symmetry condition}
&&f(sy)=\sgn_{\g/\h}(s) f(y) \qquad (s\in W(\G,\h),\ y\in W(\G,\h) \bigcup_{\hs1}\tau(\reg{\hs1})),
\end{eqnarray}
where $\sgn_{\g/\h}$ is defined by
\begin{equation}\label{sgn_g/h}
\pi_{\g/\h}(sy)=\sgn_{\g/\h}(s)\,\pi_{\g/\h}(y)  \qquad (y \in \h)\,.
\end{equation}
One motivation for such a definition of the extension is that left hand side of the equality in Lemma \ref{pullback and orbital integral} extends to all $y\in\h$ and satisfies the symmetry condition (\ref{extension by the symmetry condition}).
We would like to extend the function $f$ from the set $W(\G,\h) \bigcup_{\hs1}\tau(\reg{\hs1})$ to $\h\cap \tau(\Wv)$. This will require some more work.

Suppose $l\leq l'$. 
Fix an elliptic Cartan subalgebra $\h'\subseteq \g'$ containing $\tau'(\hs1)$. We may assume that $\h'$ does not depend on $\hs1$.
Let $\h'{}^{In-reg}\subseteq \h'$ be the subset where no non-compact roots vanish.  Set $\hs1^{In-reg}=\tau'{}^{-1}(\h'{}^{In-reg})\cap\hs1$. Then $\tau(\hs1^{In-reg})$ is the set of the elements $y\in \tau(\hs1)$ such that, under the identification (\ref{the identification}), no non-compact imaginary root of $\h'$ in $\g'_\C$ vanishes on $y$.
\begin{lem}\label{hinreg}
Suppose $l\leq l'$. 
For our specific Cartan subspace (\ref{a cartan subspace}), the set $\tau(\hs1^{In-reg})$ consists of elements $y=\sum_{j=1}^l y_jJ_j$, such that
\begin{eqnarray}\label{hin-reg}
&&y_j>0\ \text{for all $j$},\ \text{if}\ \G=\Og_{2l}\ \text{or}\ \G=\Og_{2l+1}\  \text{or}\ \G=\Sp_l\ \text{with}\ l<l'\ \text{or}\ 1=l=l',\nn\\
&&y_j\geq 0\ \text{for all $j$}\,, \ y_j+y_k>0\ \text{for all $j\ne k$},\ \text{if}\ \G=\Sp_l\ \text{and}\ 1<l=l',
\end{eqnarray}
and
\begin{eqnarray}\label{hin-reg-C}
&&y_j>0\ \text{if}\ j\leq m\ \text{and}\ l-m<q; \quad y_j\geq 0\ \text{if}\ j\leq m  \ \text{and}\ l-m=q \ \text{when $l \geq q$}\,;\nn\\
&&y_j<0\ \text{if}\ m<j\ \text{and}\ m<p; \quad y_j\leq0\ \text{if}\ m<j\ \text{and}\ m=p \ \text{when  $l \geq p$}\,; \nn\\
&&y_j-y_k>0\ \text{if}\ \ j\leq m<k\,,
\end{eqnarray}
if $\G'=\Ug_{p,q}$ and $\hs1=\h_{\overline{1},m}$. In particular, in the last case,
\begin{eqnarray}\label{hin-reg-D}
&&\tau(\h_{\overline 1,m}^{In-reg})\cap \tau(\h_{\overline 1,m'}^{In-reg})\ne \emptyset\  \ \text{implies}\ \  |m-m'|\leq 1,\\
&&\tau(\h_{\overline 1,m}^{In-reg})\cap \tau(\h_{\overline 1,m+1}^{In-reg})
\subseteq \{\sum_{j=1}^ly_jJ_j;\ y_1,\dots,y_m\geq 0=y_{m+1}\geq y_{m+2},\dots,y_l\}.\nn
\end{eqnarray}
\end{lem}
\begin{prf}
We see from (\ref{explicit tau and tau' on cartan subspace}) that the set $\tau(\hs1)$ consists of elements $y=\sum_{j=1}^l y_jJ_j$, such that $\delta_jy_j\geq 0$ for all $1\leq j\leq l$. Hence $\sum_{j=1}^l y_jJ_j'\in\h'$ not annihilated by any imaginary non-compact root of $\h'$ in $\g'_\C$ implies (\ref{hin-reg}) when $\Bbb D\neq \C$.

If $\G'=\Ug_{p,q}$, then the non-compact roots of  of $\h'$ in $\g'_\C$ acting on elements of $\h\subseteq \h'$ are given by
\begin{eqnarray*}
&&\h\ni \sum_{j=1}^l y_jJ_j'\to \pm i (y_j-y_k)\in i\R,\ \  \text{if}\ \ j\leq m<k\ \text{or}\ k\leq m<j,\\
&&\h\ni \sum_{j=1}^l y_jJ_j'\to \pm i y_j\in i\R, \ \text{if}\ j\leq m\ \text{and}\ l-m<q\ \ \text{or}\ \ m<j\ \text{and}\ m<p.
\end{eqnarray*}
Hence,  (\ref{hin-reg-C}) follows. The last statement follows from the equality
\begin{eqnarray}\label{intersection of the image of the hs1 under tau 10}
&&\tau(\hs1{}_{,m})\cap \tau(\hs1{}_{,m+k})\\
&=&\{\sum_{j=1}^ly_jJ_j;\ y_1,\dots,y_m\geq 0=y_{m+1}=\dots=y_{m+k}\geq y_{m+k+1},\dots,y_l\},\nn
\end{eqnarray}
which is a consequence of (\ref{the image of hs1 under tau 1}).
\end{prf}
\begin{lem}\label{extension to non-compact reg}
Suppose $l\leq l'$. 
For a fixed Cartan subspace $\hs1$, the function
\begin{equation}\label{extension to non-compact reg1}
\tau'_*\circ f:\tau(\reg{\hs1})\to \Ss(\g')^{\G'}
\end{equation}
extends to a smooth function 
\begin{equation}\label{extension to non-compact reg2}
\tau'_*\circ f:\tau(\hs1^{In-reg})\to \Ss(\g')^{\G'}
\end{equation}
whose all derivatives are bounded. Further, any derivative of  (\ref{extension to non-compact reg2}) extends to a continuous function on the closure of any connected component of $\tau(\hs1^{In-reg})$.
\end{lem}
\begin{prf}
For a moment let us exclude the case $\G=\Og_{2l+1}$ with $l<l'$. Let $\psi\in C_c^\infty(\g')$. Then
\begin{eqnarray}\label{muwphi}
\tau'_*(\mu_{\Oo(w)})(\psi)=\int_{\Sg/\Sg^{\hs1}}\psi(\tau'(s.w))\,d(s\Sg^{\hs1}).
\end{eqnarray}
Let $\Zg'\subseteq \G'$ is the centralizer of $\h\subseteq \g'$ and recall formula (\ref{eq:Shone}) for $\Sg^{\hs1}$. 
Since $\G$ is compact, (\ref{muwphi}) is a constant multiple of
\begin{equation}\label{G'overZ'}
\int_{\G'/\Zg'}\psi(g'.y)\,d(g'\Zg').
\end{equation}
As checked in \cite[(23)]{McKeePasqualePrzebindaSuper}, there is a positive constant $C$ such that
\begin{eqnarray}\label{singular integral as derivative of Harish-Chandra integral}
&&\pi_{\g'/\z'}(y)\int_{\G'/\Zg'}\psi(g'.y)\,d(g'\Zg')\\
&=&C\partial(\pi_{\z'/\h'})\left(\pi_{\g'/\h'}(y+y'')\int_{\G'}\psi(g'.(y+y'')\,dg'\right)|_{y''=0},\nn 
\end{eqnarray}
where $y\in\h$ and $y''\in \h'\cap[\z',\z']$. Hence, the lemma follows from \cite[Theorem 2, page 207 and Lemma 25, page 232]{HC-57Fourier} and the fact the space of the distributions is weakly complete, \cite[Theorem 2.1.8]{Hormander}.

Suppose $\G=\Og_{2l+1}$ with $l<l'$. Let $w_0\in \ss1(\V^0)$ be as in (\ref{muwforoodd}). Then $(w+w_0)^2=w^2+w_0^2$. Hence,
\begin{eqnarray}\label{taumuwforoodd}
&&\tau'_*(\mu_{\Oo(w)})(\psi)=\int_{\Sg/\Sg^{\hs1+w_0}}\psi(\tau'(s.(w+w_0)))\,d(s\Sg^{\hs1+w_0})\\
&=&\int_{\Sg/\Sg^{\hs1+w_0}}\psi(s.(\tau'(w)+\tau'(w_0)))\,d(s\Sg^{\hs1+w_0})\nn\\
&=&C_1\int_{\G'/\Zg'{}^{n}}\psi(g.(y+n))\,d(g\Zg'{}^{n}),\nn
\end{eqnarray}
where  $C_1$ is a positive constants, $y=\tau'(w)$, $n=\tau'(w_0)$ and $\Zg'{}^{n}$ is the centralizer of $n$ in $\Zg'$. 

Let $\pi_{\z'/\h'}^{short}$ denote the product of the positive short roots of $\h'$ in $\z'_\C$. 
As checked in \cite[(35)]{McKeePasqualePrzebindaSuper}, there is a positive constant $C$ such that
\begin{eqnarray}\label{rosmann}
&&\partial(\pi_{\z'/\h'}^{short})\left.\left(\pi_{\g'/\h'}(y+x)\int_{\G'/\H'}\psi(g.(y+y''))\,d(g\H)\right)\right|_{y''=0}\\
&=&C\pi_{\g'/\z'}(y)\int_{\G'/\Zg'{}^{n}}\psi(g.(y+n))\,d(g\Zg'{}^n).\nn
\end{eqnarray}
Hence the lemma follows from theorems of Harish-Chandra, as before.
\end{prf}
For a test function $\phi$ on a vector space $\Uv$ set 
\begin{equation}\label{dilations 1}
\phi_t(u)=t^{-\dim\,\Uv}\phi(t^{-1}u) \qquad (t>0,\ u\in \Uv).
\end{equation}
Then a distribution $\Phi$ on $\Uv$ is homogeneous of degree $a\in \C$ if and only if
\begin{equation*}
\Phi(\phi_t)=t^a\Phi(\phi) \qquad (t>0,\ \phi\in C_c^\infty(\Uv)).
\end{equation*}
\begin{lem}\label{case r=1 homogenity}
Suppose $l=1$. Set 
\[
f^{(k)}=\underset{y\to 0}{\lim}\ \partial(J_1^k) f(y) \qquad (k=0,\ 1,\ ,\dots).
\]
Then $\tau_*'(f^{(k)})$ is homogeneous of degree 
\begin{eqnarray}\label{case r=1 homogenity2}
&&-\dim(\g')+\deg(\pi_{\g'/\z'})+l'-1-k,\ \text{if}\ \G=\Og_{2l+1}\ \text{and}\ l<l',\\
&&-\dim(\g')+\deg(\pi_{\g'/\z'})-k,\ \text{otherwise}.\nn
\end{eqnarray}
(Here $\deg(\pi_{\g'/\z'})$ denotes the degree of the polynomial $\pi_{\g'/\z'}$.)
Furthermore,
\begin{equation}\label{case r=1 homogenity3}
\supp(\tau_*'(f^{(k)}))\subseteq \tau'(\tau^{-1}(0)).
\end{equation}
\end{lem}
\begin{prf}
It suffices to consider the restriction of $f$ to $\tau(\reg{\hs1})$ for one of the Cartan subspaces $\hs1$. 
Let $\psi\in C_c^\infty(\g')$. For a moment let us exclude the case $\G=\Og_{2l+1}$, $l<l'$. 
As we have seen in the proof of Lemma \ref{extension to non-compact reg}, there is a non-zero constant $C$, such that for $t>0$
\begin{eqnarray*}
&&\tau_*'(f^{(0)})(\psi_t)=C\,\underset{y\to 0}{\lim}\ \pi_{\g'/\z'}(y)\int_{\G'/\Zg'}\psi_t(g.y)\,d(g\Zg')\\
&=&C\,\underset{y\to 0}{\lim}\ \pi_{\g'/\z'}(y)\int_{\G'/\Zg'}t^{-\dim(\g')}\psi(g.t^{-1}y)\,d(g\Zg')\\
&=&t^{-\dim(\g')+\deg(\pi_{\g'/\z'})}\tau_*'(f^{(0)})(\psi).
\end{eqnarray*}
Thus, by taking the derivative, (\ref{case r=1 homogenity2}) follows.

Let $U\subseteq \Wv$ be an open subset with the compact closure $\overline U$ such that $\overline U\cap \tau^{-1}(0)=\emptyset$. For $w'\in \overline U$ let $w'=w'_s+w'_n$ be the Jordan decomposition and let $\epsilon$ be the minimum of all the $|w'_s|$ (for some fixed norm $|\cdot |$ on $\Wv$) such that $w'\in \overline U$. Then $\epsilon>0$ because otherwise there would be a non-zero nilpotent element of $\Wv$ outside of $\tau^{-1}(0)$, which is impossible. Hence
\begin{equation}\label{SwU}
\Sg. w\cap U=\emptyset \qquad (|w|<\epsilon, w\in\hs1).
\end{equation}
Since $\supp f(\tau(w))=\Sg w$ this implies (\ref{case r=1 homogenity3}).

Suppose $\G=\Og_{2l+1}$ and $l<l'$. Then for $t>0$, 
\begin{eqnarray*}
&&\tau_*'(f^{(0)})(\psi_t)=C\,\underset{y\to 0}{\lim}\ \pi_{\g'/\z'}(y)\int_{\Sg/\Sg^{\hs1}}\int_{\ss1(\V^0)}\psi_t(\tau'(s.(w+w^0)))\,dw^0\,d(s\Sg^\hs1)
\end{eqnarray*}
But, 
\begin{eqnarray*}
&&\pi_{\g'/\z'}(y)\int_{\Sg/\Sg^{\hs1}}\int_{\ss1(\V^0)}\psi_t(\tau'(s.(w+w^0)))\,dw^0\,d(s\Sg^\hs1)\\
&=&t^{-\dim(\g')}\pi_{\g'/\z'}(y)\int_{\Sg/\Sg^{\hs1}}\int_{\ss1(\V^0)}\psi(\tau'(s.(t^{-1/2}w+t^{-1/2}w^0)))\,dw^0\,d(s\Sg^\hs1)\\
&=&t^{-\dim(\g')+\frac{1}{2}\dim(\ss1(\V^0))}\pi_{\g'/\z'}(y)\int_{\Sg/\Sg^{\hs1}}\int_{\ss1(\V^0)}\psi(\tau'(s.(t^{-1/2}w+w^0)))\,dw^0\,d(s\Sg^\hs1)\\
&=&t^{-\dim(\g')+\deg(\pi_{\g'/\z'})+\frac{1}{2}\dim(\ss1(\V^0))}\pi_{\g'/\z'}(t^{-1}y)\int_{\G'/\Zg'{}^{n}}\psi(g.(t^{-1}y+n))\,d(g\Zg'{}^n)).
\end{eqnarray*}
Hence, by taking the limit if $y\to 0$ we conclude that
\[
\tau_*'(f^{(0)})(\psi_t)=t^{-\dim(\g')+\deg(\pi_{\g'/\z'})+\frac{1}{2}\dim(\ss1(\V^0))}\tau_*'(f^{(0)})(\psi).
\]
Since $\dim(\ss1(\V^0))=2l'-2$, (\ref{case r=1 homogenity2}) follows.

Also, with the above notation, $w+w_0$ is a Jordan sum with $w$, the semisimple part, and $w_0$, the nilpotent part. Hence, as in (\ref{SwU}), we have
\[
\Sg. (w+w_0)\cap U=\emptyset \qquad (|w|<\epsilon, w\in\hs1).
\]
Since  $\supp(f(\tau(w)))=\Sg. (w+w_0)$,  (\ref{case r=1 homogenity3}) follows.
\end{prf}
\begin{lem}\label{vanishing of derivatives}
Let $l=1$. Then $\h=\R J_1$ and 
\begin{eqnarray*}
&&W(\G,\h)\bigcup_{\hs1}\tau(\reg{\hs1})=
\left\{
\begin{array}{lll}
\R^+ J_1\ &\text{if}\ (\G, \G')=(\Ug_1,\Ug_{l'}=\Ug_{l',0}),\\
\R^- J_1\ &\text{if}\ (\G, \G')=(\Ug_1,\Ug_{l'}=\Ug_{0,l'}),\\
\R^\times J_1\ &\text{if}\ (\G, \G')=(\Og_3,\Sp_{2l'}),\ (\Og_2,\Sp_{2l'}),\ (\Sp_1, \Og_{2l'}^*)\\ 
&\text{or}\ (\Ug_1,\Ug_{p,q})\ \text{with}\ 1\leq p\leq q.
\end{array}
\right.
\end{eqnarray*}
Let $f(y)$ denote the function (\ref{extension by the symmetry condition}). 
For an integer $k=0, 1, 2, \dots$ define
\[
\langle f^{(k)}\rangle =\underset{y\to 0\pm}{\lim}\ \partial(J_1^k) f(yJ_1) 
\]
if  $(\G, \G')=(\Ug_1,\Ug_{l'})$ and 
\[
\langle f^{(k)}\rangle =\underset{y\to 0+}{\lim}\ (\partial(J_1^k) f(yJ_1) - \underset{y\to 0-}{\lim}\ (\partial(J_1^k) f(yJ_1)
\]
in the remaining cases. Assume that $1<l'$. Then
\[
\langle f^{(k)}\rangle=0\ \text{if}\  0\leq k<
\begin{cases}
2l'-2 &\text{if $\Dc= \R$ and $\G=\Og_{3}$},\\
2l'-1 &\text{if $\Dc= \R$ and $\G=\Og_{2}$},\\
l'-1 &\text{if $\Dc= \C$},\\
2(l'-1) &\text{if $\Dc=\Ha$}.
\end{cases}
\]
\end{lem}
\begin{prf}
Suppose $(\G, \G')=(\Og_3,\Sp_{2l'})$. We know from Lemma \ref{case r=1 homogenity} that the distribution $\tau'_*(\langle f^{(k)}\rangle)$ is supported in 
$\tau'(\tau^{-1}(0))$. However Lemma \ref{pullback and orbital integral} shows that for any $\phi\in C_c^\infty(\Wv_\g)$, $f(y)(\phi)$ is a smooth function of $y\in \h$. Therefore,
$\langle f^{(k)}\rangle|_{\Wv_\g}=0$. Hence,
\[
\supp(\tau'_*(\langle f^{(k)}\rangle))\subseteq \tau'(\tau^{-1}(0)\setminus \Wv_\g).
\]
A straightforward argument shows that $\tau'(\tau^{-1}(0)\setminus \Wv_\g)$ is the union of one of the two minimal nilpotent orbits in $\g'$, call it $\mathcal O_{min}$, and the zero orbit. Furthermore, $\dim(\mathcal O_{min})=2l'$. (See Appendix A.)
Lemma \ref{case r=1 homogenity} and (\ref{product of positive roots for g}) show that $\tau'_*(\langle f^{(k)}\rangle)$ is a homogeneous distribution of degree 
\[
-\dim\,\g'+\deg(\pi_{\g'/\z'})+l'-1-k=-\dim\,\g'+3l'-2-k
\]
However, as shown in \cite[Lemma 6.2]{WallachSpringer}, $\tau'_*(\langle f^{(k)}\rangle)=0$ if the homogeneity degree is greater than $-\dim\,\g'+\frac{1}{2}\dim\, \mathcal O_{min}$. Hence the claim follows.

Exactly the same argument works if $(\G, \G')=(\Og_2,\Sp_{2l'})$, or $(\Sp_1, \Og_{2l'}^*)$, or $(\Ug_1,\Ug_{p,q})$ with $1\leq p\leq q$, except that 
$\tau'(\tau^{-1}(0)\setminus \Wv_\g)=\{0\}$, see Appendix A. So, instead of relying on  \cite[Lemma 6.2]{WallachSpringer}, we may use the classical description of distributions supported at $\{0\}$, \cite[Theorem 2.3.4.]{Hormander}.

Suppose $ (\G, \G')=(\Ug_1,\Ug_{l'})$. Then (\ref{muwphi}) and (\ref{G'overZ'}) show that for $\psi\in C_c^\infty(\g')$ and $0\ne y=\tau(w)=\tau'(w)$,
\[
\tau'_*(f(y))(\psi)=const\,\pi_{\g'/\z'}(y)\int_{\G'}\psi(g'.y)\,d(g').
\]
Since the group $\G'$ is compact, the last integral defines a smooth function of $y=y'J_1$. Also, in this case, $\pi_{\g'/\z'}(y)=(iy')^{l'-l}$. Hence, the claim follows.
\end{prf}
\begin{lem}\label{vanishing of derivatives in general}
Suppose $l\leq l'$.
Let $f(y)$ denote the function (\ref{extension by the symmetry condition}), with $y=\sum_{j=1}^ly_jJ_j\in W(\G,\h) \bigcup_{\hs1} \tau(\hs1^{reg})$. 
For any multiindex $\alpha=(\alpha_1,\dots,\alpha_l)$ set $\partial(J)^\alpha=\partial(J_1)^{\alpha_1}\dots\partial(J_l)^{\alpha_l}$. For $1\leq j\leq l$ define
\[
\langle \partial(J)^\alpha f\rangle_{y_j=0} =\underset{y_j\to 0\pm}{\lim}\ \partial(J)^\alpha f(y)  
\]
if  $\{y_j\ne 0;\ y\in W(\G,\h) \bigcup_{\hs1} \tau(\hs1^{reg})\}=\R^{\pm}$,
and 
\[
\langle \partial(J)^\alpha f\rangle_{y_j=0} =\underset{y_j\to 0+}{\lim}\ \partial(J)^\alpha f(y) - \underset{y_j\to 0-}{\lim}\ \partial(J)^\alpha f(y)
\]
if   $\{y_j\ne 0;\ y\in W(\G,\h) \bigcup_{\hs1} \tau(\hs1^{reg})\}=\R^{\times}$.
Then for $1\leq j\leq l$
\[
\langle \partial(J)^\alpha f\rangle_{y_j=0}=0\; \text{if}\; \ 0\leq \alpha_j<
\begin{cases}
2(l'-l)+1 &\text{if $\Dc=\R$ and $\G=\Og_{2l}$},\\
2l'-2l &\text{if $\Dc=\R$ and $\G=\Og_{2l+1}$},\\
l'-l &\text{if $\Dc=\C$},\\
2(l'-l) &\text{if  $\Dc=\Ha$}.
\end{cases}
\]
(Here $\langle \partial(J)^\alpha f\rangle_{y_j=0}$ is a function of the $y$ with $y_j=0$.)
\end{lem}
\begin{prf}
Without any loss of generality we may assume that $j=l$. Let $w=\sum_{j=1}^{l-1}w_ju_j$, where $\delta_jw_j^2=y_j$, $1\leq j\leq l-1$.
Recall the decomposition (\ref{decomposition of space for a cartan subspace}). 
The centralizer of $w$ in $\Wv=\ss1$ is equal to
\begin{equation}\label{centralizer of v}
\ss1^w=\ss1(\V)^w=\ss1(\V^1)^w\oplus\dots \oplus \ss1(\V^{l-1})^w \oplus \ss1(\V^0\oplus\V^l).
\end{equation}
As checked in the proof of \cite[Theorem 4.5]{PrzebindaLocal} \footnote{The statement of that theorem needs to be modified as follows. ``Let $x\in\g_1$ be semisimple. Then $\g_1^x$ has a basis of $\G^x$-invariant neighborhoods  of $x$ consisting of admissible slices $\Ug_x$ through $x$. If $\ker(x)=0$ then one may choose the $\Ug_x$ so that, for $i=\overline 0,\overline 1$,
\[
\Ug_x\ni y\to y^2|_{\V_i}\in \g_0(\V_i)^{x^2}
\]
is an (injective) immersion."}, there is a slice through $w$ equal to
\[
U_w=(w_1-\epsilon, w_1+\epsilon)u_1 + \dots + (w_{l-1}-\epsilon, w_{l-1}+\epsilon)u_{l-1} + \ss1(\V^0\oplus\V^l),
\]
where $\epsilon>0$ is sufficiently small. In order to indicate its dependence on the graded space $\V$, let us denote by $f_\V(y)$ the function (\ref{extension by the symmetry condition}). Recall that $y=\sum_{j=1}^ly_jJ_j\in W(\G,\h) \bigcup_{\hs1} \tau(\hs1^{reg})$ and let $w_y$ be such that $\tau'(w_y)=y$.  The Lebesgue measure on $\ss1(\V)$ is fixed and the orbital integral $\mu_{\Oo(w_y)}$ is normalized as in (\ref{weyl int on w 1}). We normalize the Lebesgue measure on each $\ss1(\V^j)$ and on  $\ss1(\V^0\oplus\V^l)$ so that via the direct sum decomposition
\[
\ss1(\V)=\ss1(\V^1)\oplus\ss1(\V^2)\oplus\dots\oplus\ss1(\V^{l-1})\oplus\ss1(\V^0\oplus\V^l)
\]
we get the same measure on $\Wv=\ss1(\V)$. Then the $\Sg(\V)$-orbital integral $\mu_{\Oo(w_y)}$ restricts to $U_w$ and the result is the tensor product of $\Sg(\V^j)^w$-orbital integrals and the $\Sg(\V^0\oplus\V^l)$-orbital integral, because $U_w$ is a slice. 
Therefore,
\begin{eqnarray}\label{vanishing of derivatives in general 1}
&&f_\V(y)|_{\Ug_w}\\
&=&P(y)(f_{\V^1}(y_1J_1)|_{(w_1-\epsilon, w_1+\epsilon)u_1}\otimes \dots \otimes f_{\V^{l-1}}(y_{l-1}J_{l-1})|_{(w_{l-1}-\epsilon, w_{l-1}+\epsilon)u_{l-1}}\otimes f_{\V^0\oplus\V^l}(y_lJ_l)),\nn
\end{eqnarray}
where $P(y)$ is a polynomial, whose precise expression may be found from (\ref{product of positive roots for g}).
In (\ref{vanishing of derivatives in general 1})
\[
f_{\V^j}(y_jJ_j)|_{(w_j-\epsilon, w_j+\epsilon)u_j}\in \mathcal D'((w_j-\epsilon, w_j+\epsilon)u_j)\qquad (1\leq j\leq l-1)
\]
and 
\begin{equation}\label{vanishing of derivatives in general 2}
 f_{\V^0\oplus\V^l}(y_lJ_l)\in \mathcal D'(\ss1(\V^0\oplus\V^l)).
\end{equation}
Here $\mathcal D'(\Xv)$ denotes the space of distributions on $\Xv$. 
Since the dimension of a Cartan subalgebra of $\Sg(\V^0\oplus\V^l)|_{\V_{\overline 1}}$ is equal to $l'-l+1$, Lemma \ref{vanishing of derivatives in general} follows from (\ref{vanishing of derivatives in general 1}), (\ref{vanishing of derivatives in general 2}) and Lemma \ref{vanishing of derivatives}. This verifies the claim with $\alpha=(0,\dots,0,k,0,\dots,0)$ with the $k$ on the place $j$. In order to complete the proof we repeat the same argument with the $f$ replaced by $\partial(J)^\beta f$, where $\beta_j=0$.
\end{prf}
In the case $l\leq l'$, Lemmas \ref{extension to non-compact reg} and \ref{vanishing of derivatives in general} provide a further extension of the function $f$ to a continuous function
\begin{eqnarray}\label{extension by the symmetry condition and closure}
&&f:W(\G,\h) \bigcup_{\hs1}\tau(\hs1^{In-reg})\to\Ss^*(\Wv)^\Sg
\end{eqnarray}
which satisfies the symmetry condition (\ref{extension by the symmetry condition}).
 
Let
\begin{equation}\label{number r}
r=\frac{2\,\dim(\g)}{\dim(\V_{\overline 0})},
\end{equation}
where we view both $\g$ and $\V_{\overline 0}$ as vector spaces over $\R$. Explicitly,
\begin{equation}\label{number r 1}
r=\left\{
\begin{array}{lll}
2l-1\ &\text{if}\ \G=\Og_{2l},\\
2l\ &\text{if}\ \G=\Og_{2l+1},\\
l\ &\text{if}\ \G=\Ug_{l},\\
l+\frac{1}{2}\ &\text{if}\ \G=\Sp_{l}.
\end{array}\right.
\end{equation}
Let 
\begin{equation}\label{eq:iota}
\iota=
\begin{cases} 
1 &\text{if $\Dc=\R$ or $\C$}\\
\frac{1}{2} &\text{if $\Dc=\Ha$.}
\end{cases}
\end {equation}
The formula (\ref{number r 1}) together with (\ref{product of positive roots for g})  show that
\begin{equation}\label{relation of r with degree}
\max\{\deg_{y_j}\pi_{\g/\h};\ 1\leq j\leq l\}=\frac{1}{\iota} (r-1),
\end {equation}
where $\deg_{y_j}\pi_{\g/\h}$ denote the degree of $\pi_{\g/\h}(\sum_{j=1}^l y_jJ_j)$ with respect to the variable $y_j$.

The following theorem collects the required properties of the Harish-Chandra almost semisimple orbital integral on $\Wv$.

\begin{thm}\label{pullback of muy 3}
Suppose $l\leq l'$.
Then the closure of the subset $W(\G,\h) \bigcup_{\hs1}\tau(\reg{\hs1})\subseteq \h$ is equal to
\begin{equation*}
\h\cap \tau(\Wv)=\left\{
\begin{array}{lll}
\h & \text{if}\ \Dc\ne \C,\\
W(\G,\h)\{\sum_{j=1}^ly_jJ_j;\ y_1,\dots,y_{\max(l-q,0)}\geq 0\geq y_{\min(p,l)+1},\dots,y_l\}\  & \text{if}\ \Dc =\C.
\end{array}\right.
\end{equation*}
The function (\ref{extension by the symmetry condition and closure})
is smooth on the subset where each $y_j\ne 0$ and, for any multi-index $\alpha=(\alpha_1,\dots,\alpha_l)$ with
\[
\max(\alpha_1,\dots,\alpha_l)\leq \left\{
\begin{array}{lll}
d'-r-1\ &\text{if}\ \Dc=\R\ \text{or}\ \C,\\
2(d'-r)\ &\text{if}\ \Dc=\Ha,\\
\end{array}\right.
\]
the function  $\partial(J^\alpha)f(y)$ extends to a continuous function on $\h\cap\tau(\Wv)$. This extension is equal to zero on the boundary of this set. We shall therefore extend it  from $W(\G,\h) \bigcup_{\hs1}\tau(\hs1^{In-reg})$ to $\h\cap \tau(\Wv)$ by zero.

Suppose now $l>l'$. Then $f$ extends to a smooth function 
\begin{eqnarray}\label{HCintegral l>l'}
f:\h'{}^{In-reg}\cap\tau'(\hs1)\to \C
\end{eqnarray}
and any derivative of $f$ extends to a continuous function on the closure of any connected component of $\h'{}^{In-reg}\cap\tau'(\hs1)$.
\end{thm}
\begin{prf}
The formula for $\mathfrak h \cap \tau(\Wv)$ follows from (\ref{explicit tau and tau' on cartan subspace}), (\ref{the image of hs1 under tau 1}) and  (\ref{classical weyl group action}),  via a case by case verification. The extension of $\partial(J^\alpha)f(y)$ is a consequence of Lemma \ref{vanishing of derivatives in general}. Finally, the extension in the case $ l\leq l'$ is done as in 
Lemma \ref{extension to non-compact reg}.
\end{prf}
In general  we shall write $f_\phi(y)$ for $f(y)(\phi)$.
\section{\bf Limits of orbital integrals.\rm}\label{Limits of orbital integrals}
In this section we consider weighted dilations of the almost semisimple orbital using a positive variable $t$. It turns out that the limit as $t$ tends to $0$ is a constant multiple of the invariant measure on a nilpotent orbit. 
So we begin by describing the nilpotent orbits. It turns out that there is one maximal orbit. 

Let $m$ denote the minimum of $d=\dim_\Dc\,\V_{\overline 0}$ and  the dimension, over $\Dc$, of a maximal isotropic subspace of $\V_{\overline 1}$ with respect to the form $(\cdot,\cdot)'$ in (\ref{super liealgebra}). Recall that the pair $(\G,\G')$ is said to be in the stable range with $\G$ the smaller member if $m=d$.

\begin{lem}\label{structure of t'-1tau(0)}
Let $m$ be as above and  let $d'=\dim_\Dc\,\V_{\overline 1}$.  Let $\SHs_k(\Dc)$ denote the space of the skew-hermitian matrices of size $k$ with coefficients on $\Dc$, $0\leq k\leq m$. 
There are nilpotent orbits $\mathcal O'_k\subseteq\g$ such that
\begin{eqnarray}\label{structure of t'-1tau(0)1}
&&\tau'\tau^{-1}(0)=\mathcal O'_m\cup\mathcal O'_{m-1}\cup\dots\cup\mathcal O'_0,\\
&&\mathcal O'_k\cup\mathcal O'_{k-1}\cup\dots\cup\mathcal O'_0\  \ \text{is the closure of}\ \ \mathcal O'_k \ \ \text{for}\ \ 0\leq k\leq m,\nn\\
&&\dim\mathcal O'_k=d'k\,\dim_\R(\Dc)-2\dim_\R\SHs_k(\Dc).\nn
\end{eqnarray}
Explicitly
\begin{equation}\label{structure of t'-1tau(0)2}
\dim \mathcal O'_k=\left\{
\begin{array}{lll}
kd'-k(k-1)\ &\ \text{if}\ \ \Dc=\R,\\
2kd'-2k^2\ &\ \text{if}\ \ \Dc=\C,\\
4kd'-2k(2k+1)\ &\ \text{if}\ \ \Dc=\Ha.
\end{array}\right.
\end{equation}
Suppose $\Dc=\R$ or $\C$.
Then the partition of $d'$ corresponding to the complexification $\Oo_k{}_\C'=\G_\C'\Oo_k'$ of the orbit $\Oo_k'$ is $\lambda'=(2^{k},1^{d'-2k})$. In other words, the Young diagram  corresponding to the orbit $\Oo_k{}_\C'$ has $k$ rows of length $2$ and $d'-2k$ rows of length $1$.
If $\Dc=\Ha$, then $\Oo_k{}_\C'$ corresponds to the partition $\lambda'=(2^{2k},1^{2d'-4k})$ of $2d'$.

The equality
\begin{equation}\label{structure of t'-1tau(0)3}
\dim\,\mathcal O'_m=\dim\,\Wv-2\,\dim\,\g,
\end{equation}
holds if and only if either the dual pair $(\G,\G')$ is in the stable range with $\G$ - the smaller member or if 
$(\G,\G')$ is one of the following pairs
\[
(\Og_{m+1}, \Sp_{2m}(\R)), \  \ (\Ug_{d'-m}, \Ug_{m,d'-m})\ \ \text{with}\ \ 2m<d'.
\]
\end{lem}
\begin{prf}
As is well known, the variety $\tau'(\tau^{-1}(0))$ is the closure of a single $\G'$ orbit $\mathcal O'_m$, see for example \cite[Lemma (2.16)]{PrzebindaUnipotent}. 

Let $w\in\ss1$ be a nilpotent element. Then
\[
\V=\V^{(0)}\oplus \V^{(1)}\oplus \dots \oplus\V^{(k)},
\]
where $\V^{(0)}$ is the kernel of $w$ and each $(w,\V^{(j)})$, $1\leq j\leq k$ is indecomposable (see \cite[Def. 3.14]{DaszKrasPrzebindaK-S2}). (If $w=0$ then $\V=\V^{(0)}$.) 
The orbit of $w$, call it $\Oo_k$, is of maximal dimension if the kernel of $w$ is minimal, which happens if and only if $k=m$.
Since the only nilpotent element of $\g$ is zero, we have
\[
0=\tau(w)=w^2|_{\V_{\overline 0}}.
\]
Fix $j\geq 1$. Then $(w,\V^{(j)})$ is non-zero and indecomposable. The structure of such elements is well known. In particular we see from  \cite[Prop.5.2(e)]{DaszKrasPrzebindaK-S2}
that there are vectors $v_1, v_3\in \V_{\overline 1}$ and  $v_2\in \V_{\overline 0}$ such that 
\[
\V^{(j)}=\Dc v_2\oplus (\Dc v_1\oplus \Dc v_3),\ \ wv_1=v_2,\ w v_2= v_3,\ w v_3=0.
\]
Hence,
\[
\V_{\overline 1}^{(j)}=\Dc v_1\oplus \Dc v_3,\ \ w^2v_1=v_3,\ w^2 v_3=0
\]
and the decomposition
\[
\V_{\overline 1}=\V_{\overline 1}^{(0)}\oplus \V_{\overline 1}^{(1)}\oplus \dots \oplus\V_{\overline 1}^{(k)},
\]
determines the $\G'$-orbit $\Oo'_k$ of $\tau'(w)=w^2|_{\V_{\overline 1}}$.

The dual pair corresponding to $\Sg|_{\V^{(j)}}$ is $(\Og_1, \Sp_2(\R))$, if $\Dc=\R$, $(\Ug_1, \Ug_{1,1})$, if $\Dc=\C$, $(\Sp_1, \Og^*_{4})$, if $\Dc=\Ha$. The complexifications of these pairs are $(\Og_1, \Sp_2(\C))$, $(\GL_1(\C), \GL_2(\C))$ and $(\Sp_2(\C), \Og_4(\C))$ respectively. 
In particular, this leads to the description of the complexification of the orbit in terms of the Young diagrams, as in \cite{CollMc}.

The closure relations between the orbits $\Oo'_k$ are well known and their dimension may be computed using \cite[Corollary 6.1.4]{CollMc} leading to (\ref{structure of t'-1tau(0)2}). The dimension formula in (\ref{structure of t'-1tau(0)1}) follows from (\ref{structure of t'-1tau(0)2}).

The fact that (\ref{structure of t'-1tau(0)3}) holds if the pair is in the stable range was checked in \cite[Lemma (2.19)]{PrzebindaUnipotent}.
The last statement follows from  (\ref{structure of t'-1tau(0)2}) via a direct computation.
\end{prf}

Now we construct a slice through an element of the maximal nilpotent orbit. 

Recall the non-degenerate bilinear form $\langle\cdot ,\cdot \rangle$ and the automorphism $\theta$ on $\mathfrak s$, \cite[sec. 2.1]{PrzebindaLocal}. (The restriction of $\langle\cdot ,\cdot \rangle$  to $\so$ is a Killing form and the restriction to $\ss1$ is a symplectic form. Also, the restriction of $\theta$  to $\so$ is a Cartan involution and the restriction of $-\theta$ to $\ss1$ is a positive definite compatible complex structure.) In particular the bilinear form $B(\cdot ,\cdot )=-\langle \theta\cdot ,\cdot \rangle$ is symmetric and positive definite. 

Fix an element $N\in \ss1$. Then $N+[\so,N]\subseteq \ss1$ may be thought of as the tangent space at $N$ to the $\Sg$-orbit through $N$. Denote by $[\so,N]^{\perp_B}\subseteq \ss1$ the $B$-orthogonal complement of $[\so,N]$. Since the form $B$ is positive definite,
\begin{equation}\label{decomposition of ss1 with respect to N}
\ss1=[\so,N]\oplus [\so,N]^{\perp_B}.
\end{equation}
Consider the map
\begin{equation}\label{the main submersion}
\sigma:\Sg\times \left(N+ [\so,N]^{\perp_B}\right)\ni (s,u)\to su\in \ss1.
\end{equation}
The range of the derivative of the map $\sigma$ at $(s,u)$ is equal to
\begin{equation}\label{the range of the derivative}
[\so,su]+s [\so,N]^{\perp_B}=s\left([\so,u]+[\so,N]^{\perp_B}\right).
\end{equation}
Let
\begin{equation}\label{definition of U}
U=\{u\in N+ [\so,N]^{\perp_B};\ [\so,u]+[\so,N]^{\perp_B}=\ss1\}.
\end{equation}
The equality (\ref{decomposition of ss1 with respect to N}) implies that  $N\in U$ and $U$ is the slice we were looking for. Next, we consider the orbits passing through $U$. The maximal nilpotent orbit corresponds to a point $N$ and the almost semisimple ones to others points in $U$ which will approach $N$ in a suitable sense, as explained below.

Notice that $U$ is the maximal open neighborhood of $N$ in $N+ [\so,N]^{\perp_B}$ such that the map
\begin{equation}\label{the main submersion on U -1}
\sigma:\Sg\times U\ni (s,u)\to su\in \ss1
\end{equation}
is a submersion. Then $\sigma(\Sg\times U)\subseteq \ss1$ is an open $\Sg$-invariant subset and
\begin{equation}\label{the main submersion on U}
\sigma:\Sg\times U\ni (s,u)\to su\in \sigma(\Sg\times U)
\end{equation}
is a surjective submersion. 

We shall use the map (\ref{the main submersion on U}) to study the $\Sg$-orbital integrals in $\ss1$. This is parallel to Ranga Rao's unpublished study of the orbital integrals in a semisimple Lie algebra, \cite{BarVogAs}.

From now on we assume that $N$ is nilpotent.
\begin{lem}\label{lemma I.4}
The map
\begin{equation}\label{lemma I.4.1}
N+ [\so,N]^{\perp_B}\ni u\to u^2\in \so
\end{equation}
is proper.
\end{lem}
\begin{prf}
We proceed in terms of matrices. Thus $\V_{\overline 0}=\Dc^d$ is a left vector space over $\Dc$ via
\[
av:=v\overline a \qquad (v\in \V_{\overline 0},\ a\in \Dc).
\]
Then $\End_\Dc(\V_{\overline 0})$ may be identified with the space of matrices $M_d(\Dc)$ acting on $\Dc^d$ via left multiplication.
Let 
\[
(v,v')=\overline v^tv' \qquad (v,v'\in \Dc^d).
\]
This is a positive definite hermitian form on $\Dc^d$. The isometry group of this form is
\[
\G=\{g\in M_d(\Dc);\ \overline g^tg=I_d\}.
\]
Similarly, $\V_{\overline 1}=\Dc^{d'}$ is a left vector space over $\Dc$ and
\[
\G'=\{g\in M_{d'}(\Dc);\ \overline g^tFg=F\},
\]
where $F=-\overline F^t\in \GL_{d'}(\Dc)$. This is the isometry group of the form
\[
(v,v')'=\overline v^tFv' \qquad (v,v'\in \Dc^{d'}).
\]
Furthermore we have the identifications
\[
\ss1=\Hom_\Dc(\V_{\overline 0}, \V_{\overline 1})=M_{d',d}(\Dc),
\]
with the symplectic form
\[
\langle w', w\rangle =tr_{\Dc/\R}(w^*w') \qquad (w,w'\in M_{d',d}(\Dc)),
\]
where $w^*=\overline w^tF$. Also, $-\theta(w)=F^{-1}w$ so that
\[
B(w', w)=tr_{\Dc/\R}(\overline w^t w').
\]
Two elements $w,w'\in M_{d',d}(\Dc)$ anticommute if and only if
\begin{equation}\label{anticommutation relations in *}
ww'{}^*+w'w^*=0\ \ \text{and}\ \ w^*w'+w'{}^*w=0.
\end{equation}
From now on we choose the matrix $F$ as follows
\begin{equation}\label{matrix F.I}
F=\left(
\begin{array}{lll}
0 & 0 & I_k\\
0 & F' & 0\\
-I_k & 0 & 0
\end{array}
\right)
\end{equation}
where $0\leq k\leq m$, where $m$ is the minimum of $d$ and the Witt index of the form $(\ ,\ )'$, as in Lemma \ref{structure of t'-1tau(0)}. Then, with the block decomposition of an element of $M_{d',d}(\Dc)=M_{d',k}(\Dc)\oplus M_{d',d-k}(\Dc)$ dictated by (\ref{matrix F.I}),
\[
\left(
\begin{array}{lll}
w_1 & w_4\\
w_2 & w_5\\
w_3 & w_6
\end{array}
\right)^*=
\left(
\begin{array}{lll}
-\overline w_3^t & \overline w_2^tF' & \overline w_1^t\\
-\overline w_6^t & \overline w_5^tF' & \overline w_4^t
\end{array}
\right).
\]
We may choose 
\begin{equation}\label{the nilpotent element}
N=N_k=\left(
\begin{array}{lll}
I_k & 0\\
0 & 0\\
0 & 0
\end{array}
\right)
\end{equation}
Notice that 
\[
[\so,N]^{\perp_B}=\theta\left([\so,N]^{\perp}\right)=\theta\left({}^N\ss1\right)={}^{\theta N}\ss1,
\]
where the second equality is taken from \cite[Lemma 3.5]{PrzebindaLocal}. Hence, a straightforward computation using (\ref{anticommutation relations in *}) shows that
\[
[\so,N]^{\perp_B}=\left\{
\left(
\begin{array}{lll}
0 & 0\\
0 & w_5\\
w_3 & w_6
\end{array}
\right);\ w_3=-\overline w_3^t
\right\}.
\]
The image of $w$ under the map (\ref{lemma I.4.1}) consists of pairs of matrices
\begin{equation}\label{I.first eq}
\left(
\begin{array}{lll}
I_k & 0\\
0 & w_5\\
w_3 & w_6
\end{array}
\right)
\left(
\begin{array}{lll}
I_k & 0\\
0 & w_5\\
w_3 & w_6
\end{array}
\right)^*
=
\left(
\begin{array}{lll}
w_3 & 0 & I_k\\
-w_5\overline w_6^t & w_5\overline w_5^t F' & 0\\
-w_3\overline w_3^t-w_6\overline w_6^t & w_6\overline w_5^t F' & w_3
\end{array}
\right)
\end{equation}
and
\begin{equation}\label{I.second eq}
\left(
\begin{array}{lll}
I_k & 0\\
0 & w_5\\
w_3 & w_6
\end{array}
\right)^*
\left(
\begin{array}{lll}
I_k & 0\\
0 & w_5\\
w_3 & w_6
\end{array}
\right)
=
\left(
\begin{array}{lll}
2w_3 & w_6\\
-\overline w_6^t & \overline w_5^t F' w_5
\end{array}
\right).
\end{equation}
Hence the claim follows.
\end{prf}
Suppose $k=m$. Then it is easy to see from (\ref{the nilpotent element}) and (\ref{xw=0 implies x=0}) that $N=N_m\in \Wv_\g$, or equivalently $U\subseteq \Wv_\g$, if and only if either the pair $(\G,\G')$ is in the stable range with $\G$ the smaller member or 
$(\G,\G')=(\Og_{l'+1},\Sp_{2l'}(\R))$.
\begin{cor}\label{tau is proper}
If $k=m$, then the map
\begin{equation}\label{the map tau I}
\tau: N+[\so,N]^{\perp_B}\ni w\to w^*w\in\g
\end{equation}
is proper.
\end{cor}
\begin{prf}
This follows from the formula (\ref{I.second eq}). Indeed, it is enough to see that the map
\[
w_5\to \overline w_5^t F' w_5
\]
is proper. The variable $w_5$ does not exist unless $\Dc=\C$ and $d>m$. This means that $m$ is the Witt index of the form $(\ ,\ )'$. Hence $iF'$ is a definite hermitian matrix. Therefore the above map is proper.
\end{prf}
\begin{cor}\label{tau is proper again}
Suppose $k=m$.
If $E\subseteq \ss1$ is a subset such that $\tau(E)\subseteq \g$ is bounded, then
\[
E\cap \left( N+[\so,N]^{\perp_B}\right)
\]
is bounded.
\end{cor}
\begin{prf}
This is immediate from Corollary \ref{tau is proper}.
\end{prf}
\begin{lem}\label{muSNK as a tempered homogeneous distribution}
For each $k=0, 1, 2, \dots, m$, the orbital integral $\mu_{\Oo_k}$ is $\Sg$ - invariant and defines a tempered distribution on $\Wv$, homogeneous of degree $\deg \mu_{\Oo_k}=\dim \Oo'_k-\dim \Wv$. 
\end{lem}
\begin{prf}
The stabilizer of the image of $N_k$ in $\V_{\overline 1}$ is a parabolic subgroup $\Pg'\subseteq \G'$ with the Langlands decomposition $P'=\GL_k(\Bbb D)\G''\N'$, where $\G''$ is an isometry group of the same type as $\G'$ and $\N'$ is the unipotent radical. As a $\GL_k(\Bbb D)$ - module, $\n'$, the Lie algebra of $\N'$, is isomorphic to $M_{k,d'-2k}(\Bbb D)\oplus \Hs_k(\Bbb D)$, where $\Hs_k(\Bbb D)\subseteq  M_{k,k}(\Bbb D)$ stands for the space of the hermitian matrices. Hence the absolute value of the determinant of the adjoint action of an element $a\in \GL_k(\Bbb D)$ on the real vector space $\n'$ is equal to 
\[
|\det \Ad(a)_{\n'}|=|\det_\R(a)|^{d'-2k+\frac{2\dim \Hs_k(\Bbb D)}{k\dim_\R\Bbb D}}
\]
Since $\G'=\K'\Pg'$, where $\K'$ is a maximal compact subgroup, the Haar measure on $\G'$ may be written as
\[
dg'=|\det \Ad(a)_{\n'}|\,dk\,da\,dg''\,dn'.
\]
Recall that $da=|\det_\R(a)|^{-k}\,d^+a$, where $d^+a$ stands for the Lebesgue measure on the real vector space $M_{k,k}(\Bbb D)$.
Also,
\[
\frac{2\dim \Hs_k(\Bbb D)}{k\dim_\R\Bbb D}-k=
\begin{cases}
1 &\text{if $\Bbb D=\R$},\\
0  &\text{if $\Bbb D=\C$},\\
-\frac{1}{2} &\text{if $\Bbb D=\Ha$}.
\end{cases}
\]
Hence,
\[
|\det \Ad(a)_{\n'}||\det_\R(a)|^{-k}=|\det_\R(a)|^{d'-2k+\frac{2\dim \Hs_k(\Bbb D)}{k\dim_\R\Bbb D}-k}
\]
is locally integrable on  the real vector space $M_{k,k}(\Bbb D)$.

Since the stabilizer of $N_k$ in $\G'$ is equal to $\G''\N'\subseteq \Pg'$, the $\G'$ orbit of $N_k$ defines a tempered distribution on $\Wv$ by
\[
\int_\Wv \phi(w)\,d\mu_{\G' N_k}(w)=\int_{\GL_k(\Bbb D)}\int_{\K'}\phi(kaN_k)\det \Ad(a)_{\n'}\,dk\,da \qquad (\phi\in \Ss(\Wv).
\]
This distribution is homogeneous of degree
\[
(d'-2k+\frac{2\dim \Hs_k(\Bbb D)}{k\dim_\R\Bbb D})dim_\R\Bbb D -\dim\Wv
\]
Thus it remains to check that
\[
(d'-2k+\frac{2\dim \Hs_k(\Bbb D)}{k\dim_\R\Bbb D})\,k \dim_\R\Bbb D=d'k \dim_\R(\Dc)-2\dim_\R\SHs_k(\Dc),
\]
which is easy, because $M_{k,k}(\Dc)=\Hs_k(\Bbb D)\oplus \SHs_k(\Dc)$. In order to conclude the proof we notice that the orbital integral on the orbit $\Sg N_k$ is the $\G$ - average of the orbital integral we just considered:
\[
\int_\Wv \phi(w)\,d\mu_{_{\Sg N_k}}(w)=\int_\G\int_{\GL_k(\Bbb D)}\int_{\K'}\phi(kaN_kg)\det \Ad(a)_{\n'}\,dk\,da\,dg.
\]
\end{prf}
Now we want to see dilations by $t>0$ in $\ss1$ as transformations in the slice $U$ modulo the action of the group $\Sg$, which is permissible as we consider $\Sg$-orbits.

For $t>0$ let
\[
s_t=\left(
\begin{array}{lll}
t^{-1} & 0 & 0\\
0 & I & 0\\
0 & 0 & t
\end{array}
\right)
\]
where the blocks are as in (\ref{matrix F.I}). Then $s_t\in \G'$. We view $s_t$ as an element $s_t\in \GL(\ss1)$ by
\[
s_t(w)=s_tw \qquad (w\in\ss1).
\]
Also, define an element $M_t\in \GL(\ss1)$ by
\[
M_t(w)=tw  \qquad (w\in\ss1).
\]
Set $g_t=M_t\circ s_t\in \GL(\ss1)$. Thus
\[
g_t(w)=ts_tw \qquad (w\in\ss1).
\]
\begin{lem}\label{lemma about g(t)}
The linear map $g_t\in \GL(\ss1)$ preserves the set $N+[\so,N]^{\perp_B}$ and the subset $U\subseteq N+[\so,N]^{\perp_B}$. In fact
\begin{equation}\label{gt acting on NperpB}
g_t\left(
\begin{array}{lll}
I_k & 0\\
0 & w_5\\
w_3 & w_6
\end{array}
\right)
=
\left(
\begin{array}{lll}
I_k & 0\\
0 & tw_5\\
t^2 w_3 & t^2w_6
\end{array}
\right).
\end{equation}
Hence,
\begin{equation}\label{gt acting on NperpB11}
\tau|_U\circ g_t|_U=M_{t^2}\circ \tau|_U.
\end{equation}
Furthermore
\begin{equation}\label{gt acting on NperpB and sigma}
g_t\circ \sigma=\sigma\circ(\Ad\,s_t \times g_t|_{N+[\so,N]^{\perp_B}}),
\end{equation}
where the $g_t|_{N+[\so,N]^{\perp_B}}$ on the right hand side stands for the restriction of $g_t$ to $N+[\so,N]^{\perp_B}$. 
In particular, the subset $\sigma(\Sg\times U)\subseteq \ss1$ is closed under the multiplication by the positive reals. Also, 
\begin{equation}\label{det gt'}
\det((g_t|_{N+[\so,N]^{\perp_B}})')=t^{\dim\ss1-\dim\Oo'_k}.
\end{equation}
and
\begin{equation}\label{det gt}
\det(g_t')=t^{\dim\ss1}
\end{equation}
\end{lem}
\begin{prf}
The formula (\ref{gt acting on NperpB}) is clear from the definition of $g_t$ and  (\ref{gt acting on NperpB11}) follows from (\ref{gt acting on NperpB}) and (\ref{I.second eq}).

In order to verify (\ref{gt acting on NperpB and sigma})
we notice that for $s\in \Sg$ and $u\in N+[\so,N]^{\perp_B}$ we have
\begin{eqnarray*}
g_t\circ \sigma(s,u)&=&g_t(su)=t(s_ts)u=(s_tss_t^{-1})(ts_tu)\\
&=&\sigma(s_tss_t^{-1}, g_tu)=\sigma\circ(\Ad\,s_t \times g_t|_{N+[\so,N]^{\perp_B}})(s,u).
\end{eqnarray*}
By the Chain Rule, the derivative of $\sigma$ at $(s,u)$ is surjective if and only if  the derivative of $g_t\circ \sigma$ at $(s,u)$ is surjective. Then (\ref{gt acting on NperpB}) shows that this happens if and only if  the derivative of $\sigma$ at $(s_tss_t^{-1},g_tu)$ is surjective. By (\ref{the range of the derivative}), the last statement is equivalent to  the derivative of $\sigma$ being surjective at $(s,g_tu)$. In other words, $g_t$ preserves $U$.

Since $g_t'=M_t'=M_t$ and since $\det s_t=1$, (\ref{det gt}) is obvious. In order to verify (\ref{det gt'}) we proceed as follows. 
The derivative of the map $g_t|_{N+[\so,N]^{\perp_B}}$ coincides with the following linear map
\[
\left(
\begin{array}{lll}
0 & 0\\
0 & w_5\\
w_3 & w_6
\end{array}
\right)
\to
\left(
\begin{array}{lll}
0 & 0\\
0 & tw_5\\
t^2w_3 & t^2w_6
\end{array}
\right).
\]
The determinant of this map is equal to the determinant of the following map
\[
\left(
\begin{array}{lll}
0 & w_4\\
0 & w_5\\
w_3 & w_6
\end{array}
\right)
\to
\left(
\begin{array}{lll}
0 &  tw_4\\
0 & tw_5\\
t^2 w_3 & t w_6
\end{array}
\right),
\]
which equals
\[
t^{2\dim_\R \SHs_k(\Dc)}t^{d'(d-k)\dim_\R\Dc}.
\]
Since, by (\ref{structure of t'-1tau(0)1}),
\[
2\dim_\R \SHs_k(\Dc)+d'(d-k)\dim_\R\Dc=\dim \ss1-\dim \Oo'_k,
\]
(\ref{det gt'}) follows.
\end{prf}
Next we consider an $S$-invariant distribution $F$ on $\sigma(\Sg\times U)$ and its restriction $F|_U$ to the slice $U$. The following lemma proves that the restriction to $U$ of the $t$-dilation of $F$ is equal to $g_t|_U$ applied to $F|_U$.

\begin{lem}\label{invariant distributions and g(t)}
Suppose $F\in \mathcal D'(\sigma(\Sg\times U))^\Sg$. Then the intersection of the wave front set of $F$ with the conormal bundle to $U$ is zero, so that the restriction 
$F|_U$ is well defined. Furthermore, $\sigma^*F=\mu_\Sg\otimes F|_{U}$. Moreover, for $t>0$,
\begin{equation}\label{invariant distributions and g(t)1}
M_t^*F=g_t^*F,
\end{equation}
and
\begin{equation}\label{invariant distributions and g(t)3}
(M_{t}^*F)|_U=(g_t|_U)^*F|_U.
\end{equation}
\end{lem}
\begin{prf}
The wave front set of $F$ is contained in the union of the conormal bundles to the $\Sg$-orbits through elements of $\ss1$. This is because the characteristic variety of the system of differential equations expressing the condition that this distribution is annihilated by the action of the Lie algebra $\so$ coincides with that set. The intersection of this set with the conormal bundle to $U$ is zero. Indeed, this intersection is equal to the orthogonal complement to the sum of the union of the tangent bundles to the orbits and the tangent bundle to $U$, which, by the submersivity of the map $\sigma$, is equal to the whole tangent bundle. Hence,  $F$ restricts uniquely to $U$.  The formula $\sigma^*F=\mu_\Sg\otimes F|_{U}$ follows from the diagram
\[
U\to \Sg\times U\to \sigma(\Sg\times U),\ \ u\to (1,u)\to u,
\]
which shows that the restriction to $U$ equals the composition of $\sigma^*$ and the pullback via the embedding of $U$ into $\Sg\times U$.

Since $s_t^*F=F$ we see that $g_t^*F=M_t^*s_t^*F=M_t^*F$. 
Hence,
\[
(g_t|_U)^*F|_U=(g_t^*F)|_U=(M_t^*F)|_U.
\]
Thus we are done with (\ref{invariant distributions and g(t)1}) and (\ref{invariant distributions and g(t)3}). 
\end{prf}
\begin{lem}\label{invariant distributions and g(t) and limits}
Suppose $F, F_0\in \mathcal D'(\sigma(\Sg\times U))^\Sg$ and $a\in\C$ are such that
\[
t^a(g_{t^{-1}}|_U)^*F|_U \underset{t\to 0+}{\to}F_0|_U.
\]
Then
\[
t^a M_{t^{-1}}^*F\underset{t\to 0+}{\to}F_0
\]
in $\mathcal D'(\sigma(\Sg\times U))$.
\end{lem}
\begin{prf}
This is immediate from (\ref{invariant distributions and g(t)3}) and Proposition \ref{I.1}.
\end{prf}
Now we are ready to compute the limit of the weighted dilatation of the unnormalized almost semisimple orbital integral $\mu_{\mathcal O}$.

\begin{pro}\label{limit of orbits I}
Let $k=m$. Let $\mathcal O\subseteq \sigma(\Sg\times U)$ be an $\Sg$-orbit 
and let $\mu_{\mathcal O}\in \mathcal D'(\ss1)$ be the corresponding orbital integral. Assume $\mu_{\mathcal O}$ is $\Sg$ - invariant.
Then
\begin{eqnarray}\label{the limit of regular orbital integral}
\underset{t\to 0+}{\lim}t^{\deg \mu_{\Oo_m}}\M_{t^{-1}}^*\mu_{\Oo}|_{\sigma(\Sg\times U)}=\mu_{\mathcal O}|_{U}(U)\,\mu_{\Oo_m}|_{\sigma(\Sg\times U)},
\end{eqnarray}
where $\mu_{\Oo_m}\in \mathcal D'(\sigma(\Sg\times U))$ is the orbital integral on the orbit $\Oo_m=\Sg N_m$ normalized so that $\mu_{\Oo_m}|_{U}$ is the Dirac delta at $N_m$ and the convergence is in $\mathcal D'(\sigma(\Sg\times U))$. This orbital integral is a homogeneous distribution of degree $\deg \mu_{\Oo_m}=\dim \Oo_m'-\dim \Wv$.
\end{pro}
A few remarks before the proof. 
The scalar $\mu_{\mathcal O}|_{U}(U)$ may be thought of as the volume of the intersection $\mathcal O\cap U$. This volume is finite because the restriction $\mu_{\mathcal O}|_{U}$ is a distribution on $U$ with the support equal to the closure of $\mathcal O\cap U$, which is compact by Corollary \ref{tau is proper again}. Hence $\mu_{\mathcal O}|_{U}$ applies to any smooth function on $U$, which may be chosen to be constant on $\mathcal O\cap U$. 

A straightforward argument shows that every regular orbit $\mathcal O$ passes through $U$, i.e. is contained in $\sigma(\Sg\times U)$, if the pair $(\G, \G')$ is in the stable range.

Our normalization of $\mu_{\Sg N_m}$ does not depend on the normalization of $\mu_{\mathcal O}$, which is absorbed by the factor $\mu_{\mathcal O}|_{U}(U)$. 
\begin{prf}
By  (\ref{det gt'}) 
\begin{eqnarray*}
&&t^{\dim \Oo'_m-\dim\ss1}(g_{t^{-1}}|_U)^*\mu_{\mathcal O}|_{U}(\psi)
=\mu_{\mathcal O}|_{U}(\psi\circ g_{t}|_U).
\end{eqnarray*}
We see from (\ref{gt acting on NperpB}) that
\[
\underset{t\to 0}{\lim}\ g_tu=N_m \qquad (u\in U).
\]
Hence, for any $\psi\in C_c^\infty(U)$,
\[
\underset{t\to 0}{\lim}\ \mu_{\mathcal O}|_{U}(\psi\circ g_t)=\mu_{\mathcal O}|_{U}(\psi(N)\Bbb I_U)
=\mu_{\mathcal O}|_{U}(\Bbb I_U)\psi(N)=\mu_{\mathcal O}|_{U}(U)\psi(N_m),
\]
where $\Bbb I_U$ is the indicator function of $U$. Thus (\ref{the limit of regular orbital integral}) follows from Lemma \ref{invariant distributions and g(t) and limits}.
\end{prf}
In order to find the limit of the weighted dilations of the normalized almost semisimple orbital integral 
$f(y)$, see Corollary \ref{limit of orbital integrals and Harish-Chandra integral on W}, we still need to compute the weight.

A direct computation involving the formulas (\ref{product of positive roots for g})  and (\ref{product of positive roots for g'/z'}) verifies the following lemma. 
\begin{lem}\label{a relation between degrees}
Suppose $d\leq d'$. Then
\begin{equation}\label{a relation between degrees 1}
\dim\,\Wv=\dim\,\g+\dim\,\g'/\z' +\dim\,\h+\dim\,\ss1(\V^0),
\end{equation}
or equivalently
\begin{equation}\label{a relation between degrees 2}
\dim\,\Wv=2\, \deg\,\pi_{\g/\h}+2\, \deg\,\pi_{\g'/\z'} +2\,\dim\,\h+\dim\,\ss1(\V^0).
\end{equation}
Here $\dim\,\ss1(\V^0)=0$ unless $\G=\Og_{2l+1}$, $\G'=\Sp_{2l'}(\R)$ and $d=2l+1<2l'=d'$.

If $d<d'$, then $\dim\,\ss1(\V^0)=0$ and 
\begin{equation}\label{a relation between degrees 1'}
\dim\,\Wv=\dim\,\g'+\dim\,\g/\z +\dim\,\h,
\end{equation}
or equivalently
\begin{equation}\label{a relation between degrees 2'}
\dim\,\Wv=2\, \deg\,\pi_{\g'/\h}+2\, \deg\,\pi_{\g/\z} +2\,\dim\,\h.
\end{equation}
\end{lem}
Recall Harish-Chandra's semisimple orbital integral on $f(y)\in \Ss^*(\Wv)^\Sg$, (\ref{extension by the symmetry condition and closure}) and (\ref{HCintegral l>l'}).
Lemma \ref{a relation between degrees} plus a direct computation implies the following lemma.
\begin{lem}\label{dilation relations 1}
For $\phi\in C_c^\infty(\Wv)$ and $t>0$
\begin{equation}\label{dilation relations 2}
\mu_{\Oo(w)}(t^{\dim\Wv}\phi_t)=t^{\dim\ss1(\V^0)}\mu_{\Oo(t^{-1}w)}(\phi) \qquad (w\in\reg\h_{\overline 1}).
\end{equation}
Equivalently,
\begin{equation}\label{dilation relations 1.10}
M_{t^{-1}}^*f(y)=t^{2\,\deg\,\pi_{\g/\z}+2\dim \h}f(t^2y).
\end{equation}
\end{lem}
Also, without any assumptions, we have the following equivalent formulas
\begin{eqnarray}\label{dilation relations 3}
&&(\psi\circ\tau')_t=t^{2\,\dim\,\g'-\dim\,\Wv}\psi_{t^2}\circ\tau' \qquad (\psi\in \Ss(\g')),\\
&&\tau'_*(M_{t^{-1}}^* u)=t^{\dim\Wv-2\dim\g'} M_{t^{-2}}^*\tau'_*(u) \qquad (u\in \Ss^*(\Wv)).\nn
\end{eqnarray}
\begin{cor}\label{limit of orbital integrals and Harish-Chandra integral on W}
Assume that $k=m$. Then, 
\begin{equation}\label{second limit formula}
\underset{t\to 0+}{\lim}\ t^{\deg\mu_{\Oo_m}} M_{t^{-1}}^*f(y)|_{\sigma(\Sg\times U)}=
f(y)|_{U}(U)\,\mu_{\Oo_m}(\phi)|_{\sigma(\Sg\times U)}.
\end{equation}
The distribution $f(0)$ is homogeneous of degree $\deg f(0)=-(2\deg \pi_{\g/\z}+2\dim\h)$. (If $l\leq l'$, then  $2\deg \pi_{\g/\z}+2\dim\h=\dim \g+\dim\h$.) 
Moreover (\ref{second limit formula}) is equivalent to
\begin{eqnarray}\label{first limit formula}
\underset{t\to 0+}{\lim}\ t^{\frac{1}{2}(\deg\mu_{\Oo_m}-\deg f(0))} f(ty)|_{\sigma(\Sg\times U)}
=f(y)|_{U}(U)\,\mu_{\Oo_m}|_{\sigma(\Sg\times U)}.
\end{eqnarray}
\end{cor}
\begin{prf}
The statement (\ref{second limit formula}) is immediate from Proposition \ref{limit of orbits I}. 
In order to see that (\ref{second limit formula}) is equivalent to (\ref{first limit formula}) we recall (\ref{dilation relations 1.10}),
which also shows that $\deg f(0)=-2\deg \pi_{\g/\z}-2\dim\h$.
\end{prf}
The set of almost semisimple elements in $\ss1$ coincides with the union of the $\Sg$-orbits through the generalized Cartan subspaces
$\bigcup_{\t\h_{\overline 1}}\Sg\t\h_{\overline 1}$, \cite[(47)]{McKeePasqualePrzebindaSuper}. Since the set of the regular almost semisimple elements is dense in $\bigcup_{\t\h_{\overline 1}}\Sg\t\h_{\overline 1}$, \cite[Theorem 19]{McKeePasqualePrzebindaSuper} implies that the set of the regular almost semisimple elements is dense in $\ss1$. Hence there is a regular $y$ such that the corresponding orbit in $\Wv$ is contained in $\sigma(\Sg\times U)$. Then all the orbits corresponding to $ty$, with $t\geq 0$, are contained in $\sigma(\Sg\times U)$.

Next we shall try to shed some light at the limits of  the derivatives of the orbital integrals.
We assume that $l\leq l'$. 
As in \cite{HC-57DifferentialOperators} we identify the symmetric algebra on $\g$ with $\C[\g]$, the algebra of the polynomials on $\g$ using the invariant symmetric bilinear form $B$ on $\g$.
\begin{lem}\label{the main localized to U lemma}
Let $y\in \h\cap \tau(\Wv)$ and let $Q\in \C[\h]$ be such that $\deg (Q)$ is small enough so that, by Corollary \ref{pullback of muy 3}, $\partial(Q)f(y)$ exists. Then
\begin{equation}\label{the main localized to U lemma1}
t^{\deg\mu_{\Oo_{m}}} M_{t^{-1}}^*\partial(Q)f(y)|_{\sigma(\Sg\times U)}
\underset{t\to 0+}{\to}C\mu_{\Oo_m},
\end{equation}
in $\mathcal D'(\sigma(\Sg\times U))$, where $C=\partial(Q)f(y)|_{U}(\Bbb I_U)$ is the value of the  compactly supported distribution $\partial(Q)f(y)|_{U}$ on $U$ applied to the constant function $\Bbb I_U$.
\end{lem}
\begin{prf}
We see from Lemma \ref{invariant distributions and g(t)} that it suffices to prove the lemma with (\ref{the main localized to U lemma1}) replaced by
\begin{equation}\label{the main localized to U lemma1'}
t^{\deg\mu_{\Oo_{m}}} g_{t^{-1}}|_U^*\partial(Q)f(y)|_{U}
\underset{t\to 0+}{\to}C\delta_{N_m},
\end{equation}
Let $\psi\in C_c^\infty(U)$. Since $\partial(Q)f(y)|_{U}$ is a compactly supported distribution on $U$,
\begin{eqnarray*}
t^{\deg\mu_{\Oo_{m}}} g_{t^{-1}}|_U^*\partial(Q)f(y)|_{U}(\psi)
&=&\partial(Q)f(y)|_{U}(\psi\circ g_t)\\
&\underset{t\to 0+}{\to}&\partial(Q)f(y)|_{U}(\psi(N_m)\Bbb I_U)\\
&=&\partial(Q)f(y)|_{U}(\Bbb I_U)\delta_{N_m}(\psi).
\end{eqnarray*}
\end{prf}
\begin{pro}\label{the main limit pro}
Let $y\in \h\cap \tau(\Wv)$ and let $Q\in \C[\h]$ be such that $\deg (Q)$ is small enough so that, by Corollary \ref{pullback of muy 3}, $\partial(Q)f(y)$ exists. Then
\begin{equation}\label{the main limit pro1}
t^{\deg\mu_{\Oo_{m}}} M_{t^{-1}}^*\partial(Q)f(y)
\underset{t\to 0+}{\to}C\mu_{\Oo_m}
\end{equation}
in the topology of $\Ss^*(\Wv)$, where $C=\partial(Q)f(y)|_{U}(\Bbb I_U)$. Moreover, there is a seminorm $q$ on $\Ss(\Wv)$ and $N\geq 0$ such that
\begin{multline}\label{the main limit pro1'}
\left|t^{\deg\mu_{\Oo_{m}}} M_{t^{-1}}^*\partial(Q)f(y)(\phi)\right|\leq (1+|y|)^N q(\phi)\\ 
\qquad (0<t\leq 1,\ y\in \h\cap \tau(\Wv),\ \phi\in \Ss(\Wv)).
\end{multline}
\end{pro}
\begin{prf}
As we have seen in (\ref{rosmann}) and (\ref{singular integral as derivative of Harish-Chandra integral}), there is a positive constant $const$ such that for $\psi\in \Ss(\g')$
\begin{eqnarray}\label{the main limit pro2}
&&\tau'_*(\partial(Q)f(y))(\psi)=\partial(Q)\tau'_*(f(y))(\psi)\\
&=&const\ \partial(Q\t\pi_{\z'/\h'})\left(\pi_{\g'/\h'}(y+y'')\int_{\G'}\psi(g.(y+y''))\,dg\right)|_{y''=0},\nn
\end{eqnarray}
where $\t\pi_{\z'/\h'}=\pi_{\z'/\h'}^{short}$ if $\G=\Og_{2l+1}$ with $l<l'$, and $\t\pi_{\z'/\h'}=\pi_{\z'/\h'}$ otherwise. Let $P\in\C[\g']^{\G'}$. Then
\begin{eqnarray}\label{the main limit pro3}
&&\partial(Q\t\pi_{\z'/\h'})\left(\pi_{\g'/\h'}(y+y'')\int_{\G'}(P\psi)(g.(y+y''))\,dg\right)|_{y''=0}\\
&=&\partial(Q\t\pi_{\z'/\h'})\left(P(y+y'')\pi_{\g'/\h'}(y+y'')\int_{\G'}\psi(g.(y+y''))\,dg\right)|_{y''=0}.\nn
\end{eqnarray}
By commuting the operators of the multiplication by a polynomial with differentiation we may write
\[
\partial(Q\t\pi_{\z'/\h'})P(y+y'')=\sum_{|\alpha|\leq\deg(Q\t\pi_{\z'/\h'})}P_\alpha(y+y'')\partial^\alpha,
\]
where $\partial^\alpha=\prod_{j=1}^{l'}\partial(J_j')^{\alpha_j}$. Hence, (\ref{the main limit pro3}) is equal to
\begin{eqnarray}\label{the main limit pro4}
\sum_{|\alpha|\leq\deg(Q\t\pi_{\z'/\h'})}P_\alpha(y)\partial^\alpha\left(\pi_{\g'/\h'}(y+y'')\int_{\G'}\psi(g.(y+y''))\,dg\right)|_{y''=0}.
\end{eqnarray}
We see from (\ref{the main limit pro2}) - (\ref{the main limit pro4}) that the range of the map
\begin{eqnarray}\label{the main limit pro5}
\C[\g']^{\G'}\ni P \to \tau'_*(\partial(Q)f(y))\circ P\in \Ss^*(\g')
\end{eqnarray}
is contained in the space spanned by the distributions
\[
\partial^\alpha\left(\pi_{\g'/\h'}(y+y'')\int_{\G'}\psi(g.(y+y''))\,dg\right)|_{y''=0} \qquad (|\alpha|\leq\deg(Q\t\pi_{\z'/\h'}).
\]
In particular this range is finite dimensional. Therefore the distribution \eqref{the main limit pro2} is annihilated by an ideal of finite co-dimension in $\C[\g']^{\G'}$. Hence the Fourier transform
\begin{equation}\label{the main limit pro2'}
\left(\tau'_*(\partial(Q)f(y))\right)\hat{}\in\Ss^*(\g')
\end{equation}
is annihilated by an ideal of finite co-dimension in $\partial(\C[\g']^{\G'})$.
Now Harish-Chandra Regularity Theorem \cite[Theorem 1, page 11]{HC-65InvariantEigendistributionsLieAlg} implies that the distribution \eqref{the main limit pro2'}
is a locally integrable function whose restriction to the set of the regular semisimple elements has a known structure. Specifically, let $\h'_1\subseteq\g'$ be a Cartan subalgebra and let $\pi_1$ be the product of the positive roots (with respect to some order of the roots). Then Harish-Chandra's formula for the radial component of a $\G'$-invariant differential operator with constant coefficients on $\g'$ together with \cite[Lemma 19]{HC-64b} show that the restriction
\[
\pi_1\left(\tau'_*(\partial(Q)f(y))\right)\hat{}\,|_{\reg{\h'_1}}
\]
is annihilated by an ideal of finite co-dimension in $\partial(\C[\h_1'])$. Hence, for any connected component $C(\reg{\h_1'})\subseteq \reg{\h_1'}$ there is an exponential polynomial $\sum_j p_je^{\lambda_j}$. such that
\begin{equation}\label{the main limit pro2''}
\left(\tau'_*(\partial(Q)f(y))\right)\hat{}\,|_{C(\reg{\h'_1})}=\frac{1}{\pi_1}\sum_j p_je^{\lambda_j}.
\end{equation}
Let
\[
F(x)=\sum_j p_j(x)e^{\lambda_j(x)} \qquad (x\in C(\reg{\h_1'})).
\]
This function extends analytically beyond the connected component and for any $k=1,2,3,...$ we have Taylor's formula, as in \cite{Hormander}, 
\begin{equation}\label{the main limit pro2'''}
F(x)=\sum_{|\alpha|<k}\partial^\alpha F(0)\frac{x^\alpha}{\alpha!} +k\int_0^1(1-t)^{k-1}\sum_{|\alpha|=k}\partial^\alpha F(tx)\,dt\frac{x^\alpha}{\alpha!}.
\end{equation}
Since the distribution \eqref{the main limit pro2'} is tempered, we see from Harish-Chandra's theory of the orbital integrals
\begin{equation}\label{the main limit pro2'''''}
\psi(x)=\pi_1(x)\int_{\G'/\H'}\psi(g.x)\,dg\H' \qquad (\psi\in \Ss(\g'))
\end{equation}
that the real parts of the $\lambda_j$ are non-positive on the $C(\reg{\h'_1})$. Furthermore, they depend linearly on $y$ and the $p_j$ depend polynomially on the $y$. Therefore a straightforward argument shows that there is $N>0$ such that
\begin{equation}\label{the main limit pro2''''}
\left|\partial^\alpha F(tx)\right|\leq constant\, (1+|y|)^{N}(1+|x|)^{N}\sum_{|\alpha|=k}\left|\frac{x^\alpha}{\alpha!}\right|.
\end{equation}
Therefore \eqref{the main limit pro2'} is a finite sum of homogeneous distributions, of possibly negative degrees, plus the error term which is bounded by \eqref{the main limit pro2''''}.
Therefore 
there is an integer $a$ such that the following limit exists in $\Ss^*(\g')$:
\begin{eqnarray}\label{the main limit pro6}
\underset{t\to 0+}{\lim}\ t^aM_t^*\left(\tau'_*(\partial(Q)f(y))\right)\hat{}.
\end{eqnarray}
Moreover,  there is a seminorm $q$ on $\Ss(\g')$ and $N\geq 0$ such that
\begin{multline}\label{the main limit pro1''}
\left|t^aM_t^*\left(\tau'_*(\partial(Q)f(y))\right)\hat{}(\psi)\right|\leq (1+|y|)^N q(\psi)\\ 
\qquad (0<t\leq 1,\ y\in \h\cap \tau(\Wv),\ \psi\in \Ss(\g')).
\end{multline}
(All we did here was an elaboration of the argument used in the proof of \cite[Theorem 3.2]{BarVogAs}.)
By taking the inverse Fourier transform we see that 
there is an integer $b$ such that the following limit exists in $\Ss^*(\g')$:
\begin{eqnarray}\label{the main limit pro7}
\underset{t\to 0+}{\lim}\ t^b M_{t^{-1}}^*\tau'_*(\partial(Q)f(y)).
\end{eqnarray}
Moreover,  there is a seminorm $q$ on $\Ss(\g')$ and $N\geq 0$ such that
\begin{multline}\label{the main limit pro1'''}
\left|\underset{t\to 0+}{\lim}\ t^b M_{t^{-1}}^*\tau'_*(\partial(Q)f(y))(\psi)\right|\leq (1+|y|)^N q(\psi)\\ 
 \qquad (0<t\leq 1,\ y\in \h\cap \tau(\Wv),\ \psi\in \Ss(\g')).
\end{multline}
But then the injectivity of the map $\tau'_*$, see Corollary \ref{injectivity of pushforward via tau'}, and (\ref{dilation relations 3}) imply that there is an integer $d$ such that the following limit exists in $\Ss^*(\Wv)$.
\begin{eqnarray}\label{the main limit pro8}
\underset{t\to 0+}{\lim}\ t^d M_{t^{-1}}^*\partial(Q)f(y).
\end{eqnarray}
Now Lemma \ref{the main localized to U lemma} shows that $d=\deg\mu_{\Oo_m}$ and the proposition follows.
\end{prf}
\section{\bf Intertwining distributions.\rm}\label{Intertwining distributions}

In the following we consider an irreducible unitary representation $\Pi$ of $\wt\G$. We suppose that $\Pi$ is genuine in the sense that it is non-trivial on the kernel of the covering map $\wt\G \to \G$. Let $\mu \in i\h^*$ represent the infinitesimal character of $\Pi$. In particular, when $\mu$ is dominant, then we will refer to it as the Harish-Chandra parameter of $\Pi$. This is consistent with the usual terminology; see e.g. \cite[Theorem 9.20]{knappLie2}.

Assume that the distribution character $\Theta_\Pi$ is supported in the preimage $\wt{\G_1}$ of the Zariski identity component $\G_1$ of $\G$. (Recall that $\G_1=\G$ unless $\G$ is an orthogonal group.) Then
\begin{equation}\label{Weyl character formula}
\Theta_\Pi(h)\Delta(h)=\sum_{s\in W(\G,\h)} \sgn_{\g/\h}(s)\xi_{s\mu}(h) \qquad (h\in \diesis{\H}_o),
\end{equation}
where we lift $\Pi$ from $\wt\G$ to $\diesis{\G}$ via the covering (\ref{acceptable covering}) if necessary, $\Delta$ is the Weyl denominator  (\ref{eq:Weyl-den})   and $\sgn_{\g/\h}$ is as in (\ref{sgn_g/h}). Since $\sgn_{\g/\h}=\sgn$, the sign character of the Weyl group $W(\g,\h)$, unless $\G=\Og_{2l}$, we need to justify the formula (\ref{Weyl character formula}) only in this case. Our assumption that the distribution character $\Theta_\Pi$ is supported in the preimage $\wt{\G_1}$ implies that 
\[
\Theta_\Pi(h)=\frac{\sum_{s\in W(\g,\h)} \sgn(s)\xi_{s\mu}(h)}{\Delta(h)}
+\frac{\sum_{s\in W(\g,\h)} \sgn(s)\xi_{s\mu}(th)}{\Delta(th)} \qquad (h\in \diesis{\H}_o),
\]
where $W(\g,\h)$ is the Weyl group of $\G_1=\SO_{2l}$ and $t$ is any element of $W(\G,\h)$ which does not belong to $W(\g,\h)$. Since 
$\sgn(s)=\sgn_{\g/\h}(s)$ for $s\in W(\g,\h)$ and
\[
\Delta(th)=\sgn_{\g/\h}(t)\Delta(h),
\]
(\ref{Weyl character formula}) follows. 

In this section we study the analytic properties of the distribution $f_{\Pi\otimes\Pi'}=T(\check\Theta_\Pi) \in  \Ss(\Wv)$ introduced in (\ref{0.1}). 
For $x\in \g$ define
\begin{equation}\label{ch}
\ch(x)=|\det_\R(x-1)|^{1/2}\,,
\end{equation}
where the subscript $\R$ indicates that the element $x\in \g\subseteq \End(\V_{\overline 0})$ is viewed as an endomorphism of $\V_{\overline 0}$ over $\R$. Since  all eigenvalues of $x \in \g$ are purely imaginary, $x-1$ is invertible and the function $\ch(x)$ is non-zero on $\g$ and we can raise it to any real power.

For an endomorphism $x$ of $\Wv$ we set 
\begin{equation}
\label{eq:chix}
\chi_x\big(\frac{1}{4}(\langle xw,w\rangle)\big) \qquad w \in \Wv\,.
\end{equation}
This definition coincides with (\ref{eq:chicg}) when $x=c(g)$ for $g \in \Sp$ and $g-1$ is invertible.

Recall the functions $\t c$ and $\t c_-$ on $\mathfrak{sp}$ to $\wt \Sp$ introduced in
 (\ref{eq:tildec}) and the constants $r$ and $\iota$ defined by (\ref{number r}) and (\ref{eq:iota}), respectively. Recall also that $d'=\dim_{\mathbb D} \V_{\overline 1}$ and that
$\wt \G$ is equipped with the Haar measure $d\wt g$ of total mass $1$.

\begin{lem}\label{general formula for the int distr}
Let $\diesis{c_-}:\h\to \diesis{\H}_o$ be a real analytic lift of $\t c_-:\h\to \wt\H$, via the covering (\ref{acceptable covering}).
For any $\phi\in \Ss(\Wv)$
\begin{eqnarray*}
T(\check\Theta_\Pi)(\phi)&=&C\,\int_\h\left((\check\Theta_\Pi\Delta)(\diesis{c_-}(x)) \ch^{d'-r-\iota}(x)\right)
\left(\pi_{\g/\h}(x)\int_\Wv \chi_x(w)\phi^\G(w)\,dw\right)\,dx\\
&=&C'\int_\h\left(\xi_{-\mu}(\diesis{c_-}(x)) \ch^{d'-r-\iota}(x)\right)
\left(\pi_{\g/\h}(x)\int_\Wv \chi_x(w)\phi^\G(w)\,dw\right)\,dx,
\end{eqnarray*}
where $\phi^\G(w)=\int_\G\phi(gw)\,dg$, $C$ is a non-zero constant, $C'=C|W(\G,\h)|$ and each consecutive integral is absolutely convergent.
\end{lem}
\begin{prf}
By definition
\begin{equation}\label{first sterp}
T(\check\Theta_\Pi)(\phi)=\int_{\wt\G}\check\Theta_\Pi(\wt g) T(\wt g)(\phi)\,d\wt g,
\end{equation}
where the integral is absolutely convergent because both, the character and the function $T(\wt g)(\phi)$ are bounded (see for example \cite[Proposition 1.13]{PrzebindaUnitary}) and the group $\wt\G$ is compact.

Since the support of the character is contained in $\wt\G_1$ and since the image of the Cayley transform $ c_-:\g\to\G$ is contained and dense in $\G_1$, we may integrate over $\g$ rather than $\wt\G$ in (\ref{first sterp}). As checked in \cite[(3.11)]{PrzebindaUnipotent}, the Jacobian of $\t c_-(x)=\t c(x)\t c(0)^{-1}$ is a constant multiple of $\ch^{-2r}(x)$. 
Also, the element $\t c(0)$ is in the center of the metaplectic group. In particular,  $\check\Theta_\Pi(\t g\t c(0)^{-1})$ is a constant multiple of $\check\Theta_\Pi(\t g)$, where the constant has the absolute value $1$. Thus
\begin{equation}\label{second step}
T(\check\Theta_\Pi)(\phi)=C_1\int_{\g}\check\Theta_\Pi(\t c_-(x)) T(\t c(x))(\phi) \ch^{-2r}(x)\,dx,
\end{equation}
where $C_1$ is a non-zero constant and, by (\ref{the omega}), 
$T(\t c(x))=\Theta(\t c(x))\chi_x \mu_\Wv$.
Since $c(g.x)=g.c(x)$, there is $s \in\wt \Sp$ in the preimage of $1\in \Sp$ such that 
$s\t g \t c(x)\t g^{-1}=\t c(g.x)$. Since $\Pi$ occurs in Weil representation, for every $\t h\in \wt \Sp$ we have $\check\Theta_\Pi(s\t h)\Theta(s \t h)=\check\Theta_\Pi(\t h)\Theta(\t h)$. 
As $\check\Theta_\Pi$ and $\Theta$ are characters of $\wt G$,  it follows that 
$\check\Theta_\Pi(\t c(g.x))\Theta(\t c(g.x))=\check\Theta_\Pi(\t c(x))\Theta(\t c(x))$.
Weyl integration formula on $\g$ shows that
\begin{equation}\label{third step}
T(\check\Theta_\Pi)(\phi)=C_2\int_{\h}|\pi_{\g/\h}(x)|^2\check\Theta_\Pi(\t c_-(x)) T(\t c(x))(\phi^\G) \ch^{-2r}(x)\,dx,
\end{equation}
where $C_2$ is a non-zero constant and $\phi^\G$ is as in the statement of the Lemma.
Recall \cite[Lemma 5.7]{PrzebindaUnitary} that $\pi_{\g/\h}(x)$ is a constant multiple of $\Delta(\diesis{c_-}(x)) \ch^{r-\iota}(x)$. Also $\overline{\pi_{\g/\h}(x)}$ is a constant multiple of $\pi_{\g/\h}(x)$. Hence, if we set $\check\Theta_\Pi(\diesis{c_-}(x))=\check\Theta_\Pi(\t c_-(x))$, then
\begin{equation}\label{forth sterp}
T(\check\Theta_\Pi)(\phi)=C_3\int_{\h}(\check\Theta_\Pi\Delta)(\diesis{c_-}(x)) \ch^{r-\iota}(x)\pi_{\g/\h}(x)T(\t c(x))(\phi^\G) \ch^{-2r}(x)\,dx,
\end{equation}
where $C_3$ is a non-zero constant. Observe that, by  (\ref{metaplectic group}) and (\ref{the omega}),  
\begin{eqnarray*}
\Theta(\t c(x))^2=i^{\dim \Wv} \det(J_{c(x)})^{-1}_\Wv=i^{\dim \Wv} \det(c(x)-1)^{-1}_\Wv=i^{\dim \Wv} \det\big(2^{-1}(x-1)\big)_\Wv\,.
\end{eqnarray*}
Since the determinant is taken on $\Wv$, (\ref{ch}) implies that $\Theta(\t c(x))$ is a constant multiple of $\ch^{d'}(x)$, where  $d'=\dim_{\mathbb D} \V_{\overline{1}}$ as before.
Hence, by (\ref{Weyl character formula}),
\begin{eqnarray}\label{fifth sterp}
&&T(\check\Theta_\Pi)(\phi)=C_4\int_{\h}(\check\Theta_\Pi\Delta)(\diesis{c_-}(x)) \ch^{d'-r-\iota}(x)\pi_{\g/\h}(x)\int_\Wv\chi_x(w)\phi^\G(w)\,dw\,dx\\
&=&C_4\sum_{s\in W(\G,\h)}\sgn_{\g/\h}(s)\int_\h\left(\xi_{-s\mu}(\diesis{c_-}(x)) \ch^{d'-r-\iota}(x)\right)
\left(\pi_{\g/\h}(x)\int_\Wv \chi_x(w)\phi^\G(w)\,dw\right)\,dx,\nn
\end{eqnarray}
where $C_4$ is a non-zero constant and each consecutive integral is absolutely convergent. Notice that for $s\in W(\G,\h)$
\begin{eqnarray*}
&&\int_\Wv \chi_{sx}(w)\phi^\G(w)\,dw=\int_\Wv \chi_x(s^{-1}w)\phi^\G(w)\,dw=\int_\Wv \chi_x(s^{-1}w)\phi^\G(s^{-1}w)\,dw\\
&=&\int_\Wv \chi_x(w)\phi^\G(w)\,dw
\end{eqnarray*}
and that
$
\pi_{\g/\h}(sx)=\sgn(s)\,\pi_{\g/\h}(x)
$.
Therefore
\begin{eqnarray*}
&&\sum_{s\in W(\G,\h)}\sgn_{\g/\h}(s)\int_\h\left(\xi_{-s\mu}(\diesis{c_-}(x)) \ch^{d'-r-\iota}(x)\right)
\left(\pi_{\g/\h}(x)\int_\Wv \chi_x(w)\phi^\G(w)\,dw\right)\,dx\\
&=&\sum_{s\in W(\G,\h)}\int_\h\left(\xi_{-s\mu}(\diesis{c_-}(x)) \ch^{d'-r-\iota}(x)\right)
\left(\pi_{\g/\h}(sx)\int_\Wv \chi_{sx}(w)\phi^\G(w)\,dw\right)\,dx\\
&=&\sum_{s\in W(\G,\h)}\int_\h\left(\xi_{-s\mu}(\diesis{c_-}(s^{-1}x)) \ch^{d'-r-\iota}(s^{-1}x)\right)
\left(\pi_{\g/\h}(x)\int_\Wv \chi_{x}(w)\phi^\G(w)\,dw\right)\,dx\\
&=&\sum_{s\in W(\G,\h)}\int_\h\left(\xi_{-\mu}(\diesis{c_-}(x)) \ch^{d'-r-\iota}(x)\right)
\left(\pi_{\g/\h}(x)\int_\Wv \chi_{x}(w)\phi^\G(w)\,dw\right)\,dx\\
&=&|W(\G,\h)|\int_\h\left(\xi_{-\mu}(\diesis{c_-}(x)) \ch^{d'-r-\iota}(x)\right)
\left(\pi_{\g/\h}(x)\int_\Wv \chi_{x}(w)\phi^\G(w)\,dw\right)\,dx.
\end{eqnarray*}
We can verify the absolute convergence as follows.
The boundedness of the function $T(\wt g)(\phi)$, $\t g\in \wt\G$, means that there is a seminorm $q(\phi)$ on $\Ss(\g)$ such that
\begin{equation}\label{sixth sterp}
\Big|\Theta(\t c(x))\int_\Wv\chi_x(w)\phi(w)\,dw\Big|\leq q(\phi) \qquad (x\in \g).
\end{equation}
Equivalently, replacing $q(\phi)$ by a constant multiple, 
\begin{equation}\label{seventh sterp}
\Big|\int_\Wv\chi_x(w)\phi(w)\,dw\Big|\leq q(\phi)\ch^{-d'}(x) \qquad (x\in \g).
\end{equation}
(This is the van der Corput estimate, \cite[formula (23) on page 345]{Stein}.)
Also,  (\ref{relation of r with degree})  and (\ref{ch explicit}) below imply that
\[
|\pi_{\g/\h}(x)|\leq C_5 \ch^{r-1}(x) \leq C_5 \ch^{r-\iota}(x) \qquad (x\in\h),
\]
where $C_5$ is a constant.
Hence
\begin{equation}\label{seventh sterp}
\Big|\pi_{\g/\h}(x)\int_\Wv\chi_x(w)\phi(w)\,dw\Big|\leq q(\phi)\ch^{-d'+r-\iota}(x) \qquad (x\in \h).
\end{equation}
Therefore the integral over $\h$ in (\ref{fifth sterp}) may be dominated by (i.e. is less or equal a constant times the following expression)
\[
\int_\h \ch^{d'-r-\iota}(x)\ch^{-d'+r-\iota}(x)\,dx=\int_\h \ch^{-2\iota}(x)\,dx<\infty.
\]
\end{prf}
Let us fix the branch of the square root:
\begin{equation}\label{square root}
\C\setminus\R^-\ni z\to z^{\frac{1}{2}}\in\C
\end{equation}
so that $z^{\frac{1}{2}}>0$, if $z>0$.
Then for $y=\sum_{j=1}^ly_jJ_j\in\h$,
\begin{eqnarray}\label{ch explicit}
\ch(y)=\prod_{j=1}^l (1+y_j^2)^{\frac{1}{2\iota}}
&=&\prod_{j=1}^l (1+iy_j)^{\frac{1}{2\iota}}(1-iy_j)^{\frac{1}{2\iota}}.
\end{eqnarray}
The elements $J_j$, $1\leq j\leq l$, form a basis of the real vector space $\h$. Let $J_j^*$, $1\leq j\leq l$, be the dual basis of the space $\h^*$ and set
\begin{eqnarray}\label{eq:ej}
e_j=-iJ_j^*, \qquad 1\leq j\leq l\,.
\end{eqnarray}
If $\mu \in i\h^*$, then $\mu=\sum_{j=1}^l \mu_j e_j$ with $\mu_j \in \R$. 
We say that $\mu$ is strictly dominant if $\mu_1>\mu_2> \dots >\mu_l$.

The action of $\Wv(\G,\mathfrak h)$ on $\mathfrak h$ extends by duality to $i\mathfrak h^*$:
if $\mu=\sum_{j=1}^l \mu_j e_j \in i\mathfrak h^*$ and $t=\epsilon \sigma \in \Wv(\G,\mathfrak h)$ is as in (\ref{classical weyl group action}), then 
\begin{equation}\label{dual weyl group action}
t\Big(\sum_{j=1}^l\mu_je_j\Big)=\sum_{j=1}^l \epsilon_{\sigma^{-1}(j)} \mu_{\sigma^{-1}(j)}e_j\,.
\end{equation}
\begin{lem}\label{ximuchexplicit}
Let $\diesis{c_-}$ be as in  Lemma \ref{general formula for the int distr}. Then 
\begin{equation}
\label{eq:ximuWeyl}
\xi_{-\mu}(\diesis{c_-}(ty))= \xi_{-t^{-1}\mu}(\diesis{c_-}(y)) \qquad (t \in \Wv(\G,\mathfrak{h}\,, \mu \in i\mathfrak h^*\,, y \in \mathfrak h)\,.
\end{equation}
Moreover, let 
\begin{equation}
\label{eq:delta}
\delta=\frac{1}{2\iota}(d'-r+\iota)\,.
\end{equation} 
Then, with the notation of Lemma \ref{general formula for the int distr} and (\ref{ch explicit}),
\begin{eqnarray}\label{ximuchexplicit1}
\xi_{-\mu}(\diesis{c_-}(y))\ch^{d'-r-\iota}(y)
&=& \prod_{j=1}^l (1+iy_j)^{-\mu_j+\delta-1}
(1- iy_j)^{\mu_j+\delta-1},
\end{eqnarray}
where all the exponents are integers:
\begin{equation}\label{ximuchexplicit2}
\pm \mu_j+\delta\in \Bbb Z \qquad (1\leq j\leq l).
\end{equation}
In particular, (\ref{ximuchexplicit1}) is a rational function in the variables $y_1$, $y_2$, ..., $y_l$.
\end{lem}
\begin{prf}
Since
\begin{equation*}
\xi_{-\mu}(\diesis{c_-}(y))
= \prod_{j=1}^l \left(\frac{1+iy_j}
{1-iy_j}\right)^{-\mu_j}
=\prod_{j=1}^l (1+iy_j)^{-\mu_j}
(1-iy_j)^{\mu_j}\,,
\end{equation*}
(\ref{eq:ximuWeyl}) and (\ref{ximuchexplicit1}) follow from  (\ref{ch explicit}). 

Let $\lambda=\sum_{j=1}^l\lambda_j e_j$ be the highest weight of the representation $\Pi$ and  
let $\rho=\sum_{j=1}^l\rho_j e_j$ be one half times the sum of the positive roots of $\h$ in $\g_\C$. If $\mu$ is the Harish-Chandra parameter of $\Pi$, then
$\mu=\lambda+\rho=\sum_{j=1}^l\mu_je_j$.
Hence, the statement (\ref{ximuchexplicit2}) is equivalent to
\begin{equation}\label{compact5}
\lambda_j+\rho_j+\frac{1}{2\iota}(d'-r+\iota)\in \Bbb Z\,,
\end{equation}
which holds because of the assumption that $\Pi$ is a genuine representation of $\wt{\G}$. 
Indeed, if $\G=\Og_d$, then with the standard choice of the positive root system, $\rho_j=\frac{d}{2}-j$. Also, $\lambda_j\in\Bbb Z$, $\iota=1$, $r=d-1$. Hence, (\ref{compact5}) follows. Similarly, if $\G=\Ug_d$, then $\rho_j=\frac{d+1}{2}-j$, $\lambda_j+\frac{d'}{2}\in\Bbb Z$, $\iota=1$, $r=d$, which implies (\ref{compact5}). If $\G=\Sp_d$, then $\rho_j=d+1-j$, $\lambda_j\in\Bbb Z$, $\iota=\frac{1}{2}$, $r=d+\frac{1}{2}$, and (\ref{compact5}) follows. 
\end{prf}
In order to study the inner integral occurring in the formula for $T(\check\Theta_\Pi)$ in Lemma \ref{general formula for the int distr}, we shall need the following lemma.

\begin{lem}\label{lemma:HC's formula moved} 
Fix an element $z\in\h$. Let $\z\subseteq \g$ and $\Zg\subseteq\G$ denote the centralizer of $z$. (Then $\Zg$ is a real reductive group with the Lie algebra $\z$.) Denote by $\c$ the center of $\z$ and by $\pi_{\g/\z}$ be the product of the positive roots for $(\g_\C,\h_\C)$ which do not vanish on $z$. Let $B(\cdot ,\cdot )$ be any non-degenerate symmetric $\G$-invariant real bilinear form on $\g$.
Then there is a constant $C_\z$ such that for and $x\in \h$ and $x'\in \c$,
\begin{equation*}
\pi_{\g/\h}(x) \pi_{\g/\z}(x')\int_\G e^{iB(g.x,x')}\,dg
=C_\z \sum_{t W(\Zg,\h)\in W(\G,\h)/W(\Zg,\h)}\sgn_{\g/\h}(t)\pi_{\z/\h}(t^{-1}x)e^{iB(x,tx')}.
\end{equation*}
(Here  $\pi_{\z/\h}=1$ if $\z=\h$.)
\end{lem}
\begin{proof}
The proof is a straightforward modification of the argument proving Harish-Chandra's formula for the Fourier transform of a regular semisimple orbit, \cite[Theorem 2, page 104]{HC-57DifferentialOperators}.
A more general result was obtained in \cite[Proposition 34, p. 49]{OrbitesDV}. 
\end{proof}
We shall fix the symplectic form $\langle\cdot,\cdot\rangle$ on $\Wv$ according to the Lie superalgebra structure introduced at the beginning of section \ref{An almost semisimple orbital integral on the symplectic space} as follows
\begin{eqnarray}\label{symplectic form 0}
\langle w,w'\rangle&=&\tr_{\Bbb D/\R}(\Sy ww') \qquad (w,w'\in\Wv).
\end{eqnarray}
Then
\begin{eqnarray}\label{symplectic form}
\langle xw,w\rangle&=&\tr_{\Bbb D/\R}(\Sy  xw^2)\qquad (x\in \g\oplus \g', w\in \Wv).
\end{eqnarray}
(See \cite[(2.4')]{PrzebindaLocal}.) Let
\begin{equation}\label{the form B}
B(x,y)=\frac{\pi}{2}\, \tr_{\Bbb D/\R}(\Sy xy)=\frac{\pi}{2}\, \tr_{\Bbb D/\R}(xy) \qquad       (x,y\in\g).
\end{equation}
Then, using the expression (\ref{unnormalized moment maps}) for the unnormalized moment map $\tau$, we have 
\begin{equation}\label{chix and the form B}
\chi_x(w)=e^{\frac{\pi i}{2}\langle x w,w\rangle}=e^{iB(x,\tau(w))} \qquad     (x\in\g, w\in \Wv).
\end{equation}

\begin{lem}\label{reduction to ss-orb-int}
Suppose $l\leq l'$. Then, with the notation of Lemma  \ref{general formula for the int distr},
\[
\pi_{\g/\h}(x)\int_\Wv\chi_x(w)\phi^\G(w)\,dw = C \int_{\h\cap \tau(\Wv)}
e^{iB(x,y)} f_{\phi}(y)\,dy\,. 
\]
where $C$ is a non-zero constant and $f_{\phi}(y)=f(y)(\phi)$ for the Harish-Chandra regular almost semisimple orbital integral $f(y)$ of Definition \ref{def of HC integral on W}.
\end{lem}
\begin{prf}
By Lemmas  \ref{roots for l<=l'}  and \ref{|| and constant} and the Weyl integration formula (\ref{weyl int on w 1}) on $\Wv$,
\[
\int_\Wv\chi_x(w)\phi^\G(w)\,dw =\sum_{\hs1}\int_{\tau(\hs1^+)}\pi_{\g/\h}(\tau(w))\pi_{\g'/\z'}(\tau(w))C(\hs1)\mu_{\Oo(w)}(\chi_x\phi^\G)\,d\tau(w).
\]
Let us consider first the case  (\ref{muwforall othergroups})
\[
\mu_{\Oo(w)}(\chi_x\phi^\G)=\int_{\Sg/\Sg^{\hs1}}(\chi_x\phi^\G)(s.w)\,d(s\Sg^{\hs1}).
\]
Recall from (\ref{the identification}) the identification $y=\tau(w)=\tau'(w)$ and let us write $s=gg'$, where $g\in\G$ and $g'\in\G'$. Then
\begin{eqnarray*}
\chi_x(s.w)=e^{i\frac{\pi}{2}\langle x(s.w), s.w\rangle} = e^{iB(x, \tau(s.w))}=e^{iB(x, g.\tau(w))}=e^{iB(x, g.y)}
\end{eqnarray*}
and
\[
\phi^\G(s.w)=\phi^\G(g'.w).
\]
Since
\[
(\{1\}\times\G')\cap\Sg^{\hs1}=\{1\}\times \Zg',
\]
we see that for a positive constant $C_1$
\[
\mu_{\Oo(w)}(\chi_x\phi^\G)=C_1\int_\G e^{iB(x,g.y)}\,dg \int_{\G'/\Zg'}\phi^\G(g'.w)\,d(g'\Zg').
\]
However we know from Harish-Chandra (Lemma \ref{lemma:HC's formula moved}) that
\begin{eqnarray*}\label{Harish-Chandra's formula for the fourier transform of orbital integral}
\pi_{\g/\h}(x) \left(\int_\G e^{iB(x,g.y)}\,dg\right)\pi_{\g/\h}(y)
=C_2 \sum_{t\in W(\G,\h)}
\sgn_{\g/\h}(t)e^{iB(x,ty)}.
\end{eqnarray*}
Hence, by Definition \ref{def of HC integral on W} and (\ref{extension by the symmetry condition}) and for some suitable positive constants $C_k$, 
\begin{eqnarray}\label{proof in the first case}
&&\pi_{\g/\h}(x)\int_\Wv\chi_x(w)\phi^\G(w)\,dw \\
&=& C_3\sum_{t\in W(\G,\h)}\sgn_{\g/\h}(t)\sum_{\hs1} \int_{\tau(\hs1^+)}e^{iB(x,ty)} C(\hs1)\pi_{\g'/\z'}(y) \int_{\G'/\Zg'}\phi^\G(g'.w)\,d(g'\Zg')\,dy\nn\\
&=& C_3\sum_{t\in W(\G,\h)}\sgn_{\g/\h}(t)\sum_{\hs1} \int_{\tau(\hs1^+)}e^{iB(x,ty)}i^{-\dim{\g/\h}} C_{\hs1}\pi_{\g'/\z'}(y) \int_{\G'/\Zg'}\phi^\G(g'.w)\,d(g'\Zg')\,dy\nn\\
&=& C_4\sum_{t\in W(\G,\h)}\sgn_{\g/\h}(t)\int_{\bigcup_{\hs1}\tau(\hs1^+)}e^{iB(x,ty)} f_{\phi^\G}(y)\,dy\nn\\
&=& C_4\sum_{t\in W(\G,\h)}\int_{\bigcup_{\hs1}\tau(\hs1^+)}e^{iB(x,ty)} f_{\phi^\G}(t.y)\,dy\nn\\
&=& C_4\int_{W(\G,\h)(\bigcup_{\hs1}\tau(\hs1^+))}e^{iB(x,y)} f_{\phi^\G}(y)\,dy\nn\\
&=& C_4 \int_{\h\cap \tau(\Wv)}e^{iB(x,y)} f_{\phi^\G}(y)\,dy.\nn
\end{eqnarray}
Since $ f_{\phi^\G}=\vol(\G) f_{\phi}$, the formula follows.

Consider now the case 
\[
\mu_{\Oo(w)}(\chi_x\phi^\G)=\int_{\Sg/\Sg^{\hs1}}\int_{\ss1(\V^0)}(\chi_x\phi^\G)(s.(w+w^0))\,dw^0\,d(s\Sg^{\hs1}) .
\]
Then, as in (\ref{muwforsome groups}),
\begin{eqnarray*}
\mu_{\Oo(w)}(\chi_x\phi^\G)=\int_{\Sg/\Sg^{\hs1+w_0}}(\chi_x\phi^\G)(s.(w+w_0))\,d(s\Sg^{\hs1+w_0}).
\end{eqnarray*}
Furthermore,
\begin{eqnarray*}
\chi_x(s.(w+w_0))&=&e^{i\frac{\pi}{2}\langle x(s.(w+w_0)), s.(w+w_0)\rangle} =e^{i\frac{\pi}{2}\langle x(s.w), sw\rangle}\\
&=& e^{iB(x, \tau(sw))}=e^{iB(x, g.\tau(w))}=e^{iB(x, g.y)}
\end{eqnarray*}
and
\[
\phi^\G(s.(w+w_0))=\phi^\G(g'.(w+w_0)).
\]
Hence, with $n=\tau'(w_0)$,
\[
\mu_{\Oo(w)}(\chi_x\phi^\G)=C_1\int_\G e^{iB(x,g.y)}\,dg \int_{\G'/\Zg'{}^n}\phi^\G(g'.w)\,d(g'\Zg'{}^n).
\]
Therefore, the computation (\ref{proof in the first case}) holds again, and we are done.
\end{prf}
\begin{lem}\label{reduction to ss-orb-int for l>l'}
Suppose $l> l'$. Let $\z\subseteq\g$ and $\Zg\subseteq \G$ be the centralizers of $\tau(\hs1)$.
Then for $\phi \in \mathcal S(\Wv)$ 
\begin{multline*}
\pi_{\g/\h}(x)\int_\Wv\chi_x(w)\phi^\G(w)\,dw\\
= C \sum_{t W(\Zg,\h)\in W(\G,\h)/W(\Zg,\h)}\sgn_{\g/\h}(t)\pi_{\z/\h}(t^{-1}.x) \int_{\tau'(\reg{\hs1})}e^{iB(x,t.y)} f_{\phi}(y)\,dy, 
\end{multline*}
where $C$ is a non-zero constant.
\end{lem}
\begin{prf}
By the Weyl integration formula (\ref{weyl int on w 1}) with the roles of $\G$ and $\G'$ reversed and Lemmas \ref{roots for l>=l'} and \ref{|| and constant},
\[
\int_\Wv\chi_x(w)\phi^\G(w)\,dw =C_1\int_{\tau'(\reg{\hs1})}\pi_{\g/\z}(y)\pi_{\g'/\h'}(y)\mu_{\Oo(w)}(\chi_x\phi^\G)\,d\tau'(w),
\]
where
\[
\mu_{\Oo(w)}(\chi_x\phi^\G)=\int_{\Sg/\Sg^{\hs1}}(\chi_x\phi^\G)(s.w)\,d(s\Sg^{\hs1}).
\]
Recall the identification $y=\tau(w)=\tau'(w)$ and let us write $s=gg'$, where $g\in\G$ and $g'\in\G'$. Then
\begin{eqnarray*}
\chi_x(s.w)=e^{i\frac{\pi}{2}\langle x(s.w), s.w\rangle} = e^{iB(x, \tau(s.w))}=e^{iB(x, g.\tau(w))}=e^{iB(x, g.y)}
\end{eqnarray*}
and
\[
\phi^\G(s.w)=\phi^\G(g'.w).
\]
Since
\[
(\{1\}\times\G')\cap\Sg^{\hs1}=\{1\}\times \H',
\]
we see that
\[
\mu_{\Oo(w)}(\chi_x\phi^\G)=C_2\int_\G e^{iB(x,g.y)}\,dg \int_{\G'/\H'}\phi^\G(g'.w)\,d(g'\H').
\]
However we know from Harish-Chandra (Lemma \ref{lemma:HC's formula moved}) that
\begin{eqnarray*}\label{Harish-Chandra's formula for the fourier transform of orbital integral}
\pi_{\g/\h}(x) \int_\G e^{iB(g.x,y)}\,dg\, \pi_{\g/\z}(y)
=C_3 \sum_{t W(\Zg,\h)\in W(\G,\h)/W(\Zg,\h)}
\sgn_{\g/\h}(t)\pi_{\z/\h}(t^{-1}x)e^{iB(x,ty)}.
\end{eqnarray*}
Hence,
\begin{eqnarray}\label{proof in the first case, l>l'}
&&\pi_{\g/\h}(x)\int_\Wv\chi_x(w)\phi^\G(w)\,dw \\
&=& C_4\sum_{t W(\Zg,\h)\in W(\G,\h)/W(\Zg,\h)}\sgn_{\g/\h}(t)\pi_{\z/\h}(t^{-1}x) \int_{\tau'(\reg{\hs1})}e^{iB(x,ty)} \pi_{\g'/\h'}(y) \int_{\G'/\H'}\phi^\G(g'.w)\,d(g'\H')\,dy\nn\\
&=& C_5\sum_{t W(\Zg,\h)\in W(\G,\h)/W(\Zg,\h)}\sgn_{\g/\h}(t)\pi_{\z/\h}(t^{-1}x) \int_{\tau'(\reg{\hs1})}e^{iB(x,ty)} f_{\phi^\G}(y)\,dy.\nn
\end{eqnarray}
Since $ f_{\phi^\G}=\vol(\G)  f_{\phi}$, the formula follows.
\end{prf}
\begin{lem}\label{vandecorput}
Suppose $l\leq l'$.
Let $f(y)$ denote the function (\ref{extension by the symmetry condition and closure}). Then there is a seminorm $q$ on $\Ss(\Wv)$ such that
\begin{multline*}
\Big|\int_{\h\cap \tau(\Wv)} f(y)(\phi)\,e^{iB(x,y)}\,dy\Big|\leq q(\phi)\,\ch(x)^{-d'+r-1}\leq q(\phi)\,\ch(x)^{-d'+r-\iota}  \\(x\in \h,\, \phi\in\Ss(\Wv)).
\end{multline*}
\end{lem}
\begin{prf}
Since $l\leq l'$, Lemma \ref{reduction to ss-orb-int} and van der Corput estimate (\ref{seventh sterp}) prove that there is a seminorm  $q$ on $\Ss(\Wv)$ such that for all $y\in \h$ and $\phi\in\Ss(\Wv)$ we have
$$
\Big|\int_{\h\cap \tau(\Wv)} f(y)(\phi)\,e^{iB(x,y)}\,dy\Big|\leq q(\phi)\,|\pi_{\g/\h}(x)|\ch(x)^{-d'}.
$$
The result therefore follows from (\ref{relation of r with degree}) and (\ref{ch explicit}).
\end{prf}
\begin{cor}\label{an intermediate cor}
Suppose $l\leq l'$. Then 
for any $\phi\in \Ss(\Wv)$
\begin{eqnarray*}
&&T(\check\Theta_\Pi)(\phi)
=C\int_\h\xi_{-\mu}(\diesis{c_-}(x)) \ch^{d'-r-\iota}(x)
\int_{\h\cap\tau(\Wv)} e^{iB(x,y)}f_{\phi}(y)\,dy\,dx,
\end{eqnarray*}
where $C$ is a non-zero constant and each consecutive integral is absolutely convergent.
\end{cor}
\begin{prf}
The equality is immediate from Lemmas \ref{general formula for the int distr} and \ref{reduction to ss-orb-int}.
The absolute convergence of the outer integral over $\h$ follows from Lemma \ref{vandecorput}.
\end{prf}
Suppose $l>l'$. Recall from (\ref{the identification}) with $l''=l'$ that $\h'=\sum_{j=1}^{l'}\R J_j'$  is identified with $\sum_{j=1}^{l'}\R J_j\subseteq \h$. Let $\h''=\sum_{j=l'+1}^{l}\R J_j$, so that
\begin{equation}\label{h' + h'' decomposition}
\h=\h'\oplus \h''.
\end{equation}
Then $\z$, the centralizer of $\tau(\h_{\overline{1})}$,  is the centralizer of $\h'$ in $\g$ and $\z=\h'\oplus\g''$, where $\g''$ is the Lie algebra of the group $\G''$ of the isometries of the restriction of the form $(\cdot ,\cdot )$ to 
\begin{equation}\label{V" from mp}
\sum_{j=l'+1}^{l}  \V_{\overline 0}^j. 
\end{equation}
Furthermore, the derived Lie algebras of $\z$ and $\g''$ coincide (i.e. $[\z,\z]=[\g'',\g'']$) and $\h''$ is a Cartan subalgebra of $\g''$.
\begin{cor}\label{an intermediate cor, l>l'}
Suppose $l>l'$. Then for any $\phi\in \Ss(\Wv)$
\begin{eqnarray*}
&&T(\check\Theta_\Pi)(\phi)\\
&=&C\sum_{s\in W(\G,\h)}\sgn_{\g/\h}(s)\int_\h\xi_{-s\mu}(\diesis{c_-}(x)) \ch^{d'-r-\iota}(x) \pi_{\z/\h}(x)
\int_{\tau'(\reg{\hs1})} e^{iB(x,y)}f_{\phi}(y)\,dy\,dx\\
\end{eqnarray*}
where $C$ is a non-zero constant and each consecutive integral is absolutely convergent.
\end{cor}
\begin{prf}  
The formula is immediate from Lemmas \ref{general formula for the int distr} and \ref{reduction to ss-orb-int for l>l'} together with formula (\ref{eq:ximuWeyl}):
\begin{eqnarray*}
&&T(\check\Theta_\Pi)(\phi)\\
&=& C_1\int_\h\left(\xi_{-\mu}(\diesis{ c_-}(x)) \ch^{d'-r-\iota}(x)\right)
\left(\pi_{\g/\h}(x)\int_\Wv \chi_x(w)\phi^\G(w)\,dw\right)\,dx\\
&=& C_2\int_\h\left(\xi_{-\mu}(\diesis{ c_-}(x)) \ch^{d'-r-\iota}(x)\right)\\
&&\left(\sum_{t W(\Zg,\h)\in W(\G,\h)/W(\Zg,\h)}\sgn_{\g/\h}(t)\pi_{\z/\h}(t^{-1}x) \int_{\tau'(\reg{\hs1})}e^{iB(x,ty)} f_{\phi}(y)\,dy\right)\,dx\\
&=& \frac{C_2}{|W(\Zg,\h)|}\int_\h\left(\xi_{-\mu}(\diesis{ c_-}(x)) \ch^{d'-r-\iota}(x)\right)\\
&&\left(\sum_{t \in W(\G,\h)}\sgn_{\g/\h}(t)\pi_{\z/\h}(t^{-1}x) \int_{\tau'(\reg{\hs1})}e^{iB(x,ty)} f_{\phi}(y)\,dy\right)\,dx\\
&=& C_3\sum_{t \in W(\G,\h)} \sgn_{\g/\h}(t) \int_\h\left(\xi_{-\mu}(\diesis{ c_-}(tx)) \ch^{d'-r-\iota}(tx)\right)\\
&&\left(\pi_{\z/\h}(x) \int_{\tau'(\reg{\hs1})}e^{iB(tx,ty)} f_{\phi}(y)\,dy\right)\,dx\\
&=& C_3\sum_{t \in W(\G,\h)}\sgn_{\g/\h}(t)\int_\h\left(\xi_{-t^{-1}\mu}(\diesis{ c_-}(x)) \ch^{d'-r-\iota}(x)\right)\\
&&\left(\pi_{\z/\h}(x) \int_{\tau'(\reg{\hs1})}e^{iB(x,y)} f_{\phi}(y)\,dy\right)\,dx.
\end{eqnarray*}
As in (\ref{relation of r with degree}) we check that
\begin{equation*}\label{relation of r with degree z}
\max\{\deg_{x_j}\pi_{\z/\h};\ 1\leq j\leq l\}=\max\{\deg_{x_j}\pi_{\z''/\h''};\ 1\leq j\leq l''\}=\frac{1}{\iota} (r''-1).
\end {equation*}
Also, $r-r''=d'$. Hence,
\[
\ch^{d'-r-1}(x) |\pi_{\z/\h}(x)|\leq const\,\ch^{d'-r-1+r''-1}(x)=const\,\ch^{-2}(x).
\]
Furthermore, $f_\phi$ is absolutely integrable. Therefore,
the absolute convergence of the last integral over $\h$ follows from the fact that $\ch^{-2\iota}$ is absolutely integrable.  
\end{prf}

Let $\beta=\dfrac{\pi}{\iota}$, where $\iota $ is as in (\ref{eq:iota}). Then 
\begin{equation}\label{the constant beta}
B(x,y)=-\sum_{j=1}^l x_j \beta y_j \qquad \big(x=\sum_{j=1}^l x_jJ_j\,, y=\sum_{j=1}^l y_jJ_j\in\h\big).
\end{equation}
Indeed, the definition of the form $B$, (\ref{the form B}), shows that
\begin{eqnarray}\label{computation of the constant beta}
B(x,y)&=&\frac{\pi}{2}\tr_{\Dc/\R}(xy)=\frac{\pi}{2}\sum_{j,k}\tr_{\Dc/\R}(J_jJ_k)x_jy_k\nn\\
&=&\frac{\pi}{2}\sum_j\tr_{\Dc/\R}(-1)x_jy_j=-\frac{\pi}{\iota}\sum_jx_jy_j.
\end{eqnarray}
For a subset $\gamma\subseteq \{1, 2, \dots, l\}$ let
\[
\h\cap\gamma^\perp=\{y\in\h;\ y_j=0\ \text{for all}\ j\in\gamma\}.
\]

\begin{thm}\label{main thm for l<l'}
Let $l\leq l'$. 
Fix a genuine representation $\Pi$ of $\wt\G$ with the Harish-Chandra parameter $\mu\in i\h^*$. 
Let $P_{a,b}$ and $Q_{a,b}$ be the polynomials defined in (\ref{D0'}) and (\ref{D0''}).
Let $\delta$ be as in (\ref{eq:delta}) and set
\begin{eqnarray}\label{eq:ajbj}
&&a_{j}=-\mu_j-\delta+1\,, \qquad  b_{j}=\mu_j-\delta+1\\
&&p_j(y_j)=P_{a_{j},b_{j}}(\beta y_j)e^{-\beta |y_j|}\,,\qquad q_j(y_j)=\beta ^{-1} Q_{a_{j},b_{j}}(\beta ^{-1}y_j) \qquad (1\leq j\leq l). \nn
\end{eqnarray}
There is a non-zero constant $C$ such that for $\phi\in\Ss(\Wv)$
\begin{eqnarray}\label{main thm for l<l' a}
T(\check\Theta_\Pi)(\phi)&=&C\int_{\h\cap\tau(\Wv)}
\left(\prod_{j=1}^l  \left(p_j(y_j) +q_j(-\partial_{y_j})\delta_0(y_j)\right)\right)\cdot f_\phi(y)\,dy\\
&=&C\sum_{\gamma\subseteq\{1, 2, \dots, l\}}
\int_{\h\cap\tau(\Wv)\cap\gamma^\perp}
\prod_{j\notin \gamma}  p_j(y_j) \cdot \prod_{j\in\gamma} q_j(\partial(J_j))\cdot f_\phi(y)\,d_\gamma y
\nn
\end{eqnarray}
where $d_\gamma y=\prod_{j\notin \gamma} dy_j$ and $\delta_0$ is the Dirac delta at $0$.
Equivalently, if we define the following generalized function
\[
u(y)=C\prod_{j=1}^l  (p_j(y_j) +q_j(-\partial_{y_j})\delta_0(y_j)) \qquad (y\in \h\cap\tau(\Wv)),
\]
then
\begin{eqnarray}\label{main thm for l<l' conscise}
T(\check\Theta_\Pi)(\phi)&=&\int_{\h\cap\tau(\Wv)}u(y)\cdot f_\phi(y)\,dy\,.
\end{eqnarray}
\end{thm}
Theorem \ref{pullback of muy 3} implies that the function $f_\phi$ has the required number of continuous derivatives for the formula (\ref{main thm for l<l' a}) to make sense. Notice that the operators appearing in the integrand of (\ref{main thm for l<l' a}) act on different variables and therefore commute. 
\begin{prf}
Notice that the $a_j, b_j$ are integers by (\ref{ximuchexplicit2}).
Equation (\ref{main thm for l<l' a}) follows from Corollary \ref{an intermediate cor}, Lemma \ref{ximuchexplicit}, Theorem \ref{pullback of muy 3} and from Proposition \ref{D1not smooth}. 
\end{prf}

In the case $\Dc=\C$ the boundary $\partial(\h\cap\tau(\Wv))$ may be non-empty. Then the integrals in (\ref{main thm for l<l' a}) with  $(\h\setminus\gamma^\perp)\cap\partial(\h\cap\tau(\Wv))\ne \emptyset$ vanish. 
In any case we sum only over the $\gamma$ such that  $\partial(\h\cap\tau(\Wv))\subseteq \gamma^\perp$.
In particular, all terms of the sum corresponding to $\gamma \neq \emptyset$ vanish provided all hyperplanes $y_j=0$ are boundaries of $\h \cap \tau(\Wv)$. From Theorem \ref{pullback of muy 3} we see that this is the case if and only if $\max(l-q,0)=\min(p,l)$. 
In turn, this is equivalent to either $l'=q\geq l$ and $p=0$ or $l=p+q=l'$. In the first case 
$(\G,\G')=(\Ug_l,\Ug_{l'})$; in the second  $(\G,\G')=(\Ug_{p+q},\Ug_{p,q})$.

Notice also that the degree of the polynomial $Q_{a_{j},b_{j}}$ is $2\delta-2$. Hence, there are no derivatives in the formula (\ref{main thm for l<l' a}) if $2\delta-2\leq 0$.
However,
\[
2\delta-2=\frac{1}{\iota} (d'-r-\iota)=
\begin{cases}
d'-r-1 &\text{if $\Dc\ne \Ha$,}\\
2(d'-r)-1 &\text{if $\Dc=\Ha$.}
\end{cases}
\]
Also,
\[
d'-r-\iota=
\begin{cases}
2l'-2l &\text{if $(\G,\G')=(\Og_{2l},\Sp_{2l'})$},\\
2l'-2l-1 &\text{if $(\G,\G')=(\Og_{2l+1},\Sp_{2l'})$},\\
l'-l-1 &\text{if $(\G,\G')=(\Ug_{l},\Ug_{p,q}), p+q=l'$},\\
l'-l &\text{if  $(\G,\G')=(\Sp_{l},\Og^*_{2l'})$}.
\end{cases}
\]
Thus, since we assume $l\leq l'$, there are no derivatives in the formula (\ref{main thm for l<l' a}) if 
\[
d'-r-\iota\leq 0,
\]
which means that $l=l'$ if $\Dc\neq \C$ and $l\in \{l'-1,l'\}$ if $\Dc=\C$.

\begin{lem}\label{another intermediate lemma} 
Suppose $l>l'$. 
In terms of Corollary \ref{an intermediate cor, l>l'} and the decomposition (\ref{h' + h'' decomposition})
\begin{eqnarray}\label{another intermediate lemma1}
&&\xi_{-s\mu}(\diesis{ c_-}(x)) \ch^{d'-r-\iota}(x) \pi_{\z/\h}(x)\\
&=&\left(\xi_{-s\mu}(\diesis{ c_-}(x')) \ch^{d'-r-\iota}(x')\right)\left(\xi_{-s\mu}(\diesis{ c_-}(x'')) \ch^{d'-r-\iota}(x'') \pi_{\g''/\h''}(x'')\right),\nn
\end{eqnarray}
where $x'\in\h'$ and $x''\in\h''$.
Moreover,
\begin{eqnarray}\label{another intermediate lemma2}
&&\int_{\h''}\xi_{-s\mu}(\diesis{ c_-}(x'')) \ch^{d'-r-\iota}(x'') \pi_{\g''/\h''}(x'')\,dx''\\
&=&C\sum_{s''\in W(\g'',\h'')}\sgn(s'') \Bbb I_{\{0\}}(-(s\mu)|_{\h''}+s''\rho''),\nn
\end{eqnarray}
where $C$ is a constant, $\rho''$ is one half times the sum of positive roots for $(\g''_\C, \h''_\C)$ and $\Bbb I_{\{0\}}$ is the indicator function of zero.
\end{lem}
\begin{prf}
Part (\ref{another intermediate lemma1}) is obvious, because $\pi_{\z/\h}(x'+x'')=\pi_{\g''/\h''}(x'')$. We shall verify (\ref{another intermediate lemma2}).
 Let $r''$ denotes the number defined in (\ref{number r}) for the Lie algebra $\g''$. A straightforward computation verifies the following table.
\begin{center}
\renewcommand{\extrarowheight}{.2em}
\begin{tabular}{c||c|c|c}\label{table 3}
$\g$ & $r$  & $r''$ & $d'-r+r''$ \\[.1em]
\hline\hline
$\u_d$ & $d$ & $d-d'$ & $0$\\[.1em]
\hline
$\o_d$ & $d-1$ & $d-d'-1$ & $0$\\[.1em]
\hline
$\sp_d$ & $d+\frac{1}{2}$ & $d-d'+\frac{1}{2}$ & $0$\\[.2em]
\hline
\end{tabular}
\end{center}
By \cite[Lemma 5.7]{PrzebindaUnitary} applied to $\G''\supseteq \H''$ and $\g''\supseteq \h''$, 
\begin{eqnarray*}
\pi_{\g''/\h''}(x'')=C_1''\Delta''(\diesis{c_-}(x'')) \ch^{r''-\iota}(x'')
\qquad (x=x'+x'',\ x'\in\h',\ x''\in\h''),
\end{eqnarray*}
where $\Delta''$ is the Weyl denominator for $\G''$,
\begin{eqnarray}\label{FT Theta ch 13}
\Delta''=\sum_{s''\in W(\g'',\h'')} \sgn(s'')\,\xi_{s''\rho''}.
\end{eqnarray}
Hence, the integral (\ref{another intermediate lemma2}) is a constant multiple of
\begin{eqnarray}\label{FT Theta ch 11}
&&\int_{\h''}\xi_{-s\mu}(\diesis{c_-}(x''))\Delta''(\diesis{c_-}(x''))\ch^{d'-r+r''}(x'')\left(\ch^{-2\iota}(x'')\,dx''\right)\nn\\
&=&\int_{\diesis{c_-}(\h'')}\xi_{-s\mu}(h)\Delta''(h)\,dh\nn,
\end{eqnarray}
where $\diesis{c_-}(\h'')\subseteq \diesis{\H''}_o$. 

Notice that the function
\begin{equation}\label{FT Theta ch 12}
\diesis{\H''}_o\ni h\to \xi_{-s\mu}(h)\Delta''(h)\in\C
\end{equation}
is constant on the fibers of the covering map $\diesis{\H''}_o\to \H''_o$. Indeed, (\ref{FT Theta ch 13}) shows that
we'll be done as soon as we check that the weight $-s\mu+s''\rho''$ is integral for the Cartan subgroup $\H''$ (i.e. it is equal to the derivative of a character of $\H''$). Only in the two cases,  $\G=\Ug_d$ and $\G=\Og_{2n+1}$, the covering (\ref{FT Theta ch 12}) is non-trivial. 

Recall the notation used in the proof of (\ref{compact5}), and let $d''$ denote the dimension of the vector space (\ref{V" from mp}) over $\Bbb D$.
Suppose $\G=\Ug_d$. Then $\G''=\Ug_{d''}$,
$\lambda_j+\frac{d'}{2}\in \Zb$ and $\rho_j+\frac{d+1}{2}\in \Zb$. Hence, $(-s\mu)_j+\frac{d'+d+1}{2}\in \Zb$. Since, 
$\rho''_j+\frac{d''+1}{2}\in \Zb$, we see that
$$
\Zb\ni (-s\mu)_j+\frac{d'+d+1}{2}+\rho''_j+\frac{d''+1}{2}=(-s\mu)_j+\rho''_j+d+1.
$$
Therefore $(-s\mu)_j+\rho''_j\in \Zb$. 

Suppose $\G=\SO_{2n+1}$. Then $\G''=\SO_{2n''+1}$, $\lambda_j\in \Zb$ and $\rho_j+\frac{1}{2}\in \Zb$. Hence, $(-s\mu)_j+\frac{1}{2}\in \Zb$. Since, 
$\rho''_j+\frac{1}{2}\in \Zb$, we see that $(-s\mu)_j+\rho''_j\in \Zb$. 

Therefore, (\ref{FT Theta ch 11}) is a constant multiple of 
\begin{eqnarray}\label{FT Theta ch 15}
&&\sum_{s''\in W(\g'',\h'')} \sgn(s'') \int_{\H''_o}\xi_{-s\mu}(h)
\xi_{s''\rho''}(h)\,dh\\
&=&\left\{
\begin{array}{ll}
\vol(\H''_o)\,\sgn(s'')\ &\ \text{if}\ (s\mu)|_{\h''}=s''\rho'',\\
0\ &\ \text{if $(s\mu)|_{\h''}\notin W(\g'',\h'')\rho''$}\nn,
\end{array}
\right.\\
&=&\vol(\H''_o)\sum_{s''\in W(\g'',\h'')} \sgn(s'') 
\Bbb I_{\{0\}}(-(s\mu)|_{\h''}+s''\rho'').\nn
\end{eqnarray}
\end{prf}
\begin{cor}\label{another intermediate cor, l>l'}
Suppose $l>l'$. Then for any $\phi\in \Ss(\Wv)$
\begin{eqnarray}\label{another intermediate cor, l>l' a}
&&T(\check\Theta_\Pi)(\phi)\\
&=&C\sum_{s\in W(\G',\h')}\sgn_{\g'/\h'}(s)\int_{\h'}\xi_{-s\mu}(\diesis{c_-}(x)) \ch^{d'-r-\iota}(x) 
\int_{\tau'(\reg{\hs1})} e^{iB(x,y)}f_{\phi}(y)\,dy\,dx\nn\\
&=&C'\int_{\h'}\xi_{-\mu}(\diesis{c_-}(x)) \ch^{d'-r-\iota}(x) 
\int_{\tau'(\reg{\hs1})} e^{iB(x,y)}f_{\phi}(y)\,dy\,dx\nn
\end{eqnarray}
where $C$ is a non-zero constant, $C'=C|W(\G',\h')|$ and each consecutive integral is absolutely convergent. The expression (\ref{another intermediate cor, l>l' a}) is zero unless one can choose the Harish-Chandra parameter $\mu$ of $\Pi$ so that
\begin{equation}\label{another intermediate cor, l>l' b}
\mu|_{\h''}=\rho''.
\end{equation}
\end{cor}
\begin{prf}
The second equality in (\ref{another intermediate cor, l>l' a}) follows from the fact that $f_{\phi}(sy)=\sgn_{\g'/\h'}(s)f_{\phi}(y)$, $s\in W(\G',\h')$.
 
Observe that $B(x'+x'',y)=B(x',y)$ for $x'\in \h'$, $x''\in \h''$ and $y \in \tau'(\reg{\hs1})\subseteq \h'$.
We see therefore from Corollary \ref{an intermediate cor, l>l'} and Lemma \ref{another intermediate lemma} that
\begin{multline}\label{another intermediate cor, l>l'1}
T(\check\Theta_\Pi)(\phi)
=C\sum_{s\in W(\G,\h)}\sum_{s''\in W(\G'',\h'')}\sgn_{\g/\h}(s)\sgn(s'')\Bbb I_{\{0\}}(-(s\mu)|_{\h''}+s''\rho'')\\
\int_{\h'}\xi_{-s\mu}(\diesis{c_-}(x)) \ch^{d'-r-\iota}(x) \int_{\tau'(\reg{\hs1})} e^{iB(x,y)}f_{\phi}(y)\,dy\,dx.
\end{multline}
Notice that for $x\in \h'$ and $s''\in W(\G'',\h'')$, we have $s''x=x$.
Thus $\xi_{-s\mu}(\diesis{c_-}(x))=\xi_{-s''s\mu}(\diesis{c_-}(x))$ by (\ref{eq:ximuWeyl}).
Notice also that $W(\G'',\h'')\subseteq W(\G,\h)$ and $\sgn(s'')=\sgn_{\g/\h}(s'')$.
Moreover, 
$\Bbb I_{\{0\}}(-(s\mu)|_{\h''}+s''\rho'')=\Bbb I_{\{0\}}(-(s''{}^{-1}s\mu)|_{\h''}+\rho'')$. Hence, replacing $s$ by $s'' s$ in (\ref{another intermediate cor, l>l'1}), we see that this expression is equal to
\begin{eqnarray}\label{another intermediate cor, l>l'2}
&=&C\sum_{s\in W(\G,\h)}\sum_{s''\in W(\G'',\h'')}\sgn_{\g/\h}(s)\Bbb I_{\{0\}}(-(s\mu)|_{\h''}+\rho'')\\
&&\int_{\h'}\xi_{-s\mu}(\diesis{c_-}(x)) \ch^{d'-r-\iota}(x) \int_{\tau'(\reg{\hs1})} e^{iB(x,y)}f_{\phi}(y)\,dy\,dx.\nn\\
&=&C|W(\G'',\h'')|\sum_{s\in W(\G,\h),\ (s\mu)|_{\h''}=\rho''}\sgn_{\g/\h}(s)\nn\\
&&\int_{\h'}\xi_{-s\mu}(\diesis{c_-}(x)) \ch^{d'-r-\iota}(x) \int_{\tau'(\reg{\hs1})} e^{iB(x,y)}f_{\phi}(y)\,dy\,dx.\nn
\end{eqnarray}
Clearly (\ref{another intermediate cor, l>l'2}) is zero if there is no $s$ such that $(s\mu)|_{\h''}=\rho''$. Since $\mu$ is determined only up to the conjugation by the Weyl group, we may thus assume that $\mu|_{\h''}=\rho''$ and (\ref{another intermediate cor, l>l' b}) follows. Under this assumption the summation in (\ref{another intermediate cor, l>l'2}) is over the $s$ such that $(s\mu)|_{\h''}=\rho''$. For such $s$ we have
\begin{equation}\label{another intermediate cor, l>l'3}
(s\mu)|_{\h'}+\rho''=s\mu=s(\mu|_{\h'}+\rho'').
\end{equation}
Since $\mu$ is regular, (\ref{another intermediate cor, l>l'3}) shows that $s\in W(\G',\h')$. Hence,  (\ref{another intermediate cor, l>l' a}) follows.
The absolute convergence of the integrals was checked in the proof of Corollary \ref{an intermediate cor, l>l'}.
\end{prf}
\begin{thm}\label{main thm for l>=l'}
Let $l>l'$. 
Fix a genuine representation $\Pi$ of $\wt\G$ with the Harish-Chandra parameter $\mu\in i\h^*$. The distribution $T(\check\Theta_\Pi)$
is zero unless one can choose $\mu$ so that
\begin{equation}\label{main thm for l>=l' a}
\mu|_{\h''}=\rho''.
\end{equation}
Let us assume (\ref{main thm for l>=l' a}) and let
\[
a_{s,j}=(s\mu)_j-\delta+1,\qquad  b_{s,j}=-(s\mu)_j-\delta+1 \qquad (s\in W(\G',\h'),\ 1\leq j\leq l').
\]
There is a constant $C$ such that for $\phi\in\Ss(\Wv)$
\begin{eqnarray}\label{main thm for l>=l' b}
T(\check\Theta_\Pi)(\phi)&=&C \sum_{s\in W(\G',\h')} \sgn_{\g'/\h'}(s)
\int_{\tau'(\reg{\hs1})}\prod_{j=1}^{l'} P_{a_{s,j},b_{s,j}}(\beta y_j)e^{-\beta |y_j|}\cdot f_\phi(y)\,dy\\
&=&C'\int_{\tau'(\reg{\hs1})}\prod_{j=1}^{l'}  p_j(y_j)\cdot f_\phi(y)\,dy,\nn
\end{eqnarray}
where $C'=C|W(\G',\h')|$, the constant $\beta$ is as in (\ref{the constant beta}) and $p_j(y_j)=P_{a_{1,j},b_{1,j}}(\beta y_j)e^{-\beta |y_j|}$. In particular $T(\check\Theta_\Pi)$ is a locally integrable  function whose restriction to $\reg{\hs1}$ is equal to a non-zero constant multiple of 
\begin{equation}\label{main thm for l>=l' c}
\frac{1}{\pi_{\g/\z}(\tau'(w))}\sum_{s\in W(\G',\h')} \sgn_{\g'/\h'}(s)
\prod_{j=1}^{l'} P_{a_{s,j},b_{s,j}}(\beta \delta_j\tau'(w)_j)e^{-\beta\tau'(w)_j} \qquad (w\in \reg{\hs1}).
\end{equation}
\end{thm}
\begin{prf}
The formula (\ref{main thm for l>=l' b}) follows from Corollary \ref{another intermediate cor, l>l'}, Lemma \ref{ximuchexplicit} and from Proposition \ref{propD2}, because under our assumption, $a_{s,j}+b_{s,j}=-2\delta+2=-\frac{1}{\iota}(d'-r)+1\geq 1$.
Weyl integration on $\Wv$, (\ref{weyl int on w 2}), together with (\ref{main thm for l>=l' b}) implies (\ref{main thm for l>=l' c}).
\end{prf}
Our formula for the intertwining distribution $T(\check\Theta_{\Pi})$ is explicit enough to find its asymptotics, see Theorem \ref{the dilation limit of intertwining distribution}. These allow us to compute the wave front set of the representation $\Pi'$ within the Classical Invariant Theory, without using \cite{VoganGelfand}. See Corollary \ref{WF of Pi'} below.
We keep the notation of section \ref{Limits of orbital integrals}.

\begin{thm}\label{the dilation limit of intertwining distribution}
In the topology of $\Ss^*(\Wv)$,
\[
t^{\deg \mu_{\Oo_m}}M_{t^{-1}}^* T(\check\Theta_{\Pi})\underset{t\to 0+}{\to}C\mu_{\Oo_m},
\]
where $C\ne 0$, if $T(\check\Theta_{\Pi})\ne 0$.
\end{thm}
\begin{prf}
Proposition \ref{the main limit pro}, the Lebesgue dominated convergence theorem and Theorem \ref{main thm for l<l'} imply that for $l\leq l'$
\begin{eqnarray}\label{the dilation limit of intertwining distribution'}
&&\underset{t\to 0+}{\lim}\ t^{\deg \mu_{\Oo_m}}M_{t^{-1}}^* T(\check\Theta_{\Pi})\\
&=&\left(C\sum_{\gamma\subseteq\{1, 2, \dots, l\}}\nn
\int_{\h\cap\tau(\Wv)\cap\gamma^\perp}
\prod_{j\notin \gamma}  p_j(y_j) \cdot \prod_{j\in\gamma} q_j(\partial(J_j))\cdot f(y)|_U(\Bbb I_U)\, d_\gamma y \right)\mu_{\Oo_m}.
\end{eqnarray}
Similarly, Proposition \ref{the main limit pro}, the Lebesgue dominated convergence theorem and Theorem \ref{main thm for l>=l'} imply that for $l\geq l'$
\begin{eqnarray}\label{the dilation limit of intertwining distribution''}
\underset{t\to 0+}{\lim}\ t^{\deg \mu_{\Oo_m}}M_{t^{-1}}^* T(\check\Theta_{\Pi})
=\left(C'\int_{\tau'(\reg{\hs1})}\prod_{j=1}^{l'}  p_j(y_j)\cdot f(y)|_U(\Bbb I_U) \,dy \right)\mu_{\Oo_m}.
\end{eqnarray}
Thus in each case the limit is a constant multiple of the measure $\mu_{\Oo_m}$. (Notice that it is not easy to see directly that the constants \eqref{the dilation limit of intertwining distribution'} and \eqref{the dilation limit of intertwining distribution''} are finite, but as the byproduct of the convergence, we see that they are.)
The constant is equal to the integral,  over $U$, of the restriction the distribution $T(\check\Theta_{\Pi})$ to $U$ and can be computed, without the use of the orbital integrals as follows.

Recall from (\ref{chix and the form B}) that 
\[
\langle xw,w\rangle=\tr_{\Dc/\R}(x\tau(w))\qquad (x\in \g, w\in \Wv).
\]
Hence
\begin{eqnarray}\label{the dilation limit of intertwining distribution1}
T(\check\Theta_{\Pi})|_U(\Bbb I_U)
&=&\int_UT(\check\Theta_{\Pi})|_U(u)\,du
=\int_U\int_{\wt\G}\Theta_{\Pi}(\t g^{-1})\Theta(\t g)\chi_{c(g)}(u)\,d\t g\,du\nn\\
&=&\int_U\int_{\wt\G}\Theta_{\Pi}(\t g^{-1})\Theta(\t g)\chi\big(\frac{1}{4}\tr_{\Dc/\R}(c(g)\tau(u))\big)\,d\t g\,du.
\end{eqnarray}
We shall use the notation introduced in the proof of Lemma \ref{lemma I.4} with $k=m$.

In the stable range the elements of $U$ are of the form 
\[
u=\left(
\begin{array}{lll}
I_m\\
0\\
w_3 
\end{array}
\right)
\]
with $w_3=-\overline{w}^t_3$. So, by (\ref{I.second eq}), $\tau(u)=2w_3$ and 
the last integral in (\ref{the dilation limit of intertwining distribution1}) is a non-zero constant multiple of
\begin{eqnarray*}
&&\int_\g\int_{\wt\G}\Theta_{\Pi}(\t g^{-1})\Theta(\t g)\chi\big(\frac{1}{4}\tr_{\Dc/\R}(c(g)x)\big)\,d\t g\,dx\\
&=&\int_{\wt\G}\Theta_{\Pi}(\t g^{-1})\Theta(\t g)\Big(\int_\g \chi\big(\frac{1}{4}\tr_{\Dc/\R}(c(g)x)\big)\,dx\Big) \,d\t g\\
&=&const\,\int_{\wt\G}\Theta_{\Pi}(\t g^{-1})\Theta(\t g)\delta_0(c(g))\,d\t g\\
&=&const\;\Theta_{\Pi}(-\t I)\Theta(-\t I),
\end{eqnarray*}
where $\Theta_{\Pi}(\t g^{-1})\Theta(\t g)$ does not depend on the preimage $\t g$ of $g \in \Sp$ and $const$ denotes some non-zero constant.

Suppose now that the dual pair is not in the stable range, so $m<d$.  Suppose moreover that $U$ consists of the matrices of the form
\[
u=\left(
\begin{array}{lll}
I_m & 0\\
w_3 & w_6
\end{array}
\right)
\]
with $w_3=-\overline{w}^t_3$. This means that $F'=0$ in (\ref{matrix F.I}), i.e. that $\G'$ is not equal to $\Ug_{p,q}$ with $p\neq q$. By (\ref{I.second eq}), 
\[
\tau(u)=\left(
\begin{array}{lll}
2w_3 & w_6\\
-\overline w_6^t & 0
\end{array}
\right).
\]
For $g \in \G$, write $c(g)=\begin{pmatrix} x_{11} &-\overline{x_{12}}^t \\ x_{12} &x_{22}\end{pmatrix}\in \g$ with 
$x_{11}= -\overline{x_{11}}^t\in M_m(\Dc)$ and $x_{22}= -\overline{x_{22}}^t\in M_{d-m}(\Dc)$. Then
$$\tr_{\Dc/\R}\Big(c(g)\left(
\begin{array}{lll}
2w_3 & w_6\\
-\overline w_6^t & 0
\end{array}
\right)\Big)=\tr_{\Dc/\R}(2x_{11}w_3+\overline{x_{12}}^t\overline{w_6}^t+x_{12}w_6)
=2\tr_{\Dc/\R}(x_{11}w_3)+2\tr_{\Dc/\R}(x_{12}w_6)\,.$$
Hence, the last integral in (\ref{the dilation limit of intertwining distribution1}) is a non-zero constant multiple of
\begin{eqnarray}\label{correction beyond stable range1}
&&\int\int\int_{\wt\G}\Theta_{\Pi}(\t g^{-1})\Theta(\t g)
\chi\Big(\frac{1}{4}\tr_{\Dc/\R}\big(c(g)\left(
\begin{array}{lll}
2w_3 & w_6\\
-\overline w_6^t & 0
\end{array}
\right)\big)\Big)\,d\t g\,dw_3\,dw_6\\
&=&C_1
\int\int\int_\g \Theta_\Pi (\t c(-x))\Theta(\t c(x))\chi\Big(\frac{1}{4}\tr_{\Dc/\R}\big(x\left(
\begin{array}{lll}
2w_3 & w_6\\
-\overline w_6^t & 0
\end{array}
\right)\big)\Big)|\det(x-1)|^{-r}\,dx\,dw_3\,dw_6\nn\\
&=&C_2
\int_\g \Theta_\Pi (\t c(-\left(
\begin{array}{lll}
x_{11} & -\overline{x_{12}}^t\\
x_{12} & x_{22}
\end{array}
\right)))\Theta(\t c(\left(
\begin{array}{lll}
x_{11} & -\overline{x_{12}}^t\\
x_{12} & x_{22}
\end{array}
\right)))
\left|\det(\left(
\begin{array}{lll}
x_{11} & -\overline{x_{12}}^t\\
x_{12} & x_{22}
\end{array}
\right)-1)\right|^{-r}\nn\\
&&\delta_0(x_{11})\delta_0(x_{12})\,dx_{11}\,dx_{12}\,dx_{22}\nn\\
&=&C_3
\int_{M_{d-m}(\Bbb D)} \Theta_\Pi (\t c(-\left(
\begin{array}{lll}
0 & 0\\
0 & x_{22}
\end{array}
\right)))\Theta(\t c(\left(
\begin{array}{lll}
0 & 0\\
0 & x_{22}
\end{array}
\right)))
\left|\det(\left(
\begin{array}{lll}
0 & 0\\
0 & x_{22}
\end{array}
\right)-1)\right|^{-r}\,dx_{22}\nn\\
&=&C_4
\int_{M_{d-m}(\Bbb D)} \Theta_\Pi (\t c(-\left(
\begin{array}{lll}
0 & 0\\
0 & x_{22}
\end{array}
\right)))
\left|\det(\left(
\begin{array}{lll}
0 & 0\\
0 & x_{22}
\end{array}
\right)-1)\right|^{d'/2-r}\,dx_{22}\,,\nn
\end{eqnarray}
where the $C_j$ are non-zero constants, and we used the fact that the Jacobian of the Cayley transform is equal to $|\det(x-1)|^{-r}$ and $\Theta(\t c(x))$ is a constant multiple of $|\det(x-1)|^{d'/2}$. (See the proof of Lemma \ref{general formula for the int distr}.) Here $r$ depends on $\G$ as in \eqref{number r 1}. Also, we see from \eqref{number r 1} that $d'/2-r$ is equal to minus ``the $r$'' for the subgroup $\G_2\subseteq \G$ consisting of the matrices of the form 
$\left(
\begin{array}{lll}
I_m & 0\\
0 & g_{22}
\end{array}
\right)$.
In the case considered the covering, $\wt\G\to\G$ splits and the representation $\Pi$ is the tensor product of the unique non-trivial character $\chi_\Pi$ of the two element group and the trivial lift of an irreducible representation $\Pi_0$ of the group $\G$. 
Hence, \eqref{correction beyond stable range1} is a non-zero constant multiple of 
\begin{eqnarray}\label{correction beyond stable range2}
&&\int_{\G_2}\Theta_{\Pi_0}\Big(
\left(
\begin{array}{lll}
-I_m & 0\\
0 & g_{22}
\end{array}
\right)^{-1}\Big)\chi_\Pi\Big(
\widetilde{\left(
\begin{array}{lll}
-I_m & 0\\
0 & g_{22}
\end{array}
\right)}\Big)\,d{g_{22}}.
\end{eqnarray}
Since the Cayley transform $c$ is defined on the whole Lie algebra and is continuous, the term $\chi_\Pi\Big(
\widetilde{\left(
\begin{array}{lll}
-I_m & 0\\
0 & g_{22}
\end{array}
\right)}\Big)$ is constant. Thus  \eqref{correction beyond stable range1} is a non-zero constant multiple of 
\begin{eqnarray}\label{correction beyond stable range3}
&&\int_{\G_2}\Theta_{\Pi_0}\Big(
\left(
\begin{array}{lll}
-I_m & 0\\
0 & g_{22}
\end{array}
\right)^{-1}\Big)\,d{g_{22}}
=\chi_{\Pi_0}(-I_d)\int_{\G_2}\Theta_{\Pi_0}\Big(
\left(
\begin{array}{lll}
I_m & 0\\
0 &g_{22}
\end{array}
\right)^{-1}\Big)\,d{g_{22}},\nn
\end{eqnarray}
where $\chi_{\Pi_0}$ is the central character of $\Pi_0$. Therefore \eqref{correction beyond stable range1} is a non-zero constant multiple of the multiplicity of the trivial representation of $\G_2$ in $\Pi_0$. However, we know the highest weight of $\Pi_0$, see Theorem \ref{main thm for l>=l'}, and see from there that this multiplicity is non-zero. Therefore \eqref{correction beyond stable range1} is non-zero.

Suppose now that $U$ consists of the matrices of the form
\[
u=\left(
\begin{array}{lll}
I_m & 0\\
0 & w_5\\
w_3 & w_6
\end{array}
\right).
\]
Then $\Dc=\C$, the dual pair is $(\Ug_d, \Ug_{m+p,m})$ and we may assume that 
\[
\tau(u)=\left(
\begin{array}{lll}
2w_3 & w_6\\
-\overline w_6^t & \overline{w_5}^tiw_5
\end{array}
\right).
\]
Computations similar to those of the previous case show that the last integral in (\ref{the dilation limit of intertwining distribution1}) is a non-zero constant multiple of
\begin{eqnarray}\label{we have to label it}
&&\int\!\!\int\!\!\int\!\!\int_{\wt \G}\Theta_{\Pi}(\t g^{-1})\Theta(\t g)\chi\Big(\frac{1}{4}\tr_{\C/\R}\big(c(g)\left(
\begin{array}{lll}
2w_3 & w_6\\
-\overline w_6^t & \overline{w_5}^tiw_5
\end{array}
\right)\big)\Big)\,d\t g\,dw_3\,dw_6\,dw_5\nn\\
&=&C_1
\int\int \Theta_\Pi (\t c(-\left(
\begin{array}{lll}
0 & 0\\
0 & x_{22}
\end{array}
\right)))\Theta(\t c(\left(
\begin{array}{lll}
0 & 0\\
0 & x_{22}
\end{array}
\right)))
\chi\big(\frac{1}{4}\tr_{\C/\R}(x_{22}
\overline{w_5}^tiw_5)\big)\nn\\
&&\left|\det(\left(
\begin{array}{lll}
0 & 0\\
0 & x_{22}
\end{array}
\right)-1)\right|^{-r}\,dx_{22}\,dw_5\nn\\
&=&C_1
\int\int \Theta_\Pi (\t c(-\left(
\begin{array}{lll}
0 & 0\\
0 & x_{22}
\end{array}
\right)))
\chi\big(\frac{1}{4}\tr_{\C/\R}(x_{22}
\overline{w_5}^tiw_5)\big)\nn\\
&&\left|\det(\left(
\begin{array}{lll}
0 & 0\\
0 & x_{22}
\end{array}\right)
-1)\right|^{d'/2-r}\,dx_{22}\,dw_5
\nn\\
&=&C_1
\int\int \Theta_\Pi (\t c(-\left(
\begin{array}{lll}
0 & 0\\
0 & x_{22}
\end{array}
\right)))
\chi\big(\frac{1}{4}\tr_{\C/\R}(x_{22}
\overline{w_5}^tiw_5)\big)\nn\\
&&|\det(x_{22}-1)|^{d'/2-r}\,dx_{22}\,dw_5.
\end{eqnarray}
Let $\Wv_5=M_{p,d-m}(\C)$ endowed with the structure of real symplectic space induced by $\Wv$ under the identification of $w_5$ with 
$\begin{pmatrix}
0 & 0\\
0 & w_5\\
0 & 0
\end{pmatrix}$.
If $g_{22}- 1$ is invertible, then \footnote{Formula (\ref{eq:integralchicg}) can be easily verified using twisted convolution $\natural$. Indeed
\[
\Theta(\t c(0))\int_\Wv T(\t g)(w)\,dw=T(\t c(0))\natural T(g)(0)=\Theta(\t c(0)\t g),
\]
so that
\[
\int_\Wv \chi_{c(g)}(w)\,dw=\frac{\Theta(\t c(0)\t g)}{\Theta(\t c(0))\Theta(\t g)}.
\]
See e.g. \cite[Lemma 57]{AubertPrzebinda_omega}.
}
\begin{equation}\label{eq:integralchicg}
\int_{\Wv_5} \chi\big(\frac{1}{4}\tr_{\C/\R}(c(g_{22})
\overline{w_5}^tiw_5)\big)
\,dw_5=\int_{\Wv_5} \chi_{c(g_{22})}(w_5)\,dw_5=\frac{\Theta_5(\t c(0)\t g_{22})}{\Theta_5(\t c(0))\Theta_5(\t g_{22})}\,,
\end{equation}
where $\Theta_5$ is the character of the Weil representation for the metaplectic group $\wt{\Sp}(\Wv_5)$,  as defined in (\ref{the omega}).

Notice that $d'/2-r=p/2-(d-m)$ and $d-m$ is ``the $r$" for the group $\Ug_{d-m}$. Hence \eqref{we have to label it} is a non-zero constant multiple of
\begin{eqnarray}
&&\,\int_{\wt \G_2}\Theta_{\Pi}(\left(
\begin{array}{lll}
\wt{-I_m} & 0\\
0 & \t g_{22}
\end{array}
\right)^{-1})
\Theta_5(
\t g_{22})
\frac{\Theta_5(\t c(0)\t g_{22})}{\Theta_5(\t c(0))\Theta_5(\t g_{22})}
\,d\t g_{22}\\
&=&C_1\,\int_{\wt \G_2}
\Theta_{\Pi}(\left(
\begin{array}{lll}
\wt{-I_m} & 0\\
0 & \t c(0)^{-1}\t g_{22}
\end{array}
\right)^{-1})
\frac{\Theta_5(\t g_{22})}{\Theta_5(\t c(0))}
\,d\t g_{22}\nn\\
&=&
C_2\,\int_{\G_2}
\Theta_{\Pi|_{\wt\G_2}}(\t g_{22}^{-1})\Theta_5(\t g_{22})
\,d\t g_{22}.\nn
\end{eqnarray}
This is a non-zero constant multiple of the sum of the multiplicities of the irreducible components of $\Pi|_{\wt\G_5}$ in $\omega_5$.
This sum  is positive because $\omega_5$ is the restriction of $\omega$ to $\wt\Sp(\Wv_5)$, $\wt\G_2=\wt\G\cap  \Sp(\Wv_5)$ and $\Pi$ occurs in $\omega$.
Also, this sum is finite because the centralizer of $\G_2$ in $\Sp(\Wv_5)$ is compact, i.e. $\G_2$ belongs to a dual pair with both members compact.
\end{prf}
\begin{cor}\label{the dilation limit of intertwining distribution on g'}
In the topology of $\Ss^*(\g')$,
\[
t^{\dim{\Oo_m'}}M_{t^{2}}^*\mathcal F(\tau'_*( T(\check\Theta_{\Pi})))\underset{t\to 0+}{\to}C\mathcal F\mu_{\Oo'_m}
\]
or equivalently
\[
t^{\dim{\Oo_m'-2\dim \g'}}M_{t^{-2}}^*\tau'_*( T(\check\Theta_{\Pi}))\underset{t\to 0+}{\to}C\mu_{\Oo'_m},
\]
where $C\ne 0$, if $T(\check\Theta_{\Pi})\ne 0$.
\end{cor}
\begin{prf}
Observe that, by the uniqueness of the $\G\G'$-invariant measure, $\tau_*'(\mu_{\Oo_m})$ is a positive constant multiple of $\mu_{\Oo'_m}$. Hence
Proposition \ref{the dilation limit of intertwining distribution} together with (\ref{dilation relations 3}) show that
\[
t^{\deg \mu_{\Oo_m}+\dim\Wv-2\dim \g'}M_{t^{-2}}^*\tau'_*( T(\check\Theta_{\Pi}))\underset{t\to 0+}{\to}C\mu_{\Oo'_m}.
\]
Now we apply the Fourier transform $\mathcal F$ and use the fact that
\[
\mathcal F \circ M_{t^{-2}}^*=t^{2\dim \g'}M_{t^{2}}^* \circ \mathcal F 
\]
to see that
\[
t^{\deg \mu_{\Oo_m}+\dim\Wv}M_{t^{2}}^*\mathcal F(\tau'_* T(\check\Theta_{\Pi}))\underset{t\to 0+}{\to}C\mathcal F\mu_{\Oo'_m}.
\]
But Lemma \ref{muSNK as a tempered homogeneous distribution} implies that $\deg \mu_{\Oo_m}+\dim\Wv=\dim\Oo_m'$. Hence the corollary follows.
\end{prf}
\begin{cor}\label{WF of Pi'}
Suppose the representation $\Pi$ occurs in Howe's correspondence and the distribution character $\Theta_\Pi$ is supported in the preimage $\wt{\G_1}$ of the Zariski identity component $\G_1$ of $\G$. Then
\[
WF(\Pi')=\overline{\Oo_m'}.
\]
\end{cor}
\begin{prf}
This is immediate from Corollary \ref{the dilation limit of intertwining distribution on g'}, Lemma \ref{wave front set 1} and the easy to verify inclusion $WF(\Pi')\subseteq\overline{\Oo_m'}$, \cite[(6.14)]{PrzebindaUnipotent}.
\end{prf}
\section{\bf The pair $\G=\Ug_{l},\ \G'=\Ug_{l'}$, $l\leq l'$.\rm}\label{ul, ul'}\label{two unitary groups}
In this section we consider a dual pair $(\G,\G')$ with both members compact. By the classification of the dual pairs, both $\G$ and $\G'$ are compact unitary groups and we may assume that $\G=\Ug_l$ and $\G'=\Ug_{l'}$ with $l\leq l'$. Furthermore, since we want to use the results of section \ref{Intertwining distributions}, we view $\Ug_{l'}$ as $\Ug_{p,q}$ with $p=0$ and $q=l'$. In particular,  by Theorem (\ref{pullback of muy 3}), 
\begin{equation*}
\h\cap \tau(\Wv)=\Big\{\sum_{j=1}^ly_jJ_j; \ \text{$y_j\leq 0$ for all $1\leq j\leq l$}\Big\}\,,
\end{equation*} 
which is a $W(\G,\h)$-invariant domain, where $W(\G,\h)$ acts on $\h$ by permutation of the coordinates, as indicated in (\ref{classical weyl group action}). The constant $\delta$ introduced in (\ref{eq:delta}) has value 
\begin{equation}\label{eq:deltaUU}
\delta=\frac{1}{2}(l'-l+1)\,,
\end{equation}
and $\beta=\pi$ in (\ref{the constant beta}).

Fix a genuine representation $\Pi$ of $\wt\G$ with the Harish-Chandra parameter $\mu\in i\h^*$
and define
\begin{eqnarray}
\label{eq:ajbj}
&&a_j=\mu_j-\delta+1,\qquad b_j=-\mu_j-\delta+1 \quad (1\leq j\leq l)\,,\\
\label{eq:asjbsj}
&&a_{s,j}=(s\mu)_j-\delta+1,\quad b_{s,j}=-(s\mu)_j-\delta+1 \quad (s\in W(\G,\h), 1\leq j\leq l)\,.
\end{eqnarray}
(Hence $a_{j}=a_{1,j}$ and $b_{j}=b_{1,j}$.)
\begin{lem}\label{ul, ul' 1}
There is a constant $C$ such that for $\phi\in \Ss(\Wv)$
\begin{eqnarray}\label{l<l' and both compact 1}
T(\check\Theta_\Pi)(\phi)&=&C 
\int_{\h\cap\tau(\Wv)} \prod_{j=1}^l  P_{a_{j},b_{j}}(\pi y_j)e^{\pi y_j}\cdot f_\phi(y)\,dy\\
&=&\frac{C}{| W(\G,\h)|} \sum_{s\in W(\G,\h)} \sgn(s)
\int_{\h\cap\tau(\Wv)}
\prod_{j=1}^l  P_{a_{s,j},b_{s,j}}(\pi y_j)e^{\pi y_j}\cdot f_\phi(y)\,dy. \nn
\end{eqnarray}
Equivalently, with a possibly different constant $C$,
\begin{equation}\label{ul, ul' 1.1}
T(\check\Theta_\Pi)(\phi)=C \int_{\h\cap\tau(\Wv)}
\prod_{j=1}^l  P_{a_{j},b_{j},-2}(\pi y_j)e^{\pi y_j}\cdot f_\phi(y)\,dy.
\end{equation}
\end{lem}
\begin{prf} 
We are going to use Theorem \ref{main thm for l<l'}.  Since we view $\Ug_{l'}$ as $\Ug_{p,q}$ with $p=0$ and $q=l'$, we have $\delta_j=-1$ for all $j$. 
Since the degree of the polynomial $Q_{a_{s,j},b_{s,j}}$ is equal to $-a_{s,j}-b_{s,j}=2\delta-2=l'-l-1$, Corollary \ref{pullback of muy 3} implies that $Q_{a_{s,j},b_{s,j}}(\pi ^{-1}\partial(J_j)) f_\phi(y)_{y_j=0}=0$. Hence all the terms with $\gamma\ne \emptyset$ are zero. Thus the first equality in (\ref{l<l' and both compact 1}) follows.
For the second, notice that, since $(sy)_j=y_{s^{-1}(j)}$, we have
\begin{eqnarray*}
&&\prod_{j=1}^l  P_{a_{s,j},b_{s,j}}(\pi y_j)e^{\pi y_j}
=e^{\sum_{j=1}^l\pi y_j}\prod_{j=1}^l  P_{a_{s,j},b_{s,j}}(\pi y_j) \\
&=&e^{\sum_{j=1}^l\pi y_{s^{-1}(j)}}\prod_{j=1}^l  P_{a_{1,j},b_{1,j}}(\pi  y_{s^{-1}(j)}) 
=\prod_{j=1}^l  P_{a_{j},b_{j}}(\pi (s^{-1}y)_j)e^{\pi (s^{-1}y)_j}
\end{eqnarray*}
and recall that $f_\phi(sy)=\sgn(s) f_\phi(y)$. 
Finally, since $\pi y_j\leq 0$, (\ref{D0'}) implies that we may replace $P_{a_{j},b_{j}}(\pi y_j)$ by $2\pi P_{a_{j},b_{j},-2}(\pi y_j)$.
\end{prf}
\begin{lem}\label{ul, ul' 2}
The distribution $T(\check\Theta_\Pi)$ is non-zero if and only if the highest weight $\lambda=\sum_{j=1}\lambda_je_j \in i\h^*$ of $\Pi$ satisfies the following condition:
\begin{equation}\label{ul, ul' 2.1}
\lambda_1\geq\lambda_2\geq\dots\geq\lambda_l\geq \frac{l'}{2}.
\end{equation}
Equivalently, if and only if the Harish-Chandra parameter $\mu$ of $\Pi$ satisfies
\begin{equation}\label{ul, ul' 2.2}
\mu_j\in \delta +\Bbb Z^+ \qquad (j=1,2,\dots, l)\,
\end{equation} 
where $\Bbb Z^+$ denotes the set of non-negative integers. 
\end{lem}
The condition (\ref{ul, ul' 2.1}) means exactly that $\Pi$ occurs in Howe's correspondence, see for example \cite[(A.5.2)]{PrzebindaInfinitesimal}. Recall that we have chosen the Harish-Chandra
parameter $\mu$ to be strictly dominant, i.e. so that $\mu_1>\mu_2>\dots>\mu_l$, but, in fact, the condition (\ref{ul, ul' 2.2}) does not depend on the choice of the order of roots.
\begin{prf} 
Let $\rho=\sum_{j=1}\rho_j e_j$ be one half times the sum of the positive roots. 
Then $\mu_j=\lambda_j+\rho_j$ and $\rho_j=\frac{l+1}{2}-j$. Hence
\begin{eqnarray*}
\mu_j-\mu_{j+1}&=&\lambda_j-\lambda_{j+1} +1 \qquad (1\leq j\leq l-1)\,,\\
\mu_j-\delta&=&\lambda_j-\frac{l'}{2}+l-j \qquad (1\leq j\leq l)\,.
\end{eqnarray*}
This proves the equivalence of the conditions (\ref{ul, ul' 2.1}) and (\ref{ul, ul' 2.2}) since
$\mu$ is strictly dominant.  

Notice that, since $a_{j}=\mu_j-\delta+1$, the condition (\ref{ul, ul' 2.2}) is also equivalent to $a_j\geq 1$ for all $1\leq j\leq l$. 
If the distribution $T(\check\Theta_\Pi)$ is non-zero, then none of the $P_{a_{j},b_{j},-2}$ can be identically $0$. So $a_j\geq 1$ for all $1\leq j\leq l$ by (\ref{eq:Pabminus2}).  

It remains to see that the condition $a_j\geq 1$ for all $1\leq j\leq l$ suffices for the non-vanishing of the expression (\ref{ul, ul' 1.1}). We are going to use a non-direct argument, though an alternative one is going to be evident from Proposition \ref{ul, ul' 3} below.

In the case when both members of the dual pair are compact
\begin{equation}\label{ul, ul' 2.20}
f_\phi(y)=C_1\pi_{\g'/\z'}(y)\int_\Sg\phi(s.w)\,ds
\end{equation}
where $w\in\hs1$ and $\tau(w)=y$ is identified with $\tau'(w)$ and $C_1$ is the appropriate constant. Clearly the integral in (\ref{ul, ul' 2.20})  converges if $\phi\in C_c^\infty(\Wv)$ and gives an element  
of $C_c^\infty(\Wv)^\G$. By Corollary \ref{pullback via taug}, it is of the form $\psi \circ \tau'$ 
where $\psi \in C_c^\infty(\g')$. Since $\G'$ is compact, a theorem of Dadok, \cite[Corollary 1.5]{Dadok82}, applied to $i\g'$ shows that the function
\[
\h\cap\tau(\Wv)\ni y\to \int_\Sg\phi(s.w)\,ds\in \C, \qquad (y=\tau(w),\ w\in \hs1)
\]
may be an arbitrary $W(\G,\h)$-invariant compactly supported $C^\infty$ function on $\h\cap\tau(\Wv)$, as $\phi$ varies through $C_c^\infty(\Wv)$. Therefore
\[
\h\ni y\to \pi_{\g/\h}(y)\int_\Sg\phi(s.w)\,ds, \qquad (y=\tau(w),\ w\in \hs1)
\]
may be an arbitrary $W(\G,\h)$-skew-invariant compactly supported $C^\infty$ function on $\h\cap\tau(\Wv)$. Hence, if (\ref{l<l' and both compact 1}) were zero, then the function 
\[
\prod_{j=1}^l  P_{a_{j},b_{j},-2}(\pi y_j)e^{-\pi y_j}\cdot \frac{\pi_{\g'/\z'}(y)}{\pi_{\g/\h}(y)}
\]
would have to be $W(\G,\h)$-invariant. Equivalently, 
\[
\prod_{j=1}^l  P_{a_{j},b_{j},-2}(\pi y_j)
\]
would have to be $W(\G,\h)$-invariant. This is not possible if $a_j\geq 1$ for all $j$. Indeed,
$\mu$ is strictly dominant and hence, 
by (\ref{eq:Pabminus2}), the $P_{a_{j},b_{j},-2}$ are non-zero polynomials of different degrees.
Thus, the distribution (\ref{l<l' and both compact 1}) is not zero.
\end{prf}
\begin{pro}\label{ul, ul' 3}
With the notation of Lemma \ref{ul, ul' 2}, let
\[
P_\mu(y)=\prod_{j=1}^{l} P_{a_{j}, b_{j}, -2}(\pi y_j) \qquad (y\in \h).
\]
The distribution $\T(\check \Theta_{\Pi})$ is a smooth $\G\G'$-invariant function on $\Wv$. For $w \in \hs1$ it is given by the following formula:
\begin{eqnarray}\label{ul, ul' 3 1}
\T(\check \Theta_{\Pi})(w)&=&c_{\Pi}\,  e^{-\frac{\pi}{2}\langle J w, w\rangle}\Big(\frac{1}{\pi_{\g/\h}(y)} 
\sum_{s\in W(\G,\h)} \sgn(s)  P_{\mu}(sy)\Big) \nn\\
&=&c_{\Pi}\,  e^{-\frac{\pi}{2}\langle J w, w\rangle}\Big(\frac{1}{\pi_{\g/\h}(y)} 
\sum_{s\in W(\G,\h)} \sgn(s)  P_{s\mu}(y)\Big) \,, 
\end{eqnarray}
where $c_{\Pi}$ is a constant, $J$ is the unique positive compatible complex structure on $\Wv$ centralized by $\G$ and $\G'$, $\beta$ is as in (\ref{the constant beta}) and $y=\tau(w) \in \h$. The sum in (\ref{ul, ul' 3 1}) is a $W(\G,\h)$-skew symmetric polynomial. Hence the quotient by $\pi_{\g/\h}$ is a $W(\G,\h)$ invariant polynomial on $\h$. It extends uniquely to a $\G$-invariant polynomial $\t P_\mu$ on $\g$. Thus
\begin{equation}\label{ul, ul' 3 1'}
\T(\check \Theta_{\Pi})(w)=c_{\Pi}\,  e^{-\frac{\pi}{2}\langle J w, w\rangle}\t P_\mu(\tau(w)) \qquad (w\in \Wv). 
\end{equation}
\end{pro}
\begin{prf}
We see from Lemma \ref{ul, ul' 1} and the formula (\ref{ul, ul' 2.20}) that for any $\phi\in \Ss(\Wv)$,
\begin{eqnarray}\label{ul, ul' 3 2}
T(\check\Theta_\Pi)(\phi)
&=&C_1 \int_{\h\cap\tau(\Wv)} e^{\pi\sum_{j=1}^ly_j}
 P_{\mu}(y)\pi_{\g'/\z'}(y)\int_{\G\times\G'} \phi(gg'.w)\,dg\,dg'\,dy.
\end{eqnarray}
However, by (\ref{computation of the constant beta}) and (\ref{chix and the form B}) with $\beta=\pi$,
\[
-\pi \sum_{j=1}^ly_j=B(\sum_{j=1}^l J_j,y)=\frac{\pi}{2}\langle \sum_{j=1}^l J_j w,w\rangle
=\frac{\pi}{2}\langle  J w,w\rangle,
\]
where $J=\sum_{j=1}^l J_j$ has the required properties. 
Furthermore,
\[
\pi_{\g'/\z'}(y)=\frac{1}{|W(\G,\h)|}\sum_{s\in W(\G,\h)} \sgn(s) \pi_{\g'/\z'}(sy).
\]
Hence, (\ref{ul, ul' 3 2}) implies 
\begin{eqnarray}\label{ul, ul' 3 3}
&&T(\check\Theta_\Pi)(\phi)\\
&=&C_2 \int_{\h\cap\tau(\Wv)}  e^{-\frac{\pi}{2}\langle J w, w\rangle}
\sum_{s\in W(\G,\h)} \sgn(s) 
 P_{\mu}(sy)\pi_{\g'/\z'}(y)\int_{\G\times\G'} \phi(gg'.w)\,dg\,dg'\,dy.\nn
\end{eqnarray}
Weyl integration formula on $\Wv$, (\ref{weyl int on w 1}), together with (\ref{ul, ul' 3 3}) implies (\ref{ul, ul' 3 1}).
\end{prf}
We now reverse the role of $\G$ and $\G'$ to compute the intertwining distribution $\T(\check \Theta_{\Pi'})$ for a genuine irreducible unitary representation $\Pi'$ of $\wt \G'$.
Since we assume that $l\leq l'$, the decomposition (\ref{h' + h'' decomposition}) becomes
\begin{equation}\label{h + h'' decomposition}
\h'=\h\oplus\h''.
\end{equation}
\begin{pro}\label{ul, ul' 4}
Let $\Pi'$ be a genuine representation of $\wt{\G'}$ with the Harish-Chandra parameter $\mu'\in i\h'{}^*$. Then $\T(\check \Theta_{\Pi'})\ne 0$ if and only if 
\begin{eqnarray}\label{ul, ul' 4 1}
&&-\mu'_j\in \delta+\Zb^+,\ \qquad (1\leq j\leq l')\\
&&\text{and}\ \ \mu'|_{\h''}=\rho''\  \qquad\text{(up to permutation of the coordinates).}\nn
\end{eqnarray}
$\T(\check \Theta_{\Pi'})$ is a non-zero constant multiple of $\T(\check \Theta_{\Pi})$ if and only if
$\mu$ and $\mu'$ can be chosen in their Weyl group orbits so that
\begin{equation}\label{ul, ul' 4 2}
\mu'|_{\h}=-\mu\quad \text{and}\quad \mu'|_{\h''}=\rho''.
\end{equation}
\end{pro}
\begin{prf}
As one may see from (\ref{the form B}), reversing the roles of the members of the dual pair in Theorem \ref{main thm for l>=l'} results in replacing the form $B$ by $-B$. The constant $\beta=\pi$ get therefore replaced by $-\beta=-\pi$. Also,  (\ref{ul, ul' 2.20}) gets replaced by
\[
f'_\phi(y)=C_1\,\pi_{\g/\h}(y)\int_\Sg \phi(s.w)\,ds
\]
with an appropriate constant $C_1$.
Since this plays no role in Lemma \ref{another intermediate lemma} and Corollary \ref{another intermediate cor, l>l'},  $T(\check\Theta_{\Pi'})$
is zero unless one can choose $\mu'$ so that
\begin{equation}\label{ul, ul' 4 3}
\mu'|_{\h''}=\rho''.
\end{equation}
Since the roles of $l$ and $l'$ are reversed, $\delta=\frac{1}{2}(l'-l+1)$ is replaced by $\delta'=\frac{1}{2}(l-l'+1)$. As before, $\delta_j=-1$ for all $j$. Let
\[
a'_{s,j}=-(s\mu')_j-\delta'+1,\ b'_{s,j}=(s\mu')_j-\delta'+1 \qquad (s\in W(\G,\h),\ 1\leq j\leq l).
\]
Then Theorem \ref{main thm for l>=l'} says that
\begin{eqnarray}\label{ul, ul' 4 4}
T(\check\Theta_{\Pi'})(\phi)&=&C \sum_{s\in W(\G,\h)} \sgn(s)
\int_{\h\cap\tau(\Wv)}\prod_{j=1}^{l}  P_{a'_{s,j},b'_{s,j}}(-\pi y_j)e^{-\pi y_j}\cdot f'_\phi(y)\,dy,
\end{eqnarray}
where there are no derivatives
because the degree of the polynomial $Q_{a'_{s,j},b'_{s,j}}$ is equal to $-a'_{s,j}-b'_{s,j}=2\delta'-2=l-l'-1<0$. 

Since $-\pi y_j\geq 0$, we have 
$P_{a'_{s,j},b'_{s,j}}(-\pi y_j)= 2\pi P_{a'_{s,j},b'_{s,j}, 2}(-\pi y_j)$. Recall, (\ref{D0}), that $P_{a'_{s,j},b'_{s,j}, 2}(-\pi y_j)=P_{b'_{s,j},a'_{s,j},-2}(\pi y_j)$. Hence, (\ref{ul, ul' 4 4})  may be rewritten as
\begin{eqnarray}\label{ul, ul' 4 5}
T(\check\Theta_{\Pi'})(\phi)&=&C \sum_{s\in W(\G,\h)} \sgn(s)
\int_{\h\cap\tau(\Wv)}\prod_{j=1}^{l}  P_{b'_{s,j},a'_{s,j},-2}(\pi y_j)e^{-\pi y_j}\cdot f'_\phi(y)\,dy,
\end{eqnarray}
with a different constant $C$.
Recall also that
\[
\pi_{\g/\h}(y)\cdot\prod_{j=1}^ly_j^{l'-l}= C_2\pi_{\g'/\z'}(y).
\]
Hence, Proposition \ref{propD3} shows that
\[
\prod_{j=1}^{l}  P_{b'_{s,j},a'_{s,j},-2}(\pi y_j)\cdot \pi_{\g/\h}(y)=C_3\prod_{j=1}^{l}  P_{b'_{s,j}-(l'-l),a'_{s,j}-(l'-l),-2}(\pi y_j)\cdot \pi_{\g'/\z'}(y).
\]
Therefore (\ref{ul, ul' 4 5}) coincides with
\begin{eqnarray}\label{ul, ul' 4 6}
&&T(\check\Theta_{\Pi'})(\phi)\\
&=&C \sum_{s\in W(\G,\h)} \sgn(s)
\int_{\h\cap\tau(\Wv)}\prod_{j=1}^{l}  P_{(s\mu')_j-\delta+1,-(s\mu')_j-\delta+1,-2}(\pi y_j)e^{-\pi y_j}\cdot f_\phi(y)\,dy,\nn
\end{eqnarray}
with a possibly different constant $C$. By comparing (\ref{ul, ul' 4 6}) with (\ref{l<l' and both compact 1}) we see that (\ref{ul, ul' 4 2}) holds.
\end{prf}
Since, by the definition (\ref{the omega}),
\[
\OP(\mathcal K( T(\check\Theta_\Pi)))=\omega(\check\Theta_\Pi)\ \text{and}\ \OP(\mathcal K (T(\check\Theta_{\Pi'})))=\omega(\check\Theta_{\Pi'}),
\]
Proposition \ref{ul, ul' 4} implies the following corollary.
\begin{cor}\label{ul, ul' 4  7}
When restricted to the group $\wt\G\wt\G'$, the oscillator representation decomposes into the Hilbert direct sum of irreducible components of the form $C_\Pi\cdot\Pi\otimes\Pi'$, where $\Pi$ is determined by $\Pi'$ via the condition (\ref{ul, ul' 4 2}) and $C_\Pi$ are some positive integral constants.
\end{cor}

The final part of this section is devoted to the proof that $C_\Pi=1$, that is that each irreducible representation 
$\Pi \otimes \Pi'$ of $\wt \G \times \wt \G'$ contained in the oscillatory representation occurs with multiplicity one, see Proposition \ref{the constant is one} below.
This is a well known and fundamental fact due to Hermann Weyl, \cite{WeylBook}. We include a proof using the interwining distributions. 

For $\beta \in \h^*$ we denote by $H_\beta$ the unique element in $\h$ such that $\beta(H)=\tr(HH_\beta)$ for all $H \in \h$. We define the operator $\partial(\beta)$ as the 
directional derivative in the direction of $H_\beta$. This defines in particular $\partial(\pi_{\g/\h})=
\prod_{\alpha>0} \partial(\alpha)$.
\begin{lem}\label{lemma:HC's formula - the constant}
The constant $C_\z$ in Lemma \ref{lemma:HC's formula moved} is equal in absolute value to
\begin{equation}\label{eq:Cz}
\frac{\pi^{-ll'+\frac{l'(l'+1)}{2}} }{\vol(\G)}\frac{|W(\z,\h)|}{|W(\g,\h)|}\, \frac{\partial(\pi_{\g/\h})(\pi_{\g/\h})}{\partial(\pi_{\z/\h})(\pi_{\z/\h})}=\frac{(l-l')!}{l!}
\frac{\pi^{-ll'+\frac{l'(l'+1)}{2}}}{\vol(\G)} \frac{\partial(\pi_{\g/\h})(\pi_{\g/\h})}{\partial(\pi_{\z/\h})(\pi_{\z/\h})}\,.
\end{equation}
\end{lem}
\begin{proof}
The proof is a straightforward modification of the argument proving Harish-Chandra's formula for the Fourier transform of a regular semisimple orbit, \cite[Theorem 2, page 104]{HC-57DifferentialOperators}.
Notice that the constant 
$\pi^{-ll'+\frac{l'(l'+1)}{2}}=\pi^{-\frac{l(l-1)}{2}}\pi^{\frac{(l-l')(l-l'-1)}{2}}$ is due to the normalization of $B$ in (\ref{computation of the constant beta}). 
\end{proof}
Recall that we denote by $\Sigma_m$ the group of permutations of $\{1,2,\dots, m\}$.

\begin{lem}\label{leftover 2}
Let $z_j \in\C$ for $1 \leq j \leq m$.
Then, with the convention that empty products are equal to 1,
\begin{equation} \label{eq:Vandermonde-1}
\sum_{s \in \Sigma_m} \sgn(s) \prod_{j=1}^m \prod_{k=1}^{s(j)-1} (z_j-k)= 
\prod_{1\leq j < k \leq m} (z_j-z_k)\,.
\end{equation}
\end{lem}
\begin{proof}
The left-hand side is a Vandermonde determinant. Indeed
\begin{eqnarray*}
&&
\sum_{s \in \Sigma_m} \sgn(s) \prod_{j=1}^m \prod_{k=1}^{s(j)-1} (z_j-k)\\
&&\quad=\det
\begin{bmatrix}
1 & (z_1-1) & (z_1-1)(z_1-2) & \dots & (z_1-1)(z_1-2)\dots(z_1-m+1)\\
1 & (z_2-1) & (z_2-1)(z_2-2) & \dots & (z_2-1)(z_2-2)\dots(z_2-m+1)\\
\vdots & \vdots & \vdots &\vdots &\vdots\\
1 & (z_m-1) & (z_m-1)(z_m-2) & \dots & (z_m-1)(z_m-2)\dots(z_m-m+1)
\end{bmatrix}\\
&&\quad=\det
\begin{bmatrix}
1 & z_1 & z_1^2 & \dots & z_1^{m-1}\\
1 & z_2 & z_2^2 & \dots & z_2^{m-1}\\
\vdots & \vdots & \vdots &\vdots &\vdots\\
1 & z_m & z_m^2 & \dots & z_m^{m-1}
\end{bmatrix}\,.
\end{eqnarray*}
This proves the result.
\end{proof}
\begin{lem} \label{lemma:Fan moved}
Let $n \in \Ze$ with $n\geq 2$ and let $a \in \C$. Set
\begin{multline*}
F(a,n)=\\
\det \begin{bmatrix}
1 & a & a(a+1) & \dots & a(a+1)\cdots (a+n-2) \\
1 & a+1 & (a+1)(a+2) & \dots & (a+1)(a+2)\cdots (a+n-1) \\
\vdots & \vdots & \vdots & \vdots &\vdots\\
1 & a+n-1 & (a+n-1)(a+n) & \dots & (a+n-1)(a+n)\cdots (a+2n-3) \\
1 & a+n & (a+n)(a+n+1) & \dots & (a+n)(a+n+1)\cdots (a+2n-2)
\end{bmatrix}\,.
\end{multline*}
Then 
$$F(a,n)=\prod_{k=1}^{n-1} k!\,.$$
In particular, $F(a,n)$ is independent of $a$.
\end{lem}
\begin{proof}
For $2\leq j \leq n$ we replace the $j$-th row by the difference between  
the $j$-th and the $(j-1)$-th row. We obtain:
\begin{eqnarray*}
F(a,n)&=&
\det \begin{bmatrix}
1 & a & a(a+1) & \dots & a(a+1)\cdots (a+n-2) \\
0 & 1 & 2(a+1) & \dots & (n-1)(a+1)\cdots (a+n-2) \\
\vdots & \vdots & \vdots & \vdots &\vdots\\
0 & 1 & 2(a+n-1) & \dots & (n-1)(a+n-1)\cdots (a+2n-4) \\
0 & 1 & 2(a+n) & \dots & (n-1)(a+n)\cdots (a+2n-3)
\end{bmatrix}\\
&=&
\det \begin{bmatrix}
1 & 2(a+1) & \dots & (n-1)(a+1)\cdots (a+n-2) \\
\vdots & \vdots & \vdots & \vdots \\
1 & 2(a+n-1) & \dots & (n-1)(a+n-1)\cdots (a+2n-4) \\
1 & 2(a+n) & \dots & (n-1)(a+n)\cdots (a+2n-3)
\end{bmatrix}\\
&=&(n-1)!
\det \begin{bmatrix}
1 & (a+1) & \dots & (a+1)\cdots (a+n-2) \\
\vdots & \vdots & \vdots & \vdots \\
1 & (a+n-1) & \dots & (a+n-1)\cdots (a+2n-4) \\
1 & (a+n) & \dots & (a+n)\cdots (a+2n-3)
\end{bmatrix}\\
&=&(n-1)! F(a+1, n-1)\,.
\end{eqnarray*}
Iterating, we conclude
\begin{equation*}
F(n,a)=(n-1)! F(a+1, n-1)= \dots= (n-1)! \cdots 2!\, F(a+n-2,2)=\prod_{k=1}^{n-1} k!
\end{equation*}
since 
\begin{equation*}
F(a+n-2,2)=\det \begin{bmatrix} 1 & a+n-2 \\ 1 & a+n-1 \end{bmatrix}=1\,.
\end{equation*}
\end{proof}
We may identify
\[
\ss1=\Wv=\Hom(\V_{\overline 1},\V_{\overline 0})=M_{l,l'}(\C),
\]
so that $\tau(w)=wi\overline w^t$
and $J(w)=-i w$, $w\in \ M_{l,l'}(\C)$, is a compatible positive complex structure on $\Wv$ contained in $\g$. Then (\ref{symplectic form}) implies that
\[
\langle J(w),w\rangle=\Re\,\tr(J\tau(w))=\tr(w\overline w^t).
\]
Hence our normalization of the Lebesgue measure $dw$ on $\Wv$ is such that for each entry, $dz=dx\, dy$ if $z=x+iy \in \C$. In particular, if we let 
$\Phi(w)=e^{-\tr(w\overline{w}^t)}$ then
\[
\int_\Wv \Phi(w)\; dw=\pi^{ll'}\,. 
\]

\begin{lem}\label{the constant for weyl integration on w moved}
With the above notation,
\begin{equation} \label{eq:gaussian-hs-GG' moved}
\int_{\hs1^2} |\pi_{\so/\hs1^2}(w^2)|\int_{\G\times\G'} \Phi(s.w) \; ds\,dw^2= \vol(G)\vol(G') 
 \prod_{k=0}^{l} k! \, 
\prod_{k=0}^{l-1} (k+l'-l)! \,.
\end{equation}
Consequently,
\begin{equation*} 
\int_{\Wv} f(w) \, dw= C \int_{\hs1^2} |\pi_{\so/\hs1^2}(w^2)|\int_{\G\G'} f(s.w) \; ds\,dw^2 \qquad (f\in C_c(\Wv))\,,
\end{equation*}
where
\begin{equation} \label{eq:C moved}
C=\frac{\pi^{ll'}}{\vol(\G)\vol(\G') \prod_{k=0}^{l} k! \prod_{k=0}^{l-1} (k+l'-l')! }\,.
\end{equation}
\end{lem}
\begin{proof}
By $\G \times \G'$-invariance,
\begin{equation*}
\int_{\G\times \G'} \Phi(s.w)\, ds=\vol(G)\vol(G') \Phi(w)\,.
\end{equation*}
Moreover, 
\begin{equation*}
|\pi_{\so/\hs1^2}(w^2)|=|\pi_{\g'/\h'}(\tau'(w))||\pi_{\g/\z}(\tau(w))|\,.
\end{equation*}
Using (\ref{product of positive roots for g})  and (\ref{product of positive roots for g'/z'}), up to a constant of absolute value one,
\begin{eqnarray}
\label{eq:comp constant superWeyl udud', comp1 moved}
&&\int_{\hs1^2} |\pi_{\so/\hs1^2}(w^2)|\int_{\G\G'} \Phi(s.w) \; ds\,dw^2 \nn\\
&&\qquad= \int_{(\R^+)^{l}} \prod_{1\leq j < k \leq l} (x_j-x_k)^2 \Big(\prod_{j=1}^{l} x_j^{l'-l}\Big) 
e^{-x_1-\dots-x_{l}} \; dx_1\cdots dx_{l}\,.
\end{eqnarray}
Recall that 
\begin{equation*}
\int_0^\infty x^\alpha e^{-x} \; dx= \alpha!\,.
\end{equation*}
Since
\begin{eqnarray*}
\prod_{1\leq j < k \leq l} (x_j - x_k)^2&=&
\Big( \sum_{s \in \Sigma_{l}} \sgn(s)\; x_1^{s(1)-1} \cdots x_{l}^{s(l)-1} \Big)^2\\
&=& \sum_{s,t \in \Sigma_l} \sgn(st) \; x_1^{s(1)+t(1)-2} \cdots x_{d'}^{s(l)+t(l)-2}\,,
\end{eqnarray*}
formula (\ref{eq:comp constant superWeyl udud', comp1 moved}) is equal to
\begin{eqnarray}
&&\sum_{s,t \in \Sigma_{l}} \sgn(st) \int_{(\R^+)^l}  x_1^{s(1)+t(1)+l'-l-2} \cdots x_{l}^{s(l)+t(l)+l'-l-2}  \nn
e^{-x_1-\dots-x_{l}} \; dx_1\cdots dx_{l}  \nn \\
&&\qquad\qquad =\sum_{s,t \in \Sigma_{l}} \sgn(st) \big(s(1)+t(1)+l'-l-2\big)! \cdots \big(s(l)+t(l)+l'-l-2\big)!\nn \\
&&\qquad\qquad = |\Sigma_{l}| \sum_{s\in \Sigma_{l}} \sgn(s) \prod_{j=1}^{l} \big(s(j)+j+l'-l-2\big)! \nn \\
&&\qquad\qquad = l! \det \left[ \big(k+j+l'-l-2\big)! \right]_{j,k=1}^{l}\,.
\end{eqnarray}
Applying Lemma \ref{lemma:Fan moved}, we obtain
\begin{multline*}
l!\det \left[ \big(k+j+l'-l-2\big)! \right]_{j,k=1}^{l}\\
=l!\prod_{k=0}^{l-1} (k+l'-l))! F(3+(l'-l-2),l)=\prod_{k=0}^{l-1} (k+l'-l)! \; \prod_{k=0}^{l} k!\,,
\end{multline*}
which proves the lemma.
\end{proof}
One can relate the Haar measure on the group to the Lebesgue measure on the Lie algebra via the following formulas 
\begin{eqnarray}\label{jacobian of c_ moved}
&&d\diesis{c_-}(x)=2^{l^2} \, \ch^{-2l}(x)\,dx \qquad (x\in\g),\\
&&d\diesis{c_-}(x)=2^l \, \ch^{-2}(x)\,dx \qquad (x\in\h)\,.\nn
\end{eqnarray}
Also, we have the following, easy to verify, equation
\begin{eqnarray}\label{Delata pi moved}
&&\pi_{\g/\h}(x) = 2^{-\frac{l(l-1)}{2}}\, \Delta(\diesis{c_-}(x))\,\ch^{l-1}(x) \qquad (x\in\h)\,.
\end{eqnarray}
\begin{lem}\label{a constant in another intermediate lemma2}
The constant $C$ in (\ref{another intermediate lemma2}), with the roles of $\G$ and $\G'$ reversed, is equal to 
\begin{equation*}
2^{-(l'-l)(l'-l+1)/2}\vol(\diesis{{\H''}}_o)\,.
\end{equation*} 
\end{lem}
\begin{prf}
This follows from (\ref{Delata pi moved}) and (\ref{jacobian of c_ moved})  with $l'-l$ in place of $l$.
\end{prf}
\begin{lem}\label{lemma:HC's formula - the constant moved}
The constant $c_\Pi$ in Proposition \ref{ul, ul' 3} is equal in absolute value to
\begin{equation} \label{eq:CPi' moved}
(2\pi)^{l} 2^{1-l'l+\frac{l(l+1)}{2}}  
 \,\frac{\partial(\pi_{\g/\h})\big(\pi_{\g/\h}\big)}{|W(\G,\h)|}\, \frac{\vol(\G)}{ \vol(\H)}\,
\pi^{-\frac{l(l-1)}{2}} 
=
(2\pi)^{l} 2^{1+l(l-l')}\,. 
\end{equation}
\end{lem}
\begin{prf}
The value of the constant $c_{\Pi}$ is obtained by repeating the computation in the proof of Proposition \ref{ul, ul' 3} keeping track of the constants, and the following formula, due to Macdonald 
\cite[p. 95]{Macdonald}: 
\begin{equation} \label{eq:MacDonald}
\frac{\partial(\pi_{\g/\h})\big(\pi_{\g/\h}\big)}{|W(\G,\h)|}\, \frac{\vol(\G)}{ \vol(\H)}=(2\pi)^N\,,
\end{equation}
where $N=\frac{l(l-1)}{2}$ is the number of positive roots. 
\end{prf}
\begin{lem} \label{lemma:partial pig'h' at 0}
For every smooth function $F:\h \to \C$ 
\begin{equation} \label{eq:partial pig'h' f at 0}
\big(\partial(\pi_{\g/\h})( \pi_{\g/\h} F)\big)(0)=\Big(\prod_{k=0}^{l} k!\Big) \,F(0)\,.
\end{equation}
\end{lem}
\begin{proof}
In terms of  the coordinates $x_j=iy_j$ ($1\leq j\leq l)$,
\begin{eqnarray}
\label{eq:partial pig'h'}
\partial(\pi_{\g/\h})&=& \prod_{1\leq j < k \leq l}  
(\partial_{x_j}-\partial_{x_k})
=\sum_{s \in \Sigma_l} \sgn(s) \partial_{x_1}^{s(1)-1} \cdots \partial_{x_l}^{s(l)-1}\,. 
\end{eqnarray}
By the product rule, 
\begin{equation*}
\big(\partial(\pi_{\g/\h})( \pi_{\g/\h} F)\big)(0)=\partial(\pi_{\g/\h})(\pi_{\g/\h}) F(0)\,.
\end{equation*}
Moreover, if $\delta_{s,t}$ denotes Kronecker's delta, then
\begin{eqnarray} \label{eq:partialdelta(delta)}
\partial(\pi_{\g/\h})(\pi_{\g/\h})&=&
\sum_{s \in \Sigma_l} \sgn(s) \partial_{x_1}^{s(1)-1} \cdots \partial_{x_{l}}^{s(l)-1}
\Big(\sum_{t \in \Sigma_l} \sgn(t) {x_1}^{t(1)-1} \cdots {x_{l}}^{t(l)-1}  \Big) \nn\\
&=&\sum_{s \in \Sigma_l} \sgn(s) \Big(\sum_{t \in \Sigma_l} \sgn(t)\,  \delta_{s,t} \; 
\prod_{k=1}^{l} (t(k)-1)!\Big) \nn\\
&=&|\Sigma_l| \; \prod_{k=1}^{l} (k-1)! \nn \\
&=&\; \prod_{k=0}^{l} k!
\end{eqnarray}
\end{proof}
\begin{lem}\label{H.13}
The following equality holds
\begin{equation}\label{eq:dimPi dualpair}
\dim\Pi'=\frac{1}{\prod_{j=1}^{l} (l'-j)!}  \; \prod_{j=1}^{l} \frac{(\delta+\mu_j-1)!}{(\mu_j-\delta)!}
 \prod_{1\leq j < k \leq l} (\mu_j -\mu_k)\,.
\end{equation} 
\end{lem}
\begin{prf}
By Weyl's dimension formula, 
\begin{equation}\label{eq:dimPi}
\dim \Pi'= \prod_{\alpha >0} \frac{\langle \mu', \alpha\rangle}{\langle\rho,\alpha\rangle}=
\frac{\prod_{k=2}^{l'} (\mu'_1-\mu'_k) \dots \prod_{k=l+1}^{l'} (\mu'_{l}-\mu'_k)}
{\prod_{k=2}^{l'} (k-1) \dots \prod_{k=l+1}^{l'} (k-l)}\,,
\end{equation}
where $\langle\cdot,\cdot\rangle$ is the form on $\h'{}^*$ induced by (any nonzero multiple) of the Killing form on $\g'$. 
Indeed, recall that, in our conventions, the positive roots of $(\g'_\C,\h'_\C)$ are of the form
$$\alpha_{j,k}(x)=x_j-x_k \qquad (1\leq j < k\leq l'\;, x \in \h')\,,$$
and 
$$
\rho'=(\rho'_1,\dots,\rho'_{l'}) \qquad\text{with \quad $\rho'_j=\tfrac{1}{2}(l'-2j+1)$ for $1\leq j \leq l'$\,.}
$$
Let $\rho''$ be the $\rho$-function for $\u_{d-d'}$. 
Let us fix the form on $\h'{}^*$ associated with the trace form $\langle x,y\rangle=-\tr(xy)$ on $\h'$.
Then
$$\langle\rho,\alpha_{j,k}\rangle=k-j\,.$$
Recall that 
\begin{eqnarray}
\label{eq:muj j small}
&&\mu'_j=\delta+n_j   \qquad (1 \leq j \leq l\;, n_j \in \Ze^+)\\
\label{eq:muj j large}
&&\mu'_j=\rho''_{j-l}=\tfrac{1}{2}(l'-l+1-2(j-l))=\delta-j+l \qquad (l+1\leq j \leq l')\,.
\end{eqnarray}
Hence
$$\langle \mu',\alpha_{j,k}\rangle=k-j \qquad (l+1<j<k\leq l')\,.$$
This proves the second equality in (\ref{eq:dimPi}).
Observe now that 
$$\prod_{k=j}^{l'} (k-j)=\prod_{k=1}^{l'-j} k=(l'-j)!\,.$$
The denominator of (\ref{eq:dimPi}) is therefore equal to 
$$\prod_{j=1}^{l} (l'-j)!$$
For the numerator of (\ref{eq:dimPi}), using (\ref{eq:muj j large}), we have
\begin{eqnarray*}
&&\prod_{k=2}^{l'} (\mu'_1-\mu'_k) \dots \prod_{k=l+1}^{l'} (\mu'_{l}-\mu'_k)\\
&&\qquad\begin{matrix}
=\; (\mu'_1-\mu'_2)\;\;(\mu'_1-\mu'_3)\;\;\cdots \;\;(\mu'_1-\mu'_{l})\;\;(\mu'_1-\mu'_{l+1})\;\;\;\cdots \;\;\;(\mu'_1-\mu'_{l'})\\
\hfill \times \;\;(\mu'_2-\mu'_3)\;\;\cdots \;\;(\mu'_2-\mu'_{l})\;\;(\mu'_2-\mu'_{l+1})\;\;\;\cdots\;\;\; (\mu'_2-\mu'_{l'})\\
\hfill \times \;\; \ddots \hskip 7truecm \vdots \qquad\\
\hfill  \times \;\; (\mu'_{l-1}-\mu'_{l})(\mu'_{l-1}-\mu'_{l+1})\cdots (\mu'_{l-1}-\mu'_{l'})\\
\hfill  \times \;\; (\mu'_{l}-\mu'_{l+1})\;\;\;\cdots\;\;\; (\mu'_{l}-\mu'_{l'})
\end{matrix}\\
&&\qquad =\prod_{1\leq j<k \leq l} (\mu'_j-\mu'_k) \; \prod_{\stackrel{1\leq j \leq l}{l+1 \leq k \leq l'}}  (\mu'_j-\mu'_k)\\
&&\qquad =\prod_{1\leq j<k \leq l} (\mu'_j-\mu'_k) \; \prod_{\stackrel{1\leq j \leq l}{l+1 \leq k \leq l'}}  (n_j+k-l)\\
&&\qquad =\prod_{1\leq j<k \leq l} (\mu'_j-\mu'_k) \; \prod_{1\leq j \leq l}\; \prod_{1 \leq k \leq l'-l}  (n_j+k)\\
&&\qquad =\prod_{1\leq j<k \leq l} (\mu'_j-\mu'_k) \; \prod_{j=1}^{l}\frac{(2\delta+ n_j-1)!}{n_j!}
\end{eqnarray*}
\end{prf}
\begin{pro}\label{the constant is one}
Up to a constant of absolute value one,  
\begin{equation}  \label{eq:TThetaPi'at0}
T(\check\Theta_{\Pi})(0)=\vol(\wt G) \cdot \dim\Pi'\,.
\end{equation}
Equivalently, $\Pi \otimes \Pi'$ is contained in $\omega$ exactly once.
\end{pro}
\begin{prf}
The projection of $\omega$ onto its isotypic component of type $\Pi$ is given by 
\begin{eqnarray*}
\vol(\wt \G) \cdot  P_\Pi&=
&\dim \Pi \cdot \int_{\wt G} \check\Theta_\Pi(\wt g) \omega(\wt g) \; d\wt g\\
&=&\dim \Pi \cdot  \int_{\wt G} \check\Theta_\Pi(\wt g)  \OP(\mathcal K( T(\wt g)))\; d\wt g\\
&=&\dim \Pi \cdot \OP \circ \mathcal K \Big( \int_{\wt G} \check\Theta_\Pi(\wt g)  T(\wt g)\; d\wt g\Big)\\
&=&\dim \Pi \cdot \OP \circ \mathcal K\big(T(\check\Theta_\Pi)\big)\,.
\end{eqnarray*}
Also, with $K=\mathcal K\big(T(\check\Theta_\Pi\big)$, we have
\begin{equation*}
\tr(\OP(K))=\int_\X K(x,x)\; dx=T(\check\Theta_\Pi)(0)\,.
\end{equation*}
It follows that the dimension of the isotypic component of type $\Pi$ is 
\begin{eqnarray} \label{eq:dim-isotypic-Pi}
\tr(P_\Pi)&=&\dim \Pi \cdot \tr \big(\OP\Big(\mathcal K(T(\check\Theta_\Pi)\Big)\big)  
\frac{1}{\vol(\wt\G)}\nn\\
&=&\dim \Pi \cdot \frac{T(\check\Theta_\Pi)(0)}{\vol(\wt\G)}\,.
\end{eqnarray}
Hence $\Pi \otimes \Pi'$ is contained in $\omega$ exactly once if and only if (\ref{eq:TThetaPi'at0}) holds.

Notice first that, by definition, $(b-1)!P_{a,b,2}(\xi)$ is a polynomial function of $a$ and $b$.
Hence, using (\ref{ul, ul' 3 1}), we have that 
$$\prod_{j=1}^{l} (-\mu_j-\delta)! \; T(\check\Theta_{\Pi})(0)$$ 
is polynomial function
of $(\mu_1,\mu_2,\dots, \mu_{l})$. It therefore suffices to prove (\ref{eq:TThetaPi'at0}) for 
$-\mu_j$ satisfying \eqref{ul, ul' 2.2}.

By Lemma \ref{lemma:partial pig'h' at 0}, $T(\check\Theta_{\Pi})(0)$ can be computed from 
(\ref{eq:partial pig'h' f at 0}) by evaluating at $0$ the function
$$e^{\frac{\pi}{4}\langle J w, w\rangle}T(\check\Theta_{\Pi'})(w)$$
determined in Proposition \ref{ul, ul' 3}.
Set $x_j=\pi y_j$ and $\partial_j=\partial_{x_j}$. 
Then, by  (\ref{ul, ul' 3 1}), (\ref{eq:CPi' moved}) and (\ref{eq:partial pig'h' f at 0}),
\begin{eqnarray} \label{eq:id at 0 for udud', comp1}
&&\Big({\prod_{k=0}^{} k!\Big)} T(\check\Theta_{\Pi})(0) \nn\\
&&\quad=c_{\Pi}   \partial(\pi_{\g/\h}) \Big( \sum_{t\in W(\G,\h)} \sgn(t) \prod_{j=1}^{l} P_{(t\mu)_j-\delta+1, -(t\mu)_j-\delta+1, -2}(\pi y_j)\Big)(0)  \nn\\
&&\quad=c_{\Pi}   (-i \pi)^{\frac{l(l-1)}{2}} \prod_{1\leq j < k\leq l} (\partial_j -\partial_k) 
\Big( \sum_{t\in W(\G,\h)} \sgn(t) \prod_{j=1}^{l} P_{(t\mu)_j-\delta+1, -(t\mu)_j-\delta+1, -2}(x_j)\Big)(0) \nn\\
&&\quad=c_{\Pi}  (-i \pi)^{\frac{l(l-1)}{2}}  \sum_{t\in W(\G,\h)} \sgn(t)  \sum_{s \in W(\G/\h)} \sgn(s)  
\prod_{j=1}^{l} \big(\partial_j^{s(j)-1}  P_{(t\mu)_j-\delta+1, -(t\mu)_j-\delta+1, -2}\big)(0) \nn\\
&&\quad=c_{\Pi}  (-i \pi)^{\frac{l(l-1)}{2}}  \sum_{t,s \in W(\G,\h)} \sgn(ts)   
\prod_{j=1}^{l} \big(\partial_j^{s(j)-1}  P_{(t\mu)_j-\delta+1, -(t\mu)_j-\delta+1, -2}\big)(0)
\nn\\
&&\quad=c_{\Pi}  (-i \pi)^{\frac{l(l-1)}{2}}  |W(\G,\h)| \sum_{s \in W(\G,\h)} \sgn(s)   
\prod_{j=1}^{l} \big(\partial_j^{s(j)-1}  P_{\mu_j-\delta+1, -\mu_j-\delta+1, -2}\big)(0)
\,.
\end{eqnarray}

According to Lemma  \ref{lemma:diff and at  for P}, we have
\begin{eqnarray}
&&\partial_j^{s(j)-1}  P_{\mu_j-\delta+1, -\mu_j-\delta+1, -2}(0)  \nn\\
&&\qquad\qquad = P_{\mu_j-s(j)-\delta+2, -\mu_j-\delta+1, -2}(0) \nn \\
&&\qquad\qquad =(-1)^{\mu_j-\delta-s(j)+2} \; 2^{s(j)+2(\delta-1)} \binom{\mu_j+\delta-1}{s(j)+2(\delta-1)},
\end{eqnarray}
where last equality holds under the assumption \eqref{ul, ul' 2.2}.
Notice that 
\begin{eqnarray}
\prod_{j=1}^{l} \binom{\mu_j+\delta-1}{s(j)+2(\delta-1)}
&=&\prod_{j=1}^{l} \frac{(\mu_j+\delta-1)!}{(s(j)+2(\delta-1))! (\mu_j-s(j)-\delta+1)!} \nn\\
&=&\prod_{j=1}^{l} \frac{1}{(j+2(\delta-1))!} \; 
\prod_{j=1}^{l} \frac{(\mu_j+\delta-1)!}{(\mu_j-s(j)-\delta+1)!} \nn\\
&=&\prod_{j=1}^{l} \frac{1}{(l'-j)!} \; \prod_{j=1}^{l} \frac{(\mu_j+\delta-1)!}{(\mu_j-s(j)-\delta+1)!}\,.
\end{eqnarray}
Hence, omitting constants of absolute value one, we have
\begin{eqnarray} \label{eq:id at 0 for udud', comp2}
&&\sum_{s \in W(\G,\h)} \sgn(s)   
\prod_{j=1}^{l} \partial_j^{s(j)-1}  P_{\mu_j-\delta+1, -\mu_j-\delta+1, -2}(0)  \\
&&\quad=2^{ll'-\frac{l(l+1)}{2}} \sum_{s \in W(\G,\h)} \sgn(s)  \binom{\mu_j+\delta-1}{s(j)+2(\delta-1)} \nn\\
&&\quad=2^{ll'-\frac{l(l+1)}{2}} \prod_{j=1}^{l} \frac{(\mu_j+\delta-1)!}{(l'-j)!}  
\sum_{s \in W(\G,\h)} \sgn(s) \prod_{j=1}^{l} \frac{1}{(\mu_j-s(j)-\delta+1)!} \nn
\end{eqnarray}
By (\ref{eq:Vandermonde-1}), 
\begin{equation} \label{eq:id at 0 for udud', comp3}
\sum_{s \in W(\G,\h)} \sgn(s) \prod_{j=1}^{l} \frac{(\mu_j-\delta)!}{(\mu_j-s(j)-\delta+1)!} =
 \prod_{1\leq j < k \leq l} (\mu_j -\mu_k)\,.
\end{equation}
Comparing Lemma \ref{H.13}, (\ref{eq:id at 0 for udud', comp1}), (\ref{eq:id at 0 for udud', comp2}) and (\ref{eq:id at 0 for udud', comp3}), 
we deduce the following equality
\begin{equation}
\Big(\prod_{k=0}^{l} k!\Big) T(\check\Theta_{\Pi})(0)=C_{\Pi} (i\pi)^{\frac{l(l-1)}{2}} 2^{ll'-\frac{l(l+1)}{2}} |W(\G,\h)| \dim\Pi'\,.
\end{equation}
We therefore conclude that 
$T(\check\Theta_{\Pi})(0)=K_{\Pi} \dim\Pi'$,
where 
\begin{eqnarray}\label{eq:compKPi'}
|K_{\Pi}|&=&C_{\Pi} \pi^{\frac{l(l-1)}{2}} 2^{ll'-\frac{l(l+1)}{2}}\, 
\frac{|W(\G,\h)| }{\prod_{k=0}^{l} k!} \nn\\
&=&2\,\frac{(2\pi)^{\frac{l(l-1)}{2}}}{\prod_{k=0}^{l-1} k!}\; (2\pi)^{l} \nn\\
&=&2\,\vol(\G)\\
&=&\vol(\wt \G) \nn\,.
\end{eqnarray}
In (\ref{eq:compKPi'}) we have used (\ref{eq:MacDonald}), (\ref{eq:partialdelta(delta)}) and that 
$\vol(\H)=(2\pi)^{l}$.
\end{prf}
\section{\bf Limits of orbital integrals in the stable range.\rm}\label{Limits of orbital integrals in the stable range.}
The results on the limits of the orbital integrals obtained in the section \ref{Intertwining distributions} did not make any assumption on the relative sizes of the groups $\G$ and $\G'$. They are based on the adaptation to orbital integrals on the symplectic space $\Wv$ of Harish-Chandra's study of orbital integrals on a Lie algebra.
However, if we restrict ourselves to the stable range, we may obtain the same results without any reference to Harish-Chandra's work. This is another indication on how natural is the stable range assumption in the theory of reductive dual pairs. For instance, recently Lock and Ma, \cite{LockMaassocvar}, computed the correspondence of the associated varieties under this assumption. This is equivalent to computing the wave front set correspondence by the work of Schmid and Vilonen, \cite{SchmidVilonen2000}.

\begin{lem}\label{difference in homogenity degrees}
The following inequality holds 
\begin{equation}\label{difference in homogenity degrees1}
\dim \Wv-\dim \Oo_m'-\dim \g-\dim \h\geq 0.
\end{equation}
The two sides of (\ref{difference in homogenity degrees1}) are equal if and only if 
\begin{equation}\label{difference in homogenity degrees2}
(\G,\G')=(\Og_2,\Sp_{2l'}(\R));\ \ (\Og_3,\Sp_{2l'}(\R));\ \ (\Ug_1, \Ug_{p,q}),\ 1\leq p\leq q.
\end{equation}
\end{lem}
\begin{prf}
Let $F(d)$ denote The quantity (\ref{difference in homogenity degrees1}) as a function of $d=\dim \V_{\overline 0}\geq m$. This is a concave down quadratic function. We know from (\ref{structure of t'-1tau(0)3}) that $F(m)=\dim\g-\dim\h\geq0$. Also, (\ref{structure of t'-1tau(0)2}) gives the following explicit formula
\[
F(d)=\left\{
\begin{array}{lll}
2dm-m^2-m-\frac{d^2}{2}\ \  & \text{if}\ \ \G=\Og_{2l},\ \ d'=2m,\\
2dm-m^2-m-\frac{d^2-1}{2}\ \  & \text{if}\ \ \G=\Og_{2l+1},\ \ d'=2m,\\
2dd'-2md'+2m^2-d^2-d\ \  & \text{if}\ \ \G=\Ug_{d},\\
8dm-4m^2+2m-2d^2-2d\ \  & \text{if}\ \ \G=\Sp_{d},\ \ d'=2m.
\end{array}
\right.
\]
Hence, $F(d')\geq 0$ if $\G\ne \Og_{2l+1}$ and $F(d'+1)\geq 0$ if $\G=\Og_{2l+1}$. This verifies (\ref{difference in homogenity degrees1}).
 
Since $m\leq d\leq d'$ if $\G\ne \Og_{2l+1}$, and $m\leq d\leq d'+1$ if $\G=\Og_{2l+1}$, the equality in (\ref{difference in homogenity degrees1}) implies that $F(m)=0$. But $F(m)=\dim\g-\dim\h$. 
Hence, (\ref{difference in homogenity degrees2}) follows.
\end{prf}
\begin{lem}\label{nilpandhcint2}
Assume $d'\geq 2r$. If  the pair $(\G,\G')$ is in the stable range with $\G$-the smaller member, 
then one may normalize the positive orbital integral $\mu_{\Oo_{m}}$ so that
\[
\underset{y\to 0}{\lim}\ \frac{1}{C_{\hs1}\pi_{\g/\h}(y)}f_\phi(y)=\mu_{\Oo_{m}}(\phi) \qquad (\phi\in \Ss(\Wv)),
\]
where the limit is taken over $y\in\reg{\h}$. (The stable range assumption implies $d'\geq 2r$.)
If the pair $(\G,\G')$ is not in the stable range with $\G$-the smaller member then the limit is zero. 
\end{lem}
\begin{prf}
Since $\G$ is compact, the orbit $\G\cdot y\subseteq \g$ is compact and hence the Fourier transform
\[
\hat\mu_{\G.y}(x)=\int_\g e^{-iB(z,x)}\,d\mu_{\G.y}(z) \qquad (x\in \g)
\]
is a uniformly bounded function as $y$ varies through $\h$.  
Furthermore, by (\ref{relation of r with degree})
\begin{eqnarray}\label{nilpandhcint2.2}
&&\int_\g \ch^{-{d'}}(x)\,dx=\int_\h|\pi_{\g/\h}(x)|^2\ch^{-{d'}}(x)\,dx\leq \int_\h\ch^{-d'+2r-2\iota}(x)\,dx\\
&&\leq \int_\h\ch^{-2\iota}(x)\,dx<\infty.\nn
\end{eqnarray}
Hence, by van der Corput estimate (\ref{seventh sterp}) for $\phi\in\Ss(\Wv)$, the consecutive integrals in the following computation are absolutely convergent:
\begin{eqnarray}\label{nilpandhcint2.3}
&&\int_\g \hat\mu_{\G.y}(x)\int_\Wv\chi_x(w)\phi(w)\,dw\,dx\\
&=&C_1\int_\h \hat\mu_{\G.y}(x)\overline{\pi_{\g/\h}(x)}\pi_{\g/\h}(x)\int_\Wv\chi_x(w)\phi^\G(w)\,dw\,dx\nn\\
&=&C_2\int_\h \frac{1}{\pi_{\g/\h}(y)}\sum_{s\in W(\G,\h)}\sgn_{\g/\h}(s) e^{-iB(y,sx)}\int_\h e^{iB(x,z)}f_\phi(z)\,dz\,dx\nn\\
&=&C_2\int_\h \frac{1}{\pi_{\g/\h}(y)}\sum_{s\in W(\G,\h)} e^{-iB(y,sx)}\int_\h e^{iB(x,z)}f_\phi(sz)\,dz\,dx\nn\\
&=&C_2\int_\h \frac{1}{\pi_{\g/\h}(y)}\sum_{s\in W(\G,\h)} e^{-iB(y,sx)}\int_\h e^{iB(sx,z)}f_\phi(z)\,dz\,dx\nn\\
&=&C_2|W(\G,\h)|\int_\h \frac{1}{\pi_{\g/\h}(y)} e^{-iB(y,x)}\int_\h e^{iB(x,z)}f_\phi(z)\,dz\,dx\nn\\
&=&C_3 \frac{1}{\pi_{\g/\h}(y)}f_\phi(y),\nn
\end{eqnarray}
where $C_1$, $C_2$ and $C_3$ are some constants and in the second equality we used Lemmas \ref{lemma:HC's formula moved} and \ref{reduction to ss-orb-int}. Hence,
\begin{eqnarray}\label{nilpandhcint2.4}
&&\underset{y\to 0}{\lim}\ C_3 \frac{1}{\pi_{\g/\h}(y)}f_\phi(y)=
\underset{y\to 0}{\lim}\ \int_\g \hat\mu_{\G.y}(x)\int_\Wv\chi_x(w)\phi(w)\,dw\,dx\\
&=&\int_\g \hat\mu_0(x)\int_\Wv\chi_x(w)\phi(w)\,dw\,dx=\int_\g \int_\Wv\chi_x(w)\phi(w)\,dw\,dx.\nn
\end{eqnarray}
By definition of $f_\phi$,  (\ref{nilpandhcint2.4}) is a limit of $\Sg$-invariant measures. Hence it is an $\Sg$-invariant measure. The last expression shows that this measure is homogeneous of degree $-2\dim\,\g$. Hence, by (\ref{dilation relations 2}), the image of it under the map $\tau'_*$ is $\G'$-invariant and homogeneous of degree $\frac{1}{2}(\dim\,\Wv-2\dim\,\g)-\dim\,\g'$. 
Therefore that image is a multiple of a nilpotent orbital integral over a nilpotent orbit of dimension $\dim\,\Wv-2\dim\,\g$, (see \cite{BarVogAs}).

Under the assumption $d'\geq 2r$, if our pair is not in the stable range then $\Dc=\C$, i.e. the pair consists of the unitary groups. 
In this case Lemma \ref{difference in homogenity degrees} shows that 
\[
\dim\,\Wv-2\dim\,\g =\dim \Oo'_k\ \ \text{where}\ \ k=d\ \ \text{or}\ \ d'-d.
\]
But the case $k=d$ is impossible, because $k\leq m$ and we assume that $m<d$. Hence the limit is a possibly zero multiple of the orbital integral over the orbit $\Oo_{d'-d}$, assuming $d'-d\leq m$. But the assumption $d'\geq 2r$ means that $d'\geq 2d$. Hence, $d'-d\geq d>m$. 
Thus the limit is zero.  

Suppose our pair is in the stable range. Then $0\in\Wv_\g$. Furthermore, by Lemma \ref{pullback and orbital integral} and the continuity of the pull-back $\tau^*$,
\begin{eqnarray}\label{nilpandhcint2.1}
&&\frac{1}{C_{\hs1}\pi_{\g'/\z'}(y)}f_\phi(y)=\tau^*(\mu_{\G.y})(\phi)\underset{y\to 0}{\to}\tau^*(\delta_0)(\phi)=\mu_{\Oo_m}(\phi)
\qquad (\phi\in C_c^\infty(\Wv_\g)).
\end{eqnarray}
Clearly, (\ref{nilpandhcint2.1}) implies that the restriction of that measure to $\Wv_\g$ is a multiple of $\mu_\Oo|_{\Wv_\g}$. Also, the map $\tau'_*$ is injective (on invariant distributions). Hence, (\ref{nilpandhcint2.1}) is a non-zero multiple of $\mu_{\Oo_m}$. Notice that the left hand side of (\ref{nilpandhcint2.1}) is a non-negative measure. Thus our lemma follows.
\end{prf}
The two lemmas below shed some light at the connection of our limit formula in Lemma \ref{nilpandhcint2} and Rossmann's result \cite{RossmannNilpotent} concerning limts of nilpotent orbital integrals.
\begin{lem}\label{O'and springer}
Suppose the pair $(\G,\G')$ is in the stable range with $\G$-the smaller member. Define a polynomial $p_{\Oo'}\in\C[\h']$ by
\begin{eqnarray*}
p_{\Oo'}(y+y'')=\left\{
\begin{array}{lll}
\pi_{\g/\h}(y)\pi_{\z''/\h''}(y'') & \text{if}\ \G\ne \Og_{2l+1},\\
\pi_{\g/\h}(y)\pi_{\z''/\h''}^{short}(y'') & \text{if}\ \G=\Og_{2l+1},
\end{array}\right.
\end{eqnarray*}
where $y\in\h,\ y''\in\h''$ and $\pi_{\z''/\h''}^{short}$ is the product of the short roots of $\h''$ in $\z''_\C$.
Then $p_{\Oo'}$ generates the representation of the Weyl group $W(\G_\C',\h_\C')$ corresponding to the orbit $\Oo_\C'$ via Springer correspondence, as explained in \cite[sec.4]{AubertKraskiewiczPrzebinda_real}.
\end{lem}
\begin{prf}
We see from (\ref{product of positive roots for g}) that the polynomial in question is equal to  
\begin{equation}\label{polynomial pO}
\begin{array}{lll}
\left(\prod_{1\leq j<k\leq l}i(- y_j+ y_k)\right)\cdot \left(\prod_{l< j<k\leq l'}i(- y_j+ y_k)\right) & \text{if} & \Dc=\C,\\
\left(\prod_{1\leq j<k\leq l}(-y_j^2+y_k^2)\cdot \prod_{j=1}^l 2iy_j\right)\cdot \left(\prod_{l< j<k\leq l'}(-y_j^2+y_k^2)\right) & \text{if} & \Dc=\Ha,\\
\left(\prod_{1\leq j<k\leq l}(-y_j^2+y_k^2)\right)\cdot \left(\prod_{l< j<k\leq l'}(-y_j^2+y_k^2)\cdot \prod_{j=l+1}^{l'} 2iy_j\right) & \text{if} & \Dc=\R\ \text{and}\ \g=\mathfrak s\mathfrak o_{2l},\\
\left(\prod_{1\leq j<k\leq l}(-y_j^2+y_k^2)\cdot \prod_{j=1}^l iy_j\right) 
\cdot \left(\prod_{l< j<k\leq l'}(-y_j^2+y_k^2)\right) & \text{if} & \Dc=\R\ \text{and}\ \g=\mathfrak s\mathfrak o_{2l+1},\\
\end{array}
\end{equation}
where the parenthesis separate the factors in the definition of $p_{\Oo'}$. In terms of \cite[sec.3]{AubertKraskiewiczPrzebinda_real} the above products may be rewritten, up to a constant multiple, as
\begin{equation}\label{standard polynomial pO}
\begin{array}{lll}
\Delta_{(1^l)}(y_1,\dots,y_l) \Delta_{(1^{l'-l})}(y_{l+1},\dots,y_{l'})  & \text{if} & \Dc=\C,\\
\Delta_{(\emptyset,1^l)}(y_1,\dots,y_l) \Delta_{(1^{l'-l},\emptyset)}(y_{l+1},\dots,y_{l'}) & \text{if} & \Dc=\Ha,\\
\Delta_{(1^l,\emptyset)}(y_1,\dots,y_l) \Delta_{(\emptyset,1^{l'-l})}(y_{l+1},\dots,y_{l'}) & \text{if} & \Dc=\R\ \text{and}\ \g=\mathfrak s\mathfrak o_{2l},\\
\Delta_{(\emptyset,1^l)}(y_1,\dots,y_l) \Delta_{(1^{l'-l},\emptyset)}(y_{l+1},\dots,y_{l'}) & \text{if} & \Dc=\R\ \text{and}\ \g=\mathfrak s\mathfrak o_{2l+1},\\
\end{array}
\end{equation}
Hence the lemma follows from \cite[Theorem 12]{AubertKraskiewiczPrzebinda_real}.
\end{prf}
Observe that for dual pairs in the stable range and for $\psi\in \Ss(\g')$ Rossmann's formula, \cite{RossmannNilpotent} indicates that
\begin{equation}\label{rossmann1}
\underset{y'\to 0}{\lim}\ \partial(p_{\Oo'})\pi_{\g'/\h'}(y')\int_{\G'}\psi(g.y')\,dy = C_1 \mu_{\Oo'}(\psi)
\end{equation}
and Lemma \ref{nilpandhcint2} shows that
\begin{equation}\label{rossmann2}
\underset{y\to 0}{\lim}\ \frac{1}{\pi_{\g/\h}(y)}\underset{y''\to 0}{\lim}\ \partial(\t\pi_{\z''/\h''})\left(\pi_{\g'/\h'}(y+y'')\int_{\G'}\psi(g.(y+y''))\,dy\right) = C_2 \mu_{\Oo'}(\psi),
\end{equation}
where $y\in \tau(\reg{\hs1})$, $\t\pi_{\z''/\h''}=\pi_{\z''/\h''}$ if $\G\ne \Og_{2l+1}$, and $\t\pi_{\z''/\h''}=\pi_{\z''/\h''}^{short}$ if $\G=\Og_{2l+1}$. In fact (\ref{rossmann2}) is a stronger version of (\ref{rossmann1}) because, in general, if $p$ is a product of linear factors vanishing at $0$ and $F$ has limit $F(0)$ at $0$, then $[\partial(p)(pF)](0)=\partial(p)(p) \, F(0)$.

\renewcommand{\thesection}{\Alph{section}}
\renewcommand{\sectionname}{Appendix \thesection:}
\renewcommand{\thethh}{\thesection.\fontindex{thh}}
\renewcommand{\theequation}{\thesection.\fontindex{equation}}

\setcounter{section}{1}
\setcounter{thh}{0}
\setcounter{equation}{0}
\section*{Appendix \thesection: A few facts about nilpotent orbits}\label{Appendix nil-orbits}

Let $\g'$ be a semisimple Lie algebra over $\C$. Then there is a unique non-zero nilpotent orbit in $\g'$ of minimal dimension which is contained in the closure of any non-zero nilpotent orbit, \cite[Theorem 4.3.3, Remark 4.3.4]{CollMc}. The dimension of that orbit is equal to one plus the number of positive roots not orthogonal to the highest root, relative to a choice of a Cartan subalgebra and a choice of positive roots,
\cite[Lemma 4.3.5]{CollMc}. Thus in the case $\g'=\sp_{2l}(\C)$,  the dimension of the minimal non-zero nilpotent orbit is equal to $2l$.
This is precisely the dimension of the defining module for the symplectic group $\Sp_{2l}(\C)$, which may be viewed as the symplectic space for the dual pair $(\Og_1, \Sp_{2l}(\C))$. 

Consider the dual pair $(\G, \G')=(\Og_1, \Sp_{2l}(\R))$, with the symplectic space $\Wv$ and the unnormalized moment map $\tau':\Wv\to \g'$. 
Since $\Wv\setminus\{0\}$ is a single $\G'$-orbit, so is $\tau'(\Wv\setminus\{0\})$. Further, $\dim(\tau'(\Wv\setminus\{0\}))=\dim(\Wv)=2l$.
Hence, $\tau'(\Wv\setminus\{0\})\subseteq \g'$ is a minimal non-zero $\G'$-orbit. In fact, there are only two such orbits, \cite[Theorem 9.3.5]{CollMc}. In terms of dual pairs, the second one is obtained from the same pair, with the symplectic form replaced by its negative (or equivalently the symmetric form on the defining module for $\Og_1$  replaced by its negative).

Consider an irreducible dual pair $(\G, \G')$ with $\G$ compact. Denote by $l$ the dimension of a Cartan subalgebra of $\g$ and by $l'$  the dimension of a Cartan subalgebra of $\g'$. 
Let us identify the corresponding symplectic space $\Wv$ with $\Hom(\V_{\overline 1}, \V_{\overline 0})$ as in \cite[sec.2]{PrzebindaUnipotent}. 

Recall that $\Wv_\g$ denotes the maximal subset of $\Wv$ on which the restriction of the unnormalized moment map $\tau:\Wv \to \g$ is a submersion. 
Then \cite[Lemma 2.6]{PrzebindaUnipotent} shows that $\Wv_\g$ consists of all the elements $w\in\Wv$ such that for any $x\in\g$,
\begin{equation}\label{xw=0 implies x=0}
xw=0\ \text{implies}\ x=0\,.
\end{equation}
The condition (\ref{xw=0 implies x=0}) means that $x$ restricted to the image of $w$ is zero. But in that case $x$ preserves the orthogonal complement of that image. Thus we need to know that the Lie algebra of the isometries of that orthogonal complement is zero. This happens if $w$ is surjective or if $\G$ is the orthogonal group and  the dimension of the image of $w$ in $\V_{\overline 0}$ is $\geq \dim(\V_{\overline 0})-1$. Thus
\begin{equation}\label{wg non-empty}
\Wv_\g\ne \emptyset\; \text{if and only if}\; l\leq l'.
\end{equation}
Consider in particular the dual pair $(\G, \G')=(\Og_3,\Sp_{2l'}(\R))$ with $1\leq l'$. We see from the above discussion that $\Wv_\g$ consists of elements of rank $\geq 2$. Hence, $\Wv\setminus(\Wv_\g\cup\{0\})$ consists of elements $w$ or rank equal $1$. By replacing $\V_{\overline 0}$ with the image of $w$ we may consider $w$ as an element of the symplectic space for the pair $(\Og_1, \Sp_{2l'})$. Hence the image of $w$ under the moment map generates a minimal non-zero nilpotent orbit in $\g'$. 

If  $(\G, \G')=(\Og_2,\Sp_{2l'}(\R))$, with $1\leq l'$, then $\Wv_\g$ consists of elements of rank $\geq 1$. Therefore $\Wv\setminus\Wv_\g=\{0\}$.

\setcounter{section}{2}
\setcounter{thh}{0}
\setcounter{equation}{0}
\section*{Appendix \thesection: Pull-back of a distribution via a submersion}

We collect here some textbook results which are attributed to Ranga Rao in \cite{BarVogAs} and in \cite{HarrisThesis}.  These results date back to the time before the textbook \cite{Hormander} was available. 

We shall use the definition of a smooth manifold and a distribution on a smooth manifold as described in \cite[sec. 6.3]{Hormander}. Thus, if $\M$ is a smooth manifold of dimension $m$ and 
\[
M\supseteq M_\kappa\overset{\kappa}{\to}\tilde M_\kappa\subseteq \R^m
\]
is any coordinate system on $M$,
then a distribution $u$ on $M$ is the collection of distributions $u_\kappa\in \mathcal D'(\tilde M_\kappa)$ such that
\begin{equation}\label{charts and distributions}
u_{\kappa_1}=(\kappa\circ\kappa_1^{-1})^*u_\kappa.
\end{equation}
Suppose $W$ is another smooth manifold of dimension $n$ and $v$ is a distribution on $W$. Thus for any coordinate system 
\[
W\supseteq W_\lambda\overset{\lambda}{\to}\tilde  W_\lambda\subseteq \R^n
\]
we have a distribution $v_\lambda\in \mathcal D'(\tilde  W_\lambda)$ such that the condition (\ref{charts and distributions}) holds. Suppose 
\[
\sigma:M\to W
\]
is a submersion. Then for every $\kappa$ there is a unique distribution $u_\kappa\in \mathcal D'(\tilde M_\kappa)$ such that
\begin{equation}\label{definition of pullback of distributions on manifolds}
u_{\kappa}|_{(\lambda\circ \sigma\circ\kappa^{-1})^{-1}(\tilde  W_\lambda)}=(\lambda\circ \sigma\circ\kappa^{-1})^*v_\lambda.
\end{equation}
Since
\begin{eqnarray*}
(\kappa\circ\kappa_1^{-1})^*(\lambda\circ \sigma\circ\kappa^{-1})^*v_\lambda
=(\lambda\circ \sigma\circ\kappa^{-1}\circ\kappa\circ\kappa_1^{-1})^*v_\lambda
=(\lambda\circ\sigma\circ\kappa_1^{-1})^*v_\lambda
\end{eqnarray*}
the $u_\kappa$ satisfy the condition (\ref{charts and distributions}). The resulting distribution $u$ is denoted by $\sigma^*v$ and is called the pullback of $v$ from $W$ to $M$ via $\sigma$.
\begin{pro}\label{I.1}
Let $M$ and $W$ be smooth manifolds and let $\sigma:M\to W$ be a surjective submersion. Suppose $u_n\in \mathcal D'(W)$ is a sequence of distributions such that
\begin{equation}\label{I.1.1}
\underset{n\to\infty}{\lim}\ \sigma^*u_n=0 \qquad \text{in}\ \mathcal D'(M).
\end{equation}
Then
\begin{equation}\label{I.1.2}
\underset{n\to\infty}{\lim}\ u_n=0 \qquad \text{in}\ \mathcal D'(W).
\end{equation}
In particular the map $\sigma^*:\mathcal D'(W)\to \mathcal D'(M)$ is injective.  

More generally, if
$u_n\in \mathcal D'(W)$ and $\t u\in \mathcal D'(M)$ are  such that
\begin{equation}\label{I.1.10}
\underset{n\to\infty}{\lim}\ \sigma^*u_n=\t u \qquad \text{in}\ \mathcal D'(M),
\end{equation}
then there is a distribution $u\in  \mathcal D'(W)$ such that
\begin{equation}\label{I.1.20}
\underset{n\to\infty}{\lim}\ u_n=u \qquad \text{in}\ D'(W)
\end{equation}
and $\t u=\sigma^* u$.
\end{pro}
\begin{prf}
By the definition of a distribution on a manifold, as in \cite[sec.6.3]{Hormander}, we may assume that $M$ is an open subset of $\R^m$ and $W$ is an open subset of $\R^n$. 

We recall the definition of the pull-back
\begin{equation}\label{I.pullback} 
\sigma^*:\mathcal D'(W)\to \mathcal D'(M)
\end{equation}
from the proof of Theorem 6.1.2 in \cite{Hormander}. Fix a point $x_0\in M$ and a smooth map $g:M\to \R^{m-n}$ such that
\[
\sigma\oplus g:M\to \R^n\times \R^{m-n}
\]
has a bijective differential at $x_0$. By Inverse Function Theorem there is an open neighborhood $M_0$ of $x_0$ in $M$ such that
\[
\left(\sigma\oplus g\right)|_{M_0}:M_0\to Y_0
\]
is a diffeomerphism onto an open neighborhood $Y_0$ of $\sigma\oplus g(x_0)=(\sigma(x_0), g(x_0))$ in $\R^n\times \R^{m-n}$.
Let 
\[
h:Y_0\to M_0
\]
denote the inverse. For $\phi\in C_c^\infty(M_0)$ define $\Phi\in C_c^\infty(Y_0)$ by 
\begin{equation}\label{Phiphi}
\Phi(y)=\phi(h(y))|\det\,h'(y)| \qquad (y\in Y_0).
\end{equation}
Then
\begin{equation}\label{I.1.3}
\sigma^*u(\phi)=u\otimes 1(\Phi) \qquad (u\in \mathcal D'(W),\ \phi\in C_c^\infty(M_0)).
\end{equation}
By localization this gives the pull-back (\ref{I.pullback}).

Let $W_0$ be an open neighborhood of $\sigma(x_0)$ in $W$ and let $X_0$ be an open neighborhood of $g(x_0)$ in $\R^{m-n}$ such that
\[
W_0\times X_0\subseteq Y_0.
\]
Fix a function $\eta\in C_c^\infty(X_0)$ such that
\[
\int_{X_0}\eta(x)\,dx=1.
\]
Given $\psi\in C_c^\infty(W_0)$ define $\phi\in C_c^\infty(M_0)$ by
\[
\Phi(x',x'')=\psi(x')\eta(x'') \qquad (x'\in W_0,\ x''\in X_0),
\]
where $\Phi$ is related to $\phi$ via (\ref{Phiphi}). Then
\[
\sigma^*u(\phi)=u(\psi).
\]
Hence the assumption (\ref{I.1.1}) implies 
\[
\underset{n\to\infty}{\lim}\ u_n(\psi)=0 \qquad (\psi\in C_c^\infty(W_0)).
\]
Thus, by \cite[Theorem 2.1.8]{Hormander}, 
\[
\underset{n\to\infty}{\lim\ }u_n|_{W_0}=0
\]
in $\mathcal D'(W_0)$. Since the point $x_0\in M$ is arbitrary, the claim (\ref{I.1.2}) follows by localization.

Similarly,  the assumption (\ref{I.1.10}) implies  that for any $\psi\in C_c^\infty(W_0)$
\[
\underset{n\to\infty}{\lim}\ u_n(\psi)= \underset{n\to\infty}{\lim}\ \sigma^*u_n(\phi)=\t u(\phi) 
\]
exists. Thus, by \cite[Theorem 2.1.8]{Hormander}, there is $u\in \mathcal D'(W_0)$ such that
\[
\underset{n\to\infty}{\lim\ }u_n|_{W_0}=u.
\]
By the continuity of $\sigma^*$, $\sigma^* u=\t u$. Again, since the point $x_0\in M$ is arbitrary, the claim follows by localization.
\end{prf}
\begin{lem}\label{pull-back of differential operators}
Let $M$ and $W$ be smooth manifolds and let $\sigma:M\to W$ be a surjective submersion. Then for any smooth differential operator $D$ on $W$ there is, not necessary unique, smooth differential operator $\sigma^*D$ on $M$ such that
\[
\sigma^*(u\circ D) =(\sigma^*u)\circ (\sigma^*D) \qquad (u\in \mathcal D'(W)).
\]
If $D$ annihilates constants then so does $\sigma^*D$. 
The operator $\sigma^*D$ is unique if $\sigma$ is a diffeomerphism.
\end{lem}
\begin{prf}
Suppose $\sigma$ is a diffeomerphism between two open subsets of $\R^n$. Then
\[
\sigma^*u(\phi)=u(\phi\circ\sigma^{-1}|\det((\sigma^{-1})')|) \qquad (\phi\in C_c^\infty(M)).
\]
Let 
\[
(\sigma^*D)(\phi)=(D(\phi\circ\sigma^{-1}))\circ\sigma \qquad (\phi\in C_c^\infty(M)).
\]
Hence
\begin{eqnarray*}
\sigma^*(u\circ D)(\phi)&=&(u\circ D)(\phi\circ\sigma^{-1}|\det((\sigma^{-1})')|)\\
&=&u(D(\phi\circ\sigma^{-1}|\det((\sigma^{-1})')|))\\
&=&u(((D(\phi\circ\sigma^{-1})\circ\sigma)\circ\sigma^{-1}  |\det((\sigma^{-1})')|))\\
\end{eqnarray*}

Using the local classification of the submersions modulo the diffeomerphism \cite[16.7.4]{DieudonneElements}, we may assume that $\sigma$ is a linear projection
\[
\sigma:\R^{m+n}\ni (x,y)\to x\in \R^n,
\]
in which case the lemma is obvious.
\end{prf}
Suppose $M$ is a Lie group. Then there are functions $m_\kappa\in C^\infty(\tilde M_\kappa)$ such that the formula
\begin{equation}\label{haar measure locally}
\int_M \phi\circ \kappa(y)\,d\mu_M(y)=\int_{\tilde M_\kappa}\phi(x) m_\kappa(x)\,dx \qquad (\phi\in C_c^\infty(\tilde M_\kappa))
\end{equation}
defines a left invariant Haar measure on $M$. We shall tie the normalization of the Haar measure $d\mu_M(y)$ on $M$ to the normalization of the Lebesgue measure $dx$ on $\R^m$ by requiring that near the identity,
\begin{equation}\label{haar measure locally normalized}
m_{\exp^{-1}}(x)=\det\left(\frac{1-e^{-\ad(x)}}{\ad(x)}\right),
\end{equation}
as in \cite[Theorem 1.14, page 96]{HelgasonGeomtric}. Collectively, the distributions $m_\kappa(x)\,dx\in \mathcal D'(\tilde M_\kappa)$ form a distribution density on $M$. (See \cite[sec. 6.3]{Hormander} for the definition of a distribution density.) 

Suppose $W$ is another Lie group with the left Haar measure given by
\[
\int_W \psi\circ \lambda(y)\,d\mu_W(y)=\int_{\tilde W_\lambda}\phi(x) w_\lambda(x)\,dx \qquad (\psi\in C_c^\infty(\tilde W_\lambda)),
\]
and let $\sigma:M\to W$ be a submersion. 
Given any distribution density $v_\lambda\in  \mathcal D'(\tilde W_\lambda)$ we associate to it a distribution on $W$ given by $\frac{1}{w_\lambda}v_\lambda\in  \mathcal D'(\tilde W_\lambda)$. We may pullback this distribution to $M$ and obtain another distribution. Then we multiply by the $m_\kappa$ and obtain a distribution density. Thus, if $\sigma:M_\kappa\to W_\lambda$ then
\begin{equation}\label{pullback of distribution densities}
(\sigma^*v)_\kappa=m_\kappa (\lambda\circ \sigma\circ \kappa^{-1})^*(\frac{1}{w_\lambda}v_\lambda).
\end{equation}
Distribution densities on $W$ are identified with the continuous linear forms on $C_c^\infty(W)$ by
\[
v(\psi\circ \lambda)=v_\lambda(\psi) \qquad (\psi\in C_c^\infty(\tilde W_\lambda)).
\]
(Here $v$ stands for the corresponding continuous linear form.) In particular if $F\in C(W)$. then $F\mu_W$ is a continuous linear form on $C_c^\infty(W)$ and for $\psi\in C_c^\infty(\tilde W_\lambda)$,
\begin{eqnarray*}
(F\mu_W)_\lambda(\psi)&=&(F\mu_W)(\psi\circ \lambda)=\int_W \psi\circ \lambda(y)F(y)\,d\mu_W(y)\\
&=&\int_{\tilde W_\lambda} \psi(x) F\circ \lambda^{-1}(x) w_\lambda(x)\,dx. 
\end{eqnarray*}
Hence, for $\phi\in C_c^\infty(\tilde M_\kappa)$, with $\sigma:M_\kappa\to W_\lambda$,
\begin{eqnarray*}
(\sigma^*(F\mu_W))_\kappa(\phi)&=&(\lambda\circ \sigma\circ \kappa^{-1})^*(\frac{1}{w_\lambda}(F \mu_W)_\lambda)(m_\kappa \phi)\\
&=&\int_{\tilde M_\kappa} m_\kappa (x) \phi(x) F\circ\lambda^{-1}\circ (\lambda \circ \sigma\circ \kappa^{-1})(x)\,dx\\
&=&\int_{\tilde M_\kappa} \phi(x) (F \circ \sigma)\circ \kappa^{-1}(x)m_\kappa (x) \,dx\\
&=&\int_M \phi\circ \kappa(y) (F \circ \sigma)(y) \,d\mu_M(y)\\
\end{eqnarray*}
Thus
\begin{equation}\label{pullback of function times Haar}
\sigma^*(F\mu_W)=F\circ \sigma \mu_M.
\end{equation}
As explained above, we identify $\mathcal D'(M)$ with the space of the continuous linear forms on $C_c^\infty(M)$ and similarly for $W$ and obtain
\begin{equation}\label{final pullback}
\sigma^*:\mathcal D'(M)\to \mathcal D'(W)
\end{equation}
as the unique continuous extension of (\ref{pullback of function times Haar}). Our identification of distribution densities with continuous linear forms on on the space of the smooth compactly supported functions applies also to submanifolds of Lie groups.

Let $\Sg$ be a Lie group acting on another Lie group $W$ and let $U\subseteq W$ be a submanifold. (In our applications $W$ is going to be a vector space.) We shall consider the following function
\begin{equation}\label{the map f}
\sigma:\Sg\times U\ni (s,u)\to su\in W.
\end{equation}
The following fact is easy to check
\begin{lem}\label{intersetion of orbit with U}
If $\mathcal O\subseteq W$ is an $\Sg$-orbit then $\sigma^{-1}(\mathcal O)=\Sg\times (\mathcal O\cap U)$. 
\end{lem}
Assume that the map (\ref{the map f}) is submersive. Let us fix Haar measures on $\Sg$ and on $W$ so that the pullback
\[
\sigma^*:\mathcal D'(W)\to \mathcal D'(\Sg\times U)
\]
is well defined, as in (\ref{final pullback}).
Denote by $\Sg^U\subseteq \Sg$ the stabilizer of $U$.
\begin{lem}\label{main pullback lemma}
Assume that the map (\ref{the map f}) is submersive and surjective. Let
$\mathcal O\subseteq W$ be an $\Sg$-orbit and let $\mu_{\mathcal O}\in \mathcal D'(W)$ be an $\Sg$-invariant positive measure supported on the closure on $\mathcal O$. Let $\mu_{\mathcal O}|_{U}\in \mathcal D'(U)$ be the restriction of $\mu_{\mathcal O}$ to $U$ in the sense of \cite[Cor. 8.2.7]{Hormander}. Then $\mu_{\mathcal O}|_{U}$ is a positive $\Sg^U$-invariant measure supported on the closure of $\mathcal O\cap U$ in $U$. Moreover, 
\begin{equation}\label{main pullback lemma1}
\sigma^*\mu_{\mathcal O}=\mu_\Sg\otimes \mu_{\mathcal O}|_{U}.
\end{equation}
\end{lem}
\begin{prf}
Let $s\in\Sg^U$. Then
\[
s^*\left(\mu_{\mathcal O}|_{U}\right)=\left(s^*\mu_{\mathcal O}\right)|_{U}=\mu_{\mathcal O}|_{U}.
\]
Hence the distribution $\mu_{\mathcal O}|_{U}$ is $\Sg^U$-invariant. Lemma \ref{I.1} implies that $\mu_{\mathcal O}|_{U}\ne 0$ and Lemma \ref{intersetion of orbit with U} that $\mu_{\mathcal O}|_{U}$ is supported in the closure of $\mathcal O\cap U$ in $U$. Since the pullback of a positive measure is a non-negative measure, $\mu_{\mathcal O}|_{U}$ is a positive $\Sg^U$-invariant measure supported on the closure of $\mathcal O\cap U$ in $U$.

Theorem 3.1.4' in \cite{Hormander} implies that there is a positive measure $\mu_{\mathcal O\cap U}$ on $U$ such that 
\[
\sigma^*\mu_{\mathcal O}=\mu_\Sg\otimes \mu_{\mathcal O\cap U}.
\]
Consider the embedding
\[
\sigma_1:U\ni u\to (1,u)\in \Sg\times U.
\]
Then $\sigma\circ \sigma_1:U\to W$ is the inclusin of $U$ into $W$. Hence,
\[
(\sigma\circ \sigma_1)^*\mu_{\mathcal O}=\mu_{\mathcal O}|_{U}.
\]
The conormal bundle to $\sigma_1$, as defined in \cite[Theorem 8.2.4]{Hormander}, is equal to
\[
N_{\sigma_1}=T_{\{1\}\times U}^*=T^*(\Sg)\times 0\subseteq T^*(\Sg)\times T^*(U)=T^*(\Sg\times U).
\]
By the $\Sg$-invariance of $\sigma^*\mu_{\mathcal O}$, 
\[
WF(\mu_\Sg\otimes \mu_{\mathcal O\cap U})\subseteq 0\times T^*(U)\subseteq T^*(\Sg\times U).
\]
Hence 
\[
N_{\sigma_1}\cap WF(\mu_\Sg\otimes \mu_{\mathcal O\cap U})=0.
\]
Therefore
\[
\mu_{\mathcal O}|_{U}=(\sigma\circ \sigma_1)^*\mu_{\mathcal O}=\sigma_1^*\circ \sigma^*\mu_{\mathcal O}=\sigma_1^*(\mu_\Sg\otimes \mu_{\mathcal O\cap U})=\mu_{\mathcal O\cap U}.
\]
This implies (\ref{main pullback lemma1}).
\end{prf}

\setcounter{section}{3}
\setcounter{thh}{0}
\setcounter{equation}{0}
\section*{Appendix \thesection: Some confluent hypergeometric polynomials}\label{Appendix C: Some confluent hypergeometric polynomials}

For two integers $a$ and $b$ define the following functions in the real variable $\xi$,
\begin{eqnarray}\label{eq:Pab2}
P_{a,b,2}(\xi)&=&\left\{
\begin{array}{ll}
\sum_{k=0}^{b-1}\frac{a(a+1)...(a+k-1)}{k!(b-1-k)!}2^{-a-k}\xi^{b-1-k}&\text{if}\ b\geq 1\\
0&\text{if}\ b\leq 0,
\end{array}
\right.\\
\label{eq:Pabminus2}
P_{a,b,-2}(\xi)&=&\left\{
\begin{array}{ll}
(-1)^{a+b-1}\sum_{k=0}^{a-1}\frac{b(b+1)...(b+k-1)}{k!(a-1-k)!}(-2)^{-b-k}\xi^{a-1-k}&\text{if}\ a\geq 1\\
0&\text{if}\ a\leq 0,
\end{array}
\right.
\end{eqnarray}
where $a(a+1)...(a+k-1)=1$ if $k=0$. 
Notice that 
\begin{equation}\label{D0}
P_{a,b,-2}(\xi)=P_{b,a,2}(-\xi) \qquad (\xi\in\R,\ a,b\in\Zb).
\end{equation}
Set
\begin{eqnarray}\label{D0'}
P_{a,b}(\xi)&=&2\pi(P_{a,b,2}(\xi)\Bbb I_{\R^+}(\xi)+P_{a,b,-2}(\xi)\Bbb I_{\R^-}(\xi))\\
&=&2\pi(P_{a,b,2}(\xi)\Bbb I_{\R^+}(\xi)+P_{b,a,2}(-\xi)\Bbb I_{\R^+}(-\xi)),\nn
\end{eqnarray}
where $\Bbb I_S$ denotes the indicator function of the set $S$.
Also, let
\begin{eqnarray}\label{D0''}
Q_{a,b}(iy)=2\pi \left\{
\begin{array}{ll}
0 & \text{if}\ a+b\geq 1,\\
\sum_{k=b}^{-a} \frac{a(a+1)...(a+k-1)}{k!}2^{-a-k}(1-iy)^{k-b} & \text{if}\ -a>b-1\geq 0,\\
\sum_{k=a}^{-b} \frac{b(b+1)...(b+k-1)}{k!}
2^{-b-k}(1+iy)^{k-a} & \text{if}\ -b>a-1\geq 0,\\
(1+iy)^{-a}(1-iy)^{-b} & \text{if}\ a\leq 0\ \text{and}\ b\leq 0.
\end{array}
\right.
\end{eqnarray}
\begin{pro}\label{D1}
For any $a,\ b\in\Bbb Z$, the formula
\begin{equation}\label{D1.1}
\int_\R (1+iy)^{-a}(1-iy)^{-b}\phi(y)\,dy \qquad (\phi\in \Ss(\R))
\end{equation}
defines a tempered distribution on $\R$. The restriction of the Fourier transform of this distribution to $\R\setminus \{0\}$ is a function given by 
\begin{eqnarray}\label{D1.2}
\int_\R
(1+iy)^{-a}(1-iy)^{-b} e^{-iy\xi}\,dy
=P_{a,b}(\xi)e^{-|\xi|}.
\end{eqnarray}
The right hand side of (\ref{D1.2}) is an absolutely integrable function on the real line and thus defines a tempered distribution on $\R$. 
Furthermore,
\begin{eqnarray}\label{D1.3}
(1+iy)^{-a}(1-iy)^{-b}=\frac{1}{2\pi}\int_\R P_{a,b}(\xi) e^{-|\xi|} e^{iy\xi}\,dy+\frac{1}{2\pi}Q_{a,b}(iy)
\end{eqnarray}
and hence,
\begin{eqnarray}\label{D1.4}
\int_\R(1+iy)^{-a}(1-iy)^{-b}e^{-iy\xi}\,dy=P_{a,b}(\xi)e^{-|\xi|}+ Q_{a,b}(-\frac{d}{d\xi})\delta_0(\xi).
\end{eqnarray}
\end{pro}
\begin{prf}
Since, $|1\pm iy|=\sqrt{1+y^2}$, (\ref{D1.1}) is clear. The integral (\ref{D1.2}) is equal to
\begin{eqnarray}\label{D1.30}
&&\frac{1}{i}\int_{i\R}
(1+z)^{-a}(1-z)^{-b} e^{-z\xi}\,dz\\
&=&2\pi (-\Bbb I_{\R^+}(\xi)\,res_{z=1}(1+z)^{-a}(1-z)^{-b} e^{-z\xi}
+\Bbb I_{\R^-}(\xi)\,res_{z=-1}(1+z)^{-a}(1-z)^{-b} e^{-z\xi}).\nn
\end{eqnarray}
The computation of the two residues is straightforward and (\ref{D1.2}) follows.

Since
$$
\int_0^\infty e^{-\xi} e^{i\xi y}\,d\xi=(1-iy)^{-1},
$$
we have 
\begin{equation}\label{d1}
\int_0^\infty \xi^m e^{-\xi} e^{i\xi y}\,d\xi
=\left(\frac{d}{d(iy)}\right)^{m}(1-iy)^{-1}=m! (1-iy)^{-m-1} \qquad (m=0, 1, 2, ...).
\end{equation}
Thus, if $b\geq 1$, then 
\begin{eqnarray*}
&&\int_0^\infty P_{a,b,2}(\xi) e^{-\xi} e^{i\xi y}\,d\xi
=\sum_{k=0}^{b-1} \frac{a(a+1)...(a+k-1)}{k!}2^{-a-k}(1-iy)^{-b+k}\\
&=&(1-iy)^{-b}2^{-a}\sum_{k=0}^{b-1} \frac{(-a)(-a-1)...(-a-k+1)}{k!}
\left(-\frac{1}{2}(1-iy)\right)^{k}.
\end{eqnarray*}
Also, if $a\leq 0$, then
\begin{eqnarray*}
&&2^a(1+iy)^{-a}=\left(1-\frac{1}{2}(1-iy)\right)^{-a}
=\sum_{k=0}^{-a}
\left(
\begin{array}{r}
-a\\
k
\end{array}
\right) 
\left(-\frac{1}{2}(1-iy)\right)^{k}\\
&=&\sum_{k=0}^{-a} \frac{(-a)(-a-1)...(-a-k+1)}{k!}
\left(-\frac{1}{2}(1-iy)\right)^{k}.
\end{eqnarray*}
Hence, 
\begin{eqnarray}\label{d2}
&&\int_0^\infty P_{a,b,2}(\xi) e^{-\xi} e^{i\xi y}\,d\xi-
(1+iy)^{-a}(1-iy)^{-b}\\
&=&(1-iy)^{-b}2^{-a}\left(\sum_{k=0}^{b-1} \frac{(-a)(-a-1)...(-a-k+1)}{k!}
\left(-\frac{1}{2}(1-iy)\right)^{k}\right.\nn\\
&&\left.\ \ \ \ \ \ \ \ \ \ \ \ \ \ \ \ \ -\sum_{k=0}^{-a} \frac{(-a)(-a-1)...(-a-k+1)}{k!}
\left(-\frac{1}{2}(1-iy)\right)^{k}\right).\nn
\end{eqnarray}
Recall that $P_{a,b,-2}=0$ if $a\leq 0$. Hence, (\ref{D1.2}) shows that (\ref{d2}) is the inverse Fourier transform of a distribution supported at $\{0\}$ - a polynomial.

Suppose $-a<b-1$. Then (\ref{d2}) is equal to
$$
2^{-a}\sum_{k=-a+1}^{b-1} \frac{(-a)(-a-1)...(-a-k+1)}{k!}\left(-\frac{1}{2}\right)^{k}(1-iy)^{k-b},
$$
which tends to zero if $y$ goes to infinity. The only polynomial with this property is the zero polynomial. Thus in this case (\ref{d2}) is zero.
If $-a=b-1$, then (\ref{d2}) is obviously zero. 

Suppose $-a>b-1$. Then (\ref{d2}) is equal to
\begin{equation}\label{d3}
-2^{-a}\sum_{k=b}^{-a} \frac{(-a)(-a-1)...(-a-k+1)}{k!}\left(-\frac{1}{2}\right)^{k}(1-iy)^{k-b}.
\end{equation}
As in (\ref{d1}) we have
\begin{equation*}\label{d4}
\int_{-\infty}^0 \xi^m e^{\xi} e^{i\xi y}\,d\xi
=\left(\frac{d}{d(iy)}\right)^{m}(1+iy)^{-1}=(-1)^m m! (1+iy)^{-m-1} \qquad (m=0, 1, 2, ...).
\end{equation*}
Suppose $a\geq 1$. Then 
\begin{eqnarray*}
&&\int_{-\infty}^0 P_{a,b,-2}(\xi) e^{\xi} e^{i\xi y}\,d\xi
=(-1)^{a+b-1}\sum_{k=0}^{a-1} \frac{b(b+1)...(b+k-1)}{k!}(-2)^{-b-k}
(-1)^{a-1+k}(1+iy)^{-a+k}\\
&=&(1+iy)^{-a}2^{-b}\sum_{k=0}^{a-1} \frac{(-b)(-b-1)...(-b-k+1)}{k!}
\left(-\frac{1}{2}(1+iy)\right)^{k}.
\end{eqnarray*}
Also, if $b\leq 0$, then
\begin{eqnarray*}
&&2^{b}(1-iy)^{-b}=
\sum_{k=0}^{-b} \frac{(-b)(-b-1)...(-b-k+1)}{k!}
\left(-\frac{1}{2}(1+iy)\right)^{k}.
\end{eqnarray*}
Hence, 
\begin{eqnarray}\label{d5}
&&\int_{-\infty}^0 P_{a,b,-2}(\xi) e^{\xi} e^{i\xi y}\,d\xi-
(1+iy)^{-a}(1-iy)^{-b}\\
&=&(1+iy)^{-a}2^{-b}\left(\sum_{k=0}^{a-1} \frac{(-b)(-b-1)...(-b-k+1)}{k!}
\left(-\frac{1}{2}(1+iy)\right)^{k}\right.\nn\\
&&\left.\ \ \ \ \ \ \ \ \ \ \ \ \ \ \ \ \ -\sum_{k=0}^{-b} \frac{(-b)(-b-1)...(-b-k+1)}{k!}
\left(-\frac{1}{2}(1+iy)\right)^{k}\right).\nn
\end{eqnarray}
As before, we show that (\ref{d5}) is zero if $-b\leq a-1$. If $-b>a-1$, then (\ref{d5}) is equal to
$$
-2^{-b}\sum_{k=a}^{-b} \frac{(-b)(-b-1)...(-b-k+1)}{k!}
\left(-\frac{1}{2}\right)^{k}(1+iy)^{k-a}.
$$
If $a\geq 1$ and $b\geq 1$, then our computations show that
\begin{eqnarray}\label{d6}
&&\int_0^{\infty} P_{a,b,2}(\xi) e^{-\xi} e^{i\xi y}\,d\xi
+\int_{-\infty}^0 P_{a,b,-2}(\xi) e^{\xi} e^{i\xi y}\,d\xi-
(1+iy)^{-a}(1-iy)^{-b}
\end{eqnarray}
is a polynomial which tends to zero if $y$ goes to infinity. Thus (\ref{d6}) is equal zero. This completes the proof of (\ref{D1.3}). The statement (\ref{D1.4}) is a direct consequence of (\ref{D1.3}).
\end{prf}
The test functions which occur in Proposition \ref{D1} don't need to be in the Schwartz space. In fact the test functions we shall use in our applications are not necessarily smooth. Therefore we'll need a more precise version of the formula (\ref{D1.4}). This requires a definition and two well known lemmas. 

Following Harish-Chandra denote by $\Ss(\R^\times)$ the space of the smooth complex valued functions defined on $\R^\times$ whose all derivatives are rapidly decreasing at infinity and have limits at zero from both sides. For $\psi\in\Ss(\R^\times)$ let
\[
\psi(0+)=\underset{x\to 0+}{\lim}\ \psi(\xi),\ \psi(0+)=\underset{x\to 0-}{\lim}\ \psi(\xi),\ \langle \psi\rangle_0=\psi(0+)-\psi(0-).
\]
In particular the condition $\langle\psi\rangle_0=0$ means that $\psi$ extends to a continuous function on $\R$.
\begin{lem}\label{D1.4lem1}
Let $c=0, 1, 2, \dots$ and let $\psi\in \Ss(\R^\times)$. Suppose
\begin{equation}\label{D1.4lem1.1}
\langle \psi\rangle_0=\dots=\langle \psi^{(c-1)}\rangle_0=0.
\end{equation}
(The condition (\ref{D1.4lem1.1}) is empty if $c=0$.) Then
\begin{equation}\label{D1.4lem1.2}
\left | \int_{\R^\times}e^{-iy\xi}\psi(\xi)\,d\xi \right | \leq \min\{1,|y|^{-c-1}\}(|\langle \psi^{(c)}\rangle_0|+\parallel \psi^{(c+1)}\parallel_1+
\parallel \psi\parallel_1)
\end{equation}
\end{lem}
\begin{prf}
Integration by parts shows that for $z\in\C^\times$
\begin{eqnarray*}
&&\int_{\R^+}e^{-z\xi}\psi(\xi)\,d\xi=z^{-1}\psi(0+)+\dots+z^{-c-1}\psi^{(c)}(0+) +z^{-c-1}\int_{\R^+}e^{-z\xi}\psi^{(c+1)}(\xi)\,d\xi,\\
&&\int_{\R^-}e^{-z\xi}\psi(\xi)\,d\xi=-z^{-1}\psi(0-)-\dots-z^{-c-1}\psi^{(c)}(0-) +z^{-c-1}\int_{\R^-}e^{-z\xi}\psi^{(c+1)}(\xi)\,d\xi.
\end{eqnarray*}
Hence,
\begin{eqnarray*}
&&\int_{\R^\times}e^{-z\xi}\psi(\xi)\,d\xi\\
&=&z^{-1}\langle \psi\rangle_0+\dots+z^{-c}\langle \psi^{(c-1)}\rangle_0 +z^{-c-1}\langle \psi^{(c)}\rangle_0 +z^{-c-1}\int_{\R^\times}e^{-z\xi}\psi^{(c+1)}(\xi)\,d\xi
\end{eqnarray*}
and (\ref{D1.4lem1.2}) follows.
\end{prf}
\begin{lem}\label{D1.4lem2}
Under the assumptions of Lemma \ref{D1.4lem1}, with $1\leq c$,
\[
\int_\R\int_{\R^\times}(iy)^k e^{-iy\xi}\psi(\xi)\,d\xi\,dy=2\pi \psi^{(k)}(0) \qquad (0\leq k\leq c-1),
\]
where each consecutive integral is absolutely convergent.
\end{lem}
\begin{prf}
Since
\[
\int_{\R} |y|^{c-1} \min\{1, |y|^{-c-1}\}\,dy<\infty,
\]
the absolute convergence follows from Lemma \ref{D1.4lem1}. Since the Fourier transform of $\psi$ is absolutely integrable and since $\psi$ is continuous at zero, Fourier inversion formula \cite[(7.1.4)]{Hormander} shows that
\begin{equation}\label{D1.4lem2.1}
\int_\R\int_{\R^\times} e^{-iy\xi}\psi(\xi)\,d\xi\,dy=2\pi \psi(0).
\end{equation}
Also, for $0<k$,
\begin{eqnarray*}
&&\int_{\R^\times}(iy)^k e^{-iy\xi}\psi(\xi)\,d\xi=\int_{\R^\times}(-\partial_\xi)\left((iy)^{k-1} e^{-iy\xi}\right)\psi(\xi)\,d\xi\\
&=&\int_{\R^+}(-\partial_\xi)\left((iy)^{k-1} e^{-iy\xi}\right)\psi(\xi)\,d\xi
+\int_{\R^-}(-\partial_\xi)\left((iy)^{k-1} e^{-iy\xi}\right)\psi(\xi)\,d\xi\\
&=&(iy)^{k-1}\psi(0+)+\int_{\R^+}(iy)^{k-1} e^{-iy\xi}\psi'(\xi)\,d\xi\\
&-&(iy)^{k-1}\psi(0-)+\int_{\R^-}(iy)^{k-1} e^{-iy\xi}\psi'(\xi)\,d\xi\\
&=&(iy)^{k-1}\langle \psi\rangle_0+\int_{\R^\times}(iy)^{k-1} e^{-iy\xi}\psi'(\xi)\,d\xi.
\end{eqnarray*}
Hence, by induction on $k$ and by our assumption
\begin{eqnarray*}
\int_{\R^\times}(iy)^k e^{-iy\xi}\psi(\xi)\,d\xi&=&(iy)^{k-1}\langle \psi\rangle_0+(iy)^{k-2}\langle \psi'\rangle_0+\dots+\langle \psi^{(k-1)}\rangle_0\\
&+&\int_{\R^\times} e^{-iy\xi}\psi^{(k)}(\xi)\,d\xi\\
&=&\int_{\R^\times}e^{-iy\xi}\psi^{(k)}(\xi)\,d\xi.
\end{eqnarray*}
Therefore our lemma follows from (\ref{D1.4lem2.1}).
\end{prf}
The following proposition is an immediate consequence of Lemmas \ref{D1.4lem1}, \ref{D1.4lem2}, and the formula (\ref{D1.3}).
\begin{pro}\label{D1not smooth}
Fix two integers $a,\ b\in\Bbb Z$ and a function $\psi\in \Ss(\R^\times)$. Let $c=-a-b$. If $c\geq 0$ assume that
\begin{equation}\label{D1.4lem1.1not smooth}
\langle \psi\rangle_0=\dots=\langle \psi^{(c)}\rangle_0=0.
\end{equation}
Then
\begin{eqnarray}\label{D1.4not smooth}
&&\int_\R\int_{\R^\times}(1+iy)^{-a}(1-iy)^{-b}e^{-iy\xi}\psi(\xi)\,d\xi\,dy\\
&=&\int_{\R^\times}P_{a,b}(\xi)e^{-|\xi|}\psi(\xi)\,d\xi+ Q_{a,b}(\partial_\xi)\psi(\xi)|_{\xi=0}\nn\\
&=&\int_{\R}\left(P_{a,b}(\xi)e^{-|\xi|}+ Q_{a,b}(-\partial_\xi)\delta_0(\xi)\right)\psi(\xi)\,d\xi\,,.\nn
\end{eqnarray}
where $\delta_0$ denotes the Dirac delta at $0$.\\
(Recall that $Q_{a,b}=0$ if $c<0$ and $Q_{a,b}$ is a polynomial of degree if $c$, if $c\geq 0$.)
\end{pro}
Let $\Ss(\R^+)$ be the space of the smooth complex valued functions whose all derivatives are rapidly decreasing at infinity and have limits at zero. 
Then $\Ss(\R^+)$ may be viewed as the subspace of the functions in $\Ss(\R^\times)$ which are zero on $\R^-$. Similarly we define $\Ss(\R^-)$.
The following proposition is a direct consequence of Proposition \ref{D1not smooth}. We sketch an independent proof below.
\begin{pro}\label{propD2}
There is a seminorm $p$ on the space $\Ss(\R^+)$ such that 
\begin{equation}\label{propD2.0}
\left| \int_{\R^+} e^{-z\xi}\psi(\xi)\,d\xi\right| \leq \min\{1,|z|^{-1}\} p(\psi) \qquad (\psi\in\Ss(\R^+),\ Re\,z\geq 0).
\end{equation}
For any integers $a,b\in\Zb$ such that $a+b\geq 1$ and any function $\psi\in\Ss(\R^+)$
\begin{equation}\label{propD2.1}
\int_\R(1+iy)^{-a}(1-iy)^{-b}\int_{\R^+} e^{-iy\xi}\psi(\xi)\,d\xi\,dy
={2\pi} \int_{\R^+}P_{a,b,2}(\xi)e^{-\xi}\psi(\xi)\,d\xi,
\end{equation}
and
\begin{equation}\label{propD2.2}
\int_\R(1+iy)^{-a}(1-iy)^{-b}\int_{\R^-} e^{-iy\xi}\psi(\xi)\,d\xi\,dy
={2\pi} \int_{\R^-}P_{a,b,-2}(\xi)e^{\xi}\psi(\xi)\,d\xi,
\end{equation}
where each consecutive integral is absolutely convergent.

Let $a,b, c\in\Zb$, $c\geq 0$, be such that $a+b+c\geq 0$. Suppose $\psi\in\Ss(\R^+)$ is such that
\begin{equation}\label{propD2.3}
\psi(0)=\psi'(0)=\dots =\psi^{(c)}(0).
\end{equation}
Then the equalities (\ref{propD2.1}) and (\ref{propD2.2}) hold too.
\end{pro}
\begin{prf}
Clearly
\[
\left| \int_{\R^+} e^{-z\xi}\psi(\xi)\,d\xi\right| \leq \int_{\R^+} e^{-\Re\,z\xi}|\psi(\xi)|\,d\xi\leq \parallel \psi\parallel_1.
\]
Integration by parts shows that for $z\ne 0$,
\[
 \int_{\R^+} e^{-z\xi}\psi(\xi)\,d\xi=z^{-1}\psi(0)+z^{-1} \int_{\R^+} e^{-z\xi}\psi'(\xi)\,d\xi.
\]
Hence (\ref{propD2.0}) follows with $p(\psi)=|\psi(0)|+\parallel\psi\parallel_1+\parallel\psi'\parallel_1$.

Let $a,b\in\Zb$ be such that $a+b\geq 1$. Then the function
\[
(1+z)^{-a}(1-z)^{-b}\int_{\R^+} e^{-z\xi}\psi(\xi)\,d\xi
\]
is continuous on $\Re z\geq 0$ and meromorphic on $\Re z>0$ and (\ref{propD2.0}) shows that it is dominated by
$|z|^{-2}$.
Therefore Cauchy's Theorem implies that the left hand side of (\ref{propD2.1}) is equal to
\begin{equation*}\label{propD2.h}
-2\pi\, {\rm res}_{z=1}\left((1+z)^{-a}(1-z)^{-b}\int_{\R^+} e^{-z\xi}\psi(\xi)\,d\xi\right).
\end{equation*}
The computation of this residue is straightforward.  This verifies (\ref{propD2.1}) 
The proof of (\ref{propD2.2}) is entirely analogous.
 
Integration by parts shows that under the assumption (\ref{propD2.3}) 
\[
\int_{\R^+} e^{-z\xi}\psi(\xi)\,d\xi=z^{-c-1}\int_{\R^+}e^{-z\xi}\psi^{(c+1)}(\xi)\,d\xi \qquad (\Re z\geq 0,\ z\ne 0).
\]
Hence, the above argument caries over and verifies the equalities  (\ref{propD2.1}) and (\ref{propD2.2}).
\end{prf}
\begin{pro}\label{propD3}
Suppose $a,b,c\in\Zb$ are such that
\begin{equation}\label{propD3.a}
b\geq 1,\  a+b+c=1\ \text{and}\ c\geq 0.
\end{equation}
Then
\begin{equation}\label{propD3.b}
P_{a,b,2}(\xi)\xi^c=\frac{(b+c-1)!}{(b-1)!2^c} P_{a+c,b+c,2}(\xi).
\end{equation}
Suppose now that 
\begin{equation}\label{propD3.aminus}
a\geq 1,\  a+b+c=1\ \text{and}\ c\geq 0.
\end{equation}
Then
\begin{equation}\label{propD3.bminus}
P_{a,b,-2}(\xi)\xi^c=(-1)^c \frac{(a+c-1)!}{(a-1)!2^c} P_{a+c,b+c,-2}(\xi).
\end{equation}

\end{pro}
\begin{prf}
Because of (\ref{D0}), it is enough to prove (\ref{propD3.b}).
We compute
\begin{eqnarray*}
P_{a,b,2}(\xi)\xi^c
&=&\sum_{k=0}^{b-1}\frac{a(a+1)\dots (a+k-1)}{k!(b-1-k)!} 2^{-a-k}\xi^{b+c-1-k}\\
&=&\sum_{k=0}^{b-1}(-1)^k\frac{(-a)(-a-1)\dots (-a-k+1)}{k!(b-1-k)!} 2^{-a-k}\xi^{b+c-1-k}\\
&=&\sum_{k=0}^{b-1}(-1)^k\frac{(-a)!}{k!(b-1-k)!(-a-k)!} 2^{-a-k}\xi^{b+c-1-k}\\
&=&\sum_{k=0}^{b-1}(-1)^k\frac{(b+c-1)!}{k!(b-1-k)!(b+c-1-k)!} 2^{-a-k}\xi^{b+c-1-k}\\
\end{eqnarray*}
and
\begin{eqnarray*}
P_{a+c,b+c,2}(\xi)
&=&\sum_{k=0}^{b+c-1}\frac{(a+c)(a+c+1)\dots (a+c+k-1)}{k!(b+c-1-k)!} 2^{-a-c-k}\xi^{b+c-1-k}\\
&=&\sum_{k=0}^{b-1}\frac{(a+c)(a+c+1)\dots (a+c+k-1)}{k!(b+c-1-k)!} 2^{-a-c-k}\xi^{b+c-1-k}\\
&=&\sum_{k=0}^{b-1}(-1)^k\frac{(-a-c)(-a-c-1)\dots (-a-c-k+1)}{k!(b+c-1-k)!} 2^{-a-c-k}\xi^{b+c-1-k}\\
&=&\sum_{k=0}^{b-1}(-1)^k\frac{(b-1)(b-1-1)\dots (b-k)}{k!(b+c-1-k)!} 2^{-a-c-k}\xi^{b+c-1-k}\\
&=&\sum_{k=0}^{b-1}(-1)^k\frac{(b-1)(b-1-1)\dots (b-k)}{k!(b+c-1-k)!} 2^{-a-c-k}\xi^{b+c-1-k}\\
&=&\sum_{k=0}^{b-1}(-1)^k\frac{(b-1)!}{k!(b+c-1-k)!(b-1-k)!} 2^{-a-c-k}\xi^{b+c-1-k}\\
\end{eqnarray*}
and the claim follows.
\end{prf}
By combining Proposition \ref{propD3} with (\ref{D0}) we obtain the following proposition.
\begin{pro}\label{propD4}
Suppose $a,b,c\in\Zb$ are such that
\begin{equation}\label{propD3.a}
a\geq 1,\  a+b+c=1\ \text{and}\ c\geq 0.
\end{equation}
Then
\begin{equation}\label{propD3.b}
P_{a,b,-2}(\xi)(-\xi)^c=\frac{(a+c-1)!}{(a-1)!2^c} P_{a+c,b+c,-2}(\xi).
\end{equation}
\end{pro}

The following lemma is an immediate consequence of the definition of $P_{a,b,2}$, $P_{a,b,-2}$ 
and their relation \eqref{D0}.
\begin{lem}  \label{lemma:diff and at  for P}
Suppose $a,b \in \Ze$.  Then 
\begin{equation} \label{eq:diffP}
P'_{a,b,2}(\xi)=P_{a,b-1,2}(\xi)\quad \text{and} \quad P'_{a,b,-2}(\xi)=P_{a-1,b,-2}(\xi)\,.
\end{equation}
If $b \geq 1$ then 
\begin{equation} \label{eq:valuePat0}
P_{a,b,2}(0)=2^{1-a-b} \, \frac{a(a+1)\dots (a+b-2)}{(b-1)!}\,.
\end{equation}
If, moreover, $a \leq 0$ and $a+b\leq 1$, then 
\begin{equation} \label{eq:valuePat0-bis}
P_{a,b,2}(0)=(-1)^b \; 2^{1-a-b} \;\binom{-a}{-a-b+1}\,.
\end{equation}
If $a \geq 1$ then 
\begin{equation} \label{eq:valuePat0}
P_{a,b,-2}(0)=2^{1-a-b} \, \frac{b(b+1)\dots (b+a-2)}{(a-1)!}\,.
\end{equation}
If, moreover, $b \leq 0$ and $a+b\leq 1$, then 
\begin{equation} \label{eq:valuePat0-bis}
P_{a,b,-2}(0)=(-1)^a \; 2^{1-a-b} \;\binom{-b}{-a-b+1}\,.
\end{equation}

\end{lem}

\setcounter{section}{4}
\setcounter{thh}{0}
\setcounter{equation}{0}
\section*{Appendix \thesection: Wave front set of an asymptotically homogeneous distribution}\label{Appendix WF}
Let
\[
\mathcal F f(x)=\int_{\R^n}f(y)e^{-2\pi ix\cdot y}\,dy
\]
denote the usual Fourier transform on $\R^n$. Recall that for $t>0$ the function $M_t:\R^n \to \R^n$ is defined by $M_t(x)=tx$.
\begin{lem}\label{wave front set 1}
Suppose $f, u\in\Ss^*(\R^n)$ and $u$ is homogeneous of degree $d\in \C$. Suppose
\begin{equation}\label{wave front set 1.1}
t^dM_{t^{-1}}^*f (\psi)\underset{t\to 0+}{\to}u (\psi) \qquad (\psi\in \Ss(\R^n)).
\end{equation}
Then
\begin{equation}\label{wave front set 1.2}
WF_0(\mathcal F^{-1} f)\supseteq \supp u.
\end{equation}
\end{lem}
\begin{prf}
Suppose $\Phi\in C_c^\infty(\R^n)$ is such that $\Phi(0)\ne 0$. We need to show that the localized Fourier transform
\[
\mathcal F((\mathcal F^{-1} f) \Phi)
\]
is not rapidly decreasing in any open cone $\Gamma$ which has a non-empty intersection with $\supp u$. (See \cite[Definition 8.1.2]{Hormander}.) In order to do it, we'll choose a function $\psi\in C_c^\infty(\Gamma)$ such that $u(\psi)\ne 0$ and show that
\begin{eqnarray}\label{wave front set 1.3}
\int_{\R^n}(t^{-1})^{-d}\mathcal F((\mathcal F^{-1} f) \Phi)(t^{-1}x)\psi(x)\,dx
\underset{t\to 0+}{\to} u(\psi),
\end{eqnarray}
assuming $\Phi(0)=1$.
Let $\phi=\mathcal F \Phi$. Then $\int_{\R^n} \phi(x)\,dx=1$. Notice that
\begin{eqnarray}\label{wave front set 1.4}
t^dM_{t^{-1}}^*(f*\phi)=(t^dM_{t^{-1}}^*f)*(t^{-n}M_{t^{-1}}^*\phi),
\end{eqnarray}
so that, by setting $\check \psi(x)=\psi(-x)$, we have
\begin{eqnarray}\label{wave front set 1.5}
&&\int_{\R^n}(t^{-1})^{-d}\mathcal F((\mathcal F^{-1} f) \Phi)(t^{-1}x)\psi(x)\,dx\\
&=&t^dM_{t^{-1}}^*(f*\phi)*\check\psi(0)=(t^dM_{t^{-1}}^*f)*\left((t^{-n}M_{t^{-1}}^*\phi)*\check\psi\right)(0).\nn
\end{eqnarray}
We'll check that for an arbitrary $\psi\in\Ss(\R^n)$
\begin{eqnarray}\label{wave front set 1.6}
(t^{-n}M_{t^{-1}}^*\phi)*\psi\underset{t\to 0+}{\to}\psi
\end{eqnarray}
in the topology of $\Ss(\R^n)$. This, together with (\ref{wave front set 1.5}) and Banach-Steinhaus Theorem, \cite[Theorem 2.6]{RudinFunc}, will imply (\ref{wave front set 1.3}). Explicitly,
\begin{eqnarray}\label{wave front set 1.7}
\big((t^{-n}M_{t^{-1}}^*\phi)*\psi\big)(x)-\psi(x)=\int_{\R^n}\phi(y)(\psi(x-ty)-\psi(x))\,dy.
\end{eqnarray}
Fix $N=0, 1, 2, \dots$ and $\epsilon>0$. Choose $R>0$ so that
\begin{eqnarray}\label{wave front set 1.8}
\int_{|y|\geq R}|\phi(y)|\,dy \cdot \left((|(1+|y|)^N+1\right)\,\sup_{x\in\R^n} (1+|x|)^N|\psi(x)|<\epsilon.
\end{eqnarray}
Let $0<t\leq 1$. Then
\begin{eqnarray}\label{wave front set 1.9}
&&(1+|x|)^N\int_{|y|\geq R}|\phi(y)||\psi(x-ty)|\,dy \\
&\leq& \int_{|y|\geq R}|\phi(y)|(1+|ty|)^N(1+|x-ty|)^N|\psi(x-ty)|\,dy\nn\\
&\leq& \int_{|y|\geq R}|\phi(y)|(1+|y|)^N\,dy\cdot \sup_{x\in\R^n}(1+|x|)^N|\psi(x)|\nn
\end{eqnarray}
and 
\begin{eqnarray}\label{wave front set 1.10}
&&(1+|x|)^N\int_{|y|\geq R}|\phi(y)||\psi(x)|\,dy \\
&\leq& \int_{|y|\geq R}|\phi(y)|\,dy\cdot \sup_{x\in\R^n}(1+|x|)^N|\psi(x)|\nn
\end{eqnarray}
so that, by (\ref{wave front set 1.8}),
\begin{eqnarray}\label{wave front set 1.11}
(1+|x|)^N\left|\int_{|y|\geq R}\phi(y)(\psi(x-ty)-\psi(x))\,dy\right|<\epsilon \qquad (0<t\leq 1,\ x\in\R^n).
\end{eqnarray}
Choose $r>0$ so that
\begin{eqnarray}\label{wave front set 1.12}
(1+|x|)^N\left|\int_{|y|\leq R}\phi(y)(\psi(x-ty)-\psi(x))\,dy\right|<\epsilon \qquad (0<t\leq 1,\ |x|\geq r).
\end{eqnarray}
Since the function $\psi$ is uniformly continuous,
\begin{eqnarray}\label{wave front set 1.13}
\underset{t\to 0+}{\rm limsup}\sup_{|x|\leq r}\left|\int_{|y|\leq R}\phi(y)(\psi(x-ty)-\psi(x))\,dy\right|=0.
\end{eqnarray}
Hence,
\begin{eqnarray}\label{wave front set 1.14}
\underset{t\to 0+}{\rm limsup}\sup_{x\in\R^n}(1+|x|)^N\left|\int_{|y|\leq R}\phi(y)(\psi(x-ty)-\psi(x))\,dy\right|\leq \epsilon.
\end{eqnarray}
By combining (\ref{wave front set 1.11}) and (\ref{wave front set 1.14})  we see that
\begin{eqnarray}\label{wave front set 1.15}
\underset{t\to 0+}{\rm limsup}\sup_{x\in\R^n}(1+|x|)^N\left|\int_{\R^n}\phi(y)(\psi(x-ty)-\psi(x))\,dy\right|\leq 2\epsilon.
\end{eqnarray}
Since the $\epsilon>0$ is arbitrary, (\ref{wave front set 1.15}) and (\ref{wave front set 1.7}) show that
\begin{eqnarray}\label{wave front set 1.16}
\underset{t\to 0+}{\rm limsup}\sup_{x\in\R^n}(1+|x|)^N\left|(t^{-n}M_{t^{-1}}^*\phi)*\psi(x)-\psi(x)\right|=0.
\end{eqnarray}
Since the differentiation commutes with the convolution, (\ref{wave front set 1.16}) implies (\ref{wave front set 1.6})  and we are done.
\end{prf}

\setcounter{section}{5}
\setcounter{thh}{0}
\setcounter{equation}{0}
\section*{Appendix \thesection: 
A proof of a cocycle property}\label{Appendix: A proof of cocycle prop}
Here we proof the formula \eqref{phases}. 
Consider $g_1, g_2\in\Sp^J$ as elements of $\End(\Wv_\C^+)$ by restriction. They preserve the positive definite hermitian form $H(\cdot,\cdot)$, \eqref{half1minusiJ 1}. Let $K_1=\Ker(g_1-1)$, $K_2=\Ker(g_2-1)$, $K_{12}=\Ker(g_1g_2-1)$,  $\Uv_1=(g_1-1)\Wv_\C^+$, $\Uv_2=(g_2-1)\Wv_\C^+$, $\Uv_{12}=(g_1g_2-1)\Wv_\C^+$ and $\Uv=\Uv_1\cap\Uv_2$.

We assume in this appendix that $K_1=\{0\}$. 
In this case $\Uv=\Uv_2$. 
Moreover,
\[
K_2\cap K_{12}=K_1\cap K_2=\{0\}.
\]
Hence there is a subspace $\Wv_2\subseteq \Wv_\C^+$ such that
\begin{equation}\label{23' R}
\Wv_\C^+=K_{12}\oplus \Wv_2\oplus K_2.
\end{equation}
Recall that 
\begin{equation}\label{orthogonal direct sum decomposition}
\Wv_\C^+=\Uv\oplus K_2
\end{equation}
is an orthogonal direct sum decomposition. 
Define an element $h\in\GL(\Wv_\C^+)$ by
\begin{equation}\label{23'' R}
h|_{K_{12}\oplus \Wv_2}=(g_1^{-1}-1)^{-1}(g_2-1),\quad  h|_{K_2}=(g_1^{-1}-1)^{-1}\,.
\end{equation}
Fix a basis $w_i$ of $\Wv_\C^+$ so that $w_i\in K_{12}$ if $i\leq a$, $w_i\in\Wv_2$ if $a<i\leq b$ and 
$w_{b+1}$, $w_{b+2}$, $\ldots$ is a basis of
$K_2$ that is orthonormal with respect to $H$.
Then
\begin{equation}\label{relation 1}
hw_i=w_i \qquad (i\leq a).
\end{equation}
\begin{lem} \label{lem:on h}
The following equalities hold:
\begin{eqnarray}\label{result 0.9}
&&\det(H(  (g_1g_2-1)w_i,hw_j)_{a<i,j})\\
&=&\det(H( \frac{1}{2}(c (g_1)+c(g_2))(g_2-1)w_i,(g_2-1)w_j)_{a<i,j\leq b})\nn\\
&=&\det(H(  (g_1g_2-1)w_i,w_j)_{a<i,j}) \overline{\det(h)}.\nn
\end{eqnarray}
Moreover, we have
\begin{equation}
\label{result 2.1}
\det(H(  w_i,(g_1^{-1}-1)hw_j)_{i,j})\\
=\overline{\det(H( (g_2-1)w_i,w_j)_{i,j\leq b})},
\end{equation}
so that
\begin{equation}\label{result 2.11}
\overline{\det(h)}=\frac{\overline{\det(H( (g_2-1)w_i,w_j)_{i,j\leq b})}}{\det(H(  w_i,(g_1^{-1}-1)w_j)_{i,j})}.
\end{equation}
\end{lem}
\begin{prf}
Notice that both 
\begin{eqnarray*}
c(g_1)&=&(g_1+1)(g_1-1)^{-1}:\Wv_\C^+\to \Wv_\C^+,\\  
c(g_2)&=&(g_2+1)(g_2-1)^{-1}:\Uv\to \Wv_\C^+
\end{eqnarray*}
are well defined on the space $\Uv$ and
\begin{eqnarray}\label{basic relation}
&&(g_1-1)\frac{1}{2}(c(g_1)+c(g_2))(g_2-1)\\
&=&\frac{1}{2}((g_1+1)(g_2-1)+(g_1-1)(g_2+1))\nn\\
&=&g_1g_2-1.\nn
\end{eqnarray}
Suppose $a<i, j\leq b$. Then (\ref{basic relation}) and \eqref{orthogonal direct sum decomposition} show that
\begin{eqnarray}\label{relation 2}
H( (g_1g_2-1)w_i,hw_j)
&=&H( (g_1g_2-1)w_i,(g_1^{-1}-1)^{-1}(g_2-1)w_j)\\
&=&H( (g_1-1)^{-1}(g_1g_2-1)w_i,(g_2-1)w_j)\nn\\
&=&H( (g_1-1)^{-1}(g_1-1)\frac{1}{2}(c(g_1)+c(g_2))(g_2-1)w_i,(g_2-1)w_j)\nn\\
&=&H( \frac{1}{2}(c(g_1)+c(g_2))(g_2-1)w_i,(g_2-1)w_j).\nn
\end{eqnarray}
Suppose $j\leq b<i$. Then $(g_1g_2-1)w_i=(g_1-1)w_i$. Hence,
\begin{eqnarray}\label{relation 3}
H( (g_1g_2-1)w_i,hw_j)
&=&H(  (g_1-1)w_i,(g_1^{-1}-1)^{-1}(g_2-1)w_j)\\
&=&H(  w_i,(g_2-1)w_j)\nn\\
&=&H(  (g_2^{-1}-1)w_i, w_j)\nn\\
&=&H( - g_2^{-1}(g_2-1)w_i, w_j)\nn\\
&=&H(  0, w_j)\nn\\
&=&0.\nn
\end{eqnarray}
If $b<i,j$, then
\begin{eqnarray}\label{relation 4}
&&H(  (g_1g_2-1)w_i,hw_j)=H(  (g_1-1)w_i,hw_j).
\end{eqnarray}
Notice that
\begin{eqnarray}\label{30'' R}
&&\det(H(  (g_1-1)w_i,hw_j)_{b<i,j})
=\det(H(  w_i,(g_1^{-1}-1)hw_j)_{b<i,j})\\
&=&\det(H(  w_i,w_j)_{b<i,j})=1.\nn
\end{eqnarray}
The first equality in~(\ref{result 0.9})
follows from relations (\ref{relation 2}), (\ref{relation 3}), (\ref{relation 4}) and (\ref{30'' R}).

Since $h$ preserves the subspace $K_{12}$, it makes sense to  define $\t h\in \GL(\Wv_\C^+/K_{12})$ by
\[
\t h(w+K_{12})=hw \qquad (w\in\Wv_\C^+).
\]
Then 
\begin{equation*}
\det(H( (g_1g_2-1)w_i,hw_j)_{a<i,j})=\det(H(  (g_1g_2-1)w_i,w_j)_{a<i,j}) \overline{\det(\t h)}.
\end{equation*}
But (\ref{relation 1}) implies $\det(\t h)=\det(h)$. Hence
the second equality in~(\ref{result 0.9}) follows.

Also, if $j\leq b<i$, then 
\[
H(   w_i,(g_1^{-1}-1)hw_j) = H(   w_i,(g_2-1)w_j)=0
\]
because of \eqref{orthogonal direct sum decomposition}. 
Hence,
\begin{eqnarray}\label{result 2.1bis}
\det(H( w_i,(g_1^{-1}-1)hw_j)_{i,j})
&=&\det(H( w_i,(g_1^{-1}-1)hw_j)_{i,j\leq b})
\det(H( w_i,(g_1^{-1}-1)hw_j)_{b<i,j})\nn\\
&=&\det(H( w_i,(g_1^{-1}-1)hw_j)_{i,j\leq b})\nn\\
&=&\det(H( w_i,(g_2-1)w_j)_{i,j\leq b})\nn\\
&=&\overline{\det(H( (g_2-1)w_i,w_j)_{i,j\leq b})}.\nn
\end{eqnarray}
This verifies (\ref{result 2.1}). The formula \eqref{result 2.11} follows immediately from (\ref{result 2.1}).
\end{prf}
\begin{cor}\label{determinants and basis}
With the notation of Lemma \ref{lem:on h}
\begin{eqnarray*}
&&\det(H(\frac{1}{2}(c (g_1)+c(g_2))(g_2-1)w_i,(g_2-1)w_j)_{a<i,j\leq b})\\
&=&\frac{\det(H( (g_1g_2-1)w_i,w_j)_{a<i,j})\overline{\det(H( (g_2-1)w_i,w_j)_{i,j\leq b})}}
{\det(H( (g_1-1)w_i,w_j)_{i,j})}
\end{eqnarray*}
\end{cor}
\begin{lem}\label{lemma 6 R}
Fix two elements $g_1, g_2\in\Sp(\Wv)$ and assume that $K_1=\{0\}$. Let $\V\subseteq\Uv$ denote the radical of the form $H(\frac{1}{2}(c(g_1)+c(g_2))\cdot ,\cdot )_{\Uv}$ and let $H(\frac{1}{2}(c(g_1)+c(g_2))\cdot ,\cdot )_{\Uv/\V}$ denote the resulting non-degenerate form on the quotient
$\Uv/\V$.
Then $\V=(g_2-1)K_{12}$ and
\begin{eqnarray}\label{lemma 6.1}
&&\frac{\det(g_1g_2-1)_{\Uv_{12}}}{\det(g_1-1)_{\Wv_\C^+} \det(g_2-1)_{\Uv}}\\
&=&\frac{\det(H(\frac{1}{2}(c(g_1)+c(g_2))\cdot ,\cdot )_{\Uv/\V})}{|\det(g_2-1\colon K_{12}\to\V)|^{2}},\nn
\end{eqnarray}
where  $|\det(g_2-1\colon K_{12}\to\V)|$ is the absolute value of the determinant of the matrix of $g_2-1$ with respect to any orthonormal basis of 
$K_{12}$ and $\V$.
\end{lem}
\begin{prf}
We use the notation of Lemma \ref{lem:on h} and make the following additional assumptions: $\Wv_2$ is the orthogonal complement of $K_{12}+K_2$ in $\Wv_\C^+$ with respect to $H$, 
$w_1$, $w_2$, ... is a basis of $\Wv_\C^+$ such that $w_1$, $w_2$, ..., $w_a$ is an orthonormal basis of $K_{12}$, and $w_{a+1}$, $w_{a+2}$, ..., $w_b$ is an orthonormal basis of $\Wv_2$. 

Let $Q\in \GL(\Wv_\C^+)$ be such that
\[
\begin{array}{l}
Qw_1, Qw_2,...\ \text{is an orthonormal basis of}\ \Wv_\C^+,\\
Qw_i=w_i\ \text{if}\ i\leq b,\\
Qw_i\ \  \text{is orthogonal to}\ \  K_{12}+\Wv_2\ \text{if}\ b<i.
\end{array}
\]
Define the matrix elements $Q_{j,i}$ by
\[
Qw_i=\sum_{j}Q_{j,i}w_j.
\]
Then
\[
Q_{j,i}=\delta_{j,i}\ \text{if}\ i\leq b.
\]
Hence,
\[
\det(Q)=\det((Q_{j,i})_{1\leq j,i})=\det((Q_{j,i})_{b< j,i})=\det((Q_{j,i})_{a<j,i})
\]
and
\begin{eqnarray*}
1=\det(H( Qw_i, Qw_j)_{1\leq i,j})=|\det(Q)|^2\,\det(H( w_i, w_j)_{1\leq i,j}).
\end{eqnarray*}
Therefore
\begin{equation}\label{6.1 R}
|\det((Q_{j,i})_{a<j,i}|^2\,\det(H( w_i, w_j)_{1\leq i,j})=1.
\end{equation}
It is easy to check from (\ref{basic relation}) that $(g_2-1)K_{12}=\V$. In particular, $\dim \V=\dim K_{12}=a$.
Let $u_1$, $u_2$, ..., $u_b$ be  an orthogonal basis of $\Uv$ such that $u_1$, $u_2$, ..., $u_a$ span $\V$. Define the matrix elements $(g_2-1)_{k,i}$ by
\[
(g_2-1)w_i=\sum_{k=1}^b (g_2-1)_{k,i} u_k \qquad (1\leq i\leq b).
\]
Hence,
\[
(g_2-1)_{k,i}=0\ \text{if}\ i\leq a<k.
\]
Therefore
\begin{eqnarray}\label{6.2 R}
&&\det(((g_2-1)_{k,i})_{1\leq k,i\leq b})\\
&=&\det(((g_2-1)_{k,i})_{1\leq k,i\leq a})\,\det(((g_2-1)_{k,i})_{a< k,i\leq b}).\nn
\end{eqnarray}
Define $h\in\GL(\Wv_\C^+)$ as in (\ref{23'' R}). 
Then, by (\ref{result 2.1}),
\begin{eqnarray}\label{6.4 R}
&&\overline{\det(h)}=\overline{\det((g_1^{-1}-1)^{-1}(g_1^{-1}-1)h)}\\
&=&\overline{\det(g_1^{-1}-1)}^{-1}\,\overline{\det((g_1^{-1}-1)h)}\nn\\
&=&\overline{\det(g_1^{-1}-1)}^{-1}\,\det(H( w_i, (g_1^{-1}-1)hw_j)_{1\leq i,j})\,\det(H( w_i, w_j)_{1\leq i,j})^{-1}\nn\\
&=&\overline{\det(g_1^{-1}-1)}^{-1}\,\overline{\det(H( (g_2-1)w_i, w_j)_{i,j\leq b})}\,\det(H( w_i, w_j)_{1\leq i,j})^{-1}\nn
\end{eqnarray}
Also,
\begin{eqnarray}\label{6.5 R}
&&\det(H(\frac{1}{2}(c(g_1)+c(g_2))(g_2-1)w_i,(g_2-1)w_j)_{a<i,j\leq b})\\
&=&\det(H(\frac{1}{2}(c(g_1)+c(g_2))u_k,u_l)_{a<k,l\leq b})\,|\det(((g_2-1)_{k,i})_{a<k,i\leq b}|^2.\nn
\end{eqnarray}
Since $\Wv_\C^+$ is the orthogonal direct sum of $K_{12}$ and $\Uv_{12}$, the vectors $Qw_j$, $a<j$, form an orthonormal basis of $\Uv_{12}$, so that
\begin{eqnarray}\label{6.6 R}
&&\det(g_1g_2-1)_{\Uv_{12}}=\det(H((g_1g_2-1)Qw_i,Qw_j)_{a<i,j})\\
&=&\,|\det((Q_{i,j})_{a<i,j})|^2\,\det(H ((g_1g_2-1)w_i,w_j)_{a<i,j}).\nn
\end{eqnarray}
Define an element $q\in\GL(\Wv_\C^+)$ by
\[
\begin{array}{l}
qw_i=u_i\ \text{if}\ i\leq b,\\
qw_i=w_i\ \text{if}\ b<i.
\end{array}
\]
Then $qw_1$, $qw_2$, $\ldots$, $qw_b$ is an orthonormal basis of 
$\Uv$ so that
\[
\det(g_2-1)_{\Uv}=\det(H((g_2-1)qw_i,qw_j)_{i,j\leq b}).
\]
Define the coefficients $q_{i,j}$ by
\[
qw_i=\sum_jq_{j,i}w_j.
\]
Then
\[
q_{j,i}=\delta_{j,i}\ \text{if}\ b<i
\]
so that
\[
\det(q)=\det((q_{j,i})_{1\leq i,j})=\det((q_{j,i})_{1\leq i,j\leq b}).
\]
Also,
\[
(g_2-1)qw_i=\sum_jq_{j,i}(g_2-1)w_j=\sum_{j\leq b}q_{j,i}(g_2-1)w_j \qquad (i\leq b).
\]
Therefore,
\[
\det(H((g_2-1)qw_i,qw_j)_{i,j\leq b})=|\det(q)|^2 \det(H((g_2-1)w_i,w_j)_{i,j\leq b}).
\]
Define the coefficients $q^{-1}_{i,j}$ of the inverse map $q^{-1}$ by
\[
w_i=q^{-1}(qw_i)=\sum_j q^{-1}_{i,j} qw_j.
\]
Since, the $qw_i$ form an orthonormal basis of $\Wv_\C^+$,
\[
q^{-1}_{i,j}=H(q^{-1}qw_i,qw_j)=H(w_i,qw_j)=\overline{H(qw_j,w_i)},
\]
so that
\begin{equation*}
q^{-1}_{i,j}=\left\{
\begin{array}{ll}
\overline{H( u_j,w_i)} \ \text{if}\ j\leq b,\\
\overline{H(w_j,w_i)} \ \text{if}\ j> b,\\
\overline{H(w_j,w_i)}=\delta_{i,j} \ \text{if}\ i,j> b.
\end{array}\right.
\end{equation*}
In particular,
\[
q^{-1}_{i,j}=0\ \text{if}\ j\leq b<i
\]
so that
\[
\det(q)^{-1}=\det(q^{-1})=\det((q^{-1}_{i,j})_{i,j\leq b})=\det(\overline{H( u_j,w_i)} _{i,j\leq b}).
\]
Thus
\begin{eqnarray}\label{formula star}
&&\overline{\det(H( (g_2-1)w_i,w_j)_{i,j\leq b})}\, 
\det(g_2-1\colon\Uv\to\Uv)\\
&=&\overline{\det(H( (g_2-1)w_i,w_j)_{i,j\leq b})}\, \det(H( (g_2-1)w_i,w_j)_{i,j\leq b})\,|\det(q)|^2\nn\\
&=&|\det(H( (g_2-1)w_i,w_j)_{i,j\leq b})|^2|\det(q)|^2\nn\\
&=&|\det(H(\sum_{k=1}^b (g_2-1)_{k,i}u_k,w_j)_{i,j\leq b})|^2
\,|\det(q)|^2\nn\\
&=&|\det((g_2-1)_{k,i})_{k,i\leq b})\det(H( u_k,w_j)_{k,j\leq b})|^2
\,|\det(q)|^2\nn\\
&=&|\det((g_2-1)_{k,i})_{k,i\leq b})|^2\nn\\
&=&|\det((g_2-1)_{k,i})_{k,i\leq a})|^2|\det((g_2-1)_{k,i})_{a<k,i\leq b})|^2,\nn
\end{eqnarray}
where the last equality follows from (\ref{6.2 R}).
Notice also that
\begin{eqnarray}\label{formula star1}
\overline{\det(g_1^{-1}-1)}&=&\overline{\det(H((g_1^{-1}-1) Qw_j,Qw_k)_{1\leq j,k})}\\
&=&\overline{\det(H( Qw_j,(g_1-1)Qw_k)_{1\leq j,k})}\nn\\
&=&\det(H((g_1-1)Qw_k, Qw_j)_{1\leq j,k})\nn\\
&=&\det(g_1-1).\nn
\end{eqnarray}
The formula (\ref{lemma 6.1}) follows from \eqref{result 0.9} and (\ref{6.1 R}) - (\ref{formula star}) via a straightforward computation:
\begin{eqnarray*}
&&\frac{\det(g_1g_2-1\colon\Uv_{12}\to\Uv_{12})}{\det(g_1-1\colon\Wv_\C^+\to\Wv_\C^+) 
\det(g_2-1\colon \Uv\to\Uv)}\\
&=&\frac{|\det((Q_{i,j})_{a<i,j})|^2\,\det(H( (g_1g_2-1)w_i,w_j)_{a<i,j})}{\det(g_1-1\colon\Wv_\C^+\to\Wv_\C^+) \det(g_2-1\colon \Uv\to\Uv)}\\
&=&\frac{|\det((Q_{i,j})_{a<i,j})|^2\,\det(H(\frac{1}{2}(c(g_1)+c(g_2))u_k,u_l)_{a<k,l\leq b})\,
|\det(((g_2-1)_{k,i})_{a<k,i\leq b})|^2}{\overline{\det(h)}\det(g_1-1)
 \det(g_2-1\colon\Uv\to\Uv)}\\
	&=&\frac{|\det((Q_{i,j})_{a<i,j})|^2\,\det(H(\frac{1}{2}(c(g_1)+c(g_2))u_k,u_l)_{a<k,l\leq b})\,|\det(((g_2-1)_{k,i})_{a<k,i\leq b})|^2}{\overline{\det(g_1^{-1}-1)}^{-1}\,\overline{\det(H( (g_2-1)w_i, w_j)_{i,j\leq b})}\,
\det(H( w_i, w_j)_{1\leq i,j})^{-1}\det(g_1-1) 
\det(g_2-1\colon\Uv\to\Uv)}\\
&=&\frac{\,\det(H(\frac{1}{2}(c(g_1)+c(g_2))u_k,u_l)_{a<k,l\leq b})\,|\det(((g_2-1)_{k,i})_{a<k,i\leq b})|^2}{\overline{\det(g_1^{-1}-1)}^{-1}\,\overline{\det(H( (g_2-1)w_i, w_j)_{i,j\leq b})}\,\det(g_1-1) 
\det(g_2-1\colon\Uv\to\Uv)}\\
&=&\frac{\,\det(H(\frac{1}{2}(c(g_1)+c(g_2))u_k,u_l)_{a<k,l\leq b})\,|\det(((g_2-1)_{k,i})_{a<k,i\leq b})|^2}{\overline{\det(H( (g_2-1)w_i, w_j)_{i,j\leq b})}\,\det(g_2-1\colon\Uv\to\Uv)}\\
&=&\frac{\det(H(\frac{1}{2}(c(g_1)+c(g_2))u_k,u_l)_{a<k,l\leq b})\,|\det(((g_2-1)_{k,i})_{a<k,i\leq b})|^2}{|\det((g_2-1)_{k,i})_{k,i\leq a})|^2|\det((g_2-1)_{k,i})_{a<k,i\leq b})|^2}\\
&=&\frac{\det(H(\frac{1}{2}(c(g_1)+c(g_2))u_k,u_l)_{a<k,l\leq b})}{|\det((g_2-1)_{k,i})_{k,i\leq a})|^2}.
\end{eqnarray*}
\end{prf}
Since, in terms of \eqref{two signatures},
\begin{eqnarray*}
&&\frac{\det(H(\frac{1}{2}(c(g_1)+c(g_2))\cdot ,\cdot )_{\Uv/\V})}{|\det(H(\frac{1}{2}(c(g_1)+c(g_2))\cdot ,\cdot )_{\Uv/\V})|}
=\frac{\det(H(i(-i)\frac{1}{2}(c(g_1)+c(g_2))\cdot ,\cdot )_{\Uv/\V})}{|\det(H(i(-i)\frac{1}{2}(c(g_1)+c(g_2))\cdot ,\cdot )_{\Uv/\V})|}\\
&=&i^{\sgn h_{g_1,g_2}},
\end{eqnarray*}
the formula \eqref{phases} follows from \eqref{lemma 6.1}.

\biblio

\newcommand{\etalchar}[1]{$^{#1}$}
\begin{thebibliography}{{Har}57b}

\bibitem[AB95]{AdamsBarbaschcomplex}
J.~Adams and D.~Barbasch.
\newblock {Reductive dual pairs correspondence for complex groups}.
\newblock {\em {J. Funct. Anal.}}, {132}:{1--42}, {1995}.

\bibitem[ABP{\etalchar{+}}07]{AdamsBarbaschPaulTrapaVogan}
J.~Adams, D.~Barbasch, A.~Paul, P.~Trapa, and D.~A. Vogan.
\newblock {Unitary Shimura corresondence for split real groups}.
\newblock {\em {Journal of the AMS}}, {20}:{701--751}, {2007}.

\bibitem[Ada83]{Adamsdiscrete}
J.~Adams.
\newblock {Discrete spectrum of the dual pair $(O(p,q), Sp(2m,\Bbb R)$}.
\newblock {\em {Invent. Math}}, {74}:{449--475}, {1983}.

\bibitem[Ada87]{Adamshighestweight}
J.~Adams.
\newblock {Unitary highest weight modules}.
\newblock {\em {Adv. in Math}}, {63}:{113--137}, {1987}.

\bibitem[Ada98]{Adamslift}
J.~Adams.
\newblock {Lifting of characters on orthogonal and metaplectic groups}.
\newblock {\em {Duke Math. J.}}, {92}(1):{129--178}, {1998}.

\bibitem[AKP13]{AubertKraskiewiczPrzebinda_real}
A.-M. Aubert, W.~Kra\'skiewicz, and T.~Przebinda.
\newblock {Howe correspondence and Springer correspondence for real reductive
  dual pairs}.
\newblock {\em {Manuscripta Mathematica}}, {143}:{81--130}, {2013}.

\bibitem[AP14]{AubertPrzebinda_omega}
A.-M. Aubert and T.~Przebinda.
\newblock {A reverse engineering approach to the Weil Representation}.
\newblock {\em {Central European Journal of Mathematics}}, {12}:{1500--1585},
  {2014}.

\bibitem[BP13]{BerPrzeCHC_inv_eig}
F.~Bernon and T.~Przebinda.
\newblock {The Cauchy Harish-Chandra integral and the invariant
  eigendistributions}.
\newblock {\em {International Mathematics Research Notices}}, {2013 (7)},
  {2013}.

\bibitem[BV80]{BarVogAs}
D.~Barbasch and D.~Vogan.
\newblock The local structure of characters.
\newblock {\em J. Funct. Anal.}, 37(1):27--55, 1980.

\bibitem[{C. }89]{Moeglinarchi}
{C. Moeglin}.
\newblock {Correspondance de Howe pour les paires duales r\'eductives duales:
  quelques calculs dans la cas archim\'dien}.
\newblock {\em {J. Funct. Anal.}}, {85}:{1--85}, {1989}.

\bibitem[CM93]{CollMc}
D.~Collingwood and W.~McGovern.
\newblock {\em {Nilpotent orbits in complex semisimple Lie algebras}}.
\newblock {Reinhold, Van Nostrand, New York}, {1993}.

\bibitem[Dad82]{Dadok82}
J.~Dadok.
\newblock {On the $C^\infty$ Chevalley Theorem}.
\newblock {\em {Advances in Math.}}, {44}:{121--131}, {1982}.

\bibitem[Die71]{DieudonneElements}
J.~Dieudonn\'e.
\newblock {\em {{\'E}lements d'Analyse}}.
\newblock {Gauthier-Villars {\'E}diteur}, {1971}.

\bibitem[DKP97]{DaszKrasPrzebindaComplex}
A.~Daszkiewicz, W.~Kra\'skiewicz, and T.~Przebinda.
\newblock {Nilpotent Orbits and Complex Dual Pairs}.
\newblock {\em {Journal of Algebra}}, {190}:{518--539}, {1997}.

\bibitem[DKP05]{DaszKrasPrzebindaK-S2}
A.~Daszkiewicz, W.~Kra\'skiewicz, and T.~Przebinda.
\newblock {Dual Pairs and Kostant-Sekiguchi Correspondence. II. Classification
  of Nilpotent Elements}.
\newblock {\em {Central European Journal of Mathematics}}, {3}:{430--464},
  {2005}.

\bibitem[DP96]{DaszPrzebindaInv}
A.~Daszkiewicz and T.~Przebinda.
\newblock {The oscillator character formula, for isometry groups of split forms
  in deep stable range}.
\newblock {\em {Invent. Math.}}, {123}({2}):{349--376}, {1996}.

\bibitem[DV90]{OrbitesDV}
M.~Duflo and M.~Vergne.
\newblock Orbites coadjointes et cohomologie \'equivariante.
\newblock In {\em The orbit method in representation theory (Copenhagen,
  1988)}, volume~82 of {\em Progr. Math.}, pages 11--60. Birkh\"auser Boston,
  Boston, MA, 1990.

\bibitem[EHW83]{EnrightHoweWallach}
T.~J. Enright, R.~Howe, and N.~R. Wallach.
\newblock {A classification of unitary highest weight modules}.
\newblock {\em {Proceedings of Utah Conference, 1982}}, pages {97--143},
  {1983}.

\bibitem[EW04]{EnrightWillenbring2004}
T.~Enright and J.~Willebbring.
\newblock {Hilbert series, Howe duality and branching for classical groups.}
\newblock {\em {Annals of Mathematics}}, {159}, {2004}.

\bibitem[GZ13]{gomezchengbozhu}
R.~Gomez and C.~Zhu.
\newblock {Local theta lifting of generalized Whittaker models associated to
  nilpotent orbits}.
\newblock {\em {Preprint to appear in Geom. Funct. Ana.. arXiv:1302.3744}},
  {2013}.

\bibitem[{Har}55]{HC1955c}
{Harish-Chandra}.
\newblock {Representations of Semisimple Lie Groups IV}.
\newblock {\em {Amer. J. Math.}}, {77}:{743--777}, {1955}.

\bibitem[{Har}57a]{HC-57DifferentialOperators}
{Harish-Chandra}.
\newblock {Differential operators on a semisimple Lie algebra}.
\newblock {\em {Amer. J. Math.}}, {79}:{87--120}, {1957}.

\bibitem[{Har}57b]{HC-57Fourier}
{Harish-Chandra}.
\newblock {Fourier Transform on a semsimple Lie algebra I}.
\newblock {\em {Amer. J. Math.}}, {79}:{193--257}, {1957}.

\bibitem[{Har}65]{HC-65InvariantEigendistributionsLieAlg}
{Harish-Chandra}.
\newblock {Invariant Eigendistributions on a Semisimple Lie algebra}.
\newblock {\em {Pubi. Math. IHES}}, {27}:{5--54}, {1965}.

\bibitem[Har11]{HarrisThesis}
B.~Harris.
\newblock {Fourier transforms of nilpotent orbits, limit formulas for reductive
  Lie groups and wave fron cycles for tempered representations}.
\newblock {\em {MIT thesis}}, {2011}.

\bibitem[HC64]{HC-64b}
Harish-Chandra.
\newblock Invariant differential operators and distributions on a semisimple
  {L}ie algebra.
\newblock {\em Amer. J. Math.}, 86:534--564, 1964.

\bibitem[He03]{HeHongyu}
H.~He.
\newblock {Unitary Representations and Theta Correspondence for Type I
  Classical Groups}.
\newblock {\em {J. Funct. Anal.}}, {199}:{92--121}, {2003}.

\bibitem[Hel84]{HelgasonGeomtric}
S.~Helgason.
\newblock {\em Groups and Geometric Analysis, Integral Geometry, Invariant
  Differential Operators, and Spherical Functions}.
\newblock {Academic Press}, {1984}.

\bibitem[H{\"o}r83]{Hormander}
L.~H{\"o}rmander.
\newblock {\em {The Analysis of Linear Partial Differential Operators I}}.
\newblock {Springer Verlag}, {1983}.

\bibitem[How79]{howetheta}
R.~Howe.
\newblock {$\theta $}-series and invariant theory.
\newblock In {\em Automorphic forms, representations and $L$-functions (Proc.
  Sympos. Pure Math., Oregon State Univ., Corvallis, Ore., 1977), Part 1},
  Proc. Sympos. Pure Math., XXXIII, pages 275--285. Amer. Math. Soc.,
  Providence, R.I., 1979.

\bibitem[How80]{HoweQuantum}
R.~Howe.
\newblock {Quantum mechanics and partial differential equations}.
\newblock {\em {J. Funct. Anal.}}, {38}:{188--254}, {1980}.

\bibitem[How81]{HoweWave}
R.~Howe.
\newblock {Wave Front Sets of Representations of Lie Groups}.
\newblock In {\em Automorphic forms, Representation Theory and Arithmetic},
  pages {117--140}. Tata Institute of Fundamental Research, Bombay, {1981}.

\bibitem[How89a]{HoweRemarks}
R.~Howe.
\newblock {Remarks on Classical Invariant Theory}.
\newblock {\em {Trans. Amer. Math. Soc.}}, {313}:{539--570}, {1989}.

\bibitem[How89b]{HoweTrans}
R.~Howe.
\newblock {Transcending Classical Invariant Theory}.
\newblock {\em {J. Amer. Math. Soc. 2}}, {2}:{535--552}, {1989}.

\bibitem[Kna86]{knappLie2}
A.~Knapp.
\newblock {\em Representation Theory of Semisimple groups, an overview based on
  examples}.
\newblock Princeton Mathematical Series. Princeton University Press, Princeton,
  New Jersey, 1986.

\bibitem[Li89]{Jian-ShuLiSingular}
Jian-Shu Li.
\newblock {Singular unitary representations of classical groups}.
\newblock {\em {Invent. Math.}}, {97}({2}):{237--255}, {1989}.

\bibitem[LM15]{LockMaassocvar}
H.~Y. Lock and J.J. Ma.
\newblock {Invariants and $K$-spectrum of local theta lifts}.
\newblock {\em {Composition Mathematica}}, {151}:{179--206}, {2015}.

\bibitem[LPTZ03]{Jian-ShuLi-Cheng-boZhu}
J.~S. Li, A.~Paul, E.~C. Tan, and C.~B. Zhu.
\newblock {The explicit duality correspondence of $(\Sp(p,q), \Og^*(2n))$.}
\newblock {\em {J. Funct. Anal.}}, {200}({}):{71--100}, {2003}.

\bibitem[Mac80]{Macdonald}
I.~G. Macdonald.
\newblock {The volume of a compact Lie group}.
\newblock {\em {Invent. Math.}}, {56}:{93--95}, {1980}.

\bibitem[{Moe}98]{Moeglinarchiwave}
{Moeglin, C}.
\newblock {Correspondance de Howe et front d'onde}.
\newblock {\em {Adv. in Math.}}, {133}:{224--285}, {1998}.

\bibitem[MPP15]{McKeePasqualePrzebindaSuper}
M.~McKee, A.~Pasquale, and T.~Przebinda.
\newblock {Semisimple orbital integrals on the symplectic space for a real
  reductive dual pair}.
\newblock {\em {J. Funct. Anal.}}, {268}:{275--335}, {2015}.

\bibitem[NOT{\etalchar{+}}01]{notyk}
Kyo Nishiyama, Hiroyuki Ochiai, Kenji Taniguchi, Hiroshi Yamashita, and Shohei
  Kato.
\newblock {\em NILPOTENT ORBITS, ASSOCIATED CYCLES AND WHITTAKER MODELS FOR
  HIGHEST WEIGHT REPRESENTATIONS}, volume {273}.
\newblock Ast´erisque, 2001.

\bibitem[Pan10]{PanReal}
Shu-Yen Pan.
\newblock {Orbit correspondence for real reductive dual pairs}.
\newblock {\em {Pacific Journal of Mathematics}}, {248}({2}):{403--427},
  {2010}.

\bibitem[Pau98]{AnnegretunitaryI}
A.~Paul.
\newblock {Howe correpondence for real unitary groups I}.
\newblock {\em {J. Funct. Anal.}}, {159}:{384--431}, {1998}.

\bibitem[Pau00]{AnnegretunitaryII}
A.~Paul.
\newblock {Howe correspondence for real unitary groups II}.
\newblock {\em {Proc. Amer. Math. Soc.}}, {128}:{3129--3136}, {2000}.

\bibitem[Pau05]{Annegretorthosymplectic}
A.~Paul.
\newblock {On the Howe correspondence for symplectic-orthogonal dual pairs}.
\newblock {\em {J. Funct. Anal.}}, {228}:{270--310}, {2005}.

\bibitem[Prz91]{PrzebindaUnipotent}
T.~Przebinda.
\newblock {Characters, dual pairs, and unipotent representations}.
\newblock {\em { J. Funct. Anal.}}, {98}({1}):{59--96}, {1991}.

\bibitem[Prz93]{PrzebindaUnitary}
T.~Przebinda.
\newblock {Characters, dual pairs, and unitary representations}.
\newblock {\em {Duke Math. J.}}, {69}({3}):{547--592}, {1993}.

\bibitem[Prz96]{PrzebindaInfinitesimal}
T.~Przebinda.
\newblock {The duality correspondence of infinitesimal characters}.
\newblock {\em { Coll. Math}}, {70}:{93--102}, {1996}.

\bibitem[Prz06]{PrzebindaLocal}
T.~Przebinda.
\newblock {Local Geometry of Orbits for an Ordinary Classical Lie Supergroup}.
\newblock {\em {Central European Journal of Mathematics}}, {4}:{449--506},
  {2006}.

\bibitem[Rao93]{RangaRaoWeil}
R.~Ranga Rao.
\newblock {On some explicit formulas in the theory of Weil representations}.
\newblock {\em {Pacific Journal of Mathematics}}, {157}:{335–--371}, {1993}.

\bibitem[Ren98]{RenardLift}
D.~Renard.
\newblock {Transfer of orbital integrals between $\Mp(2n,\R)$ and
  $\SOg(n,n+1)$}.
\newblock {\em {Duke Math J.}}, {95}:{125--450}, {1998}.

\bibitem[Ros90]{RossmannNilpotent}
W.~Rossmann.
\newblock {\em Nilpotent Orbital Integrals in a Real Semisimple Lie Algebra and
  Representations of Weyl Groups}, volume~92 of {\em Progress in Math.}
\newblock Birkh\"auser Boston, Boston, MA, 1990.

\bibitem[Ros95]{Ross95}
W.~Rossmann.
\newblock {Picard-Lefschetz theory and characters of semisimple a Lie group}.
\newblock {\em Invent. Math.}, 121:579--611, 1995.

\bibitem[{Rud}91]{RudinFunc}
{Rudin, W.}
\newblock {\em Functional Analysis}.
\newblock McGraw-Hill, Inc, 1991.

\bibitem[Sch74]{Schwarz74}
G.~Schwarz.
\newblock {Smooth functions invariant under the action of a compact Lie group}.
\newblock {\em {Topology}}, {14}:{63--68}, {1974}.

\bibitem[Ste93]{Stein}
E.M. Stein.
\newblock {\em Harmonic Analysis, Real-Variable Methods, Orthogonality, and
  Oscillatory Integrals}.
\newblock Princeton University Press, Princeton, NJ, 1993.

\bibitem[SV00]{SchmidVilonen2000}
W.~Schmid and K.~Vilonen.
\newblock Characteristic cycles and wave front cycles of representations of
  reductive lie groups.
\newblock {\em Annals of Math. (2)}, 151(3):1071–1118, 2000.

\bibitem[Tho09]{TerujiWeyl}
T.~Thomas.
\newblock {Weil representation, Weyl transform, and transfer factor}.
\newblock {\em {preprint on author's webpage}}, {2009}.

\bibitem[Vog78]{VoganGelfand}
D.~Vogan.
\newblock {Gelfand-Kirillov dimension for Harish-Chandra modules}.
\newblock {\em Invent. Math.}, 48:75--98, 1978.

\bibitem[Vog89]{Vogan89}
D.~Vogan.
\newblock {\em Associated varieties and unipotent representations.}
\newblock Harmonic analysis on reductive groups (Brunswick, ME, 1989).
  Birkh\"auser Boston Inc., Boston, MA, 1989.

\bibitem[Wal84]{Wallachholomorphic}
N.~Wallach.
\newblock The asymptotic behavior of holomorphic representations.
\newblock {\em M\'em. Soc. Math. France (N.S.)}, (15):291--305, 1984.

\bibitem[Wal93]{WallachSpringer}
N.~Wallach.
\newblock {Invariant Differential Operators on a Reductive Lie Algebra and Weyl
  Group Representations.}
\newblock {\em {J. Amer. Math. Soc.6}}, {4}:{779--816}, {1993}.

\bibitem[{Wey}46]{WeylBook}
{Weyl, H.}
\newblock {\em The classical groups}.
\newblock Princeton Univ. Press, Princeton, N.J., 1946.

\end{thebibliography}
\end{document}